\documentclass[11pt]{article}
\usepackage{amsmath, amssymb, amscd, epsfig, amsthm}
\usepackage{graphicx}
\usepackage{caption, subcaption}
\usepackage{dsfont}

\usepackage[all,cmtip]{xy}

\newtheorem{thm}{Theorem}
\newtheorem{prop}[thm]{Proposition}
\newtheorem{lem}[thm]{Lemma}

\newtheorem{cor}[thm]{Corollary}
\newtheorem{rem}[thm]{Remark}
\newtheorem{df}[thm]{Definition}

\renewcommand{\epsilon}{\varepsilon}
\renewcommand{\phi}{\varphi}
\renewcommand{\deg}{\operatorname{deg}}

\newcommand{\proj}{\operatorname{P}}
\newcommand{\BB}{\mathbb}
\newcommand{\g}{\mathfrak}

\newcommand{\im}{\operatorname{Im}}
\newcommand{\re}{\operatorname{Re}}
\newcommand{\tr}{\operatorname{Tr}}

\newcommand{\HC}{\BB H_{\BB C}}

\newcommand{\degt}{\widetilde{\operatorname{deg}}}

\newcommand{\M}{\operatorname{Mx}}
\newcommand{\Xm}{\operatorname{Xm}}

\newcommand{\DR}{\operatorname{Diag}}
\newcommand{\mult}{\operatorname{Mult}}
\newcommand{\sgn}{\operatorname{sign}}

\usepackage[OT2,T1]{fontenc}
\newcommand\textcyr[1]{{\fontencoding{OT2}\fontfamily{wncyr}\selectfont #1}}
\newcommand{\Zh}{\textit{\textcyr{Zh}}}
\newcommand{\Sh}{\textit{\textcyr{Sh}}}

\begin{document}

\title{\bf Quasi Regular Functions in Quaternionic Analysis}
\author{Igor Frenkel and Matvei Libine}
\maketitle

\begin{abstract}
We study a new class of functions that arise naturally in quaternionic
analysis, we call them ``quasi regular functions''.
Like the well-known quaternionic regular functions, these functions
provide representations of the quaternionic conformal group.
However, unlike the regular functions, the quasi regular ones do not
admit an invariant unitary structure but rather a pseudounitary equivalent.
The reproducing kernels of these functions have an especially simple form:
$(Z-W)^{-1}$.
We describe the $K$-type bases of quasi regular functions and derive
the reproducing kernel expansions.
We also show that the restrictions of the irreducible representations
formed from the quasi regular functions to the Poincar\'e group have
three irreducible components.
  
Our interest in the quasi regular functions arises from an application
to the study of conformal-invariant algebras of quaternionic functions.
We also introduce a factorization of certain intertwining operators between
tensor products of spaces of quaternionic functions.
This factorization is obtained using fermionic Fock spaces constructed
from the quasi regular functions.
\end{abstract}

\section{Introduction}

Quaternionic analysis begins with the concepts of left and right regular
functions, that is, functions on the space of quaternions $\BB H$ satisfying
respectively
\begin{align}
\nabla^+f &= \frac{\partial f}{\partial z^0} +
i\frac{\partial f}{\partial z^1} + j\frac{\partial f}{\partial z^2} +
k\frac{\partial f}{\partial z^3} =0,   \label{left-reg}  \\
g\nabla^+ &= \frac{\partial z}{\partial z^0} +
\frac{\partial g}{\partial z^1}i + \frac{\partial g}{\partial z^2}j +
\frac{\partial g}{\partial z^3}k =0,   \label{right-reg}
\end{align}
and the quaternionic counterpart of the Cauchy's integral formula for
such functions:
\begin{align*}
f(Z_0) &= \frac 1 {2\pi^2} \int_{Z \in \partial U}
\frac{(Z-Z_0)^{-1}}{\det(Z-Z_0)} \cdot Dz \cdot f(Z), \\
g(Z_0) &= \frac 1 {2\pi^2} \int_{Z \in \partial U}
g(Z) \cdot Dz \cdot \frac {(Z-Z_0)^{-1}}{\det(Z-Z_0)},
\qquad \forall Z_0 \in U,
\end{align*}
where $U \subset \BB H$ is a bounded set and the determinant is taken in
the standard realization of $\BB H$ as $2 \times 2$ complex matrices.
(These are called the Cauchy-Fueter formulas.)
In the original setup, the functions $f$ and $g$ were quaternionic valued,
but we can consider regular functions with values in the two-dimensional
left and right $\BB H$-modules that we denote by $\BB S$ and $\BB S'$
respectively.
We denote by ${\cal V}$ the space of left regular Laurent polynomials on
$\BB H^{\times}$ with values in $\BB S$. Similarly, we denote by ${\cal V}'$
the space of right regular Laurent polynomials on $\BB H^{\times}$ with values
in $\BB S'$.

In our first paper \cite{FL1} we suggested a quaternionic counterpart of the
Cauchy's formula for the second order pole.
We found out that the role of the derivative in the classical complex
analysis formula
$$
\frac{df}{dz}(z_0) = \frac 1{2\pi i} \oint \frac {f(z)\,dz}{(z-z_0)^2}
$$
in the quaternionic setting is played by a certain second order differential
operator (``Maxwell operator'')
$$  %\label{Mx-intro}
\M: (\rho'_2, {\cal W}') \to (\rho_2, {\cal W}), \qquad
\M f = \nabla f \nabla - \square f^+,
$$
where $(\rho_2, {\cal W})$ and $(\rho'_2, {\cal W}')$ are two modules over
the conformal Lie algebra consisting of the space of quaternionic
valued\footnote{Strictly speaking, we need to complexify this space:
  $\BB C \otimes \BB H [z^0,z^1,z^2,z^3, (\det Z)^{-1}]$, as stated in
  Subsection \ref{Subsect8.1}.}
Laurent polynomials on $\BB H^{\times}$
$$
\BB H [z^0,z^1,z^2,z^3, (\det Z)^{-1}]
$$
with the Lie algebra actions obtained by differentiating the
natural actions of the conformal group:
\begin{align*}
\rho_2(h): \: F(Z) \quad &\mapsto \quad \bigl( \rho_2(h)F \bigr)(Z) =
\frac {(cZ+d)^{-1}}{\det(cZ+d)} \cdot F \bigl( (aZ+b)(cZ+d)^{-1} \bigr) \cdot
\frac {(a'-Zc')^{-1}}{\det(a'-Zc')},  \\
\rho'_2(h): \: G(Z) \quad &\mapsto \quad \bigl( \rho'_2(h)G \bigr)(Z) =
\frac {a'-Zc'}{\det(a'-Zc')} \cdot G \bigl( (aZ+b)(cZ+d)^{-1} \bigr)
\cdot \frac {cZ+d}{\det(cZ+d)},
\end{align*}
where $F, G \in \BB H [z^0,z^1,z^2,z^3, (\det Z)^{-1}]$,
$h = \bigl(\begin{smallmatrix} a' & b' \\ c' & d' \end{smallmatrix}\bigr)
\in GL(2,\BB H)$ and 
$h^{-1} = \bigl(\begin{smallmatrix} a & b \\ c & d \end{smallmatrix}\bigr)$.
This operator $\M$ has a remarkable property that its kernel consists of
the solutions of the Maxwell equations in the Euclidean signature.
Moreover, the Maxwell operator intertwines the actions $\rho'_2$ and $\rho_2$.
Additionally, $(\rho_2, {\cal W})$ and $(\rho'_2, {\cal W}')$ are dual
to each other, with the invariant bilinear product expressible as an integral
over a four-dimensional contour (Subsection 4.2 in \cite{FL1}).

It has been known for a long time that the spaces of left and right regular
functions ${\cal V}$ and ${\cal V}'$ are conformally invariant.
The conformal group acts on ${\cal V}$ and ${\cal V}'$ respectively by
\begin{align*}
\pi_l(h): \: f(Z) \quad &\mapsto \quad \bigl( \pi_l(h)f \bigr)(Z) =
\frac{(cZ+d)^{-1}}{\det(cZ+d)} \cdot f \bigl( (aZ+b)(cZ+d)^{-1} \bigr),  \\
\pi_r(h): \: g(Z) \quad &\mapsto \quad \bigl( \pi_r(h)g \bigr)(Z) =
g \bigl( (a'-Zc')^{-1}(-b'+Zd') \bigr) \cdot
\frac{(a'-Zc')^{-1}}{\det(a'-Zc')},
\end{align*}
where $f \in {\cal V}$, $g \in {\cal V}'$,
$h = \bigl(\begin{smallmatrix} a' & b' \\ c' & d' \end{smallmatrix}\bigr)
\in GL(2,\BB H)$ and 
$h^{-1} = \bigl(\begin{smallmatrix} a & b \\ c & d \end{smallmatrix}\bigr)$.
The representation $(\rho_2, {\cal W})$ can be regarded as the product of the
representations $(\pi_l, {\cal V})$ and $(\pi_r, {\cal V}')$.
More precisely, the multiplication map
$$
(\pi_l, {\cal V}) \otimes (\pi_r, {\cal V}') \to (\rho_2, {\cal W}), \qquad
f \otimes g \mapsto fg,
$$
is an intertwining operator.
The duality between $(\rho_2, {\cal W})$ and $(\rho'_2, {\cal W}')$ suggests
that there should be an analogous factorization for $(\rho'_2, {\cal W}')$,
namely that there should be some natural representations $\pi'_l$ and $\pi'_r$
of the conformal group in the spaces ${\cal U}$ and ${\cal U}'$ of
$\BB S$ and $\BB S'$ valued functions such that the multiplication map
\begin{equation}  \label{UU'->W'}
(\pi'_l, {\cal U}) \otimes (\pi'_r, {\cal U}') \to (\rho'_2, {\cal W}'), \qquad
f \otimes g \mapsto fg,
\end{equation}
is an intertwining operator.
It turns out that such spaces ${\cal U}$ and ${\cal U}'$ indeed exist and
their study is the main subject of this paper.
First of all, the intertwining operator \eqref{UU'->W'} suggests the
following actions of the conformal group:
\begin{align}
\bigl( \pi'_l(h)f \bigr)(Z) &=
\frac{a'-Zc'}{\det(a'-Zc')} \cdot f \bigl( (a'-Zc')^{-1}(-b'+Zd') \bigr),
\label{pi'_l-intro}
\\
\bigl( \pi'_r(h)g \bigr)(Z) &=
g \bigl( (aZ+b)(cZ+d)^{-1} \bigr) \cdot \frac{cZ+d}{\det(cZ+d)},
\label{pi'_r-intro}
\end{align}
where
$h = \bigl( \begin{smallmatrix} a' & b' \\ c' & d' \end{smallmatrix} \bigr)
\in GL(2,\BB H)$ and
$h^{-1} = \bigl( \begin{smallmatrix} a & b \\ c & d \end{smallmatrix} \bigr)$.
These actions preserve the spaces of what we call the
left and right quasi anti regular functions with values in $\BB S$ and
$\BB S'$, that is, functions satisfying respectively
$$
\nabla \square f =0 \qquad \text{and} \qquad
(\square g) \overleftarrow{\nabla} =0,
$$
where
$$
\square = \nabla\nabla^+ = \nabla^+\nabla
= \frac{\partial^2}{(\partial z^0)^2} + \frac{\partial^2}{(\partial z^1)^2}
+ \frac{\partial^2}{(\partial z^2)^2} + \frac{\partial^2}{(\partial z^3)^2}.
$$
Note that the requirement that \eqref{UU'->W'} is an intertwining operator
and the actions \eqref{pi'_l-intro}-\eqref{pi'_r-intro} dictate that we use
quasi anti regular functions (functions annihilated by $\nabla \square$).
If one wants to consider quasi regular functions (functions annihilated by
$\nabla^+ \square$), these can be obtained by applying quaternionic conjugation.

It turns out that the quasi regular functions have already been studied, and
an analogue of the Cauchy-Fueter formulas for these functions is known
\cite{R, LR}.
Let $U \subset \BB H$ be an open bounded set with (piecewise) smooth boundary
$\partial U$, and let $\overrightarrow{n}$ denote the outward pointing normal
unit vector to $\partial U$, then we have
\begin{equation*}  %\label{Cauchy-Fueter-left-intro}
\begin{split}
f(Z_0) &= \frac1{2\pi^2} \int_{Z \in \partial U}
\frac{Z-Z_0}{\det(Z-Z_0)^2} \cdot (Dz)^+ \cdot f(Z)  \\
&+ \frac1{8\pi^2} \int_{Z \in \partial U}
\frac{\partial}{\partial \overrightarrow{n}} \frac{Z-Z_0}{\det(Z-Z_0)}
\cdot \nabla f(Z) \,dS  \\
&- \frac1{8\pi^2} \int_{Z \in \partial U} \frac{Z-Z_0}{\det(Z-Z_0)} \cdot
\frac{\partial}{\partial \overrightarrow{n}} \nabla f(Z) \,dS,
\end{split}
\end{equation*}
\begin{equation*}   %\label{Cauchy-Fueter-right-intro}
\begin{split}
g(Z_0) &= \frac1{2\pi^2} \int_{Z \in \partial U} g(Z) \cdot (Dz)^+ \cdot
\frac{Z-Z_0}{\det(Z-Z_0)^2}  \\
&+ \frac1{8\pi^2} \int_{Z \in \partial U} (g \overleftarrow{\nabla})(Z) \cdot
\frac{\partial}{\partial \overrightarrow{n}} \frac{Z-Z_0}{\det(Z-Z_0)}\,dS  \\
&- \frac1{8\pi^2} \int_{Z \in \partial U}
\frac{\partial}{\partial \overrightarrow{n}}
(g \overleftarrow{\nabla})(Z) \cdot \frac{Z-Z_0}{\det(Z-Z_0)} \,dS,
\end{split}
\end{equation*}
where $f$ (respectively $g$) is a quasi left (respectively right) anti regular
function and $Z_0 \in U$.

We prove that the spaces ${\cal U}$ and ${\cal U}'$ of left and right
(anti) regular Laurent polynomial functions on $\BB H^{\times}$ are
invariant under the action of the conformal Lie algebra.
%$\pi'_l$ and $\pi'_r$ respectively.
As in the case of the regular functions, we obtain a decomposition into
irreducible representations
$$
{\cal U} = {\cal U}^+ \oplus {\cal U}^- \qquad \text{and} \qquad
{\cal U}' = {\cal U}'^+ \oplus {\cal U}'^-.
$$
We show that, unlike ${\cal V}^{\pm}$ and ${\cal V}'^{\pm}$,
these representations do not have $\mathfrak{u}(2,2)$-invariant
unitary structures, but they still admit invariant pseudounitary forms.
It is a well-known fact that the restrictions of the representations
${\cal V}^{\pm}$, ${\cal V}'^{\pm}$ to the Poincar\'e group remain irreducible
\cite{MT}. In the case of the representations ${\cal U}^{\pm}$, ${\cal U}'^{\pm}$
it is no longer true, but we still obtain a sufficiently simple decomposition
of these representations into three irreducible components with respect to the
Poincar\'e group.

We describe the $K$-type bases of ${\cal U}$ and ${\cal U}'$, then decompose
the reproducing kernel with respect to these bases.
These kernel expansions are quasi regular analogues of the matrix coefficient
expansions of the Cauchy-Fueter kernel
$$
\frac{(Z-W)^{-1}}{\det(Z-W)}, \qquad Z, W \in \BB H,
$$
given in \cite{FL1}.
But now the reproducing kernel has an especially simple form
-- it is the quaternionic conjugate of
$$
(Z-W)^{-1}, \qquad Z, W \in \BB H,
$$
and looks exactly like its complex counterpart!
As in some of our previous papers, we use extensively the $K$-type
decompositions of the spaces of regular and quasi regular functions
as well as the spaces ${\cal W}$ and ${\cal W}'$ of quaternionic valued
functions. They all should be regarded as quaternionic analogues of the
spaces of Laurent polynomials, though the algebra structure of ${\cal W}$
and ${\cal W}'$ is much less straightforward.
In any case, the matrix coefficients $t^l_{n \,\underline{m}}(Z)$ of $SU(2)$
extended to the quaternions are a very convenient substitute for the Laurent
monomials.

In our first paper \cite{FL1}, studying regular quaternionic functions
also led us to analysis of harmonic functions on the space of quaternions;
such harmonic functions can be regarded as a scalar analogue of regular
functions.
Similarly, the quasi regular functions have a natural scalar analogue.
This analogue could be called quasi harmonic functions, but already has a
name of biharmonic functions and has been studied by Fueter \cite{F1}
and others.
The biharmonic functions also appeared in our study of the composition series
of the spaces ${\cal W}$ and  ${\cal W}'$ \cite{ATMP}.
In this paper we study the spaces of biharmonic functions as scalar versions
of quasi regular functions. In particular, we study their $K$-types and the
invariant bilinear pairing.
Like the spaces of quasi regular functions, the spaces of biharmonic
functions do not have $\mathfrak{u}(2,2)$-invariant unitary structures,
but instead they admit invariant pseudounitary structures.

One of the main reasons for our interest in quasi regular functions is
their relation to conformal-invariant quaternionic algebra structures.
The reproducing kernel for the quasi regular functions allows us to define
intertwining operators for the tensor powers of ${\cal W}$, such as
$$
{\cal W} \otimes {\cal W} \otimes {\cal W} \to \BB C
$$
given by
\begin{multline}  \label{6easy-W-3form-intro}
 F \otimes G \otimes H  \mapsto \Bigl( \frac{i}{2\pi^3} \Bigr)^3
\iiint_{\begin{smallmatrix} Z_1 \in U(2)_{R_1} \\ Z_2 \in U(2)_{R_2} \\
Z_3 \in U(2)_{R_3} \end{smallmatrix}}
\tr \biggl[ F(Z_1) \cdot \frac1{(Z_1-Z_2)^+} \cdot G(Z_2) \\
\times \frac1{(Z_2-Z_3)^+} \cdot H(Z_3) \cdot
\frac1{(Z_3-Z_1)^+} \biggr] \,dV_{Z_1}dV_{Z_2}dV_{Z_3},
\end{multline}
where the contours of integration
$U(2)_{R_i} = \{ Z \in \BB H \otimes \BB C ;\: ZZ^* =R_i^2 \}$, $i=1,2,3$,
can be arranged in any of the six possible ways.
Similar intertwiners exist for ${\cal W}'$ as well as
$\Zh = \BB C [z^0,z^1,z^2,z^3, (\det Z)^{-1}]$. In fact,
\begin{align*}
\text{the kernel for ${\cal W}'$ is }
&\frac{(Z-W)^{-1}}{\det(Z-W)} \text{ of order $-3$, and}  \\
\text{the kernel for $\Zh$ is }
&\frac1{\det(Z-W)} \text{ of order $-2$.}
\end{align*}
These intertwiners are closely related to the definition of
multiplication operations on $\Zh$, ${\cal W}$ and ${\cal W}'$
that we constructed in \cite{ATMP}.
However, in order to construct nontrivial multiplications on
$\Zh$, ${\cal W}$ and ${\cal W}'$ that intertwine the conformal group
actions, one needs more involved cycles of integration than those in
\eqref{6easy-W-3form-intro}.
Furthermore, we believe that a certain multiplication on ${\cal W}$
based on the simplest kernel $(Z-W)^{-1}$ of order $-1$ is the best candidate
for a quaternionic algebra of functions.

The integrals such as \eqref{6easy-W-3form-intro} also appear as correlation
functions in Clifford algebra (or fermionic) representations of central
extensions of loop algebras, where the reproducing kernel is
$(z-w)^{-1}$, $z, w \in \BB C$.
One can repeat some steps of this construction for the Clifford algebras based
on ${\cal V}$ and ${\cal V}'$ and also ${\cal U}$ and ${\cal U}'$.
This will yield factorizations of the intertwiners for
${\cal W}'$ and ${\cal W}$ respectively.
The first one -- based on ${\cal V}$ and ${\cal V}'$ -- is a version of
the second quantization of the massless Dirac equation.
The second one -- based on ${\cal U}$ and ${\cal U}'$ -- uses a non-unitary
structure and appears to be new.
We hope that it can be used for a construction of a representation of a
quaternionic algebra based on ${\cal W}$.
In this paper we study the intertwining operators from
${\cal U} \otimes {\cal U}'$ to ${\cal W}'$ and the dual intertwining
operators from ${\cal W}$ to a certain holomorphic completion of
${\cal U} \otimes {\cal U}'$ in detail.

The paper is organized as follows.
In Section \ref{Sect2} we study the space of biharmonic functions,
which can be viewed as a scalar counterpart of quasi (anti) regular functions.
We derive the reproducing formulas (Theorem \ref{BH-repro-thm}),
invariant forms (Proposition \ref{biharm-pairing-prop} and Theorem
\ref{BH-unitary-thm}) and the $K$-types (Proposition \ref{K-type-biharm_prop}).
In Section \ref{Sect3} we begin to study the spaces of quasi anti regular
functions in detail. We first consider the conformal transformations
(Theorem \ref{qr-action-thm}) and then produce various examples of
quasi anti regular functions (Corollary \ref{series-cor}).
We conclude this section with the reproducing formula (Theorem
\ref{QAR-repro-thm}) originally derived in a more general context in
\cite{R, LR}.
In Section \ref{Sect4} we study the $K$-type basis of the spaces of quasi
anti regular functions (Propositions \ref{K-types-prop},
\ref{K-typebasis+_prop}, \ref{K-typebasis-_prop})
and the action of conformal algebra in this basis.
We also show that the spaces quasi (anti) regular functions ${\cal U}$ and
its dual ${\cal U}'$ decompose into two irreducible components each:
${\cal U} = {\cal U}^+ \oplus {\cal U}^-$ and
${\cal U}' = {\cal U}'^+ \oplus {\cal U}'^-$ (Proposition \ref{irred-prop}).
Then in Section \ref{Sect5} we consider how these irreducible components
${\cal U}^{\pm}$ and ${\cal U}'^{\pm}$ decompose further after restricting to
various subgroups of the conformal group.
In particular, we analyze the case of the Poincar\'e group and establish that
each of the ${\cal U}^{\pm}$, ${\cal U}'^{\pm}$ decomposes into three simple
components that can be nicely described using basic differential operators
(Theorem \ref{Poincare-restrict-thm}).
In Section \ref{Sect6} we construct an invariant bilinear pairing between
${\cal U}$ and ${\cal U}'$ (Proposition \ref{bilinear-pairing-prop}),
then derive invariant pseudounitary structures on these spaces
(Theorem \ref{unitary-thm}).
Section \ref{Sect7} is dedicated to the expansion of the reproducing kernel
for the quasi (anti) regular functions (Theorem \ref{kernel-exp-thm}).
In Section \ref{Sect8} we study the fundamental intertwining operators
\eqref{fork'} from the space of quaternionic functions ${\cal W}$ to
the tensor product of the spaces of quasi anti regular functions
${\cal U} \otimes {\cal U}'$ (Theorems \ref{J'-thm}, \ref{J+-_thm}
and \ref{Xm^0-operator-thm}).
In particular, we evaluate the intertwining map \eqref{fork'} on an
element $N(Z)^{-2} \cdot Z^+$ (Proposition \ref{J-gen-prop}).
This element $N(Z)^{-2} \cdot Z^+$ spans the trivial one-dimensional
component of ${\cal W}$ and generates a larger indecomposable
subrepresentation of ${\cal W}$ that we explicitly describe.
Section \ref{IntertwiningOper-section} is dedicated to the study of
the conformal-invariant algebra structures on the spaces of quaternionic
functions $\Zh$, ${\cal W}$ and ${\cal W}'$ using the theory of regular
and quasi regular functions developed before.
We start with a more tractable example of scalar quaternionic functions $\Zh$
(Theorem \ref{Zh-mult-thm}), then proceed to the main examples,
namely the spaces ${\cal W}$ and ${\cal W}'$.
Finally, in Section \ref{Sect10} we introduce a factorization of intertwining
operators between the tensor products of the spaces of quaternionic functions
${\cal W}$ and ${\cal W}'$.
We use a fermionic Fock space representation of an infinite dimensional
Clifford algebra associated with the spaces of regular functions
${\cal V}$, ${\cal V}'$ for ${\cal W}'$ and quasi regular functions
${\cal U}$, ${\cal U}'$ for ${\cal W}$.

I.~Frenkel thanks Toshiyuki Kobayashi for discussion and correspondence
on constructions presented in Subsection \ref{IntertwiningOper-subsection}.

\section{Biharmonic Functions}  \label{Sect2}

\subsection{Definitions and Conformal Properties}

Biharmonic functions in the context of quaternionic analysis were studied by
Fueter \cite{F1}. In this subsection, we briefly summarize the results on
biharmonic functions given in \cite{ATMP}.

We continue to use notations established in \cite{FL1,desitter,ATMP}
and direct the reader to Section 2 of \cite{desitter} for a summary.
In particular, $e_0$, $e_1$, $e_2$, $e_3$ denote the units of the classical
quaternions $\BB H$ corresponding to the more familiar $1$, $i$, $j$, $k$
(we reserve the symbol $i$ for $\sqrt{-1} \in \BB C$).
Thus $\BB H$ is an algebra over $\BB R$ generated by $e_0$, $e_1$, $e_2$, $e_3$,
and the multiplicative structure is determined by the rules
$$
e_0 e_i = e_i e_0 = e_i, \qquad
(e_i)^2 = e_1e_2e_3 = - e_0, \qquad
e_ie_j=-e_ie_j, \qquad 1 \le i< j \le 3,
$$
and the fact that $\BB H$ is a division ring.
Next we consider the algebra of complexified quaternions
(also known as biquaternions) $\HC = \BB C \otimes_{\BB R} \BB H$ and
write elements of $\HC$ as
$$
Z = z^0e_0 + z^1e_1 + z^2e_2 + z^3e_3, \qquad z^0,z^1,z^2,z^3 \in \BB C,
$$
so that $Z \in \BB H$ if and only if $z^0,z^1,z^2,z^3 \in \BB R$:
$$
\BB H = \{ X = x^0e_0 + x^1e_1 + x^2e_2 + x^3e_3; \: x^0,x^1,x^2,x^3 \in \BB R \}.
$$
For $Z = z^0e_0 + z^1e_1 + z^2e_2 + z^3e_3 \in \HC$, we denote by $Z^+$ its
{\em quaternionic conjugate}:
\[
Z^+ = z^0e_0 - z^1e_1 - z^2e_2 - z^3e_3
\]
(the bar notation is reserved for complex conjugation). And
\[
N(Z) = ZZ^+ = Z^+Z  \in \BB C.
\]
As an algebra, $\HC$ can naturally be identified with $2 \times 2$ matrices
with complex entries, and we frequently use this identification.
We denote by $\BB S$ (respectively $\BB S'$)
the irreducible 2-dimensional left (respectively right) $\HC$-module,
as described in Subsection 2.3 of \cite{FL1}.
The spaces $\BB S$ and $\BB S'$ can be realized as respectively
columns and rows of complex numbers. Then
\begin{equation*}  %\label{SotimesS}
\BB S \otimes \BB S' \simeq \HC.
\end{equation*}

Recall the differential operators
\begin{align}
\nabla^+ &= e_0 \frac{\partial}{\partial z^0}
+ e_1 \frac{\partial}{\partial z^1}
+ e_2 \frac{\partial}{\partial z^2}
+ e_3 \frac{\partial}{\partial z^3}, \\
\nabla &= e_0 \frac{\partial}{\partial z^0}
- e_1 \frac{\partial}{\partial z^1}
- e_2 \frac{\partial}{\partial z^2}
- e_3 \frac{\partial}{\partial z^3}, \\
\square = \nabla\nabla^+ = \nabla^+\nabla
&= \frac{\partial^2}{(\partial z^0)^2} + \frac{\partial^2}{(\partial z^1)^2}
+ \frac{\partial^2}{(\partial z^2)^2} + \frac{\partial^2}{(\partial z^3)^2}.
\label{Laplacian}
\end{align}

We frequently use the matrix coefficients $t^l_{n\,\underline{m}}(Z)$'s
of $SU(2)$ described by equation (27) of \cite{FL1} (cf. \cite{V}):
\begin{equation}  \label{t}
t^l_{n\,\underline{m}}(Z) = \frac 1{2\pi i}
\oint (sz_{11}+z_{21})^{l-m} (sz_{12}+z_{22})^{l+m} s^{-l+n} \,\frac{ds}s,
\qquad
\begin{smallmatrix} l = 0, \frac12, 1, \frac32, \dots, \\ m,n \in \BB Z +l, \\
 -l \le m,n \le l, \end{smallmatrix}
\end{equation}
where $Z=\bigl(\begin{smallmatrix} z_{11} & z_{12} \\
z_{21} & z_{22} \end{smallmatrix}\bigr) \in \HC$,
the integral is taken over a loop in $\BB C$ going once around the origin
in the counterclockwise direction.
We regard these functions as polynomials on $\HC$, and we frequently use them
to construct $K$-type bases of various spaces as in, for example,
\cite{FL1}, \cite{desitter}.

We consider a space
$$
\Sh' = \bigl\{\text{$\BB C$-valued polynomial functions on
$\HC^{\times}$}\bigr\}
= \BB C[z_{11},z_{12},z_{21},z_{22}, N(Z)^{-1}].
$$
The Lie algebra $\g{gl}(2,\HC)$ acts on this space by differentiating
the following group action:
$$
\rho'(h): \: f(Z) \quad \mapsto \quad \bigl( \rho'(h)f \bigr)(Z) =
f \bigl( (aZ+b)(cZ+d)^{-1} \bigr), \qquad f \in \Sh',
$$
$h \in GL(2,\HC)$ with 
$h^{-1} = \bigl(\begin{smallmatrix} a & b \\ c & d \end{smallmatrix}\bigr)$.
%This action is spelled out in Lemma 24 in \cite{ATMP}.
The space of biharmonic functions is a $\mathfrak{gl}(2,\HC)$-invariant
subspace of $(\Sh',\rho')$.
We have a description of biharmonic functions in terms of harmonic ones.

\begin{prop} [Proposition 34 in \cite{ATMP}] \label{biharmonic-prop}
We have:
$$
\{ f \in \Sh' ;\: (\square \circ \square) f=0 \} = {\cal BH}^+ + {\cal BH}^-,
$$
where
\begin{align*}
{\cal BH}^+ &= \BB C\text{-span of }
\bigl\{ N(Z)^k \cdot t^l_{n\,\underline{m}}(Z);\: 0 \le k \le 1 \bigr\}, \\
{\cal BH}^- &= \BB C\text{-span of }
\bigl\{ N(Z)^k \cdot t^l_{n\,\underline{m}}(Z);\: -1 \le 2l+k \le 0 \bigr\}.
\end{align*}
In other words, a function $f \in \Sh'$ is biharmonic if and only if it can
be written as
$$
f(Z) = h_0(Z)+h_1(Z) \cdot N(Z)
$$
with $h_0$ and $h_1$ harmonic.
\end{prop}

We introduce a notation for the space of biharmonic functions:
$$
{\cal BH} = {\cal BH}^+ + {\cal BH}^-
= \{ f \in \Sh' ;\: (\square \circ \square) f=0 \}.
$$
Theorem 32 in \cite{ATMP} describes all the $\mathfrak{gl}(2,\HC)$-invariant
subspaces of $(\rho',\Sh')$.
As an immediate consequence, we obtain the invariant subspaces of the space of
biharmonic functions into irreducible components:

\begin{thm}  \label{biharm-irred-thm}
The only proper $\mathfrak{gl}(2,\HC)$-invariant subspaces of the space of
biharmonic functions $(\rho',{\cal BH})$ are
\begin{align*}
{\cal I}'_0 &= \BB C = \BB C\text{-span of }
\bigl\{ N(Z)^0 \cdot t^0_{0\,\underline{0}}(Z) \bigr\}, \\
{\cal BH}^+ &= \BB C\text{-span of }
\bigl\{ N(Z)^k \cdot t^l_{n\,\underline{m}}(Z);\: 0 \le k \le 1 \bigr\}, \\
{\cal BH}^- &= \BB C\text{-span of }
\bigl\{ N(Z)^k \cdot t^l_{n\,\underline{m}}(Z);\: -1 \le 2l+k \le 0 \bigr\}.
\end{align*}
The irreducible components of $(\rho',{\cal BH})$ are the trivial
subrepresentation $(\rho', {\cal I}'_0)$ and the quotients
$$
(\rho', {\cal BH}^+/{\cal I}'_0), \quad (\rho', {\cal BH}^-/{\cal I}'_0).
$$
\end{thm}

\subsection{$K$-Type Basis of Biharmonic Functions}  \label{K-type-biharmonic}

Using notations
$$
\partial = \begin{pmatrix} \partial_{11} & \partial_{21} \\
\partial_{12} & \partial_{22} \end{pmatrix} = \tfrac 12 \nabla, \qquad
\partial^+ = \begin{pmatrix} \partial_{22} & -\partial_{21} \\
-\partial_{12} & \partial_{11} \end{pmatrix} = \tfrac 12 \nabla^+,
\qquad \partial_{ij} = \frac{\partial}{\partial z_{ij}},
$$
we can describe the Lie algebra action on ${\cal BH}$ (and $\Sh'$):

\begin{lem} [Lemma 24 in \cite{ATMP}]
The Lie algebra action $\rho'$ of $\mathfrak{gl}(2,\HC)$ on ${\cal BH}$
is given by
\begin{align*}
\rho' \bigl( \begin{smallmatrix} A & 0 \\ 0 & 0 \end{smallmatrix} \bigr) &:
f(Z) \mapsto - \tr (AZ \partial) f,  \\
\rho' \bigl( \begin{smallmatrix} 0 & B \\ 0 & 0 \end{smallmatrix} \bigr) &:
f(Z) \mapsto - \tr (B \partial) f,  \\
\rho' \bigl( \begin{smallmatrix} 0 & 0 \\ C & 0 \end{smallmatrix} \bigr) &:
f(Z) \mapsto \tr (ZCZ \partial) f,  \\
\rho' \bigl( \begin{smallmatrix} 0 & 0 \\ 0 & D \end{smallmatrix} \bigr) &:
f(Z) \mapsto \tr (ZD \partial) f.
\end{align*}
\end{lem}

We realize $SU(2) \times SU(2)$ as diagonal elements of $GL(2,\HC)$:
\begin{equation*}
SU(2) \times SU(2) = \left\{
\left(\begin{smallmatrix} a & 0 \\ 0 & d \end{smallmatrix}\right)
\in GL(2,\HC);\: a,d \in \BB H, N(a)=N(d)=1 \right\}.
\end{equation*}
Similarly, we realize $\mathfrak{sl}(2,\BB C) \times \mathfrak{sl}(2,\BB C)$
as diagonal elements of $\mathfrak{gl}(2,\HC)$:
\begin{equation}  \label{sl2xsl2}
\mathfrak{sl}(2,\BB C) \times \mathfrak{sl}(2,\BB C) = \left\{
\left(\begin{smallmatrix} A & 0 \\ 0 & D \end{smallmatrix}\right)
\in \mathfrak{gl}(2,\HC);\: A,D \in \HC, \re(A)=\re(D)=0 \right\}.
\end{equation}
From Proposition \ref{biharmonic-prop} we obtain
a description of the $K$-types of $(\rho', {\cal BH})$. Let
$$
{\cal BH}(d) = \{ f \in {\cal BH};\: \text{$f$ is homogeneous of degree $d$} \},
\qquad d \in \BB Z.
$$
Recall that we denote by $(\tau_{\frac12},\BB S)$ the tautological 2-dimensional
representation of $\HC^{\times}$.
Then, for $l=0,\frac12,1,\frac32, \dots$, we denote by $(\tau_l,V_l)$ the
$2l$-th symmetric power product of $(\tau_{\frac12},\BB S)$.
(In particular, $(\tau_0,V_0)$ is the trivial one-dimensional representation.)
Thus, each $(\tau_l,V_l)$ is an irreducible representation of $\HC^{\times}$
of dimension $2l+1$.
A concrete realization of $(\tau_l,V_l)$ as well as an isomorphism
$V_l \simeq \BB C^{2l+1}$ suitable for our purposes are described in
Subsection 2.5 of \cite{FL1}.

\begin{prop}  \label{K-type-biharm_prop}
Each ${\cal BH}(d)$ is invariant under the $\rho'$ action restricted to
$\mathfrak{sl}(2,\BB C) \times \mathfrak{sl}(2,\BB C)$, and we have the
following decomposition into irreducible components:
\begin{align*}
{\cal BH}(2l) &= \bigl( V_l \boxtimes V_l \bigr) \oplus
\bigl( V_{l-1} \boxtimes V_{l-1} \bigr),  \\
{\cal BH}(-2l) &= \bigl( V_l \boxtimes V_l \bigr) \oplus
\bigl( V_{l-1} \boxtimes V_{l-1} \bigr),
\end{align*}
$l=0,\frac12,1,\frac32,\dots$.
In particular, $\dim_{\BB C} {\cal BH}(0) = 1$ and
$\dim_{\BB C} {\cal BH}(d) = 2d^2+2$ for $d \ne 0$.

The functions
$$
\phi^{(1)}_{l,m,n}(Z) = t^l_{n \,\underline{m}}(Z),
\qquad m,n =-l ,-l+1,\dots,l,
$$
$$
\phi^{(2)}_{l,m,n}(Z) = N(Z) \cdot t^{l-1}_{n \,\underline{m}}(Z),
\qquad m,n  =-l+1 ,-l+2,\dots,l-1,
$$
span respectively the $V_l \boxtimes V_l$ and $V_{l-1} \boxtimes V_{l-1}$
components of ${\cal BH}(2l)$, $l=0,\frac12,1,\frac32,\dots$.

The functions
$$
\tilde \phi^{(1)}_{l,m,n}(Z) = t^l_{m \,\underline{n}}(Z^{-1}),
\qquad m,n =-l ,-l+1,\dots,l,
$$
$$
\tilde \phi^{(2)}_{l,m,n}(Z) = N(Z)^{-1} \cdot t^{l-1}_{m \,\underline{n}}(Z^{-1}),
\qquad m,n  =-l+1 ,-l+2,\dots,l-1,
$$
span respectively the $V_l \boxtimes V_l$ and $V_{l-1} \boxtimes V_{l-1}$
components of ${\cal BH}(-2l)$, $l=0,\frac12,1,\frac32,\dots$.
\end{prop}

(Note that $\tilde \phi^{(1)}_{0,0,0}(Z) = \phi^{(1)}_{0,0,0}(Z) = 1$.)

The element
$\bigl( \begin{smallmatrix} 0 & 1 \\ 1 & 0 \end{smallmatrix} \bigr)
\in GL(2,\HC)$ acts on biharmonic functions by an inversion
$$
\rho' \bigl( \begin{smallmatrix} 0 & 1 \\ 1 & 0 \end{smallmatrix} \bigr): \:
\phi(Z) \mapsto \phi(Z^{-1}).
$$

\begin{lem}  \label{inversion-bh-lem}
The inversion
$\rho' \bigl( \begin{smallmatrix} 0 & 1 \\ 1 & 0 \end{smallmatrix} \bigr)$
acts on the basis of biharmonic functions by sending
$$
\phi^{(1)}_{l,m,n}(Z) \longleftrightarrow \tilde \phi^{(1)}_{l,n,m}(Z), \qquad
\phi^{(2)}_{l,m,n}(Z) \longleftrightarrow \tilde \phi^{(2)}_{l,n,m}(Z).
$$
\end{lem}

Next, we want to find the actions of
$\rho' \bigl( \begin{smallmatrix} 0 & B \\ C & 0 \end{smallmatrix} \bigr)$,
$B, C \in \HC$, on the basis vectors given in
Proposition \ref{K-type-biharm_prop}.
(Note that the actions of
$\rho' \bigl( \begin{smallmatrix} A & 0 \\ 0 & D \end{smallmatrix} \bigr)$,
$A, D \in \HC$, were effectively described in
Proposition \ref{K-type-biharm_prop}.)

By direct computation using Lemma 22 in \cite{FL1}:
\begin{equation}  \label{dt}
\partial t^l_{n\,\underline{m}}(Z) =
\begin{pmatrix} (l-m) t^{l-\frac12}_{n+\frac12\,\underline{m+\frac12}}(Z) &
(l-m) t^{l-\frac12}_{n-\frac12\,\underline{m+\frac12}}(Z) \\
(l+m) t^{l-\frac12}_{n+\frac12\,\underline{m-\frac12}}(Z) &
(l+m) t^{l-\frac12}_{n-\frac12\,\underline{m-\frac12}}(Z) \end{pmatrix},
\end{equation}
equation (37) in \cite{ATMP}:
\begin{multline}  \label{dt-inverse}
\partial \bigl( N(Z)^{-1} \cdot t^l_{n\,\underline{m}}(Z^{-1}) \bigr)  \\
= - \frac1{N(Z)} \begin{pmatrix}
(l-n+1) t^{l+\frac12}_{n-\frac12\,\underline{m-\frac12}}(Z^{-1}) &
(l-n+1) t^{l+\frac12}_{n-\frac12\,\underline{m+\frac12}}(Z^{-1})  \\
(l+n+1) t^{l+\frac12}_{n+\frac12\,\underline{m-\frac12}}(Z^{-1}) &
(l+n+1) t^{l+\frac12}_{n+\frac12\,\underline{m+\frac12}}(Z^{-1}) \end{pmatrix},
\end{multline}
equation (34) in \cite{ATMP}:
\begin{multline}  \label{Zt-identity}
Z \cdot t^l_{n\,\underline{m}}(Z) = \frac1{2l+1} \begin{pmatrix}
(l-n+1) t^{l+\frac12}_{n-\frac12\,\underline{m-\frac12}}(Z) &
(l-n+1) t^{l+\frac12}_{n-\frac12\,\underline{m+\frac12}}(Z) \\
(l+n+1) t^{l+\frac12}_{n+\frac12\,\underline{m-\frac12}}(Z) &
(l+n+1) t^{l+\frac12}_{n+\frac12\,\underline{m+\frac12}}(Z) \end{pmatrix}  \\
+ \frac{N(Z)}{2l+1} \cdot \begin{pmatrix}
(l+m) t^{l-\frac12}_{n-\frac12\,\underline{m-\frac12}}(Z) &
-(l-m) t^{l-\frac12}_{n-\frac12\,\underline{m+\frac12}}(Z) \\
-(l+m) t^{l-\frac12}_{n+\frac12\,\underline{m-\frac12}}(Z) &
(l-m) t^{l-\frac12}_{n+\frac12\,\underline{m+\frac12}}(Z) \end{pmatrix},
\end{multline}
equations (27)-(28) in \cite{desitter} and identities
\begin{equation}  \label{Ct}
Z \bigl( \partial t^l_{n\,\underline{m}}(Z) \bigr) Z + Z t^l_{n\,\underline{m}}(Z)
= \begin{pmatrix} (l-n+1) t^{l+\frac12}_{n-\frac12\,\underline{m-\frac12}}(Z) &
(l-n+1) t^{l+\frac12}_{n-\frac12\,\underline{m+\frac12}}(Z) \\
(l+n+1) t^{l+\frac12}_{n+\frac12\,\underline{m-\frac12}}(Z) &
(l+n+1) t^{l+\frac12}_{n+\frac12\,\underline{m+\frac12}}(Z) \end{pmatrix},
\end{equation}
\begin{multline}  \label{Ct-inverse}
Z\Bigl(\partial\bigl( N(Z)^{-1} \cdot t^l_{n\,\underline{m}}(Z^{-1}) \bigr)\Bigr)Z
+ \frac{Z}{N(Z)} \cdot t^l_{n\,\underline{m}}(Z^{-1})  \\
= -\frac1{N(Z)} \begin{pmatrix}
(l-m) t^{l-\frac12}_{n+\frac12\,\underline{m+\frac12}}(Z^{-1}) &
(l-m) t^{l-\frac12}_{n-\frac12\,\underline{m+\frac12}}(Z^{-1}) \\
(l+m) t^{l-\frac12}_{n+\frac12\,\underline{m-\frac12}}(Z^{-1}) &
(l+m) t^{l-\frac12}_{n-\frac12\,\underline{m-\frac12}}(Z^{-1}) \end{pmatrix},
\end{multline}
we find:
{\small
$$
\rho' \bigl( \begin{smallmatrix} 0 & B \\ 0 & 0 \end{smallmatrix} \bigr)
\phi^{(1)}_{l,m,n}(Z) = -\tr \left( B \begin{pmatrix}
(l-m) \phi^{(1)}_{l-\frac12,m+\frac12,n+\frac12}(Z) &
(l-m) \phi^{(1)}_{l-\frac12,m+\frac12,n-\frac12}(Z)  \\
(l+m) \phi^{(1)}_{l-\frac12,m-\frac12,n+\frac12}(Z) &
(l+m) \phi^{(1)}_{l-\frac12,m-\frac12,n-\frac12}(Z)
\end{pmatrix} \right),
$$
\begin{multline*}
\rho' \bigl( \begin{smallmatrix} 0 & B \\ 0 & 0 \end{smallmatrix} \bigr)
\phi^{(2)}_{l,m,n}(Z) = - \frac1{2l-1} \tr \left( B \begin{pmatrix}
(l+n) \phi^{(1)}_{l-\frac12,m+\frac12,n+\frac12}(Z) &
-(l-n) \phi^{(1)}_{l-\frac12,m+\frac12,n-\frac12}(Z)  \\
-(l+n) \phi^{(1)}_{l-\frac12,m-\frac12,n+\frac12}(Z) &
(l-n) \phi^{(1)}_{l-\frac12,m-\frac12,n-\frac12}(Z)
\end{pmatrix} \right)  \\
- \frac{2l}{2l-1} \tr \left( B \begin{pmatrix}
(l-m-1) \phi^{(2)}_{l-\frac12,m+\frac12,n+\frac12}(Z) &
(l-m-1) \phi^{(2)}_{l-\frac12,m+\frac12,n-\frac12}(Z)  \\
(l+m-1) \phi^{(2)}_{l-\frac12,m-\frac12,n+\frac12}(Z) &
(l+m-1) \phi^{(2)}_{l-\frac12,m-\frac12,n-\frac12}(Z)
\end{pmatrix} \right),
\end{multline*}
\begin{multline*}
\rho' \bigl( \begin{smallmatrix} 0 & 0 \\ C & 0 \end{smallmatrix} \bigr)
\phi^{(1)}_{l,m,n}(Z) = \frac{2l}{2l+1}\tr \left( C \begin{pmatrix}
(l-n+1) \phi^{(1)}_{l+\frac12,m-\frac12,n-\frac12}(Z) &
(l-n+1) \phi^{(1)}_{l+\frac12,m+\frac12,n-\frac12}(Z)  \\
(l+n+1) \phi^{(1)}_{l+\frac12,m-\frac12,n+\frac12}(Z) &
(l+n+1) \phi^{(1)}_{l+\frac12,m+\frac12,n+\frac12}(Z) \end{pmatrix} \right)  \\
+ \frac1{2l+1} \tr \left( C \begin{pmatrix}
-(l+m) \phi^{(2)}_{l+\frac12,m-\frac12,n-\frac12}(Z) &
(l-m) \phi^{(2)}_{l+\frac12,m+\frac12,n-\frac12}(Z)  \\
(l+m) \phi^{(2)}_{l+\frac12,m-\frac12,n+\frac12}(Z) &
-(l-m) \phi^{(2)}_{l+\frac12,m+\frac12,n+\frac12}(Z) \end{pmatrix} \right),
\end{multline*}
$$
\rho' \bigl( \begin{smallmatrix} 0 & 0 \\ C & 0 \end{smallmatrix} \bigr)
\phi^{(2)}_{l,m,n}(Z) = \tr \left( C \begin{pmatrix}
(l-n) \phi^{(2)}_{l+\frac12,m-\frac12,n-\frac12}(Z) &
(l-n) \phi^{(2)}_{l+\frac12,m+\frac12,n-\frac12}(Z)  \\
(l+n) \phi^{(2)}_{l+\frac12,m-\frac12,n+\frac12}(Z) &
(l+n) \phi^{(2)}_{l+\frac12,m+\frac12,n+\frac12}(Z) \end{pmatrix} \right);
$$
\begin{multline*}
\rho' \bigl( \begin{smallmatrix} 0 & B \\ 0 & 0 \end{smallmatrix} \bigr)
\tilde \phi^{(1)}_{l,m,n}(Z) = \frac{2l}{2l+1}\tr \left( B \begin{pmatrix}
(l-m+1) \tilde \phi^{(1)}_{l+\frac12,m-\frac12,n-\frac12}(Z) &
(l-m+1) \tilde \phi^{(1)}_{l+\frac12,m-\frac12,n+\frac12}(Z)  \\
(l+m+1) \tilde \phi^{(1)}_{l+\frac12,m+\frac12,n-\frac12}(Z) &
(l+m+1) \tilde \phi^{(1)}_{l+\frac12,m+\frac12,n+\frac12}(Z) \end{pmatrix} \right)  \\
+ \frac1{2l+1} \tr \left( B \begin{pmatrix}
-(l+n) \tilde \phi^{(2)}_{l+\frac12,m-\frac12,n-\frac12}(Z) &
(l-n) \tilde \phi^{(2)}_{l+\frac12,m-\frac12,n+\frac12}(Z)  \\
(l+n) \tilde \phi^{(2)}_{l+\frac12,m+\frac12,n-\frac12}(Z) &
-(l-n) \tilde \phi^{(2)}_{l+\frac12,m+\frac12,n+\frac12}(Z) \end{pmatrix} \right),
\end{multline*}
$$
\rho' \bigl( \begin{smallmatrix} 0 & B \\ 0 & 0 \end{smallmatrix} \bigr)
\tilde \phi^{(2)}_{l,m,n}(Z) = \tr \left( B \begin{pmatrix}
(l-m) \tilde \phi^{(2)}_{l+\frac12,m-\frac12,n-\frac12}(Z) &
(l-m) \tilde \phi^{(2)}_{l+\frac12,m-\frac12,n+\frac12}(Z)  \\
(l+m) \tilde \phi^{(2)}_{l+\frac12,m+\frac12,n-\frac12}(Z) &
(l+m) \tilde \phi^{(2)}_{l+\frac12,m+\frac12,n+\frac12}(Z) \end{pmatrix} \right),
$$
$$
\rho' \bigl( \begin{smallmatrix} 0 & 0 \\ C & 0 \end{smallmatrix} \bigr)
\tilde \phi^{(1)}_{l,m,n}(Z) = -\tr \left( C \begin{pmatrix}
(l-n) \tilde \phi^{(1)}_{l-\frac12,m+\frac12,n+\frac12}(Z) &
(l-n) \tilde \phi^{(1)}_{l-\frac12,m-\frac12,n+\frac12}(Z)  \\
(l+n) \tilde \phi^{(1)}_{l-\frac12,m+\frac12,n-\frac12}(Z) &
(l+n) \tilde \phi^{(1)}_{l-\frac12,m-\frac12,n-\frac12}(Z)
\end{pmatrix} \right),
$$
\begin{multline*}
\rho' \bigl( \begin{smallmatrix} 0 & 0 \\ C & 0 \end{smallmatrix} \bigr)
\tilde \phi^{(2)}_{l,m,n}(Z) = \frac1{2l-1} \tr \left( C \begin{pmatrix}
-(l+m) \tilde \phi^{(1)}_{l-\frac12,m+\frac12,n+\frac12}(Z) &
(l-m) \tilde \phi^{(1)}_{l-\frac12,m-\frac12,n+\frac12}(Z)  \\
(l+m) \tilde \phi^{(1)}_{l-\frac12,m+\frac12,n-\frac12}(Z) &
-(l-m) \tilde \phi^{(1)}_{l-\frac12,m-\frac12,n-\frac12}(Z)
\end{pmatrix} \right)  \\
- \frac{2l}{2l-1} \tr \left( C \begin{pmatrix}
(l-n-1) \tilde \phi^{(2)}_{l-\frac12,m+\frac12,n+\frac12}(Z) &
(l-n-1) \tilde \phi^{(2)}_{l-\frac12,m-\frac12,n+\frac12}(Z)  \\
(l+n-1) \tilde \phi^{(2)}_{l-\frac12,m+\frac12,n-\frac12}(Z) &
(l+n-1) \tilde \phi^{(2)}_{l-\frac12,m-\frac12,n-\frac12}(Z)
\end{pmatrix} \right).
\end{multline*}}

\subsection{Invariant Bilinear Pairing}

In this subsection we construct a symmetric $\mathfrak{gl}(2,\HC)$-invariant
bilinear pairing on ${\cal BH}$.
Let $S^3_R = \{ X \in \BB H ;\: N(X)=R^2 \}$ be the sphere of radius $R>0$
centered at the origin, and let $dS$ be the usual Euclidean volume element
on $S^3_R$.

\begin{df}  \label{biharm-pairing-def}
We define a symmetric bilinear pairing $\langle \phi_1, \phi_2 \rangle_{\cal BH}$
on ${\cal BH}$ as follows.
The pairing between ${\cal BH}^+$ and ${\cal BH}^-$ is defined by
\begin{align*}
  \bigl\langle \phi^{(1)}_{l,m,n}(Z), \tilde \phi^{(1)}_{l',m',n'}(Z)
  \bigr\rangle_{\cal BH}
&=
\frac{2l(2l+1)}{2\pi^2} \int_{X \in S^3_R}
\phi^{(1)}_{l,m,n}(X) \cdot \tilde \phi^{(1)}_{l',m',n'}(X) \,\frac{dS}{R^3},  \\
\bigl\langle \phi^{(2)}_{l,m,n}(Z), \tilde \phi^{(2)}_{l',m',n'}(Z)
\bigr\rangle_{\cal BH}
&= -\frac{2l(2l-1)}{2\pi^2} \int_{X \in S^3_R}
\phi^{(2)}_{l,m,n}(X) \cdot \tilde \phi^{(2)}_{l',m',n'}(X) \,\frac{dS}{R^3},  \\
\bigl\langle \phi^{(\alpha)}_{l,m,n}(Z), \tilde \phi^{(\alpha')}_{l',m',n'}(Z)
\bigr\rangle_{\cal BH} &= 0 \qquad \text{if $\alpha \ne \alpha'$}.
\end{align*}
We extend the pairing to ${\cal BH}^- \times {\cal BH}^+$ by symmetry,
and declare it to be zero on ${\cal BH}^+ \times {\cal BH}^+$
and ${\cal BH}^- \times {\cal BH}^-$.
\end{df}

\begin{prop}  \label{biharm-pairing-prop}
The symmetric bilinear pairing $\langle \phi_1, \phi_2 \rangle_{\cal BH}$
on ${\cal BH} \times {\cal BH}$ is $SU(2) \times SU(2)$ and
$\mathfrak{gl}(2,\HC)$-invariant, independent of the choice of $R>0$
and satisfies the following orthogonality relations:
\begin{align*}
  \bigl\langle \phi^{(1)}_{l,m,n}(Z), \tilde \phi^{(1)}_{l',m',n'}(Z)
  \bigr\rangle_{\cal BH}
&=
  \bigl\langle \tilde \phi^{(1)}_{l,m,n}(Z), \phi^{(1)}_{l',m',n'}(Z)
  \bigr\rangle_{\cal BH}
= 2l \delta_{ll'} \delta_{mm'} \delta_{nn'}, \\
\bigl\langle \phi^{(2)}_{l,m,n}(Z), \tilde \phi^{(2)}_{l',m',n'}(Z)
\bigr\rangle_{\cal BH}
&=
\bigl\langle \tilde \phi^{(2)}_{l,m,n}(Z), \phi^{(2)}_{l',m',n'}(Z)
\bigr\rangle_{\cal BH}
= -2l \delta_{ll'} \delta_{mm'} \delta_{nn'}, \\
\bigl\langle \phi^{(1)}_{l,m,n}(Z), \phi^{(1)}_{l',m',n'}(Z) \bigr\rangle_{\cal BH}
&= \bigl\langle \phi^{(1)}_{l,m,n}(Z), \phi^{(2)}_{l',m',n'}(Z)
\bigr\rangle_{\cal BH} =
\bigl\langle \phi^{(1)}_{l,m,n}(Z), \tilde \phi^{(2)}_{l',m',n'}(Z)
\bigr\rangle_{\cal BH}
=0, \\
\bigl\langle \phi^{(2)}_{l,m,n}(Z), \phi^{(1)}_{l',m',n'}(Z) \bigr\rangle_{\cal BH}
&= \bigl\langle \phi^{(2)}_{l,m,n}(Z), \phi^{(2)}_{l',m',n'}(Z)
\bigr\rangle_{\cal BH} =
\bigl\langle \phi^{(2)}_{l,m,n}(Z), \tilde \phi^{(1)}_{l',m',n'}(Z)
\bigr\rangle_{\cal BH}
=0, \\
\bigl\langle \tilde \phi^{(1)}_{l,m,n}(Z), \phi^{(2)}_{l',m',n'}(Z)
\bigr\rangle_{\cal BH}
&= \bigl\langle \tilde \phi^{(1)}_{l,m,n}(Z),
\tilde \phi^{(1)}_{l',m',n'}(Z) \bigr\rangle_{\cal BH} =
\bigl\langle \tilde \phi^{(1)}_{l,m,n}(Z),
\tilde \phi^{(2)}_{l',m',n'}(Z) \bigr\rangle_{\cal BH} =0, \\
\bigl\langle \tilde \phi^{(2)}_{l,m,n}(Z), \phi^{(1)}_{l',m',n'}(Z)
\bigr\rangle_{\cal BH}
&= \bigl\langle \tilde \phi^{(2)}_{l,m,n}(Z),
\tilde \phi^{(1)}_{l',m',n'}(Z) \bigr\rangle_{\cal BH} =
\bigl\langle \tilde \phi^{(2)}_{l,m,n}(Z),
\tilde \phi^{(2)}_{l',m',n'}(Z) \bigr\rangle_{\cal BH} =0.
\end{align*}

The pairing is also invariant under the inversion:
\begin{equation}  \label{inversion-invar=eqn}
\Bigl\langle
\rho' \bigl( \begin{smallmatrix} 0 & 1 \\ 1 & 0 \end{smallmatrix} \bigr) \phi_1,
\rho' \bigl( \begin{smallmatrix} 0 & 1 \\ 1 & 0 \end{smallmatrix} \bigr)
\phi_2 \Bigr\rangle_{\cal BH}
= \langle \phi_1, \phi_2 \rangle_{\cal BH}, \qquad \phi_1, \phi_2 \in {\cal BH}.
\end{equation}
\end{prop}

\begin{rem}
The pairing is obviously degenerate: the function
$\phi^{(1)}_{0,0,0}(Z) = \tilde \phi^{(1)}_{0,0,0}(Z) =1$
spans the kernel of the pairing on ${\cal BH}$.
Hence the pairing descends to a non-degenerate pairing on
${\cal BH}/{\cal I}'_0 \times {\cal BH}/{\cal I}'_0$.
\end{rem}
  
\begin{proof}
It is easy to see that the pairing $\langle \phi_1, \phi_2 \rangle_{\cal BH}$
is $SU(2) \times SU(2)$ invariant, independent of the choice of $R>0$
and satisfies the above orthogonality relations.
The invariance under the scalar matrices 
$\bigl( \begin{smallmatrix} \lambda & 0 \\
0 & \lambda \end{smallmatrix} \bigr)$, $\lambda \in \BB R^{\times}$,
and under the dilation matrices 
$\bigl( \begin{smallmatrix} \lambda & 0 \\
0 & \lambda^{-1} \end{smallmatrix} \bigr)$, $\lambda \in \BB R^{\times}$,
is trivial.
The invariance under the inversion \eqref{inversion-invar=eqn}
follows from Lemma \ref{inversion-bh-lem} and the orthogonality relations.
  
Next, we verify the invariance of the pairing under
$\bigl( \begin{smallmatrix} 0 & B \\ 0 & 0 \end{smallmatrix} \bigr)
\in \mathfrak{gl}(2,\HC)$, $B \in \HC$.
Since we have already established the $SU(2) \times SU(2)$ invariance,
it is sufficient to check for
$B = \bigl( \begin{smallmatrix} 1 & 0 \\ 0 & 0 \end{smallmatrix} \bigr)$
only. We do so by verifying that
\begin{equation}  \label{B-invar}
\Bigl\langle
\rho' \bigl( \begin{smallmatrix} 0 & B \\ 0 & 0 \end{smallmatrix} \bigr)
\phi^{(\alpha)}_{l,m,n}, \tilde \phi^{(\alpha')}_{l',m',n'} \Bigr\rangle_{\cal BH}
+ \Bigl\langle \phi^{(\alpha)}_{l,m,n},
\rho' \bigl( \begin{smallmatrix} 0 & B \\ 0 & 0 \end{smallmatrix} \bigr)
\tilde \phi^{(\alpha')}_{l',m',n'} \Bigr\rangle_{\cal BH} = 0
\end{equation}
for the basis functions, where $\alpha,\alpha'=1,2$.
Note that the actions of
$\rho' \bigl( \begin{smallmatrix} 0 & B \\ 0 & 0 \end{smallmatrix} \bigr)$
were spelled out in Subsection \ref{K-type-biharmonic}.
By the orthogonality relations, both summands in \eqref{B-invar} are zero
unless $l'=l-\frac12$, $m'=m+\frac12$, $n'=n+\frac12$. In this case,
\begin{align*}
\Bigl\langle
\rho' \bigl( \begin{smallmatrix} 0 & B \\ 0 & 0 \end{smallmatrix} \bigr)
\phi^{(1)}_{l,m,n}, \tilde \phi^{(1)}_{l',m',n'} \Bigr\rangle_{\cal BH}
&= -(2l-1)(l-m)
= - \Bigl\langle \phi^{(1)}_{l,m,n},
\rho' \bigl( \begin{smallmatrix} 0 & B \\ 0 & 0 \end{smallmatrix} \bigr)
\tilde \phi^{(1)}_{l',m',n'} \Bigr\rangle_{\cal BH}.  \\
\Bigl\langle
\rho' \bigl( \begin{smallmatrix} 0 & B \\ 0 & 0 \end{smallmatrix} \bigr)
\phi^{(1)}_{l,m,n}, \tilde \phi^{(2)}_{l',m',n'} \Bigr\rangle_{\cal BH}
&= \Bigl\langle \phi^{(1)}_{l,m,n},
\rho' \bigl( \begin{smallmatrix} 0 & B \\ 0 & 0 \end{smallmatrix} \bigr)
\tilde \phi^{(2)}_{l',m',n'} \Bigr\rangle_{\cal BH} =0,  \\
\Bigl\langle
\rho' \bigl( \begin{smallmatrix} 0 & B \\ 0 & 0 \end{smallmatrix} \bigr)
\phi^{(2)}_{l,m,n}, \tilde \phi^{(1)}_{l',m',n'} \Bigr\rangle_{\cal BH}
&= -(l+n)
= - \Bigl\langle \phi^{(2)}_{l,m,n},
\rho' \bigl( \begin{smallmatrix} 0 & B \\ 0 & 0 \end{smallmatrix} \bigr)
\tilde \phi^{(1)}_{l',m',n'} \Bigr\rangle_{\cal BH},  \\
\Bigl\langle
\rho' \bigl( \begin{smallmatrix} 0 & B \\ 0 & 0 \end{smallmatrix} \bigr)
\phi^{(2)}_{l,m,n}, \tilde \phi^{(2)}_{l',m',n'} \Bigr\rangle_{\cal BH}
&= 2l(l-m-1)
= - \Bigl\langle \phi^{(2)}_{l,m,n},
\rho' \bigl( \begin{smallmatrix} 0 & B \\ 0 & 0 \end{smallmatrix} \bigr)
\tilde \phi^{(2)}_{l',m',n'} \Bigr\rangle_{\cal BH}.
\end{align*}
This proves the invariance of the pairing under
$\bigl( \begin{smallmatrix} 0 & B \\ 0 & 0 \end{smallmatrix} \bigr)
\in \mathfrak{gl}(2,\HC)$, $B \in \HC$.

Finally, it follows from \eqref{inversion-invar=eqn} that the
pairing is also invariant under
$\bigl( \begin{smallmatrix} 0 & 0 \\ C & 0 \end{smallmatrix} \bigr)
\in \mathfrak{gl}(2,\HC)$, $C \in \HC$.
\end{proof}

\subsection{Reproducing Kernel for ${\cal BH}/{\cal I}'_0$}

In this subsection we provide two expansions of the reproducing kernel
for ${\cal BH}/{\cal I}'_0$ in terms of the basis functions given in
Proposition \ref{K-type-biharm_prop}.
These expansions should be considered as biharmonic analogues
of the matrix coefficient expansions of $N(Z-W)^{-1}$ from
Proposition 25 in \cite{FL1} (see also Proposition 112 in \cite{ATMP}).

Recall that $\BB D^+$ is certain open domain in $\HC$ having $U(2)$
as Shilov boundary:
$$
\BB D^+ = \{ Z \in \HC;\: ZZ^*<1 \},
$$
where the inequality $ZZ^*<1$ means that the matrix $ZZ^*-1$ is
negative definite.
We use the principal branch of the complex logarithm function
with a branch cut along the negative real axis.

\begin{thm}
We have the reproducing kernel expansions:
\begin{equation}  \label{1st-bh-kernel-expansion}
\log\bigl[ N(1-ZW^{-1}) \bigr]
= \sum_{\genfrac{}{}{0pt}{}{l \ge 1/2}{m, n}}
\frac1{2l} \phi^{(1)}_{l,m,n}(Z) \cdot \tilde \phi^{(1)}_{l,m,n}(W)
- \sum_{\genfrac{}{}{0pt}{}{l \ge 1}{m, n}}
\frac1{2l} \phi^{(2)}_{l,m,n}(Z) \cdot \tilde \phi^{(2)}_{l,m,n}(W),
\end{equation}
which converges uniformly on compact subsets in the region
$\{ (Z,W) \in \HC \times \HC^{\times} ; \: ZW^{-1} \in \BB D^+ \}$.
The sum is taken first over all applicable $m$ and $n$,
then over $l=\frac 12, 1, \frac 32,2,\dots$.

Similarly,
\begin{equation}  \label{2nd-bh-kernel-expansion}
\log\bigl[ N(1-Z^{-1}W) \bigr]
= \sum_{\genfrac{}{}{0pt}{}{l \ge 1/2}{m, n}}
\frac1{2l} \tilde \phi^{(1)}_{l,m,n}(Z) \cdot \phi^{(1)}_{l,m,n}(W)
- \sum_{\genfrac{}{}{0pt}{}{l \ge 1}{m, n}}
\frac1{2l} \tilde \phi^{(2)}_{l,m,n}(Z) \cdot \phi^{(2)}_{l,m,n}(W),
\end{equation}
which converges uniformly on compact subsets in the region
$\{ (Z,W) \in \HC^{\times} \times \HC; \: Z^{-1}W \in \BB D^+ \}$.
The sum is taken first over all applicable $m$ and $n$,
then over $l=\frac 12, 1, \frac 32,2,\dots$.
\end{thm}

\begin{proof}
We compute:
\begin{multline*}
\sum_{\genfrac{}{}{0pt}{}{l \ge 1/2}{m, n}}
\frac1{2l} \phi^{(1)}_{l,m,n}(Z) \cdot \tilde \phi^{(1)}_{l,m,n}(W)
- \sum_{\genfrac{}{}{0pt}{}{l \ge 1}{m, n}}
\frac1{2l} \phi^{(2)}_{l,m,n}(Z) \cdot \tilde \phi^{(2)}_{l,m,n}(W)  \\
= \sum_{\genfrac{}{}{0pt}{}{l \ge 1/2}{m, n}} \frac1{2l} t^l_{n \,\underline{m}}(Z)
\cdot t^l_{m \,\underline{n}}(W^{-1})
- \sum_{\genfrac{}{}{0pt}{}{l \ge 1}{m, n}}
\frac1{2l} N(Z) \cdot t^{l-1}_{n \,\underline{m}}(Z)
\cdot N(W)^{-1} \cdot t^{l-1}_{m \,\underline{n}}(W^{-1})  \\
= \sum_{l \ge 1/2,\: n} \frac1{2l} t^l_{n \,\underline{n}}(ZW^{-1})
- N(ZW^{-1}) \sum_{l \ge 1,\: n}
\frac1{2l} t^{l-1}_{n \,\underline{n}}(ZW^{-1}).
\end{multline*}
Assume further that $ZW^{-1}$ can be diagonalized as
$\bigl(\begin{smallmatrix} \lambda_1 & 0 \\
0 & \lambda_2 \end{smallmatrix}\bigr)$ with
$\lambda_1 \ne \lambda_2$.
This is allowed since the set of matrices with distinct eigenvalues is
dense in $\HC$.
Then the sum $\sum_{n} t^l_{n \, \underline{n}}(ZW^{-1})$
is just the character $\chi_l(ZW^{-1})$ of the irreducible representation of
$GL(2,\BB C)$ of dimension $2l+1$ and equals
$\frac{\lambda_1^{2l+1}-\lambda_2^{2l+1}}{\lambda_1-\lambda_2}$.
We continue our computations:
\begin{multline*}
\sum_{\genfrac{}{}{0pt}{}{l \ge 1/2}{m, n}}
\frac1{2l} \phi^{(1)}_{l,m,n}(Z) \cdot \tilde \phi^{(1)}_{l,m,n}(W)
- \sum_{\genfrac{}{}{0pt}{}{l \ge 1}{m, n}}
\frac1{2l} \phi^{(2)}_{l,m,n}(Z) \cdot \tilde \phi^{(2)}_{l,m,n}(W)  \\
= \frac{\lambda_1}{\lambda_1-\lambda_2} \sum_{l \ge 1/2} \frac{\lambda_1^{2l}}{2l}
- \frac{\lambda_2}{\lambda_1-\lambda_2} \sum_{l \ge 1/2} \frac{\lambda_2^{2l}}{2l}
- \frac{\lambda_2}{\lambda_1-\lambda_2} \sum_{l \ge 1} \frac{\lambda_1^{2l}}{2l}
+ \frac{\lambda_1}{\lambda_1-\lambda_2} \sum_{l \ge 1} \frac{\lambda_2^{2l}}{2l}\\
= \lambda_1 + \lambda_2 + \sum_{l \ge 1} \frac{\lambda_1^{2l}}{2l}
+ \sum_{l \ge 1} \frac{\lambda_2^{2l}}{2l}
= \sum_{l \ge 1/2} \frac{\lambda_1^{2l}}{2l}
+ \sum_{l \ge 1/2} \frac{\lambda_2^{2l}}{2l}  \\
= \log(1-\lambda_1) + \log(1-\lambda_2)
= \log\bigl[ (1-\lambda_1)(1-\lambda_2) \bigr]
= \log\bigl[ N(1-ZW^{-1}) \bigr].
\end{multline*}
This proves \eqref{1st-bh-kernel-expansion}.

The proof of \eqref{2nd-bh-kernel-expansion} is similar.
Alternatively, \eqref{2nd-bh-kernel-expansion} follows from
\eqref{1st-bh-kernel-expansion} by applying the inversion and
Lemma \ref{inversion-bh-lem}.
\end{proof}

%\begin{rem}
%Observe that 
%\begin{equation*}
%\begin{split}
%\nabla^+_Z \log\bigl[ N(1-ZW^{-1}) \bigr]
%&= 2 \frac{Z-W}{N(Z-W)},  \\
%\nabla^+_W \log\bigl[ N(1-Z^{-1}W) \bigr]
%&= - 2 \frac{Z-W}{N(Z-W)},
%\end{split}
%\end{equation*}
%and $\frac{Z-W}{N(Z-W)}$ is the reproducing kernel for quasi anti regular
%functions.
%\end{rem}

\subsection{Reproducing Formula for Biharmonic Functions}

In this subsection we give a reproducing formula for the
biharmonic functions, it could be considered as an analogue
of the Poisson formula (cf. Theorem 34 in \cite{FL1}).

Recall the degree operator $\deg$ acting on functions on $\BB H$:
$$
\deg f = x^0\frac{\partial f}{\partial x^0} +
x^1\frac{\partial f}{\partial x^1} + x^2\frac{\partial f}{\partial x^2}
+ x^3\frac{\partial f}{\partial x^3},
$$
and let $\degt$ denote the degree operator plus the identity:
$$
\degt f = \deg f + f.
$$
Similarly, we can define operators $\deg$ and $\degt$ acting on functions
on $\HC$.

Also recall the $(\deg+2)^{-1}$ operator discussed in detail in
Subsection 2.4 of \cite{ATMP}. This operator has the property that
if $\phi$ is a function on $\BB H^{\times}$ that is homogeneous of degree $d$,
then
$$
(\deg+2)^{-1} \phi = \tfrac1{d+2} \phi.
$$
Observe that if $\phi \in {\cal BH}$ and $\square \phi \ne 0$,
then $\square \phi$ {\em cannot} be homogeneous of degree $-2$,
thus $(\deg+2)^{-1} \square \phi$ is well-defined.

When restricted to ${\cal BH}^+ \times {\cal BH}^-$, the invariant bilinear
pairing on ${\cal BH}$ given in Definition \ref{biharm-pairing-def}
can be expressed as follows:
\begin{multline}  \label{biharm-pairing-integral}
\langle f, g \rangle_{\cal BH} = 
\frac1{8\pi^2} \int_{X \in S^3_R} f(X) \cdot \bigl( \degt \square g(X) \bigr)
\,\frac{dS}R
- \frac1{8\pi^2} \int_{X \in S^3_R} \bigl( \degt \square f(X) \bigr) \cdot g(X)
\,\frac{dS}R  \\
+ \frac1{16\pi^2} \int_{X \in S^3_R} \bigl( \degt (\deg+2)^{-1} \square f(X) \bigr)
\cdot \bigl( (\deg+2)^{-1} \square g(X) \bigr) \,\frac{dS}R.
\end{multline}
This can be verified by checking that the orthogonality relations from
Proposition \ref{biharm-pairing-prop} are satisfied.

\begin{rem}
  On ${\cal BH}^- \times {\cal BH}^+$, the formula
  \eqref{biharm-pairing-integral} produces the negative of the pairing;
  and on ${\cal BH}^+ \times {\cal BH}^+$, ${\cal BH}^- \times {\cal BH}^-$
  the formula \eqref{biharm-pairing-integral} can produce non-zero values,
  while the bilinear pairing is zero.
\end{rem}

Let $B_R$ denote the open ball in $\BB H$ of radius $R$ centered at the origin.
Substituting the reproducing kernel $\log\bigl[ N(1-Z^{-1}W) \bigr]$
into \eqref{biharm-pairing-integral}, we obtain the following reproducing
formula for biharmonic functions.

\begin{thm}  \label{BH-repro-thm}
  Suppose that $f(X)$ is biharmonic on a neighborhood of the closure
  $\overline{B_R}$ and $X_0 \in B_R$, then
\begin{multline*}
f(X_0) = f(0)
+ \bigl\langle f(X), \log\bigl[ N(1-X^{-1}X_0) \bigr] \bigr\rangle_{\cal BH}  \\
= f(0)
+ \frac1{2\pi^2} \int_{X \in S^3_R} f(X) \cdot
\degt \Bigl( \frac1{N(X-X_0)} - \frac1{N(X)} \Bigr) \,\frac{dS}R  \\
- \frac1{8\pi^2} \int_{X \in S^3_R} \bigl( \degt \square f(X) \bigr)
\cdot \log\bigl[ N(1-X^{-1}X_0) \bigr] \,\frac{dS}R  \\
+ \frac1{4\pi^2} \int_{X \in S^3_R} \bigl( \degt (\deg+2)^{-1} \square f(X) \bigr)
\cdot (\deg+2)^{-1} \Bigl( \frac1{N(X-X_0)} - \frac1{N(X)} \Bigr) \,\frac{dS}R.
\end{multline*}
\end{thm}

%{\bf Quasi (anti) regular functions are biharmonic;
%  and if $f \in {\cal BH}^+$, then $\tilde f = \nabla^+ f$
%  is a function satisfying $\nabla \square \tilde f =0$.
%  Can the reproducing formula for one type of functions be derived
%  from the other?}

\subsection{Invariant Pseudounitary Structures}

In this subsection we construct pseudounitary (non-degenerate indefinite)
$\mathfrak{u}(2,2)$-invariant structures on
${\cal BH}^+/{\cal I}'_0$ and ${\cal BH}^-/{\cal I}'_0$.

The group $U(2,2)$ can be realized as the subgroup of elements of $GL(2,\HC)$
preserving the Hermitian form on $\BB C^4$ given by the $4 \times 4$ matrix
$\bigl(\begin{smallmatrix} 1 & 0 \\ 0 & -1 \end{smallmatrix}\bigr)$. Explicitly,
\begin{equation}  \label{U(2,2)}
\begin{split}
U(2,2) &= \Biggl\{ \begin{pmatrix} a & b \\ c & d \end{pmatrix};\:
a,b,c,d \in \HC,\:
\begin{matrix} a^*a = 1+c^*c \\ d^*d = 1+b^*b \\ a^*b=c^*d \end{matrix}
\Biggr\}  \\
&= \Biggl\{ \begin{pmatrix} a & b \\ c & d \end{pmatrix};\:
a,b,c,d \in \HC,\:
\begin{matrix} a^*a = 1+b^*b \\ d^*d = 1+c^*c \\ ac^*=bd^* \end{matrix}
\Biggr\}.
\end{split}
\end{equation}
The Lie algebra of $U(2,2)$ is
\begin{equation}  \label{u(2,2)-algebra}
\mathfrak{u}(2,2) = \bigl\{
\bigl(\begin{smallmatrix} A & B \\ B^* & D \end{smallmatrix}\bigr) ;\:
A,B,D \in \HC ,\: A=-A^*, D=-D^* \bigr\}.
%\mathfrak{u}(2,2) = \biggl\{
%\begin{pmatrix} A & B \\ B^* & D \end{pmatrix} ;\: A,B,D \in \HC ,\:
%A=-A^*, D=-D^* \biggr\}.
\end{equation}

%\begin{lem}
%Let $\lambda \in \BB C$ with $|\lambda|=1$, and consider an element
%$\Lambda = \left( \begin{smallmatrix} \lambda & 0 & 0 & 0 \\
%  0 & \lambda & 0 & 0 \\ 0 & 0 & \lambda^{-1} & 0 \\
%  0 & 0 & 0 & \lambda^{-1} \end{smallmatrix} \right) \in U(2) \times U(2)$.
%Then for $f \in {\cal BH}(d)$ we have:
%$$
%\rho'(\Lambda) f = \lambda^{-2d} f.
%$$
%\end{lem}

\begin{thm}  \label{BH-unitary-thm}
The spaces $(\rho', {\cal BH}^+)$ and $(\rho', {\cal BH}^-)$ have
pseudounitary structures that are preserved by the real form
$\mathfrak{u}(2,2)$ of $\mathfrak{gl}(2,\HC)$.
The invariant pseudounitary structure on ${\cal BH}^+$ can be described as
\begin{align*}
\bigl( \phi^{(1)}_{l,m,n}(Z), \phi^{(1)}_{l',m',n'}(Z) \bigr)_{{\cal BH}^+} &=
\frac{2l(2l+1)}{2\pi^2} \int_{X \in S^3}
\phi^{(1)}_{l,m,n}(X) \cdot \overline{\phi^{(1)}_{l',m',n'}(X)} \,dS,  \\
\bigl( \phi^{(2)}_{l,m,n}(Z), \phi^{(2)}_{l',m',n'}(Z) \bigr)_{{\cal BH}^+} &=
-\frac{2l(2l+1)}{2\pi^2} \int_{X \in S^3}
\phi^{(2)}_{l,m,n}(X) \cdot \overline{\phi^{(2)}_{l',m',n'}(X)} \,dS,  \\
\bigl( \phi^{(\alpha)}_{l,m,n}(Z), \phi^{(\alpha')}_{l',m',n'}(Z) \bigr)_{{\cal BH}^+}
&= 0 \qquad \text{if $\alpha \ne \alpha'$}.
\end{align*}
This structure descends to a non-degenerate product on
${\cal BH}^+/{\cal I}'_0$.

And the invariant pseudounitary structures on ${\cal BH}^-$ and
${\cal BH}^-/{\cal I}'_0$ are similar.
\end{thm}
  
\begin{proof}
Clearly, the product $(\phi_1,\phi_2)_{{\cal BH}^+}$ is $U(2) \times U(2)$
invariant.
From \cite{V} we have:
$$
\overline{ t^l_{n\,\underline{m}}(X)} =
\frac{(l-m)!(l+m)!}{(l-n)!(l+n)!} t^l_{m\,\underline{n}}(X^+), \qquad X \in \BB H,
$$
and orthogonality relations:
\begin{equation}  \label{t-orthog-rels}
\frac1{2\pi^2} \int_{X \in S^3} t^l_{n\,\underline{m}}(X) \cdot
\overline{ t^{l'}_{n'\,\underline{m'}}(X)} \,dS
= \frac{(l-m)!(l+m)!}{(l-n)!(l+n)!} \frac{\delta_{ll'} \delta_{mm'} \delta_{nn'}}
{2l+1}.
\end{equation}
Then we have orthogonality relations for the $\phi^{(\alpha)}_{l,m,n}(Z)$'s:
\begin{align*}
\bigl( \phi^{(1)}_{l,m,n}(Z), \phi^{(1)}_{l',m',n'}(Z) \bigr)_{{\cal BH}^+} &=
2l \frac{(l-m)!(l+m)!}{(l-n)!(l+n)!}
\delta_{ll'} \delta_{mm'} \delta_{nn'},  \\
\bigl( \phi^{(2)}_{l,m,n}(Z), \phi^{(2)}_{l',m',n'}(Z) \bigr)_{{\cal BH}^+} &=
- \frac{2l(2l+1)}{2l-1} \frac{(l-m-1)!(l+m-1)!}{(l-n-1)!(l+n-1)!}
\delta_{ll'} \delta_{mm'} \delta_{nn'},  \\
\bigl( \phi^{(\alpha)}_{l,m,n}(Z), \phi^{(\alpha')}_{l',m',n'}(Z) \bigr)_{{\cal BH}^+}
&= 0 \qquad \text{if $\alpha \ne \alpha'$}.
\end{align*}
Consider
\begin{equation}  \label{XYinu(2,2)}
X=\left(\begin{smallmatrix} 0 & 0 & 1 & 0 \\ 0 & 0 & 0 & 0 \\ 1 & 0 & 0 & 0 \\
  0 & 0 & 0 & 0 \end{smallmatrix} \right), \quad
Y=\left(\begin{smallmatrix} 0 & 0 & i & 0 \\ 0 & 0 & 0 & 0 \\ -i & 0 & 0 & 0 \\
  0 & 0 & 0 & 0 \end{smallmatrix} \right) \quad
\in \mathfrak{u}(2,2).
\end{equation}
To prove that $(\phi_1,\phi_2)_{{\cal BH}^+}$ is $\mathfrak{u}(2,2)$-invariant,
it is sufficient to show that $(\phi_1,\phi_2)_{{\cal U}^+}$ is invariant under
$\rho'(X)$ and $\rho'(Y)$.
From Subsection \ref{K-type-biharmonic} we find:
\begin{multline*}
\rho'(X) \phi^{(1)}_{l,m,n}(Z) = -(l-m)\phi^{(1)}_{l-\frac12,m+\frac12,n+\frac12}(Z)  \\
+ \frac{2l(l-n+1)}{2l+1} \phi^{(1)}_{l+\frac12,m-\frac12,n-\frac12}(Z)
- \frac{l+m}{2l+1} \phi^{(2)}_{l+\frac12,m-\frac12,n-\frac12}(Z),
\end{multline*}
\begin{multline*}
\rho'(X) \phi^{(2)}_{l,m,n}(Z)
= -\frac{l+n}{2l-1} \phi^{(1)}_{l-\frac12,m+\frac12,n+\frac12}(Z)  \\
- \frac{2l(l-m-1)}{2l-1} \phi^{(2)}_{l-\frac12,m+\frac12,n+\frac12}(Z)
+ (l-n) \phi^{(2)}_{l+\frac12,m-\frac12,n-\frac12}(Z).
\end{multline*}
And we can find similar expressions for
$$
\rho'(Y) \phi^{(1)}_{l,m,n}(Z), \qquad \rho'(Y) \phi^{(2)}_{l,m,n}(Z).
$$
Then, by direct calculation,
\begin{align*}
\bigl( \phi^{(1)}_{l,m,n}(Z), \rho'(X) \phi^{(1)}_{l+\frac12,m-\frac12,n-\frac12}(Z)
\bigr)_{{\cal BH}^+}
&= - 2l \frac{(l-m+1)!(l+m)!}{(l-n)!(l+n)!},  \\
&= - \bigl( \rho'(X) \phi^{(1)}_{l,m,n}(Z), \phi^{(1)}_{l+\frac12,m-\frac12,n-\frac12}(Z)
\bigr)_{{\cal BH}^+},  \\
\bigl( \phi^{(2)}_{l,m,n}(Z), \rho'(X) \phi^{(2)}_{l+\frac12,m-\frac12,n-\frac12}(Z)
\bigr)_{{\cal BH}^+}
&= (2l+1) \frac{(l-m)!(l+m-1)!}{(l-n-1)!(l+n-1)!},  \\
&= -\bigl( \rho'(X) \phi^{(2)}_{l,m,n}(Z), \phi^{(2)}_{l+\frac12,m-\frac12,n-\frac12}(Z)
\bigr)_{{\cal BH}^+},  \\
\bigl( \phi^{(1)}_{l,m,n}(Z), \rho'(X) \phi^{(2)}_{l+\frac12,m-\frac12,n-\frac12}(Z)
\bigr)_{{\cal BH}^+}
&= - \frac{(l-m)!(l+m)!}{(l-n)!(l+n-1)!},  \\
&= - \bigl( \rho'(X) \phi^{(1)}_{l,m,n}(Z), \phi^{(2)}_{l+\frac12,m-\frac12,n-\frac12}(Z)
\bigr)_{{\cal BH}^+}.
\end{align*}
This proves the $\rho'(X)$-invariance of the pseudounitary structure.

The calculations for $\rho'(Y)$ are nearly identical.
\end{proof}

Recall that, by Theorem \ref{biharm-irred-thm},
$(\rho', {\cal BH}^+/{\cal I}'_0)$ and $(\rho', {\cal BH}^-/{\cal I}'_0)$
are irreducible representations of $\mathfrak{gl}(2,\HC)$ and hence
$\mathfrak{u}(2,2)$, and the (pseudo)unitary structure on an irreducible
representation is unique up to rescaling. Thus we obtain:

\begin{cor}
The spaces $(\rho', {\cal BH}^+/{\cal I}'_0)$ and
$(\rho', {\cal BH}^-/{\cal I}'_0)$
do {\em not} have $\mathfrak{u}(2,2)$-invariant unitary structures.
\end{cor}

\section{Quasi Anti Regular Functions}  \label{Sect3}

\subsection{Definitions and Conformal Invariance}  \label{QR-def-subsection}

We introduce left and right quasi anti regular functions defined on
open subsets of $\BB H$ and $\HC$.

\begin{df}  \label{qr-definition}
Let $U$ be an open subset of $\BB H$.
A ${\cal C}^3$-function $f: U \to \BB S$ is
{\em quasi left anti regular} (QLAR for short) if it satisfies
$$
\nabla \square f =0 \qquad \text{at all points in $U$}.
$$

Similarly, a ${\cal C}^3$-function $g: U \to \BB S'$ is
{\em quasi right anti regular} (QRAR for short) if
$$
(\square g) \overleftarrow{\nabla} =0 \qquad \text{at all points in $U$}.
$$
\end{df}

We also can talk about quasi regular functions defined on open subsets of
$\HC$. In this case we require such functions to be holomorphic.

\begin{df}
Let $U$ be an open subset of $\HC$.
A holomorphic function $f: U \to \BB S$ is
{\em quasi left anti regular} (QLAR for short) if it satisfies
$\nabla \square f =0$ at all points in $U$.

Similarly, a holomorphic function $g: U \to \BB S'$ is
{\em quasi right anti regular} (QRAR for short) if
$(\square g) \overleftarrow{\nabla} =0$ at all points in $U$.
\end{df}

It is clear from \eqref{Laplacian} that QLAR and QRAR functions are biharmonic,
i.e. annihilated by $\square^2$.
One way to construct QLAR functions is to start with a biharmonic
function $\phi: \BB H \to \BB S$, then $\nabla^+ \phi$ is QLAR.
Similarly, if $\phi: \BB H \to \BB S'$ is biharmonic, then
$\phi \overleftarrow{\nabla^+}$ is QRAR.

Let $\tilde{\cal U}$ and $\tilde{\cal U}'$ denote respectively the
spaces of (holomorphic) quasi left and right anti regular functions on $\HC$,
possibly with singularities.

\begin{thm}  \label{qr-action-thm}
\begin{enumerate}
\item
The space $\tilde{\cal U}$ of quasi left anti regular functions
$\HC \to \BB S$ (possibly with singularities)
is invariant under the following action of $GL(2,\HC)$:
\begin{multline}  \label{pi'_l}
\pi'_l(h): \: f(Z) \: \mapsto \: \bigl( \pi'_l(h)f \bigr)(Z) =
\frac{a'-Zc'}{N(a'-Zc')} \cdot f \bigl( (a'-Zc')^{-1}(-b'+Zd') \bigr),  \\
h = \bigl( \begin{smallmatrix} a' & b' \\ c' & d' \end{smallmatrix} \bigr)
\in GL(2,\HC).
\end{multline}
\item
The space $\tilde{\cal U}'$ of quasi right anti regular functions
$\HC \to \BB S'$ (possibly with singularities)
is invariant under the following action of $GL(2,\HC)$:
\begin{multline}  \label{pi'_r}
\pi'_r(h): \: g(Z) \: \mapsto \: \bigl( \pi'_r(h)g \bigr)(Z) =
g \bigl( (aZ+b)(cZ+d)^{-1} \bigr) \cdot \frac{cZ+d}{N(cZ+d)},  \\
h^{-1} = \bigl( \begin{smallmatrix} a & b \\ c & d \end{smallmatrix} \bigr)
\in GL(2,\HC).
\end{multline}
\end{enumerate}
\end{thm}

\begin{proof}
It is easy to see that the formulas describing the actions
$\pi'_l$ and $\pi'_r$ also produce well-defined actions on the spaces
of all functions on $\HC$ (possibly with singularities) with values in
$\BB S$ and $\BB S'$ respectively.
Differentiating $\pi'_l$ and $\pi'_r$, we obtain actions of the Lie algebra
$\mathfrak{gl}(2,\HC)$, which we still denote by $\pi'_l$ and $\pi'_r$
respectively. We spell out these Lie algebra actions:

\begin{lem}  \label{pi'-Lie_alg-action}
The Lie algebra action $\pi'_l$ of $\mathfrak{gl}(2,\HC)$ on
$\tilde{\cal U}$ is given by
\begin{align*}
\pi'_l \bigl( \begin{smallmatrix} A & 0 \\ 0 & 0 \end{smallmatrix} \bigr) &:
f(Z) \mapsto - \tr (AZ \partial + A) f + Af,  \\
\pi'_l \bigl( \begin{smallmatrix} 0 & B \\ 0 & 0 \end{smallmatrix} \bigr) &:
f(Z) \mapsto - \tr (B \partial) f,  \\
\pi'_l \bigl( \begin{smallmatrix} 0 & 0 \\ C & 0 \end{smallmatrix} \bigr) &:
f(Z) \mapsto \tr (ZCZ \partial +ZC) f - ZCf,  \\
\pi'_l \bigl( \begin{smallmatrix} 0 & 0 \\ 0 & D \end{smallmatrix} \bigr) &:
f(Z) \mapsto \tr (ZD \partial) f.
\end{align*}

Similarly, the Lie algebra action $\pi'_r$ of $\mathfrak{gl}(2,\HC)$ on
$\tilde{\cal U}'$ is given by
\begin{align*}
\pi'_r \bigl( \begin{smallmatrix} A & 0 \\ 0 & 0 \end{smallmatrix} \bigr) &:
g(Z) \mapsto - \tr (AZ \partial) g,  \\
\pi'_r \bigl( \begin{smallmatrix} 0 & B \\ 0 & 0 \end{smallmatrix} \bigr) &:
g(Z) \mapsto - \tr (B \partial) g,  \\
\pi'_r \bigl( \begin{smallmatrix} 0 & 0 \\ C & 0 \end{smallmatrix} \bigr) &:
g(Z) \mapsto \tr (ZCZ \partial +CZ) g - gCZ,  \\
\pi'_r \bigl( \begin{smallmatrix} 0 & 0 \\ 0 & D \end{smallmatrix} \bigr) &:
g(Z) \mapsto \tr (ZD \partial + D) g - gD.
\end{align*}
\end{lem}

%\begin{proof}
%These formulas are obtained by differentiating \eqref{pi'_l} and \eqref{pi'_r}.
%\end{proof}

We continue our proof of Theorem \ref{qr-action-thm}.
Since the Lie group $GL(2,\HC) \simeq GL(4,\BB C)$ is connected, it is
sufficient to show that, if $f \in \tilde{\cal U}$,
$g \in \tilde{\cal U}'$ and
$\bigl( \begin{smallmatrix} A & B \\ C & D \end{smallmatrix} \bigr) \in
\mathfrak{gl}(2,\HC)$, then
$\pi'_l \bigl( \begin{smallmatrix} A & B \\ C & D \end{smallmatrix} \bigr)f
\in \tilde{\cal U}$ and
$\pi'_r \bigl( \begin{smallmatrix} A & B \\ C & D \end{smallmatrix} \bigr)g
\in \tilde{\cal U}'$.
Consider, for example, the case of 
$\pi'_l \bigl( \begin{smallmatrix} 0 & 0 \\ C & 0 \end{smallmatrix} \bigr)f$,
the other cases are similar and easier. We have:
$$
\partial \pi'_l
\bigl( \begin{smallmatrix} 0 & 0 \\ C & 0 \end{smallmatrix} \bigr) f
= \partial \bigl( \tr (ZCZ \partial +ZC) f - ZCf \bigr)
= \tr (ZCZ \partial +ZC) \partial f + CZ\partial f - Cf,
$$
\begin{multline*}
\tfrac18 \nabla \square \pi'_l
\bigl( \begin{smallmatrix} 0 & 0 \\ C & 0 \end{smallmatrix} \bigr) f
= \partial ^+ \partial \partial \pi'_l
\bigl( \begin{smallmatrix} 0 & 0 \\ C & 0 \end{smallmatrix} \bigr) f
= \partial^+ \partial
\bigl( \tr (ZCZ \partial +ZC) \partial f + CZ\partial f - Cf \bigr)  \\
= \partial^+ \bigl( \tr (ZCZ \partial +ZC) \partial \partial f
+ CZ\partial\partial f + \partial(ZC\partial f) + \partial(CZ\partial f)
- C\partial f - \partial Cf \bigr)  \\
= \tr (ZCZ \partial +ZC) \partial^+ \partial \partial f
+ \partial^+(Z^+C^+\partial\partial f) + C^+Z^+\partial^+\partial\partial f
- C^+\partial\partial f  \\
+ \partial^+(CZ\partial\partial f) + \partial^+\partial(ZC\partial f)
+ \partial^+\partial(CZ\partial f) - \partial^+C\partial f
- \partial^+\partial Cf  \\
= \partial^+(Z^+C^+\partial\partial f) + \partial^+(CZ\partial\partial f)
- C^+\partial\partial f + \partial^+\partial(ZC\partial f)  \\
+ C\partial^+\partial(Z\partial f) - \partial^+C\partial f
- C\partial^+\partial f =0,
\end{multline*}
hence
$\pi'_l \bigl( \begin{smallmatrix} 0 & 0 \\ C & 0 \end{smallmatrix} \bigr)f
\in \tilde{\cal U}$.
\end{proof}

\subsection{Constructing Quasi Anti Regular Functions}

As was mentioned at the beginning of Subsection \ref{QR-def-subsection},
one can construct left and right anti regular functions from biharmonic
functions. In this subsection, we describe another way to construct
anti regular functions following Laville-Ramadanoff \cite{LR}.

\begin{lem}
An $\BB S$-valued function $f$ is QLAR if and only if $\square^2f=0$ and
$\square^2(Z^+f)=0$.
Similarly, an $\BB S'$-valued function $g$ is QRAR if and only if
$\square^2g=0$ and $\square^2(gZ^+)=0$.
\end{lem}

\begin{proof}
By direct computation,
$$
\square(Z^+f) = 2\nabla f + Z^+\square f,
$$
hence
$$
2\nabla\square f = \square(Z^+\square f) - Z^+\square^2f
= \square^2(Z^+ f) - 2\square\nabla f - Z^+\square^2f,
$$
$$
4\nabla\square f = \square^2(Z^+f)-Z^+\square^2f.
$$
The last equation implies the statement about the QLAR functions.
The QRAR case is similar.
\end{proof}
  
\begin{lem}  \label{harm-comp-lem}
If $h(w_1,w_2)$ is a holomorphic function of two variables $\BB C^2 \to \BB C$
satisfying
$$
\frac{\partial^2 h}{\partial w_1^2} + \frac{\partial^2 h}{\partial w_2^2} =0,
$$
then the function
$$
h \bigl( z^0, \sqrt{(z^1)^2+(z^2)^2+(z^3)^2} \bigr) : \HC \to \BB C, \qquad
Z=e_0z^0+e_1z^1+e_2z^2+e_3z^3 \in \HC,
$$
is biharmonic.
\end{lem}

\begin{proof}
The proof is by direct computation of
$\square^2 h \bigl( z^0, \sqrt{(z^1)^2+(z^2)^2+(z^3)^2} \bigr)$.
\end{proof}

\begin{lem}
Functions $\exp(\lambda Z^+) = \sum_{n=0}^{\infty} (\lambda Z^+)^n/n!$,
$\lambda \in \BB R$, and $(Z^+)^n$, $n \in \BB Z$, satisfy
$$
\nabla \square \exp(\lambda Z^+) = 0, \quad
\nabla \square (Z^+)^n = 0 \qquad \text{and} \qquad
\square \exp(\lambda Z^+) \overleftarrow{\nabla} = 0, \quad
\square (Z^+)^n \overleftarrow{\nabla} = 0.
$$
\end{lem}

\begin{proof}
Write $Z=z_0+z$, where $z_0=\re Z$ and $z = \im Z$, and
let $|z|= \sqrt{(z^1)^2+(z^2)^2+(z^3)^2}$.
By Lemma \ref{harm-comp-lem}, the function
$e^{\lambda z_0} \cos(\lambda|z|)$ is biharmonic.
On the other hand,
$$
\exp(\lambda Z^+) = e^{\lambda z_0} \sum_{n=0}^{\infty} (-\lambda z)^n/n!
= e^{\lambda z_0} \Bigl( \cos(\lambda|z|)
- \frac{z}{|z|} \sin (\lambda|z|) \Bigr)
= \frac1{\lambda} \nabla^+ e^{\lambda z_0} \cos(\lambda|z|),
$$
which proves
$\nabla \square \exp(\lambda Z^+)
= \square \exp(\lambda Z^+) \overleftarrow{\nabla} = 0$.

Differentiating $\exp(\lambda Z^+)$ with respect to $\lambda$ and setting
$\lambda=0$ proves
$$
\nabla \square (Z^+)^n = \square (Z^+)^n \overleftarrow{\nabla} = 0
$$
for non-negative integer powers.
Applying the inversion
$\bigl( \begin{smallmatrix} 0 & 1 \\ 1 & 0 \end{smallmatrix} \bigr)
\in GL(2,\HC)$ to $(Z^+)^n$ proves
$\nabla \square (Z^+)^n = \square (Z^+)^n \overleftarrow{\nabla} = 0$
for negative integer powers.
\end{proof}

\begin{cor}
The columns of $(Z^+)^n$, $n \in \BB Z$, are QLAR;
the rows of $(Z^+)^n$, $n \in \BB Z$, are QRAR.
\end{cor}

\begin{cor}  \label{series-cor}
Convergent ``power'' series of the form
$$  
\sum_{n \ge 0} \text{first column of } (Z^+)^n a_n
+ \sum_{n \ge 0} \text{second column of } (Z^+)^n b_n, \qquad a_n, b_n \in \HC,
$$
are QLAR.
Convergent ``power'' series of the form
$$  
\sum_{n \ge 0} c_n \cdot \text{ first row of } (Z^+)^n
+ \sum_{n \ge 0} d_n \cdot \text{ second row of } (Z^+)^n,
\qquad c_n, d_n \in \HC,
$$
are QRAR.
\end{cor}

\subsection{Reproducing Formulas}

In this subsection we prove reproducing formulas for the quasi left and right
anti regular functions similar to the Cauchy-Fueter formulas.
We follow the argument presented in \cite{LR}.

Recall the quaternionic-valued $3$-form on $\BB H$
$$
Dx= e_0 dx^1 \wedge dx^2 \wedge dx^3 - e_1 dx^0 \wedge dx^2 \wedge dx^3
+ e_2 dx^0 \wedge dx^1 \wedge dx^3 - e_3 dx^0 \wedge dx^1 \wedge dx^2
$$
that appears in the Cauchy-Fueter formulas. Its quaternionic conjugate
$$
(Dx)^+ = e_0 dx^1 \wedge dx^2 \wedge dx^3 + e_1 dx^0 \wedge dx^2 \wedge dx^3
- e_2 dx^0 \wedge dx^1 \wedge dx^3 + e_3 dx^0 \wedge dx^1 \wedge dx^2
$$
has the property that, if $f$ and $g$ are differentiable functions,
\begin{equation*}  %\label{dfDg}
d( g \cdot (Dx)^+ \cdot f)
= \bigl( (g \overleftarrow{\nabla})f + g(\nabla f) \bigr)dV, \qquad
dV = dx^0 \wedge dx^1 \wedge dx^2 \wedge dx^3.
\end{equation*}

\begin{thm}  \label{QAR-repro-thm}
Let $U \subset \BB H$ be an open bounded subset with piecewise ${\cal C}^1$
boundary $\partial U$, and let $\overrightarrow{n}$ denote the outward
pointing normal unit vector to $\partial U$.
Suppose that $f(X)$ is quasi left anti regular on a
neighborhood of the closure $\overline{U}$, then
\begin{multline}  \label{Cauchy-Fueter-left}
\frac1{2\pi^2} \int_{X \in \partial U}
\frac{X-X_0}{N(X-X_0)^2} \cdot (Dx)^+ \cdot f(X)
+ \frac1{8\pi^2} \int_{X \in \partial U}
\frac{\partial}{\partial \overrightarrow{n}} \frac{X-X_0}{N(X-X_0)}
\cdot \nabla f(X) \,dS  \\
- \frac1{8\pi^2} \int_{X \in \partial U} \frac{X-X_0}{N(X-X_0)} \cdot
\frac{\partial}{\partial \overrightarrow{n}} \nabla f(X) \,dS
= \begin{cases}
f(X_0) & \text{if $X_0 \in U$;} \\
0 & \text{if $X_0 \notin \overline{U}$.}
\end{cases}
\end{multline}
If $g(X)$ is quasi right anti regular on a neighborhood of the closure
$\overline{U}$, then
\begin{multline}   \label{Cauchy-Fueter-right}
\frac1{2\pi^2} \int_{X \in \partial U} g(X) \cdot (Dx)^+ \cdot
\frac{X-X_0}{N(X-X_0)^2}
+ \frac1{8\pi^2} \int_{X \in \partial U} (g \overleftarrow{\nabla})(X) \cdot
\frac{\partial}{\partial \overrightarrow{n}} \frac{X-X_0}{N(X-X_0)}\,dS  \\
- \frac1{8\pi^2} \int_{X \in \partial U}
\frac{\partial}{\partial \overrightarrow{n}}
(g \overleftarrow{\nabla})(X) \cdot \frac{X-X_0}{N(X-X_0)} \,dS
= \begin{cases}
g(X_0) & \text{if $X_0 \in U$;} \\
0 & \text{if $X_0 \notin \overline{U}$.}
\end{cases}
\end{multline}
\end{thm}

\begin{proof}
We give a proof for the case $f$ is QLAR and $X_0 \in U$ only;
the other cases are similar and easier.
%Recall that, in general, if $f$ and $g$ are sufficiently smooth functions,
%$$
%d( g \cdot (Dx)^+ \cdot f)
%= \bigl( (g \overleftarrow{\nabla})f + g(\nabla f) \bigr) dV.
%$$
We have:
$$
\square_X \frac{X-Y}{N(X-Y)} = \square_Y \frac{X-Y}{N(X-Y)}
= -4 \frac{X-Y}{N(X-Y)^2}.
$$
Let $S^3_{\epsilon}(X_0)$ be the $3$-sphere in $\BB H$ of radius $\epsilon$
centered at $X_0$.
Then, for any $\BB S$-valued function $f$, we have:
\begin{multline*}
8\pi^2 \cdot f(X_0)
= 4 \lim_{\epsilon \to 0^+} \int_{X \in S^3_{\epsilon}(X_0)}
\frac{X-X_0}{N(X-X_0)^2} \cdot (Dx)^+ \cdot f(X)  \\
= 4 \int_{X \in \partial U} \frac{X-X_0}{N(X-X_0)^2} \cdot (Dx)^+ \cdot f(X)
- 4 \lim_{\epsilon \to 0^+} \int_{X \in U \setminus B_{\epsilon}(X_0)}
\frac{X-X_0}{N(X-X_0)^2} \cdot \nabla f(X) \,dV  \\
= - \int_{X \in \partial U} \square_X \frac{X-X_0}{N(X-X_0)} \cdot (Dx)^+ \cdot f(X)
+ \lim_{\epsilon \to 0^+} \int_{X \in U \setminus B_{\epsilon}(X_0)}
\square_X \frac{X-X_0}{N(X-X_0)} \cdot \nabla f(X) \,dV  \\
= - \int_{X \in \partial U} \square_X \frac{X-X_0}{N(X-X_0)} \cdot (Dx)^+ \cdot f(X)
+\int_{X \in \partial U} \frac{\partial}{\partial \overrightarrow{n}}
\frac{X-X_0}{N(X-X_0)} \cdot \nabla f(X) \,dS  \\
- \int_{X \in \partial U} \frac{X-X_0}{N(X-X_0)} \cdot
\frac{\partial}{\partial \overrightarrow{n}} \nabla f(X) \,dS
+ \int_{X \in U} \frac{X-X_0}{N(X-X_0)} \cdot \nabla \square f(X) \,dV.
\end{multline*}
At the last step we used Green's identity
$$
\int_U u \square v = \int_U v \square u
+ \int_{\partial U} u \frac{\partial v}{\partial \overrightarrow{n}}
- \int_{\partial U} v \frac{\partial u}{\partial \overrightarrow{n}}
$$
with $u = \nabla f(X)$ and $v = \frac{X-X_0}{N(X-X_0)}$.
If $f$ is QLAR, the integral over $U$ is zero, and we obtain
\eqref{Cauchy-Fueter-left}.

The proof of \eqref{Cauchy-Fueter-right} is similar.
\end{proof}

%Similarly, for any $\BB S'$-valued function $g$, we have:
%\begin{multline*}
%8\pi^2 \cdot g(X_0)
%= 4\int_{X \in \partial U} g(X) \cdot (Dx)^+ \cdot \frac{X-X_0}{N(X-X_0)^2}
%- 4\int_{X \in U} (g \overleftarrow{\nabla})(X) \cdot
%\frac{X-X_0}{N(X-X_0)^2} \,dV  \\
%= - \int_{X \in \partial U} g(X) \cdot (Dx)^+ \cdot \square_X \frac{X-X_0}{N(X-X_0)}
%+ \int_{X \in U} (g \overleftarrow{\nabla})(X) \cdot
%\square_X \frac{X-X_0}{N(X-X_0)} \,dV  \\
%= - \int_{X \in \partial U} g(X) \cdot (Dx)^+ \cdot \square_X \frac{X-X_0}{N(X-X_0)}
%+ \int_{X \in \partial U} (g \overleftarrow{\nabla})(X) \cdot
%\frac{\partial}{\partial \overrightarrow{n}} \frac{X-X_0}{N(X-X_0)}\,dS  \\
%- \int_{X \in \partial U} \frac{\partial}{\partial \overrightarrow{n}}
%(g \overleftarrow{\nabla})(X) \cdot \frac{X-X_0}{N(X-X_0)} \,dS
%+ \int_{X \in U}
%(\square g \overleftarrow{\nabla})(X) \cdot \frac{X-X_0}{N(X-X_0)} \,dV.
%\end{multline*}
%We used Green's identity
%$$
%\int_U u \square v = \int_U v \square u
%+ \int_{\partial U} u \frac{\partial v}{\partial \overrightarrow{n}}
%- \int_{\partial U} v \frac{\partial u}{\partial \overrightarrow{n}}
%$$
%with $u = (g \overleftarrow{\nabla})(X)$ and $v = \frac{X-X_0}{N(X-X_0)}$.
%If $g$ is QRAR, the integral over $U$ is zero, and we obtain
%\eqref{Cauchy-Fueter-right}.

In the special case when the open set $U$ is $B_R$ -- the open ball in $\BB H$
of radius $R$ centered at the origin, $\partial U = S^3_R$
-- the $3$-sphere of radius $R$ centered at the origin -- and
$$
\frac{\partial}{\partial \overrightarrow{n}} = \frac1R \deg, \qquad
(Dx)^+ \Bigr|_{S^3_R} = RX^{-1}\,dS
$$
(Lemma 6 in \cite{FL1}). Thus we obtain:

\begin{cor}
Suppose that $f(X)$ is quasi left anti regular on a
neighborhood of the closure $\overline{B_R}$, then
\begin{multline*}
\frac{R}{2\pi^2} \int_{X \in S^3_R}
\frac{X-X_0}{N(X-X_0)^2} \cdot X^{-1} \cdot f(X) \,dS
+ \frac1{8\pi^2} \int_{X \in S^3_R} \Bigl( \deg_X \frac{X-X_0}{N(X-X_0)} \Bigr)
\cdot \nabla f(X) \,\frac{dS}R  \\
- \frac1{8\pi^2} \int_{X \in S^3_R} \frac{X-X_0}{N(X-X_0)}
\cdot \deg \nabla f(X) \,\frac{dS}R
= \begin{cases}
f(X_0) & \text{if $X_0 \in B_R$;} \\
0 & \text{if $X_0 \notin \overline{B_R}$.}
\end{cases}
\end{multline*}
If $g(X)$ is quasi right anti regular on a neighborhood of the closure
$\overline{B_R}$, then
\begin{multline*}
\frac{R}{2\pi^2}
\int_{X \in S^3_R} g(X) \cdot X^{-1} \cdot \frac{X-X_0}{N(X-X_0)^2}\,dS
+ \frac1{8\pi^2} \int_{X \in S^3_R} (g \overleftarrow{\nabla})(X) \cdot
\deg_X \frac{X-X_0}{N(X-X_0)}\,\frac{dS}R  \\
- \frac1{8\pi^2} \int_{X \in S^3_R}
\Bigl( \deg (g \overleftarrow{\nabla})(X) \Bigr)
\cdot \frac{X-X_0}{N(X-X_0)} \,\frac{dS}R
= \begin{cases}
g(X_0) & \text{if $X_0 \in B_R$;} \\
0 & \text{if $X_0 \notin \overline{B_R}$.}
\end{cases}
\end{multline*}
\end{cor}

%Let $f \in {\cal BH}^+$, then $\tilde f = \nabla^+ f: \BB H \to \BB H$
%is a function satisfying $\nabla \square \tilde f =0$, and the reproducing
%formulas from the previous section apply to $\tilde f$.

\section{$K$-Types of Quasi Anti Regular Functions}  \label{Sect4}

In this section we study the polynomial quasi left and right regular functions.
We identify the $K$-types of these functions with respect to the action of
$SU(2) \times SU(2)$ realized as the subgroup of diagonal matrices in
$GL(2,\BB H)$.
We spell out the bases of these $K$-types as well as actions of
$\mathfrak{gl}(2,\HC)$ in these bases.
As an immediate consequence, we see that certain spaces of
polynomial quasi left and right anti regular functions form irreducible
representations of $\mathfrak{gl}(2,\HC)$ (Proposition \ref{irred-prop}).
In subsequent sections we use these results to construct an invariant
bilinear pairing between left and right quasi regular functions,
pseudounitary structures and two expansions of the reproducing kernel.

\subsection{Polynomial Quasi Anti Regular Functions}  \label{polynomial-subsection}

We introduce the following spaces of quasi left and right anti regular
functions:
%$$
%{\cal U}^+ = \{ f \in \BB C[z^0,z^1,z^2,z^3] \otimes \BB S ;\:
%\nabla\square f =0 \},
%$$
$$
{\cal U} = \{ f \in \BB C[z^0,z^1,z^2,z^3,N(Z)^{-1}] \otimes \BB S ;\:
\nabla\square f =0 \},
$$
$$
{\cal U}(d) = \{ f \in {\cal U};\: \text{$f$ is homogeneous of degree $d$} \},
\qquad d \in \BB Z.
$$
Then
$$
{\cal U} = \bigoplus_{d \in \BB Z} {\cal U}(d) = {\cal U}^+ \oplus {\cal U}^-,
\quad \text{where}
$$
$$
{\cal U}^+ = \bigoplus_{d \in \BB Z,\: d \ge 0} {\cal U}(d), \qquad
{\cal U}^- = \bigoplus_{d \in \BB Z,\: d < 0} {\cal U}(d).   
$$
Note that ${\cal U}$ and a subspace
$$
\{ f \in \BB C[z^0,z^1,z^2,z^3] \otimes \BB S ;\: \nabla\square f =0 \}
\subset {\cal U}^+
$$
are preserved by the $\pi'_l$ action of $\mathfrak{gl}(2,\HC)$.

Similarly, we introduce
%$$
%{\cal U}'^+ = \{ g \in \BB C[z^0,z^1,z^2,z^3] \otimes \BB S' ;\:
%(\square g) \overleftarrow{\nabla} =0 \},
%$$
$$
{\cal U}' = \{ g \in \BB C[z^0,z^1,z^2,z^3,N(Z)^{-1}] \otimes \BB S ;\:
(\square g) \overleftarrow{\nabla} =0 \},
$$
$$
{\cal U}'(d) = \{ g \in {\cal U}';\: \text{$g$ is homogeneous of degree $d$} \},
\qquad d \in \BB Z.
$$
Then
$$
{\cal U}' = \bigoplus_{d \in \BB Z} {\cal U}'(d) = {\cal U}'^+ \oplus {\cal U}'^-,
\quad \text{where}
$$
$$
{\cal U}'^+ = \bigoplus_{d \in \BB Z,\: d \ge 0} {\cal U}'(d), \qquad
{\cal U}'^- = \bigoplus_{d \in \BB Z,\: d < 0} {\cal U}'(d).   
$$
Note that ${\cal U}'$ and a subspace
$$
\{ g \in \BB C[z^0,z^1,z^2,z^3] \otimes \BB S' ;\:
(\square g) \overleftarrow{\nabla} =0 \} \subset {\cal U}'^+
$$
are preserved by the $\pi'_r$ action of $\mathfrak{gl}(2,\HC)$.

\begin{lem}  \label{degree-decomp-lem}
Let $\lambda \in \BB C$ with $|\lambda|=1$, and consider an element
$\Lambda = \left( \begin{smallmatrix} \lambda & 0 & 0 & 0 \\
  0 & \lambda & 0 & 0 \\ 0 & 0 & \lambda^{-1} & 0 \\
  0 & 0 & 0 & \lambda^{-1} \end{smallmatrix} \right) \in U(2) \times U(2)$.
Then, for $f \in {\cal U}(d)$ and $g \in {\cal U}'(d)$, we have:
$$
\pi'_l(\Lambda) f = \lambda^{-2d-1} f \quad \text{and} \quad
\pi'_r(\Lambda) g = \lambda^{-2d-1} g.
$$
\end{lem}

\begin{proof}
The result follows immediately from equations \eqref{pi'_l}-\eqref{pi'_r}.
\end{proof}

\begin{lem}  \label{biharm-decomp-lem}
Let $f \in {\cal U}(d)$, then $f$ can be written uniquely as
$f=\phi + N(Z) \cdot f'$, where $\phi$ is harmonic of degree $d$,
while $f'$ has degree $d-2$ and satisfies $\nabla f'=0$:
$$
\phi \in {\cal H}(d) \otimes \BB S =
\{ h \in \BB C[z^0,z^1,z^2,z^3,N(Z)^{-1}] \otimes \BB S ;\:
\square h =0,\: \text{$h$ is homogeneous of degree $d$} \},
$$
$$
f' = \{ f \in \BB C[z^0,z^1,z^2,z^3,N(Z)^{-1}] \otimes \BB S ;\:
\nabla f =0,\: \text{$f$ is homogeneous of degree $d-2$} \}.
$$
Conversely, each sum of this form $\phi + N(Z) \cdot f'$
belongs to ${\cal U}(d)$.

Similarly, let $g \in {\cal U}'(d)$, then $g$ can be written uniquely as
$g=\psi + N(Z) \cdot g'$, where $\psi$ is harmonic of degree $d$,
while $g'$ has degree $d-2$ and satisfies $g' \overleftarrow{\nabla}=0$:
$$
\psi \in {\cal H}(d) \otimes \BB S' =
\{ h \in \BB C[z^0,z^1,z^2,z^3,N(Z)^{-1}] \otimes \BB S' ;\:
\square h =0,\: \text{$h$ is homogeneous of degree $d$} \},
$$
$$
g' = \{ g \in \BB C[z^0,z^1,z^2,z^3,N(Z)^{-1}] \otimes \BB S' ;\:
g \overleftarrow{\nabla} =0,\: \text{$g$ is homogeneous of degree $d-2$} \}.
$$
Conversely, each sum of this form $\psi + N(Z) \cdot g'$ belongs to
${\cal U}'(d)$.
\end{lem}

\begin{proof}
Write $f= \bigl( \begin{smallmatrix} f_1 \\ f_2 \end{smallmatrix} \bigr)$
for some $f_1, f_2 \in \BB C[z^0,z^1,z^2,z^3,N(Z)^{-1}]$.
Since $\square^2 f_1 = \square^2 f_2 =0$, by Proposition \ref{biharmonic-prop},
we have $f_1 = \phi_1 + N(Z) \cdot f'_1$ and $f_2 = \phi_2 + N(Z) \cdot f'_2$
for some harmonic functions
$\phi_1, \phi_2, f'_1, f'_2 \in \BB C[z^0,z^1,z^2,z^3,N(Z)^{-1}]$
of appropriate degrees.
The harmonic functions $\phi_1,\phi_2,f_1,f_2$ are unique, except when $d=0$,
in which case we may assume that $f'_1=f'_2=0$. Let
$\phi = \bigl( \begin{smallmatrix} \phi_1 \\ \phi_2 \end{smallmatrix} \bigr)$,
$f'= \bigl( \begin{smallmatrix} f'_1 \\ f'_2 \end{smallmatrix} \bigr)$.
It remains to show that $f'$ satisfies $\nabla f'=0$.
By direct calculation,
\begin{equation}  \label{Laplacian-Nf}
\square \bigl( N(Z)^k \cdot h(Z) \bigr) =
4kN(Z)^{k-1} \cdot (\deg+k+1)h(Z) +N(Z)^k \cdot \square h(Z).
\end{equation}
%Recall equation (33) from \cite{ATMP}:
%$$
%\square \bigl( N(Z)^k \cdot h_l(Z) \bigr) = 4k(d'+k+1) N(Z)^{k-1} \cdot h_{d'}(Z),
%$$
%$$
%h_{d'} \in {\cal H}(d') = \{ h \in \BB C[z^0,z^1,z^2,z^3,N(Z)^{-1}] ;\:
%\square h =0,\: \text{$h$ is homogeneous of degree $d'$} \},
%$$
%Considering separately the cases $d>0$ and $d<0$, we see that
Hence,
$$
\square f = \square (N(Z) \cdot f') = 4d f',
$$
and $\nabla\square f=0$ if and only if $\nabla f'=0$.

The case of $\BB S'$-valued functions is similar.
\end{proof}

\begin{cor}
For $d \ge 0$, we have:
$$
{\cal U}^+(d) = \{ f \in \BB C[z^0,z^1,z^2,z^3] \otimes \BB S ;\:
\nabla\square f =0,\: \text{$f$ is homogeneous of degree $d$} \},
$$
$$
{\cal U}'^+(d) = \{ g \in \BB C[z^0,z^1,z^2,z^3] \otimes \BB S' ;\:
(\square g)\overleftarrow{\nabla} =0,\:
\text{$g$ is homogeneous of degree $d$} \}.
$$
In particular,
$$
{\cal U}^+ = \{ f \in \BB C[z^0,z^1,z^2,z^3] \otimes \BB S ;\:
\nabla\square f =0 \},
$$
$$
{\cal U}'^+ = \{ g \in \BB C[z^0,z^1,z^2,z^3] \otimes \BB S' ;\:
(\square g)\overleftarrow{\nabla} =0 \}.
$$
\end{cor}

\begin{proof}
  The result follows from the corresponding statements for harmonic
  and anti regular functions: for $d \ge 0$,
\begin{align*}
{\cal H}(d) &=
\{ h \in \BB C[z^0,z^1,z^2,z^3,N(Z)^{-1}] ;\:
\square h =0,\: \text{$h$ is homogeneous of degree $d$} \}  \\
&= \{ h \in \BB C[z^0,z^1,z^2,z^3] ;\:
\square h =0,\: \text{$h$ is homogeneous of degree $d$} \},
\end{align*}
\begin{multline*}
\{ f \in \BB C[z^0,z^1,z^2,z^3,N(Z)^{-1}] \otimes \BB S ;\:
\nabla f =0,\: \text{$f$ is homogeneous of degree $d$} \}  \\
= \{ f \in \BB C[z^0,z^1,z^2,z^3] \otimes \BB S ;\:
\nabla f =0,\: \text{$f$ is homogeneous of degree $d$} \},
\end{multline*}
\begin{multline*}
\{ g \in \BB C[z^0,z^1,z^2,z^3,N(Z)^{-1}] \otimes \BB S' ;\:
g \overleftarrow{\nabla} =0,\: \text{$g$ is homogeneous of degree $d$} \}  \\
= \{ g \in \BB C[z^0,z^1,z^2,z^3] \otimes \BB S' ;\:
g \overleftarrow{\nabla} =0,\: \text{$g$ is homogeneous of degree $d$} \},
\end{multline*}
  and Lemma \ref{biharm-decomp-lem}.
\end{proof}

As usual, we realize $\mathfrak{sl}(2,\BB C) \times \mathfrak{sl}(2,\BB C)$
as diagonal elements of $\mathfrak{gl}(2,\HC)$
(recall equation \eqref{sl2xsl2}).
%$$
%\mathfrak{sl}(2,\BB C) \times \mathfrak{sl}(2,\BB C) = \left\{
%\left(\begin{smallmatrix} A & 0 \\ 0 & D \end{smallmatrix}\right)
%\in \mathfrak{gl}(2,\HC);\: A,D \in \HC, \re(A)=\re(D)=0 \right\}.
%$$

\begin{prop}  \label{K-types-prop}
Each ${\cal U}(d)$ is invariant under the $\pi'_l$ action restricted to
$\mathfrak{sl}(2,\BB C) \times \mathfrak{sl}(2,\BB C)$, and we have the
following decomposition into irreducible components:
\begin{align}
{\cal U}(2l) &= \bigl( V_{l+\frac12} \boxtimes V_l \bigr) \oplus
\bigl( V_{l-\frac12} \boxtimes V_l \bigr) \oplus
\bigl( V_{l-\frac12} \boxtimes V_{l-1} \bigr),  \label{U(2l)-decomp}  \\
{\cal U}(-2l-1) &= \bigl( V_l \boxtimes V_{l+\frac12} \bigr) \oplus
\bigl( V_l \boxtimes V_{l-\frac12} \bigr) \oplus
\bigl( V_{l-1} \boxtimes V_{l-\frac12} \bigr),
\end{align}
$l=0,\frac12,1,\frac32,\dots$.
Similarly, each ${\cal U}'(d)$ is invariant under the $\pi'_r$ action
restricted to $\mathfrak{sl}(2,\BB C) \times \mathfrak{sl}(2,\BB C)$,
and we have the following decomposition into irreducible components:
\begin{align*}
{\cal U}'(2l) &= \bigl( V_l \boxtimes V_{l+\frac12} \bigr) \oplus
\bigl( V_l \boxtimes V_{l-\frac12} \bigr) \oplus
\bigl( V_{l-1} \boxtimes V_{l-\frac12} \bigr),  \\
{\cal U}'(-2l-1) &= \bigl( V_{l+\frac12} \boxtimes V_l \bigr) \oplus
\bigl( V_{l-\frac12} \boxtimes V_l \bigr) \oplus
\bigl( V_{l-\frac12} \boxtimes V_{l-1} \bigr),
\end{align*}
$l=0,\frac12,1,\frac32,\dots$.
In particular, $\dim_{\BB C} {\cal U}(d) = \dim_{\BB C} {\cal U}'(d) = 3d^2+3d+2$.
\end{prop}

\begin{proof}
From Lemma \ref{biharm-decomp-lem},
$\dim_{\BB C} {\cal U}(d) = \dim_{\BB C} {\cal U}'(d) = 3d^2+3d+2$.
Then, for example, ${\cal U}(2l)$ is an
($\mathfrak{sl}(2,\BB C) \times \mathfrak{sl}(2,\BB C)$)-invariant subspace of
\begin{multline*}
\Bigl( \bigl( V_l \boxtimes V_l \bigr) \otimes
\bigl( V_{\frac12} \boxtimes V_0 \bigr) \Bigr)
\oplus \Bigl( \bigl( V_{l-1} \boxtimes V_{l-1} \bigr) \otimes
\bigl( V_{\frac12} \boxtimes V_0 \bigr) \Bigr)  \\
= \bigl( V_{l+\frac12} \boxtimes V_l \bigr) \oplus
\bigl( V_{l-\frac12} \boxtimes V_l \bigr) \oplus
\bigl( V_{l-\frac12} \boxtimes V_{l-1} \bigr) \oplus
\bigl( V_{l-\frac32} \boxtimes V_{l-1} \bigr),
\end{multline*}
hence ${\cal U}(2l)$ must contain some of these four irreducible components.
The harmonic functions produce
$$
\bigl( V_l \boxtimes V_l \bigr) \otimes \bigl( V_{\frac12} \boxtimes V_0 \bigr)
= \bigl( V_{l+\frac12} \boxtimes V_l \bigr) \oplus
\bigl( V_{l-\frac12} \boxtimes V_l \bigr).
$$
And the functions of the type $N(Z) \cdot f'$, where $f'$ has degree $2l-2$
and satisfies $\nabla f=0$, form an invariant subspace of
$$
\bigl( V_{l-1} \boxtimes V_{l-1} \bigr) \otimes
\bigl( V_{\frac12} \boxtimes V_0 \bigr)
= \bigl( V_{l-\frac12} \boxtimes V_{l-1} \bigr) \oplus
\bigl( V_{l-\frac32} \boxtimes V_{l-1} \bigr).
$$
Comparing the dimensions, we see that \eqref{U(2l)-decomp} is the only
possibility.
\end{proof}

\subsection{$K$-Type Basis of Quasi Anti Regular Functions}

In this subsection we find a basis for each ${\cal U}(d)$ and ${\cal U}'(d)$.
These bases consist of weight vectors of the irreducible components of
${\cal U}(d)$, ${\cal U}'(d)$ described in Proposition \ref{K-types-prop}
relative to the maximal torus of $SU(2) \times SU(2)$
$$
\left\{ \Delta(a,b) =
\left( \begin{smallmatrix} a & 0 & 0 & 0 \\ 0 & a^{-1} & 0 & 0 \\
  0 & 0 & b^{-1} & 0 \\ 0 & 0 & 0 & b \end{smallmatrix} \right)
\in SU(2) \times SU(2);\: a, b \in \BB C,\: |a|=|b|=1 \right\}.
$$
First, we find three families of QLAR functions and three families of QRAR
functions for $d \ge 0$.

\begin{prop}  \label{K-typebasis+_prop}
The functions
\begin{align*}
f^{(1)}_{l,m,n}(Z) &= \begin{pmatrix} t^l_{n-\frac12 \,\underline{m}}(Z) \\
- t^l_{n+\frac12 \,\underline{m}}(Z) \end{pmatrix},
\qquad
\begin{smallmatrix}
m =-l ,-l+1,\dots,l,  \\
n =-l-\frac12,-l+\frac12,\dots,l+\frac12,
\end{smallmatrix}  \\
f^{(2)}_{l,m,n}(Z) &= \begin{pmatrix} (l-n+\frac12) t^l_{n-\frac12 \,\underline{m}}(Z)\\
(l+n+\frac12) t^l_{n+\frac12 \,\underline{m}}(Z) \end{pmatrix},
\qquad
\begin{smallmatrix}
m =-l ,-l+1,\dots,l,  \\
n =-l+\frac12,-l+\frac32,\dots,l-\frac12,
\end{smallmatrix}  \\
f^{(3)}_{l,m,n}(Z)
&= \begin{pmatrix} N(Z) \cdot t^{l-1}_{n-\frac12 \,\underline{m}}(Z) \\
-N(Z) \cdot  t^{l-1}_{n+\frac12 \,\underline{m}}(Z) \end{pmatrix},
\qquad
\begin{smallmatrix}
m =-l+1 ,-l+2,\dots,l-1,  \\
n =-l+\frac12,-l+\frac32,\dots,l-\frac12,
\end{smallmatrix}
\end{align*}
span respectively the $V_{l+\frac12} \boxtimes V_l$,
$V_{l-\frac12} \boxtimes V_l$ and $V_{l-\frac12} \boxtimes V_{l-1}$
components of ${\cal U}(2l)$, $l=0,\frac12,1,\frac32,\dots$.

The functions
\begin{align*}
g^{(1)}_{l,m,n}(Z) &= \bigl( (l+m+\tfrac12) t^l_{n\,\underline{m-\frac12}}(Z),
-(l-m+\tfrac12) t^l_{n\,\underline{m+\frac12}}(Z) \bigr),
\qquad
\begin{smallmatrix}
m =-l-\frac12 ,-l+\frac12,\dots,l+\frac12,  \\
n =-l,-l+1,\dots,l,
\end{smallmatrix}  \\
g^{(2)}_{l,m,n}(Z) &= \bigl( t^l_{n\,\underline{m-\frac12}}(Z),
t^l_{n\,\underline{m+\frac12}}(Z) \bigr),
\qquad
\begin{smallmatrix}
m =-l+\frac12 ,-l+\frac32,\dots,l-\frac12,  \\
n =-l,-l+1,\dots,l,
\end{smallmatrix}  \\
g^{(3)}_{l,m,n}(Z) &= N(Z) \cdot
\bigl( (l+m-\tfrac12) t^{l-1}_{n\,\underline{m-\frac12}}(Z),
-(l-m-\tfrac12) t^{l-1}_{n\,\underline{m+\frac12}}(Z) \bigr),
\qquad
\begin{smallmatrix}
m =-l+\frac12 ,-l+\frac32,\dots,l-\frac12,  \\
n =-l+1,-l+2,\dots,l-1,
\end{smallmatrix}
\end{align*}
span respectively the $V_l \boxtimes V_{l+\frac12}$,
$V_l \boxtimes V_{l-\frac12}$ and $V_{l-1} \boxtimes V_{l-\frac12}$
components of ${\cal U}'(2l)$, $l=0,\frac12,1,\frac32,\dots$.
\end{prop}

%And we find three families of QRAR functions:
%$$
%g^{(1)}_{l,m,n}(Z) = \bigl( (l+m) t^{l-\frac12}_{n\,\underline{m-\frac12}}(Z),
%-(l-m) t^{l-\frac12}_{n\,\underline{m+\frac12}}(Z) \bigr),
%\qquad
%\begin{smallmatrix}
%m =-l,-l+1,\dots,l,  \\
%n =-l+\frac12,-l+\frac32,\dots,l-\frac12,
%\end{smallmatrix}
%$$
%$$
%g^{(2)}_{l,m,n}(Z) = \bigl( t^{l-\frac12}_{n\,\underline{m-\frac12}}(Z),
%t^{l-\frac12}_{n\,\underline{m+\frac12}}(Z) \bigr),
%\qquad
%\begin{smallmatrix}
%m =-l+1 ,-l+2,\dots,l-1,  \\
%n =-l+\frac12,-l+\frac32,\dots,l-\frac12,
%\end{smallmatrix}
%$$
%$$
%g^{(3)}_{l,m,n}(Z)
%= N(Z) \cdot \bigl( (l+m-1) t^{l-\frac32}_{n\,\underline{m-\frac12}}(Z),
%-(l-m-1) t^{l-\frac32}_{n\,\underline{m+\frac12}}(Z) \bigr),
%\qquad
%\begin{smallmatrix}
%m =-l+1,-l+2,\dots,l-1,  \\
%n =-l+\frac32,-l+\frac52,\dots,l-\frac32;
%\end{smallmatrix}
%$$
%they span respectively the $V_{l-\frac12} \boxtimes V_l$,
%$V_{l-\frac12} \boxtimes V_{l-1}$ and $V_{l-\frac32} \boxtimes V_{l-1}$
%components of ${\cal U}'(2l-1)$, $l=\frac12,1,\frac32,2,\dots$.

\begin{rem}
The functions $f^{(1)}_{l,m,n}(Z)$'s form a $K$-type basis of the space
of left anti regular functions,
$f^{(3)}_{l,m,n}(Z) = N(Z) \cdot f^{(1)}_{l-1,m,n}(Z)$,
and, by Lemma 23 in \cite{FL1},
$$
f^{(2)}_{l,m,n}(Z) = Z \cdot
\begin{pmatrix} (l-m) t^{l-\frac12}_{n \,\underline{m+\frac12}}(Z) \\
(l+m) t^{l-\frac12}_{n \,\underline{m-\frac12}}(Z) \end{pmatrix}.
$$
Thus, $f^{(2)}_{l,m,n}(Z)$s are $Z$-multiples of the left regular
basis functions from Proposition 24 in \cite{FL1}.

Similarly, the functions $g^{(1)}_{l,m,n}(Z)$'s form a $K$-type basis of
the space of right anti regular functions,
$g^{(3)}_{l,m,n}(Z) = N(Z) \cdot g^{(1)}_{l-1,m,n}(Z)$,
and, by Lemma 23 in \cite{FL1},
$$
g^{(2)}_{l,m,n}(Z) = \bigl( t^{l-\frac12}_{n+\frac12\,\underline{m}}(Z),
t^{l-\frac12}_{n-\frac12\,\underline{m}}(Z) \bigr) \cdot Z.
$$
Thus, $g^{(2)}_{l,m,n}(Z)$s are $Z$-multiples of the right regular
basis functions from Proposition 24 in \cite{FL1}.
\end{rem}

\begin{proof}
By direct computation,
$$
\pi'_l(\Delta(a,b)) \begin{pmatrix} \alpha t^l_{n\,\underline{m}}(Z) \\
\beta t^{l'}_{n'\,\underline{m'}}(Z) \end{pmatrix}
= \begin{pmatrix} \alpha a^{2n+1} b^{2m} \cdot t^l_{n\,\underline{m}}(Z) \\
\beta a^{2n'-1} b^{2m'} \cdot t^{l'}_{n'\,\underline{m'}}(Z) \end{pmatrix},
$$
$$
\pi'_r(\Delta(a,b)) \bigl( \alpha t^l_{n\,\underline{m}}(Z),
\beta t^{l'}_{n'\,\underline{m'}}(Z) \bigr)
= \bigl( \alpha a^{2n} b^{2m+1} \cdot t^l_{n\,\underline{m}}(Z),
\beta a^{2n'} b^{2m'-1} \cdot t^{l'}_{n'\,\underline{m'}}(Z) \bigr).
$$
Introduce the following elements of
$\mathfrak{sl}(2,\BB C) \times \mathfrak{sl}(2,\BB C)$:
\begin{equation}  \label{EF}
E_1=\begin{pmatrix} \begin{smallmatrix} 0 & 1 \\ 0 & 0 \end{smallmatrix} & 0 \\
  0 & 0 \end{pmatrix}, \quad
F_1=\begin{pmatrix} \begin{smallmatrix} 0 & 0 \\ 1 & 0 \end{smallmatrix} & 0 \\
  0 & 0 \end{pmatrix}, \qquad
E_2 = \begin{pmatrix} 0 & 0 \\
  0 & \begin{smallmatrix} 0 & 0 \\ 1 & 0 \end{smallmatrix} \end{pmatrix}, \quad
F_2 = \begin{pmatrix} 0 & 0 \\
  0 & \begin{smallmatrix} 0 & 1 \\ 0 & 0 \end{smallmatrix} \end{pmatrix}.
\end{equation}
We can find their actions on ${\cal U}(d)$ and ${\cal U}'(d)$ using
Lemma 23 from \cite{FL1} together with
Lemma \ref{pi'-Lie_alg-action} and equation \eqref{dt}:
$$
\pi'_l(E_1) \begin{pmatrix} \alpha t^l_{n-\frac12 \,\underline{m}}(Z) \\
\beta t^l_{n+\frac12 \,\underline{m}}(Z) \end{pmatrix}
= \begin{pmatrix} \bigl(\beta -\alpha (l+n+\frac12) \bigr)
t^l_{n+\frac12 \,\underline{m}}(Z) \\
-\beta (l+n+\frac32) t^l_{n+\frac32 \,\underline{m}}(Z) \end{pmatrix},
$$
$$
\pi'_l(F_1) \begin{pmatrix} \alpha t^l_{n-\frac12 \,\underline{m}}(Z) \\
\beta t^l_{n+\frac12 \,\underline{m}}(Z) \end{pmatrix}
= \begin{pmatrix} -\alpha (l-n+\frac32) t^l_{n-\frac32 \,\underline{m}}(Z) \\
\bigl(\alpha -\beta (l-n+\frac12)\bigr) t^l_{n-\frac12 \,\underline{m}}(Z)
\end{pmatrix},
$$
$$
\pi'_l(E_2) \begin{pmatrix} \alpha t^l_{n-\frac12 \,\underline{m}}(Z) \\
\beta t^l_{n+\frac12 \,\underline{m}}(Z) \end{pmatrix}
= (l-m) \begin{pmatrix} \alpha t^l_{n-\frac12 \,\underline{m+1}}(Z) \\
\beta t^l_{n+\frac12 \,\underline{m+1}}(Z) \end{pmatrix},
$$
$$
\pi'_l(F_2) \begin{pmatrix} \alpha t^l_{n-\frac12 \,\underline{m}}(Z) \\
\beta t^l_{n+\frac12 \,\underline{m}}(Z) \end{pmatrix}
= (l+m) \begin{pmatrix} \alpha t^l_{n-\frac12 \,\underline{m-1}}(Z) \\
\beta t^l_{n+\frac12 \,\underline{m-1}}(Z) \end{pmatrix},
$$
$$
\pi'_r(E_1) \bigl( \alpha t^l_{n\,\underline{m-\frac12}}(Z),
\beta t^l_{n\,\underline{m+\frac12}}(Z) \bigr)
= -(l+n+1) \bigl( \alpha t^l_{n+1\,\underline{m-\frac12}}(Z),
\beta t^l_{n+1\,\underline{m+\frac12}}(Z) \bigr),
$$
$$
\pi'_r(F_1) \bigl( \alpha t^l_{n\,\underline{m-\frac12}}(Z),
\beta t^l_{n\,\underline{m+\frac12}}(Z) \bigr)
= -(l-n+1) \bigl( \alpha t^l_{n-1\,\underline{m-\frac12}}(Z),
\beta t^l_{n-1\,\underline{m+\frac12}}(Z) \bigr),
$$
\begin{multline*}
\pi'_r(E_2) \Bigl( \alpha t^l_{n\,\underline{m-\frac12}}(Z),
\beta t^l_{n\,\underline{m+\frac12}}(Z) \bigr)  \\
= \bigl( \bigl(\alpha (l-m+\tfrac12) -\beta\bigr) t^l_{n\,\underline{m+\frac12}}(Z),
\beta (l-m-\tfrac12) t^l_{n\,\underline{m+\frac32}}(Z) \Bigr),
\end{multline*}
\begin{multline*}
\pi'_r(F_2) \bigl( \alpha t^l_{n\,\underline{m-\frac12}}(Z),
\beta t^l_{n\,\underline{m+\frac12}}(Z) \bigr)  \\
= \Bigl( (l+m-\tfrac12) t^l_{n\,\underline{m-\frac32}}(Z),
\bigl(\beta (l+m+\tfrac12) - \alpha\bigr) t^l_{n\,\underline{m-\frac12}}(Z) \Bigr).
\end{multline*}
We see that each family of functions is invariant under the
$\mathfrak{sl}(2,\BB C) \times \mathfrak{sl}(2,\BB C)$-action and
spans an irreducible component of ${\cal U}(2l)$ or ${\cal U}'(2l)$
as stated.
\end{proof}

We find three more families of QLAR functions and
three more families of QRAR functions of negative degrees:

\begin{prop}  \label{K-typebasis-_prop}
The functions
\begin{align*}
\tilde f^{(1)}_{l,m,n}(Z) &= \begin{pmatrix}
t^{l+\frac12}_{m\,\underline{n+\frac12}}(Z^{-1}) \\
- t^{l+\frac12}_{m\,\underline{n-\frac12}}(Z^{-1}) \end{pmatrix},
\qquad
\begin{smallmatrix}
m =-l-\frac12 ,-l+\frac12,\dots,l+\frac12,  \\
n =-l,-l+1,\dots,l,
\end{smallmatrix}  \\
\tilde f^{(2)}_{l,m,n}(Z) &= \begin{pmatrix}
(l-n) N(Z)^{-1} \cdot t^{l-\frac12}_{m\,\underline{n+\frac12}}(Z^{-1}) \\
(l+n) N(Z)^{-1} \cdot t^{l-\frac12}_{m\,\underline{n-\frac12}}(Z^{-1}) \end{pmatrix},
\qquad
\begin{smallmatrix}
m =-l+\frac12 ,-l+\frac32,\dots,l-\frac12,  \\
n =-l,-l+1,\dots,l,
\end{smallmatrix}  \\
\tilde f^{(3)}_{l,m,n}(Z) &= \begin{pmatrix}
N(Z)^{-1} \cdot t^{l-\frac12}_{m\,\underline{n+\frac12}}(Z^{-1}) \\
-N(Z)^{-1} \cdot t^{l-\frac12}_{m\,\underline{n-\frac12}}(Z^{-1}) \end{pmatrix},
\qquad
\begin{smallmatrix}
m =-l+\frac12 ,-l+1,\dots,l-\frac12,  \\
n =-l+1,-l+2,\dots,l-1,
\end{smallmatrix}
\end{align*}
span respectively the $V_l \boxtimes V_{l+\frac12}$,
$V_l \boxtimes V_{l-\frac12}$ and $V_{l-1} \boxtimes V_{l-\frac12}$
components of ${\cal U}(-2l-1)$, $l=0,\frac12,1,\frac32\dots$.

The functions
%$$
%\bigl( (l+m+1/2) t^l_{m+\frac12\,\underline{n}}(Z^{-1}),
%-(l-m+1/2) t^l_{m-\frac12\,\underline{n}}(Z^{-1}) \bigr),
%\qquad
%\begin{matrix}
%m =-l+\frac12 ,-l+\frac32,\dots,l-\frac12,  \\
%n =-l,-l+1,\dots,l;
%\end{matrix}
%$$
\begin{align*}
\tilde g^{(1)}_{l,m,n}(Z)
&= \bigl( (l+m+1) t^{l+\frac12}_{m+\frac12\,\underline{n}}(Z^{-1}),
-(l-m+1) t^{l+\frac12}_{m-\frac12\,\underline{n}}(Z^{-1}) \bigr),
\qquad
\begin{smallmatrix}
m =-l,-l+1,\dots,l,  \\
n =-l-\frac12,-l+\frac12,\dots,l+\frac12,
\end{smallmatrix}  \\
\tilde g^{(2)}_{l,m,n}(Z)
&= N(Z)^{-1} \cdot \bigl( t^{l-\frac12}_{m+\frac12\,\underline{n}}(Z^{-1}),
t^{l-\frac12}_{m-\frac12\,\underline{n}}(Z^{-1}) \bigr),
\qquad
\begin{smallmatrix}
m =-l,-l+1,\dots,l,  \\
n =-l+\frac12,-l+\frac32,\dots,l-\frac12,
\end{smallmatrix}  \\
%$$
%N(Z)^{-1} \cdot \bigl( (l+m+1/2) t^l_{m+\frac12\,\underline{n}}(Z^{-1}),
%-(l-m+1/2) t^l_{m-\frac12\,\underline{n}}(Z^{-1}) \bigr),
%\qquad
%\begin{matrix}
%m =-l+\frac12 ,-l+\frac32,\dots,l-\frac12,  \\
%n =-l,-l+1,\dots,l,
%\end{matrix}
%$$
\tilde g^{(3)}_{l,m,n}(Z)
&= N(Z)^{-1} \cdot \bigl( (l+m) t^{l-\frac12}_{m+\frac12\,\underline{n}}(Z^{-1}),
-(l-m) t^{l-\frac12}_{m-\frac12\,\underline{n}}(Z^{-1}) \bigr),
\qquad
\begin{smallmatrix}
m =-l+1,-l+2,\dots,l-1,  \\
n =-l+\frac12,-l+\frac32,\dots,l-\frac12,
\end{smallmatrix}
\end{align*}
span respectively the $V_{l+\frac12} \boxtimes V_l$,
$V_{l-\frac12} \boxtimes V_l$ and $V_{l-\frac12} \boxtimes V_{l-1}$
components of ${\cal U}'(-2l-1)$, $l=0,\frac12,1,\frac32,\dots$.
\end{prop}

\begin{rem}
The functions $\tilde f^{(3)}_{l,m,n}(Z)$'s form a $K$-type basis of the space
of left anti regular functions,
$\tilde f^{(1)}_{l,m,n}(Z) = N(Z) \cdot f^{(3)}_{l+1,m,n}(Z)$,
and, by Lemma 23 in \cite{FL1},
$$
\tilde f^{(2)}_{l,m,n}(Z) = Z \cdot
\begin{pmatrix} (l-m+\frac12) N(Z)^{-1} \cdot t^l_{m-\frac12 \,\underline{n}}(Z^{-1})\\
(l+m+\frac12) N(Z)^{-1} \cdot t^l_{m+\frac12 \,\underline{n}}(Z^{-1}) \end{pmatrix}.
$$
Thus, $\tilde f^{(2)}_{l,m,n}(Z)$s are $Z$-multiples of the left regular
basis functions from Proposition 24 in \cite{FL1}.

Similarly, the functions $\tilde g^{(3)}_{l,m,n}(Z)$'s form a $K$-type basis of
the space of right anti regular functions,
$g^{(1)}_{l,m,n}(Z) = N(Z) \cdot g^{(3)}_{l+1,m,n}(Z)$,
and, by Lemma 23 in \cite{FL1},
$$
\tilde g^{(2)}_{l,m,n}(Z) = N(Z)^{-1} \cdot
\bigl( t^l_{m\,\underline{n-\frac12}}(Z^{-1}),
t^l_{m\,\underline{n+\frac12}}(Z^{-1}) \bigr) \cdot Z.
$$
Thus, $\tilde g^{(2)}_{l,m,n}(Z)$s are $Z$-multiples of the right regular
basis functions from Proposition 24 in \cite{FL1}.
\end{rem}

\begin{proof}
We have:
$$
\pi'_l(\Delta(a,b)) \begin{pmatrix}
\alpha N(Z)^{-1} \cdot t^l_{m\,\underline{n}}(Z^{-1}) \\
\beta N(Z)^{-1} \cdot t^{l'}_{m'\,\underline{n'}}(Z^{-1}) \end{pmatrix}
= \begin{pmatrix}
\alpha a^{-2n+1} b^{-2m} N(Z)^{-1} \cdot t^l_{m\,\underline{n}}(Z^{-1}) \\
\beta a^{-2n'-1} b^{-2m'} N(Z)^{-1} \cdot t^{l'}_{m'\,\underline{n'}}(Z^{-1})
\end{pmatrix},
$$
\begin{multline*}
\pi'_r(\Delta(a,b)) \bigl(
\alpha N(Z)^{-1} \cdot t^l_{m\,\underline{n}}(Z^{-1}),
\beta N(Z)^{-1} \cdot t^{l'}_{m'\,\underline{n'}}(Z^{-1}) \bigr)  \\
= \bigl( \alpha a^{-2n} b^{-2m+1} N(Z)^{-1} \cdot t^l_{m\,\underline{n}}(Z^{-1}),
\beta a^{-2n'} b^{-2m'-1} N(Z)^{-1} \cdot t^{l'}_{m'\,\underline{n'}}(Z^{-1}) \bigr).
\end{multline*}
Let $E_1$, $E_2$, $F_1$ and $F_2$ be as in \eqref{EF}.
Using Lemma 23 from \cite{FL1} together with
Lemma \ref{pi'-Lie_alg-action} and equation \eqref{dt-inverse}, we find:
$$
\pi'_l(E_1) \begin{pmatrix}
\alpha N(Z)^{-1} \cdot t^l_{m\,\underline{n+\frac12}}(Z^{-1}) \\
\beta N(Z)^{-1} \cdot t^l_{m\,\underline{n-\frac12}}(Z^{-1}) \end{pmatrix}
= \begin{pmatrix} \bigl(\alpha (l+n+\frac12) +\beta\bigr)
N(Z)^{-1} \cdot t^l_{m\,\underline{n-\frac12}}(Z^{-1}) \\
\beta (l+n-\frac12) N(Z)^{-1} \cdot t^l_{m\,\underline{n-\frac32}}(Z^{-1})
\end{pmatrix},
$$
$$
\pi'_l(F_1) \begin{pmatrix}
\alpha N(Z)^{-1} \cdot t^l_{m\,\underline{n+\frac12}}(Z^{-1}) \\
\beta N(Z)^{-1} \cdot t^l_{m\,\underline{n-\frac12}}(Z^{-1}) \end{pmatrix}
= \begin{pmatrix} \alpha (l-n-\frac12)
N(Z)^{-1} \cdot t^l_{m\,\underline{n+\frac32}}(Z^{-1}) \\
\bigl(\alpha + \beta (l-n+\frac12)\bigr)
N(Z)^{-1} \cdot t^l_{m\,\underline{n+\frac12}}(Z^{-1}) \end{pmatrix},
$$
$$
\pi'_l(E_2) \begin{pmatrix}
\alpha N(Z)^{-1} \cdot t^l_{m\,\underline{n+\frac12}}(Z^{-1}) \\
\beta N(Z)^{-1} \cdot t^l_{m\,\underline{n-\frac12}}(Z^{-1}) \end{pmatrix}
= -(l-m+1)
\begin{pmatrix} \alpha N(Z)^{-1} \cdot t^l_{m-1\,\underline{n+\frac12}}(Z^{-1}) \\
\beta N(Z)^{-1} \cdot t^l_{m-1\,\underline{n-\frac12}}(Z^{-1}) \end{pmatrix},
$$
$$
\pi'_l(F_2) \begin{pmatrix}
\alpha N(Z)^{-1} \cdot t^l_{m\,\underline{n+\frac12}}(Z^{-1}) \\
\beta N(Z)^{-1} \cdot t^l_{m\,\underline{n-\frac12}}(Z^{-1}) \end{pmatrix}
= -(l+m+1)
\begin{pmatrix} \alpha N(Z)^{-1} \cdot t^l_{m+1\,\underline{n+\frac12}}(Z^{-1}) \\
\beta N(Z)^{-1} \cdot t^l_{m+1\,\underline{n-\frac12}}(Z^{-1}) \end{pmatrix},
$$
\begin{multline*}
\pi'_r(E_1) \bigl( \alpha N(Z)^{-1} \cdot t^l_{m+\frac12\,\underline{n}}(Z^{-1}),
\beta N(Z)^{-1} \cdot t^l_{m-\frac12\,\underline{n}}(Z^{-1}) \bigr)  \\
= (l+n) \bigl( \alpha N(Z)^{-1} \cdot t^l_{m+\frac12\,\underline{n-1}}(Z^{-1}),
\beta N(Z)^{-1} \cdot t^l_{m-\frac12\,\underline{n-1}}(Z^{-1}) \bigr),
\end{multline*}
\begin{multline*}
\pi'_r(F_1) \bigl( \alpha N(Z)^{-1} \cdot t^l_{m+\frac12\,\underline{n}}(Z^{-1}),
\beta N(Z)^{-1} \cdot t^l_{m-\frac12\,\underline{n}}(Z^{-1}) \bigr)  \\
= (l-n) \bigl( \alpha N(Z)^{-1} \cdot t^l_{m+\frac12\,\underline{n+1}}(Z^{-1}),
\beta N(Z)^{-1} \cdot t^l_{m-\frac12\,\underline{n+1}}(Z^{-1}) \bigr),
\end{multline*}
\begin{multline*}
\pi'_r(E_2) \bigl( \alpha N(Z)^{-1} \cdot t^l_{m+\frac12\,\underline{n}}(Z^{-1}),
\beta N(Z)^{-1} \cdot t^l_{m-\frac12\,\underline{n}}(Z^{-1}) \bigr)  \\
= - \Bigl( \bigl(\alpha (l-m+\tfrac12) + \beta\bigr)
N(Z)^{-1} \cdot t^l_{m-\frac12\,\underline{n}}(Z^{-1}),
\beta (l-m+\tfrac32) N(Z)^{-1} \cdot t^l_{m-\frac32\,\underline{n}}(Z^{-1}) \Bigr),
\end{multline*}
\begin{multline*}
\pi'_r(E_2) \bigl( \alpha N(Z)^{-1} \cdot t^l_{m+\frac12\,\underline{n}}(Z^{-1}),
\beta N(Z)^{-1} \cdot t^l_{m-\frac12\,\underline{n}}(Z^{-1}) \bigr)  \\
= -\Bigl( \alpha (l+m+\tfrac32) N(Z)^{-1} \cdot t^l_{m+\frac32\,\underline{n}}(Z^{-1}),
\bigl(\alpha+\beta (l+m+\tfrac12)\bigr)
N(Z)^{-1} \cdot t^l_{m+\frac12\,\underline{n}}(Z^{-1}) \Bigr).
\end{multline*}
We see that each family of functions is invariant under the
$\mathfrak{sl}(2,\BB C) \times \mathfrak{sl}(2,\BB C)$-action and
spans an irreducible component of ${\cal U}(-2l-1)$ or ${\cal U}'(-2l-1)$
as stated.
\end{proof}

\subsection{Action of $\mathfrak{gl}(2,\HC)$ on the $K$-Type Basis}  \label{K-type-action}

In this subsection we find the actions of
$\pi'_l \bigl( \begin{smallmatrix} 0 & B \\ C & 0 \end{smallmatrix} \bigr)$ and
$\pi'_r \bigl( \begin{smallmatrix} 0 & B \\ C & 0 \end{smallmatrix} \bigr)$,
$B, C \in \HC$, on the basis vectors given in
Propositions \ref{K-typebasis+_prop}, \ref{K-typebasis-_prop}.
(Note that the actions of
$\pi'_l \bigl( \begin{smallmatrix} A & 0 \\ 0 & D \end{smallmatrix} \bigr)$ and
$\pi'_r \bigl( \begin{smallmatrix} A & 0 \\ 0 & D \end{smallmatrix} \bigr)$,
$A, D \in \HC$, were effectively described in Propositions \ref{K-types-prop},
\ref{K-typebasis+_prop}, \ref{K-typebasis-_prop}.)

By direct computation using equations (27)-(28) in \cite{desitter},
Lemma \ref{pi'-Lie_alg-action} and equations \eqref{dt}-\eqref{Ct-inverse},
we find:
{\footnotesize
$$
\pi'_l \bigl( \begin{smallmatrix} 0 & B \\ 0 & 0 \end{smallmatrix} \bigr)
f^{(1)}_{l,m,n}(Z) = -\tr \left( B \begin{pmatrix}
(l-m) f^{(1)}_{l-\frac12,m+\frac12,n+\frac12}(Z) &
(l-m) f^{(1)}_{l-\frac12,m+\frac12,n-\frac12}(Z)  \\
(l+m) f^{(1)}_{l-\frac12,m-\frac12,n+\frac12}(Z) &
(l+m) f^{(1)}_{l-\frac12,m-\frac12,n-\frac12}(Z)
\end{pmatrix} \right),
$$
\begin{multline*}
\pi'_l \bigl( \begin{smallmatrix} 0 & B \\ 0 & 0 \end{smallmatrix} \bigr)
f^{(2)}_{l,m,n}(Z) =
- \frac{2l+1}{2l} \tr \left( B \begin{pmatrix}
(l-m) f^{(2)}_{l-\frac12,m+\frac12,n+\frac12}(Z) &
(l-m) f^{(2)}_{l-\frac12,m+\frac12,n-\frac12}(Z)  \\
(l+m) f^{(2)}_{l-\frac12,m-\frac12,n+\frac12}(Z) &
(l+m) f^{(2)}_{l-\frac12,m-\frac12,n-\frac12}(Z)
\end{pmatrix} \right)  \\
- \frac1{2l} \tr \left( B \begin{pmatrix}
(l-m)(l+n+\frac12) f^{(1)}_{l-\frac12,m+\frac12,n+\frac12}(Z) &
-(l-m)(l-n+\frac12) f^{(1)}_{l-\frac12,m+\frac12,n-\frac12}(Z)  \\
(l+m)(l+n+\frac12) f^{(1)}_{l-\frac12,m-\frac12,n+\frac12}(Z) &
-(l+m)(l-n+\frac12) f^{(1)}_{l-\frac12,m-\frac12,n-\frac12}(Z)
\end{pmatrix} \right),
\end{multline*}
\begin{multline*}
\pi'_l \bigl( \begin{smallmatrix} 0 & B \\ 0 & 0 \end{smallmatrix} \bigr)
f^{(3)}_{l,m,n}(Z) =
- \frac1{2l} \tr \left( B \begin{pmatrix}
(l+n+\frac12) f^{(1)}_{l-\frac12,m+\frac12,n+\frac12}(Z) &
-(l-n+\frac12) f^{(1)}_{l-\frac12,m+\frac12,n-\frac12}(Z)  \\
-(l+n+\frac12) f^{(1)}_{l-\frac12,m-\frac12,n+\frac12}(Z) &
(l-n+\frac12) f^{(1)}_{l-\frac12,m-\frac12,n-\frac12}(Z)
\end{pmatrix} \right)  \\
- \frac1{2l(2l-1)} \tr \left( B \begin{pmatrix}
- f^{(2)}_{l-\frac12,m+\frac12,n+\frac12}(Z) &
- f^{(2)}_{l-\frac12,m+\frac12,n-\frac12}(Z)  \\
f^{(2)}_{l-\frac12,m-\frac12,n+\frac12}(Z) &
f^{(2)}_{l-\frac12,m-\frac12,n-\frac12}(Z)
\end{pmatrix} \right)  \\
- \frac{2l}{2l-1} \tr \left( B \begin{pmatrix}
(l-m-1) f^{(3)}_{l-\frac12,m+\frac12,n+\frac12}(Z) &
(l-m-1) f^{(3)}_{l-\frac12,m+\frac12,n-\frac12}(Z)  \\
(l+m-1) f^{(3)}_{l-\frac12,m-\frac12,n+\frac12}(Z) &
(l+m-1) f^{(3)}_{l-\frac12,m-\frac12,n-\frac12}(Z)
\end{pmatrix} \right);
\end{multline*}
\begin{multline*}
\pi'_l \bigl( \begin{smallmatrix} 0 & 0 \\ C & 0 \end{smallmatrix} \bigr)
f^{(1)}_{l,m,n}(Z) =
\frac{2l+1}{2l+2} \tr \left( C \begin{pmatrix}
(l-n+\frac32) f^{(1)}_{l+\frac12,m-\frac12,n-\frac12}(Z) &
(l-n+\frac32) f^{(1)}_{l+\frac12,m+\frac12,n-\frac12}(Z)  \\
(l+n+\frac32) f^{(1)}_{l+\frac12,m-\frac12,n+\frac12}(Z) &
(l+n+\frac32) f^{(1)}_{l+\frac12,m+\frac12,n+\frac12}(Z)
\end{pmatrix} \right)  \\
+ \frac1{(2l+1)(2l+2)} \tr \left( C \begin{pmatrix}
- f^{(2)}_{l+\frac12,m-\frac12,n-\frac12}(Z) &
- f^{(2)}_{l+\frac12,m+\frac12,n-\frac12}(Z)  \\
f^{(2)}_{l+\frac12,m-\frac12,n+\frac12}(Z) &
f^{(2)}_{l+\frac12,m+\frac12,n+\frac12}(Z)
\end{pmatrix} \right)  \\
+ \frac1{2l+1} \tr \left( C \begin{pmatrix}
-(l+m) f^{(3)}_{l+\frac12,m-\frac12,n-\frac12}(Z) &
(l-m) f^{(3)}_{l+\frac12,m+\frac12,n-\frac12}(Z)  \\
(l+m) f^{(3)}_{l+\frac12,m-\frac12,n+\frac12}(Z) &
-(l-m) f^{(3)}_{l+\frac12,m+\frac12,n+\frac12}(Z),
\end{pmatrix} \right),
\end{multline*}
\begin{multline*}
\pi'_l \bigl( \begin{smallmatrix} 0 & 0 \\ C & 0 \end{smallmatrix} \bigr)
f^{(2)}_{l,m,n}(Z) =
\frac{2l}{2l+1} \tr \left( C \begin{pmatrix}
(l-n+\frac12) f^{(2)}_{l+\frac12,m-\frac12,n-\frac12}(Z) &
(l-n+\frac12) f^{(2)}_{l+\frac12,m+\frac12,n-\frac12}(Z)  \\
(l+n+\frac12) f^{(2)}_{l+\frac12,m-\frac12,n+\frac12}(Z) &
(l+n+\frac12) f^{(2)}_{l+\frac12,m+\frac12,n+\frac12}(Z)
\end{pmatrix} \right)  \\
+ \frac1{2l+1} \tr \left( C \begin{pmatrix}
-(l+m)(l-n+\frac12) f^{(3)}_{l+\frac12,m-\frac12,n-\frac12}(Z) &
(l-m)(l-n+\frac12) f^{(3)}_{l+\frac12,m+\frac12,n-\frac12}(Z)  \\
-(l+m)(l+n+\frac12) f^{(3)}_{l+\frac12,m-\frac12,n+\frac12}(Z) &
(l-m)(l+n+\frac12) f^{(3)}_{l+\frac12,m+\frac12,n+\frac12}(Z),
\end{pmatrix} \right),
\end{multline*}
$$
\pi'_l \bigl( \begin{smallmatrix} 0 & 0 \\ C & 0 \end{smallmatrix} \bigr)
f^{(3)}_{l,m,n}(Z) = \tr \left( C \begin{pmatrix}
(l-n+\frac12) f^{(3)}_{l+\frac12,m-\frac12,n-\frac12}(Z) &
(l-n+\frac12) f^{(3)}_{l+\frac12,m+\frac12,n-\frac12}(Z)  \\
(l+n+\frac12) f^{(3)}_{l+\frac12,m-\frac12,n+\frac12}(Z) &
(l+n+\frac12) f^{(3)}_{l+\frac12,m+\frac12,n+\frac12}(Z)
\end{pmatrix} \right);
$$
$$
\pi'_r \bigl( \begin{smallmatrix} 0 & B \\ 0 & 0 \end{smallmatrix} \bigr)
g^{(1)}_{l,m,n}(Z) = -\tr \left( B \begin{pmatrix}
(l-m+\frac12) g^{(1)}_{l-\frac12,m+\frac12,n+\frac12}(Z) &
(l-m+\frac12) g^{(1)}_{l-\frac12,m+\frac12,n-\frac12}(Z)  \\
(l+m+\frac12) g^{(1)}_{l-\frac12,m-\frac12,n+\frac12}(Z) &
(l+m+\frac12) g^{(1)}_{l-\frac12,m-\frac12,n-\frac12}(Z)
\end{pmatrix} \right),
$$
\begin{multline*}
\pi'_r \bigl( \begin{smallmatrix} 0 & B \\ 0 & 0 \end{smallmatrix} \bigr)
g^{(2)}_{l,m,n}(Z) = - \frac1{2l} \tr \left( B \begin{pmatrix}
g^{(1)}_{l-\frac12,m+\frac12,n+\frac12}(Z) &
g^{(1)}_{l-\frac12,m+\frac12,n-\frac12}(Z)  \\
- g^{(1)}_{l-\frac12,m-\frac12,n+\frac12}(Z) &
- g^{(1)}_{l-\frac12,m-\frac12,n-\frac12}(Z)
\end{pmatrix} \right)  \\
- \frac{2l+1}{2l} \tr \left( B \begin{pmatrix}
(l-m-\frac12) g^{(2)}_{l-\frac12,m+\frac12,n+\frac12}(Z) &
(l-m-\frac12) g^{(2)}_{l-\frac12,m+\frac12,n-\frac12}(Z)  \\
(l+m-\frac12) g^{(2)}_{l-\frac12,m-\frac12,n+\frac12}(Z) &
(l+m-\frac12) g^{(2)}_{l-\frac12,m-\frac12,n-\frac12}(Z)
\end{pmatrix} \right),
\end{multline*}
\begin{multline*}
\pi'_r \bigl( \begin{smallmatrix} 0 & B \\ 0 & 0 \end{smallmatrix} \bigr)
g^{(3)}_{l,m,n}(Z) = - \frac1{2l} \tr \left( B \begin{pmatrix}
(l+n) g^{(1)}_{l-\frac12,m+\frac12,n+\frac12}(Z) &
-(l-n) g^{(1)}_{l-\frac12,m+\frac12,n-\frac12}(Z)  \\
-(l+n) g^{(1)}_{l-\frac12,m-\frac12,n+\frac12}(Z) &
(l-n) g^{(1)}_{l-\frac12,m-\frac12,n-\frac12}(Z)
\end{pmatrix} \right)  \\
- \frac1{2l(2l-1)} \tr \left( B \begin{pmatrix}
-(l-m-\frac12)(l+n) g^{(2)}_{l-\frac12,m+\frac12,n+\frac12}(Z) &
(l-m-\frac12)(l-n) g^{(2)}_{l-\frac12,m+\frac12,n-\frac12}(Z)  \\
- (l+m-\frac12)(l+n) g^{(2)}_{l-\frac12,m-\frac12,n+\frac12}(Z) &
(l+m-\frac12)(l-n) g^{(2)}_{l-\frac12,m-\frac12,n-\frac12}(Z)
\end{pmatrix} \right)  \\
- \frac{2l}{2l-1} \tr \left( B \begin{pmatrix}
(l-m-\frac12) g^{(3)}_{l-\frac12,m+\frac12,n+\frac12}(Z) &
(l-m-\frac12) g^{(3)}_{l-\frac12,m+\frac12,n-\frac12}(Z)  \\
(l+m-\frac12) g^{(3)}_{l-\frac12,m-\frac12,n+\frac12}(Z) &
(l+m-\frac12) g^{(3)}_{l-\frac12,m-\frac12,n-\frac12}(Z)
\end{pmatrix} \right);
\end{multline*}
\begin{multline*}
\pi'_r \bigl( \begin{smallmatrix} 0 & 0 \\ C & 0 \end{smallmatrix} \bigr)
g^{(1)}_{l,m,n}(Z) = \frac{2l+1}{2l+2}\tr \left( C \begin{pmatrix}
(l-n+1) g^{(1)}_{l+\frac12,m-\frac12,n-\frac12}(Z) &
(l-n+1) g^{(1)}_{l+\frac12,m+\frac12,n-\frac12}(Z)  \\
(l+n+1) g^{(1)}_{l+\frac12,m-\frac12,n+\frac12}(Z) &
(l+n+1) g^{(1)}_{l+\frac12,m+\frac12,n+\frac12}(Z)
\end{pmatrix} \right)  \\
+ \tr \left( C \begin{pmatrix}
-\frac{(l+m+\frac12)(l-n+1)}{(2l+1)(2l+2)} g^{(2)}_{l+\frac12,m-\frac12,n-\frac12}(Z) &
\frac{(l-m+\frac12)(l-n+1)}{(2l+1)(2l+2)} g^{(2)}_{l+\frac12,m+\frac12,n-\frac12}(Z) \\
-\frac{(l+m+\frac12)(l+n+1)}{(2l+1)(2l+2)} g^{(2)}_{l+\frac12,m-\frac12,n+\frac12}(Z) &
\frac{(l-m+\frac12)(l+n+1)}{(2l+1)(2l+2)} g^{(2)}_{l+\frac12,m+\frac12,n+\frac12}(Z)
\end{pmatrix} \right)  \\
+ \frac1{2l+1} \tr \left( C \begin{pmatrix}
-(l+m+\frac12) g^{(3)}_{l+\frac12,m-\frac12,n-\frac12}(Z) &
(l-m+\frac12) g^{(3)}_{l+\frac12,m+\frac12,n-\frac12}(Z)  \\
(l+m+\frac12) g^{(3)}_{l+\frac12,m-\frac12,n+\frac12}(Z) &
-(l-m+\frac12) g^{(3)}_{l+\frac12,m+\frac12,n+\frac12}(Z)
\end{pmatrix} \right),
\end{multline*}
\begin{multline*}
\pi'_r \bigl( \begin{smallmatrix} 0 & 0 \\ C & 0 \end{smallmatrix} \bigr)
g^{(2)}_{l,m,n}(Z) = \frac{2l}{2l+1}\tr \left( C \begin{pmatrix}
(l-n+1) g^{(2)}_{l+\frac12,m-\frac12,n-\frac12}(Z) &
(l-n+1) g^{(2)}_{l+\frac12,m+\frac12,n-\frac12}(Z)  \\
(l+n+1) g^{(2)}_{l+\frac12,m-\frac12,n+\frac12}(Z) &
(l+n+1) g^{(2)}_{l+\frac12,m+\frac12,n+\frac12}(Z)
\end{pmatrix} \right)  \\
+ \frac1{2l+1} \tr \left( C \begin{pmatrix}
- g^{(3)}_{l+\frac12,m-\frac12,n-\frac12}(Z) &
- g^{(3)}_{l+\frac12,m+\frac12,n-\frac12}(Z)  \\
g^{(3)}_{l+\frac12,m-\frac12,n+\frac12}(Z) &
g^{(3)}_{l+\frac12,m+\frac12,n+\frac12}(Z)
\end{pmatrix} \right),
\end{multline*}
$$
\pi'_r \bigl( \begin{smallmatrix} 0 & 0 \\ C & 0 \end{smallmatrix} \bigr)
g^{(3)}_{l,m,n}(Z) = \tr \left( C \begin{pmatrix}
(l-n) g^{(3)}_{l+\frac12,m-\frac12,n-\frac12}(Z) &
(l-n) g^{(3)}_{l+\frac12,m+\frac12,n-\frac12}(Z)  \\
(l+n) g^{(3)}_{l+\frac12,m-\frac12,n+\frac12}(Z) &
(l+n) g^{(3)}_{l+\frac12,m+\frac12,n+\frac12}(Z)
\end{pmatrix} \right);
$$
\begin{multline*}
\pi'_l \bigl( \begin{smallmatrix} 0 & B \\ 0 & 0 \end{smallmatrix} \bigr)
\tilde f^{(1)}_{l,m,n}(Z) = \frac{2l+1}{2l+2} \tr \left( B \begin{pmatrix}
(l-m+\frac32) \tilde f^{(1)}_{l+\frac12,m-\frac12,n-\frac12}(Z) &
(l-m+\frac32) \tilde f^{(1)}_{l+\frac12,m-\frac12,n+\frac12}(Z)  \\
(l+m+\frac32) \tilde f^{(1)}_{l+\frac12,m+\frac12,n-\frac12}(Z) &
(l+m+\frac32) \tilde f^{(1)}_{l+\frac12,m+\frac12,n+\frac12}(Z)
\end{pmatrix} \right)  \\
+ \frac1{(2l+1)(2l+2)} \tr \left( B \begin{pmatrix}
- \tilde f^{(2)}_{l+\frac12,m-\frac12,n-\frac12}(Z) &
- \tilde f^{(2)}_{l+\frac12,m-\frac12,n+\frac12}(Z)  \\
\tilde f^{(2)}_{l+\frac12,m+\frac12,n-\frac12}(Z) &
\tilde f^{(2)}_{l+\frac12,m+\frac12,n+\frac12}(Z)
\end{pmatrix} \right)  \\
+ \frac1{2l+1} \tr \left( B \begin{pmatrix}
-(l+n) \tilde f^{(3)}_{l+\frac12,m-\frac12,n-\frac12}(Z) &
(l-n) \tilde f^{(3)}_{l+\frac12,m-\frac12,n+\frac12}(Z)  \\
(l+n) \tilde f^{(3)}_{l+\frac12,m+\frac12,n-\frac12}(Z) &
-(l-n) \tilde f^{(3)}_{l+\frac12,m+\frac12,n+\frac12}(Z)
\end{pmatrix} \right),
\end{multline*}
\begin{multline*}
\pi'_l \bigl( \begin{smallmatrix} 0 & B \\ 0 & 0 \end{smallmatrix} \bigr)
\tilde f^{(2)}_{l,m,n}(Z) = \frac{2l}{2l+1} \tr \left( B \begin{pmatrix}
(l-m+\frac12) \tilde f^{(2)}_{l+\frac12,m-\frac12,n-\frac12}(Z) &
(l-m+\frac12) \tilde f^{(2)}_{l+\frac12,m-\frac12,n+\frac12}(Z)  \\
(l+m+\frac12) \tilde f^{(2)}_{l+\frac12,m+\frac12,n-\frac12}(Z) &
(l+m+\frac12) \tilde f^{(2)}_{l+\frac12,m+\frac12,n+\frac12}(Z)
\end{pmatrix} \right)  \\
+ \frac1{2l+1} \tr \left( B \begin{pmatrix}
-(l-m+\frac12)(l+n) \tilde f^{(3)}_{l+\frac12,m-\frac12,n-\frac12}(Z) &
(l-m+\frac12)(l-n) \tilde f^{(3)}_{l+\frac12,m-\frac12,n+\frac12}(Z)  \\
-(l+m+\frac12)(l+n) \tilde f^{(3)}_{l+\frac12,m+\frac12,n-\frac12}(Z) &
(l+m+\frac12)(l-n) \tilde f^{(3)}_{l+\frac12,m+\frac12,n+\frac12}(Z)
\end{pmatrix} \right),
\end{multline*}
$$
\pi'_l \bigl( \begin{smallmatrix} 0 & B \\ 0 & 0 \end{smallmatrix} \bigr)
\tilde f^{(3)}_{l,m,n}(Z) = \tr \left( B \begin{pmatrix}
(l-m+\frac12) \tilde f^{(3)}_{l+\frac12,m-\frac12,n-\frac12}(Z) &
(l-m+\frac12) \tilde f^{(3)}_{l+\frac12,m-\frac12,n+\frac12}(Z)  \\
(l+m+\frac12) \tilde f^{(3)}_{l+\frac12,m+\frac12,n-\frac12}(Z) &
(l+m+\frac12) \tilde f^{(3)}_{l+\frac12,m+\frac12,n+\frac12}(Z)
\end{pmatrix} \right);
$$
$$
\pi'_l \bigl( \begin{smallmatrix} 0 & 0 \\ C & 0 \end{smallmatrix} \bigr)
\tilde f^{(1)}_{l,m,n}(Z) = - \tr \left( C \begin{pmatrix}
(l-n) \tilde f^{(1)}_{l-\frac12,m+\frac12,n+\frac12}(Z) &
(l-n) \tilde f^{(1)}_{l-\frac12,m-\frac12,n+\frac12}(Z)  \\
(l+n) \tilde f^{(1)}_{l-\frac12,m+\frac12,n-\frac12}(Z) &
(l+n) \tilde f^{(1)}_{l-\frac12,m-\frac12,n-\frac12}(Z)
\end{pmatrix} \right),
$$
\begin{multline*}
\pi'_l \bigl( \begin{smallmatrix} 0 & 0 \\ C & 0 \end{smallmatrix} \bigr)
\tilde f^{(2)}_{l,m,n}(Z) =  - \frac{2l+1}{2l} \tr \left( C \begin{pmatrix}
(l-n) \tilde f^{(2)}_{l-\frac12,m+\frac12,n+\frac12}(Z) &
(l-n) \tilde f^{(2)}_{l-\frac12,m-\frac12,n+\frac12}(Z)  \\
(l+n) \tilde f^{(2)}_{l-\frac12,m+\frac12,n-\frac12}(Z) &
(l+n) \tilde f^{(2)}_{l-\frac12,m-\frac12,n-\frac12}(Z)
\end{pmatrix} \right)  \\
+ \frac1{2l} \tr \left( C \begin{pmatrix}
-(l+m+\frac12)(l-n) \tilde f^{(1)}_{l-\frac12,m+\frac12,n+\frac12}(Z) &
(l-m+\frac12)(l-n) \tilde f^{(1)}_{l-\frac12,m-\frac12,n+\frac12}(Z)  \\
-(l+m+\frac12)(l+n) \tilde f^{(1)}_{l-\frac12,m+\frac12,n-\frac12}(Z) &
(l-m+\frac12)(l+n) \tilde f^{(1)}_{l-\frac12,m-\frac12,n-\frac12}(Z)
\end{pmatrix} \right),
\end{multline*}
\begin{multline*}
\pi'_l \bigl( \begin{smallmatrix} 0 & 0 \\ C & 0 \end{smallmatrix} \bigr)
\tilde f^{(3)}_{l,m,n}(Z) = \frac1{2l} \tr \left( C \begin{pmatrix}
-(l+m+\frac12) \tilde f^{(1)}_{l-\frac12,m+\frac12,n+\frac12}(Z) &
(l-m+\frac12) \tilde f^{(1)}_{l-\frac12,m-\frac12,n+\frac12}(Z)  \\
(l+m+\frac12) \tilde f^{(1)}_{l-\frac12,m+\frac12,n-\frac12}(Z) &
-(l-m+\frac12) \tilde f^{(1)}_{l-\frac12,m-\frac12,n-\frac12}(Z)
\end{pmatrix} \right)  \\
+ \frac1{2l(2l-1)} \tr \left( C \begin{pmatrix}
\tilde f^{(2)}_{l-\frac12,m+\frac12,n+\frac12}(Z) &
\tilde f^{(2)}_{l-\frac12,m-\frac12,n+\frac12}(Z)  \\
-\tilde f^{(2)}_{l-\frac12,m+\frac12,n-\frac12}(Z) &
-\tilde f^{(2)}_{l-\frac12,m-\frac12,n-\frac12}(Z)
\end{pmatrix} \right)  \\
- \frac{2l}{2l-1} \tr \left( C \begin{pmatrix}
(l-n-1) \tilde f^{(3)}_{l-\frac12,m+\frac12,n+\frac12}(Z) &
(l-n-1) \tilde f^{(3)}_{l-\frac12,m-\frac12,n+\frac12}(Z)  \\
(l+n-1) \tilde f^{(3)}_{l-\frac12,m+\frac12,n-\frac12}(Z) &
(l+n-1) \tilde f^{(3)}_{l-\frac12,m-\frac12,n-\frac12}(Z)
\end{pmatrix} \right);
\end{multline*}
\begin{multline*}
\pi'_r \bigl( \begin{smallmatrix} 0 & B \\ 0 & 0 \end{smallmatrix} \bigr)
\tilde g^{(1)}_{l,m,n}(Z) = \frac{2l+1}{2l+2} \tr \left( B \begin{pmatrix}
(l-m+1) \tilde g^{(1)}_{l+\frac12,m-\frac12,n-\frac12}(Z) &
(l-m+1) \tilde g^{(1)}_{l+\frac12,m-\frac12,n+\frac12}(Z)  \\
(l+m+1) \tilde g^{(1)}_{l+\frac12,m+\frac12,n-\frac12}(Z) &
(l+m+1) \tilde g^{(1)}_{l+\frac12,m+\frac12,n+\frac12}(Z)
\end{pmatrix} \right)  \\
+ \tr \left( B \begin{pmatrix}
-\frac{(l-m+1)(l+n+\frac12)}{(2l+1)(2l+2)}
\tilde g^{(2)}_{l+\frac12,m-\frac12,n-\frac12}(Z) &
\frac{(l-m+1)(l-n+\frac12)}{(2l+1)(2l+2)}
\tilde g^{(2)}_{l+\frac12,m-\frac12,n+\frac12}(Z)  \\
-\frac{(l+m+1)(l+n+\frac12)}{(2l+1)(2l+2)}
\tilde g^{(2)}_{l+\frac12,m+\frac12,n-\frac12}(Z) &
\frac{(l+m+1)(l-n+\frac12)}{(2l+1)(2l+2)}
\tilde g^{(2)}_{l+\frac12,m+\frac12,n+\frac12}(Z)
\end{pmatrix} \right)  \\
+ \frac1{2l+1} \tr \left( B \begin{pmatrix}
-(l+n+\frac12) \tilde g^{(3)}_{l+\frac12,m-\frac12,n-\frac12}(Z) &
(l-n+\frac12) \tilde g^{(3)}_{l+\frac12,m-\frac12,n+\frac12}(Z)  \\
(l+n+\frac12) \tilde g^{(3)}_{l+\frac12,m+\frac12,n-\frac12}(Z) &
-(l-n+\frac12) \tilde g^{(3)}_{l+\frac12,m+\frac12,n+\frac12}(Z)
\end{pmatrix} \right),
\end{multline*}
\begin{multline*}
\pi'_r \bigl( \begin{smallmatrix} 0 & B \\ 0 & 0 \end{smallmatrix} \bigr)
\tilde g^{(2)}_{l,m,n}(Z) = \frac{2l}{2l+1} \tr \left( B \begin{pmatrix}
(l-m+1) \tilde g^{(2)}_{l+\frac12,m-\frac12,n-\frac12}(Z) &
(l-m+1) \tilde g^{(2)}_{l+\frac12,m-\frac12,n+\frac12}(Z)  \\
(l+m+1) \tilde g^{(2)}_{l+\frac12,m+\frac12,n-\frac12}(Z) &
(l+m+1) \tilde g^{(2)}_{l+\frac12,m+\frac12,n+\frac12}(Z)
\end{pmatrix} \right)  \\
+ \frac1{2l+1} \tr \left( B \begin{pmatrix}
-\tilde g^{(3)}_{l+\frac12,m-\frac12,n-\frac12}(Z) &
-\tilde g^{(3)}_{l+\frac12,m-\frac12,n+\frac12}(Z)  \\
\tilde g^{(3)}_{l+\frac12,m+\frac12,n-\frac12}(Z) &
\tilde g^{(3)}_{l+\frac12,m+\frac12,n+\frac12}(Z)
\end{pmatrix} \right),
\end{multline*}
$$
\pi'_r \bigl( \begin{smallmatrix} 0 & B \\ 0 & 0 \end{smallmatrix} \bigr)
\tilde g^{(3)}_{l,m,n}(Z) = \tr \left( B \begin{pmatrix}
(l-m) \tilde g^{(3)}_{l+\frac12,m-\frac12,n-\frac12}(Z) &
(l-m) \tilde g^{(3)}_{l+\frac12,m-\frac12,n+\frac12}(Z)  \\
(l+m) \tilde g^{(3)}_{l+\frac12,m+\frac12,n-\frac12}(Z) &
(l+m) \tilde g^{(3)}_{l+\frac12,m+\frac12,n+\frac12}(Z)
\end{pmatrix} \right);
$$
$$
\pi'_r \bigl( \begin{smallmatrix} 0 & 0 \\ C & 0 \end{smallmatrix} \bigr)
\tilde g^{(1)}_{l,m,n}(Z) = - \tr \left( C \begin{pmatrix}
(l-n+\frac12) \tilde g^{(1)}_{l-\frac12,m+\frac12,n+\frac12}(Z) &
(l-n+\frac12) \tilde g^{(1)}_{l-\frac12,m-\frac12,n+\frac12}(Z)  \\
(l+n+\frac12) \tilde g^{(1)}_{l-\frac12,m+\frac12,n-\frac12}(Z) &
(l+n+\frac12) \tilde g^{(1)}_{l-\frac12,m-\frac12,n-\frac12}(Z)
\end{pmatrix} \right),
$$
\begin{multline*}
\pi'_r \bigl( \begin{smallmatrix} 0 & 0 \\ C & 0 \end{smallmatrix} \bigr)
\tilde g^{(2)}_{l,m,n}(Z) = \frac1{2l}  \tr \left( C \begin{pmatrix}
- \tilde g^{(1)}_{l-\frac12,m+\frac12,n+\frac12}(Z) &
- \tilde g^{(1)}_{l-\frac12,m-\frac12,n+\frac12}(Z)  \\
\tilde g^{(1)}_{l-\frac12,m+\frac12,n-\frac12}(Z) &
\tilde g^{(1)}_{l-\frac12,m-\frac12,n-\frac12}(Z)
\end{pmatrix} \right)  \\
- \frac{2l+1}{2l} \tr \left( C \begin{pmatrix}
(l-n-\frac12) \tilde g^{(2)}_{l-\frac12,m+\frac12,n+\frac12}(Z) &
(l-n-\frac12) \tilde g^{(2)}_{l-\frac12,m-\frac12,n+\frac12}(Z)  \\
(l+n-\frac12) \tilde g^{(2)}_{l-\frac12,m+\frac12,n-\frac12}(Z) &
(l+n-\frac12) \tilde g^{(2)}_{l-\frac12,m-\frac12,n-\frac12}(Z)
\end{pmatrix} \right),
\end{multline*}
\begin{multline*}
\pi'_r \bigl( \begin{smallmatrix} 0 & 0 \\ C & 0 \end{smallmatrix} \bigr)
\tilde g^{(3)}_{l,m,n}(Z) = \frac1{2l} \tr \left( C \begin{pmatrix}
-(l+m) \tilde g^{(1)}_{l-\frac12,m+\frac12,n+\frac12}(Z) &
(l-m) \tilde g^{(1)}_{l-\frac12,m-\frac12,n+\frac12}(Z)  \\
(l+m) \tilde g^{(1)}_{l-\frac12,m+\frac12,n-\frac12}(Z) &
-(l-m)\tilde g^{(1)}_{l-\frac12,m-\frac12,n-\frac12}(Z)
\end{pmatrix} \right)  \\
+ \frac1{2l(2l-1)} \tr \left( C \begin{pmatrix}
(l+m)(l-n-\frac12) \tilde g^{(2)}_{l-\frac12,m+\frac12,n+\frac12}(Z) &
-(l-m)(l-n-\frac12) \tilde g^{(2)}_{l-\frac12,m-\frac12,n+\frac12}(Z)  \\
(l+m)(l+n-\frac12) \tilde g^{(2)}_{l-\frac12,m+\frac12,n-\frac12}(Z) &
-(l-m)(l+n-\frac12) \tilde g^{(2)}_{l-\frac12,m-\frac12,n-\frac12}(Z)
\end{pmatrix} \right)  \\
- \frac{2l}{2l-1} \tr \left( C \begin{pmatrix}
(l-n-\frac12) \tilde g^{(3)}_{l-\frac12,m+\frac12,n+\frac12}(Z) &
(l-n-\frac12) \tilde g^{(3)}_{l-\frac12,m-\frac12,n+\frac12}(Z)  \\
(l+n-\frac12) \tilde g^{(3)}_{l-\frac12,m+\frac12,n-\frac12}(Z) &
(l+n-\frac12) \tilde g^{(3)}_{l-\frac12,m-\frac12,n-\frac12}(Z)
\end{pmatrix} \right).
\end{multline*}
}

As a consequence of these calculations, we obtain that the spaces ${\cal U}^+$,
${\cal U}^-$ of QLAR functions as well as the spaces ${\cal U}'^+$,
${\cal U}'^-$ of QRAR functions are irreducible.

\begin{prop}  \label{irred-prop}
The spaces $(\pi'_l, {\cal U}^+)$, $(\pi'_l, {\cal U}^-)$,
$(\pi'_r, {\cal U}'^+)$ and $(\pi'_r, {\cal U}'^-)$ are irreducible
representations of $\mathfrak{sl}(2,\HC)$ (as well as $\mathfrak{gl}(2,\HC)$).
\end{prop}

\begin{proof}
The following diagram illustrates the action of
$\pi'_l \bigl( \begin{smallmatrix} 0 & B \\ 0 & 0 \end{smallmatrix} \bigr)$,
$B \in \HC$, on ${\cal U}(2l)$:
$$
\xymatrix{
{\cal U}(2l): & V_{l+\frac12} \boxtimes V_l \ar[d] &
V_{l-\frac12} \boxtimes V_l \ar[dl] \ar[d] &
V_{l-\frac12} \boxtimes V_{l-1} \ar[dll] \ar[dl] \ar[d] \\
{\cal U}(2l-1): & V_l \boxtimes V_{l-\frac12} &
V_{l-1} \boxtimes V_{l-\frac12} & V_{l-1} \boxtimes V_{l-\frac32}}
$$
And this diagram illustrates the action of 
$\pi'_l \bigl( \begin{smallmatrix} 0 & 0 \\ C & 0 \end{smallmatrix} \bigr)$,
$C \in \HC$, on ${\cal U}(2l)$:
$$
\xymatrix{
{\cal U}(2l+1): & V_{l+1} \boxtimes V_{l+\frac12} &
V_l \boxtimes V_{l+\frac12} & V_l \boxtimes V_{l+\frac32}  \\
{\cal U}(2l): & V_{l+\frac12} \boxtimes V_l \ar[u] \ar[ur] \ar[urr] &
V_{l-\frac12} \boxtimes V_l \ar[ur] \ar[u] &
V_{l-\frac12} \boxtimes V_{l-1} \ar[u]}
$$
(note that the arrows get reversed).
The diagrams expressing the actions of
$\bigl( \begin{smallmatrix} 0 & B \\ 0 & 0 \end{smallmatrix} \bigr)$
and $\bigl( \begin{smallmatrix} 0 & 0 \\ C & 0 \end{smallmatrix} \bigr)$
on ${\cal U}(-2l-1)$, ${\cal U}'(2l)$ and ${\cal U}'(-2l-1)$ are similar.

It follows from Lemma \ref{degree-decomp-lem} and this description of
the actions of
$\pi'_l \bigl( \begin{smallmatrix} 0 & B \\ C & 0 \end{smallmatrix} \bigr)$ and
$\pi'_r \bigl( \begin{smallmatrix} 0 & B \\ C & 0 \end{smallmatrix} \bigr)$,
$B, C \in \HC$, on the spaces ${\cal U}^+$, ${\cal U}^-$,
${\cal U}'^+$ and ${\cal U}'^-$ that any non-zero vector generates
the whole space. Hence the result.
\end{proof}

We can find the effect of the inversion on these families of QLAR and QRAR
functions. The element
$\bigl( \begin{smallmatrix} 0 & 1 \\ 1 & 0 \end{smallmatrix} \bigr)
\in GL(2,\HC)$ acts on QLAR functions by
$$
\pi'_l \bigl( \begin{smallmatrix} 0 & 1 \\ 1 & 0 \end{smallmatrix} \bigr): \:
f(Z) \mapsto -\frac{Z}{N(Z)} f(Z^{-1})
$$
and on QRAR functions by
$$
\pi'_r \bigl( \begin{smallmatrix} 0 & 1 \\ 1 & 0 \end{smallmatrix} \bigr): \:
G(Z) \mapsto g(Z^{-1}) \frac{Z}{N(Z)}.
$$
By direct calculation using the identity \eqref{Zt-identity}, we obtain:

\begin{lem}  \label{inversion-action-lem}
The inversion
$\pi'_l \bigl( \begin{smallmatrix} 0 & 1 \\ 1 & 0 \end{smallmatrix} \bigr)$
acts on the basis of QLAR functions by sending
$$
f^{(1)}_{l,m,n}(Z) \longleftrightarrow -\tilde f^{(1)}_{l,n,m}(Z), \qquad
f^{(2)}_{l,m,n}(Z) \longleftrightarrow -\tilde f^{(2)}_{l,n,m}(Z), \qquad
f^{(3)}_{l,m,n}(Z) \longleftrightarrow -\tilde f^{(3)}_{l,n,m}(Z).
$$
The inversion
$\pi'_r \bigl( \begin{smallmatrix} 0 & 1 \\ 1 & 0 \end{smallmatrix} \bigr)$
acts on the basis of QRAR functions by sending
$$
g^{(1)}_{l,m,n}(Z) \longleftrightarrow \tilde g^{(1)}_{l,n,m}(Z), \qquad
g^{(2)}_{l,m,n}(Z) \longleftrightarrow \tilde g^{(2)}_{l,n,m}(Z), \qquad
g^{(3)}_{l,m,n}(Z) \longleftrightarrow \tilde g^{(3)}_{l,n,m}(Z).
$$
\end{lem}

\section{Restricting Quasi Regular Functions to Subgroups of $GL(2,\HC)$}  \label{Sect5}

%\section{Restricting to Subgroups of $GL(2,\HC)$}

In Proposition \ref{irred-prop} we saw that $(\pi'_l, {\cal U}^+)$,
$(\pi'_l, {\cal U}^-)$, $(\pi'_r, {\cal U}'^+)$ and $(\pi'_r, {\cal U}'^-)$
are irreducible representations of $\mathfrak{gl}(2,\HC)$.
In this section we investigate how these spaces decompose into irreducible
components after restricting to certain natural subgroups of $GL(2,\HC)$,
namely the group of upper triangular matrices in $GL(2,\HC)$,
the group of rigid motions of $\BB H$ and the Poincar\'e group.

\subsection{The Upper Triangular Subgroup of $GL(2,\HC)$}

Consider the group of upper triangular matrices in $GL(2,\HC)$:
$$
R_{\BB C} = \bigl\{
\bigl( \begin{smallmatrix} a & b \\ 0 & d \end{smallmatrix}\bigr);\:
a,b,d \in \HC, \: N(a) \ne 0, \: N(d) \ne 0 \bigr\}.
$$
Its Lie algebra is
$$
\bigl\{
\bigl( \begin{smallmatrix} A & B \\ 0 & D \end{smallmatrix}\bigr);\:
A,B,D \in \HC \bigr\}
\subset \mathfrak{gl}(2,\HC).
$$

Let us consider the space $(\pi'_l, {\cal U}^+)$.
From Subsection \ref{K-type-action} we see that $(\pi'_l, {\cal U}^+)$
has $R_{\BB C}$-invariant finite-dimensional subspaces
\begin{align*}
{\cal U}^+_{(1)}(L) = \BB C\text{-span of }
&\bigl\{ f^{(1)}_{L,m,n}(Z),\: f^{(1)}_{l,m,n}(Z),\: f^{(2)}_{l,m,n}(Z),\:
f^{(3)}_{l,m,n}(Z) ;\: 0 \le l \le L-\tfrac12 \bigr\}, \\
&\qquad L = 0,\tfrac12,1,\tfrac32,\dots, \\
{\cal U}^+_{(2)}(L) = \BB C\text{-span of }
&\bigl\{ f^{(1)}_{L,m,n}(Z),\: f^{(2)}_{L,m,n}(Z),\: f^{(1)}_{l,m,n}(Z),\:
f^{(2)}_{l,m,n}(Z),\: f^{(3)}_{l,m,n}(Z);\: 0 \le l \le L-\tfrac12 \bigr\}, \\
&\qquad L = \tfrac12,1,\tfrac32,2,\dots, \\
{\cal U}^+_{(3)}(L) = \BB C\text{-span of }
&\bigl\{ f^{(1)}_{l,m,n}(Z),\: f^{(2)}_{l,m,n}(Z),\:
f^{(3)}_{l,m,n}(Z);\: 0 \le l \le L \bigr\},  \\
&\qquad L = 1,\tfrac32,2,\tfrac52,\dots.
\end{align*}
Furthermore, these invariant subspaces form a filtration of ${\cal U}^+$
\begin{multline}  \label{max-filtration}
\{0\} \subset {\cal U}^+_{(1)}(0) \subset {\cal U}^+_{(1)}(\tfrac12) \subset
{\cal U}^+_{(2)}(\tfrac12) \subset {\cal U}^+_{(1)}(1) \subset
{\cal U}^+_{(2)}(1) \subset {\cal U}^+_{(3)}(1) \\
\subset {\cal U}^+_{(1)}(\tfrac32) \subset {\cal U}^+_{(2)}(\tfrac32)
\subset {\cal U}^+_{(3)}(\tfrac32) \subset \dots,
\end{multline}
and the quotient of two subsequent subspaces is irreducible.
Indeed, when restricted to $SU(2) \times SU(2)$,
\begin{align*}
{\cal U}^+_{(1)}(L) / {\cal U}^+_{(3)}(L-\tfrac12) &=
V_{L+\frac12} \boxtimes V_L,  \\
{\cal U}^+_{(2)}(L)/ {\cal U}^+_{(1)}(L) &= V_{L-\frac12} \boxtimes V_L,  \\
{\cal U}^+_{(3)}(L) / {\cal U}^+_{(2)}(L) &= V_{L-\frac12} \boxtimes V_{L-1}.
\end{align*}
Thus, \eqref{max-filtration} is a maximal filtration of $(\pi'_l, {\cal U}^+)$
by $R_{\BB C}$-invariant subspaces, and we obtain a decomposition of
$(\pi'_l, {\cal U}^+)$ into $R_{\BB C}$-irreducible components.

The story for the other representations $(\pi'_l, {\cal U}^-)$,
$(\pi'_r, {\cal U}'^+)$ and $(\pi'_r, {\cal U}'^-)$ is very similar.

\subsection{The Group of Rigid Motions of $\BB H$}
%{$SO(4) \ltimes \BB R^4$}

We consider the group of rigid motions on $\BB H$:
$$
SO(4) \ltimes \BB R^4 = \bigl\{
\bigl( \begin{smallmatrix} a & b \\ 0 & d \end{smallmatrix}\bigr);\:
a,b,d \in \BB H, \: N(a)=N(d)=1 \bigr\}
\subset GL(2,\BB H).
$$
Its Lie algebra is
$$
\bigl\{
\bigl( \begin{smallmatrix} A & B \\ 0 & D \end{smallmatrix}\bigr);\:
A,B,D \in \BB H, \: \re A = \re D =0 \bigr\}
\subset \mathfrak{gl}(2,\BB H).
$$

Since $SO(4) \ltimes \BB R^4$ is a subgroup of the group of upper triangular
matrices $R_{\BB C}$, the filtration \eqref{max-filtration} of
$(\pi'_l, {\cal U}^+)$ is invariant under $SO(4) \ltimes \BB R^4$,
and it is easy to see that the quotient of two subsequent subspaces remains
irreducible under the action of $SO(4) \ltimes \BB R^4$.
Thus, we obtain a decomposition of $(\pi'_l, {\cal U}^+)$ into irreducible
components for the action of $SO(4) \ltimes \BB R^4$.

The story for the other representations $(\pi'_l, {\cal U}^-)$,
$(\pi'_r, {\cal U}'^+)$ and $(\pi'_r, {\cal U}'^-)$ is very similar.

\subsection{The Poincar\'e Group $O(3,1) \ltimes \BB R^{3,1}$}

First, consider the group of ``motions'' of the Minkowski space $\BB M$:
$$
O(3,1) \ltimes \BB R^{3,1} = \bigl\{
\bigl( \begin{smallmatrix} a & b \\ 0 & d \end{smallmatrix}\bigr);\:
a,b,d \in \HC , \: ad^*=1,\: ab^*+ba^*=0 \bigr\}
\subset GL(2,\HC).
$$
This is one of many possible realizations of the Poincar\'e group,
and we denote it by $P'_{\BB R}$. Its Lie algebra is
$$
\mathfrak{p}'_{\BB R} = \bigl\{
\bigl( \begin{smallmatrix} A & B \\ 0 & -A^* \end{smallmatrix}\bigr);\:
A,B \in \HC, \: \re A =0,\: B^*=-B \bigr\}
\subset \mathfrak{gl}(2,\HC).
$$
Since $P'_{\BB R}$ is a subgroup of the group of upper triangular matrices
$R_{\BB C}$, the filtration \eqref{max-filtration} of $(\pi'_l, {\cal U}^+)$ is
$\mathfrak{p}'_{\BB R}$-invariant, and
everything that was said about $SO(4) \ltimes \BB R^4$ applies here as well.

Note that this group of motions of $\BB M$ is a subgroup of the full
conformal group of $\BB M$
$$
U(2,2)' = \biggl\{ \begin{pmatrix} a & b \\ c & d \end{pmatrix};\:
a,b,c,d \in \HC,\:
\begin{smallmatrix} ab^*+ba^*=0 \\ cd^*+dc^*=0 \\ ad^*+bc^*=1 \end{smallmatrix}
\biggr\}
$$
with Lie algebra
$$
\mathfrak{u}(2,2)' = 
\bigl\{ \bigl(\begin{smallmatrix} A & B \\ C & -A^* \end{smallmatrix}\bigr) ;\:
A,B,C \in \HC,\: B^*=-B,\: C^*=-C \bigr\}.
$$
Clearly, $\mathfrak{p}'_{\BB R}$ is a subalgebra of $\mathfrak{u}(2,2)'$.

Things become more interesting when we consider a different copy of
$U(2,2)$ and the corresponding realization of Poincar\'e group.
Consider $U(2,2) \subset GL(2,\HC)$ as realized by \eqref{U(2,2)};
%$$
%U(2,2) = \biggl\{ \bigl(\begin{smallmatrix} a & b \\
%    c & d \end{smallmatrix}\bigr);\: a,b,c,d \in \HC,\:
%  \begin{smallmatrix} a^*a = 1+c^*c \\ d^*d = 1+b^*b \\ a^*b=c^*d
%  \end{smallmatrix} \biggr\},
%$$
its conformal action preserves $U(2) \subset \HC$.  
The Lie algebra of $U(2,2)$ is given by \eqref{u(2,2)-algebra}.
This copy of $U(2,2)$ is conjugate to $U(2,2)'$ (that preserves $\BB M$)
via the Cayley transform.
The corresponding realization of the Poincar\'e group $P_{\BB R}$ has Lie algebra
$$
\mathfrak{p}_{\BB R} = \bigl\{
\bigl( \begin{smallmatrix} A-A^*+B & -i(A+A^*-B) \\
  i(A+A^*+B) & A-A^*-B \end{smallmatrix}\bigr);\:
A,B \in \HC, \: \re A =0,\: B^*=-B \bigr\}
\subset \mathfrak{gl}(2,\HC).
$$

We consider the following spaces of holomorphic functions on $\BB D^+$
with values in $\BB S$:
\begin{align*}
{\cal U}^+(\BB D^+) &=\{ f: \BB D^+ \to \BB S;\:
\text{$f$ is holomorphic and $\nabla\square f =0$} \},  \\
{\cal U}^+_1(\BB D^+) &= \{ f: \BB D^+ \to \BB S;\:
\text{$f$ is holomorphic and $\nabla [N(Z+i)^{-1} f(Z)] =0$} \},  \\
{\cal U}^+_2(\BB D^+) &= \{ f: \BB D^+ \to \BB S;\:
\text{$f$ is holomorphic and $\square [(Z+i)^{-1} f(Z)] =0$} \},  \\
{\cal V}^+(\BB D^+) &= \{ f: \BB D^+ \to \BB S;\:
\text{$f$ is holomorphic and $\nabla^+ f =0$} \},  \\
{\cal V}^+_a(\BB D^+) &=\{ f: \BB D^+ \to \BB S;\:
\text{$f$ is holomorphic and $\nabla f =0$} \}.
\end{align*}
Each of these spaces has the standard topology of uniform convergence on
compact subsets.
We discuss the space of left anti regular functions ${\cal V}^+_a$ and the
conformal action $\pi_{la}$ in the next subsection.

\begin{thm}  \label{Poincare-restrict-thm}
The spaces ${\cal U}^+_1(\BB D^+)$ and ${\cal U}^+_2(\BB D^+)$ are closed
subspaces of ${\cal U}^+(\BB D^+)$ that are invariant under the $\pi'_l$
action of the Poincar\'e group. The representations of the Poincar\'e group
$P_{\BB R}$
$$
\bigl(\pi'_l, {\cal U}^+_1(\BB D^+)\bigr), \qquad
\bigl(\pi'_l, {\cal U}^+(\BB D^+)/{\cal U}^+_2(\BB D^+)\bigr)
$$
are both isomorphic to $\bigl(\pi_{la}, {\cal V}^+_a(\BB D^+)\bigr)$,
while the representation 
$$
\bigl(\pi'_l, {\cal U}^+_2(\BB D^+)/{\cal U}^+_1(\BB D^+)\bigr)
$$
is isomorphic to $\bigl(\pi_l, {\cal V}^+(\BB D^+)\bigr)$.
In particular, the three representations of the Poincar\'e group $P_{\BB R}$
$$
\bigl(\pi'_l, {\cal U}^+_1(\BB D^+)\bigr), \qquad
\bigl(\pi'_l, {\cal U}^+_2(\BB D^+)/{\cal U}^+_1(\BB D^+)\bigr), \qquad  
\bigl(\pi'_l, {\cal U}^+(\BB D^+)/{\cal U}^+_2(\BB D^+)\bigr)
$$
are irreducible.
\end{thm}

This result follows from Lemma \ref{D-T-bijections} and
Proposition \ref{Poincare-iso-prop}.

%For the purpose of decomposition of various spaces of functions into
%irreducible components, consider the following conjugates of
%$\mathfrak{p}_{\BB R}$:
%$$
%\bigl(\begin{smallmatrix} 1 & i \\ 0 & 1 \end{smallmatrix}\bigr)
%\mathfrak{p}_{\BB R}
%\bigl(\begin{smallmatrix} 1 & -i \\ 0 & 1 \end{smallmatrix}\bigr)
%= \bigl\{ \bigl(\begin{smallmatrix} -2A^* & 0 \\
%  i(A+A^*+B) & 2A \end{smallmatrix}\bigr);\:
%A,B \in \HC, \: \re A =0,\: B^*=-B \bigr\},
%$$
%$$
%\bigl(\begin{smallmatrix} 1 & 0 \\ -i & 1 \end{smallmatrix}\bigr)
%\mathfrak{p}_{\BB R}
%\bigl(\begin{smallmatrix} 1 & 0 \\ i & 1 \end{smallmatrix}\bigr)
%= \bigl\{ \bigl(\begin{smallmatrix} 2A & -i(A+A^*-B) \\
%  0 & -2A^* \end{smallmatrix}\bigr);\:
%A,B \in \HC, \: \re A =0,\: B^*=-B \bigr\}.
%$$
%Note that the actions of
%$\bigl(\begin{smallmatrix} 1 & \pm i \\ 0 & 1 \end{smallmatrix}\bigr)$
%are just parallel translations by $\pm i$ and preserve the spaces
%of polynomials, ${\cal V}^+$, ${\cal V}'^+$, ${\cal U}^+$, ${\cal U}'^+$. 
%Similarly, the actions of
%$\bigl(\begin{smallmatrix} 1 & 0 \\ \pm i & 1 \end{smallmatrix}\bigr)$
%preserve the space of Laurent polynomials regular at infinity,
%${\cal V}^-$, ${\cal V}'^-$, ${\cal U}^-$, ${\cal U}'^-$. 
%However, these operators are not unitary, most likely not continuous,
%and most likely do not extend to the globalizations of
%${\cal V}^{\pm}$, ${\cal V}'^{\pm}$, ${\cal U}^{\pm}$, ${\cal U}'^{\pm}$
%to representations of $U(2,2)$.

\subsection{Anti Regular Functions}

Since the spaces of left and right anti regular functions appear in the
decompositions of $(\pi'_l, {\cal U}^{\pm})$ and $(\pi'_r, {\cal U}'^{\pm})$
restricted to the Poincar\'e group, we briefly review their
definitions and conformal properties.
We start with the notions of left and right anti regular functions defined on
open subsets of $\BB H$ and $\HC$.

\begin{df}  \label{r-definition}
Let $U$ be an open subset of $\BB H$.
A ${\cal C}^1$-function $f: U \to \BB S$ is
{\em left anti regular} if it satisfies
$$
\nabla f =0 \qquad \text{at all points in $U$}.
$$

Similarly, a ${\cal C}^1$-function $g: U \to \BB S'$ is
{\em right anti regular} if
$$
g \overleftarrow{\nabla} =0 \qquad \text{at all points in $U$}.
$$
\end{df}

We also can talk about regular functions defined on open subsets of
$\HC$. In this case we require such functions to be holomorphic.

\begin{df}
Let $U$ be an open subset of $\HC$.
A holomorphic function $f: U \to \BB S$ is
{\em left anti regular} if it satisfies
$\nabla f =0$ at all points in $U$.

Similarly, a holomorphic function $g: U \to \BB S'$ is
{\em right anti regular} if
$g \overleftarrow{\nabla} =0$ at all points in $U$.
\end{df}

It is clear from \eqref{Laplacian} that anti regular functions are harmonic,
i.e. annihilated by $\square$.
One way to construct left anti regular functions is to start with a harmonic
function $\phi: \BB H \to \BB S$, then $\nabla^+ \phi$ is left anti regular.
Similarly, if $\phi: \BB H \to \BB S'$ is harmonic, then
$\phi \overleftarrow{\nabla^+}$ is right anti regular.

Let $\tilde{\cal V}_a$ and $\tilde{\cal V}'_a$ denote respectively the
spaces of (holomorphic) left and right anti regular functions on $\HC$,
possibly with singularities.

\begin{thm}  \label{r-action-thm}
\begin{enumerate}
\item
The space $\tilde{\cal V}_a$ of left anti regular functions
$\HC \to \BB S$ (possibly with singularities)
is invariant under the following action of $GL(2,\HC)$:
\begin{multline}  \label{pi_la}
\pi_{la}(h): \: f(Z) \: \mapsto \: \bigl( \pi_{la}(h)f \bigr)(Z) =
\frac{a'-Zc'}{N(a'-Zc')^2} \cdot f \bigl( (a'-Zc')^{-1}(-b'+Zd') \bigr),  \\
h = \bigl( \begin{smallmatrix} a' & b' \\ c' & d' \end{smallmatrix} \bigr)
\in GL(2,\HC).
\end{multline}
\item
The space $\tilde{\cal V}'_a$ of right anti regular functions
$\HC \to \BB S'$ (possibly with singularities)
is invariant under the following action of $GL(2,\HC)$:
\begin{multline}  \label{pi_ra}
\pi_{ra}(h): \: g(Z) \: \mapsto \: \bigl( \pi_{ra}(h)g \bigr)(Z) =
g \bigl( (aZ+b)(cZ+d)^{-1} \bigr) \cdot \frac{cZ+d}{N(cZ+d)^2},  \\
h^{-1} = \bigl( \begin{smallmatrix} a & b \\ c & d \end{smallmatrix} \bigr)
\in GL(2,\HC).
\end{multline}
\end{enumerate}
\end{thm}

Differentiating $\pi_{la}$ and $\pi_{ra}$, we obtain actions of the Lie algebra
$\mathfrak{gl}(2,\HC)$, which we still denote by $\pi_{la}$ and $\pi_{ra}$
respectively. We spell out these Lie algebra actions:

\begin{lem}  \label{pi-Lie_alg-action}
The Lie algebra action $\pi_{la}$ of $\mathfrak{gl}(2,\HC)$ on
$\tilde{\cal V}_a$ is given by
\begin{align*}
\pi_{la} \bigl( \begin{smallmatrix} A & 0 \\ 0 & 0 \end{smallmatrix} \bigr) &:
f(Z) \mapsto - \tr (AZ \partial + 2A) f + Af,  \\
\pi_{la} \bigl( \begin{smallmatrix} 0 & B \\ 0 & 0 \end{smallmatrix} \bigr) &:
f(Z) \mapsto - \tr (B \partial) f,  \\
\pi_{la} \bigl( \begin{smallmatrix} 0 & 0 \\ C & 0 \end{smallmatrix} \bigr) &:
f(Z) \mapsto \tr (ZCZ \partial + 2ZC) f - ZCf,  \\
\pi_{la} \bigl( \begin{smallmatrix} 0 & 0 \\ 0 & D \end{smallmatrix} \bigr) &:
f(Z) \mapsto \tr (ZD \partial) f.
\end{align*}

Similarly, the Lie algebra action $\pi_{ra}$ of $\mathfrak{gl}(2,\HC)$ on
$\tilde{\cal V}'_a$ is given by
\begin{align*}
\pi_{ra} \bigl( \begin{smallmatrix} A & 0 \\ 0 & 0 \end{smallmatrix} \bigr) &:
g(Z) \mapsto - \tr (AZ \partial) g,  \\
\pi_{ra} \bigl( \begin{smallmatrix} 0 & B \\ 0 & 0 \end{smallmatrix} \bigr) &:
g(Z) \mapsto - \tr (B \partial) g,  \\
\pi_{ra} \bigl( \begin{smallmatrix} 0 & 0 \\ C & 0 \end{smallmatrix} \bigr) &:
g(Z) \mapsto \tr (ZCZ \partial + 2CZ) g - gCZ,  \\
\pi_{ra} \bigl( \begin{smallmatrix} 0 & 0 \\ 0 & D \end{smallmatrix} \bigr) &:
g(Z) \mapsto \tr (ZD \partial + 2D) g - gD.
\end{align*}
\end{lem}

\subsection{Cayley Transform}

Another ingredient in identifying the irreducible components of
$(\pi'_l, {\cal U}^+)$ restricted to the Poincar\'e group is the
Cayley transform from \cite{FL1}:
\begin{align*}
\gamma = \bigl(\begin{smallmatrix} i & 1 \\ i & -1 \end{smallmatrix}\bigr):
&\: Z \mapsto (Z-i)(Z+i)^{-1} \text{ maps $\BB U(2) \to M$,
  $\BB D^{\pm} \to \BB T^{\pm}$ (with singularities)},  \\
\gamma^{-1} = \tfrac12
\bigl(\begin{smallmatrix} -i & -i \\ 1 & -1 \end{smallmatrix}\bigr):
&\: Z \mapsto -i(Z+1)(Z-1)^{-1}
\text{ maps $\BB M \to U(2)$, $\BB T^{\pm} \to \BB D^{\pm}$ (with singularities)}.
\end{align*}
Recall that $\BB T^+ = \BB M + i\operatorname{C}^+$ and
$\BB T^- = \BB M - i\operatorname{C}^+$ are tube domains in $\HC$, where
$\operatorname{C}^+$ is the open cone
\begin{equation*}
\operatorname{C}^+ = \{ Y \in \BB M ;\: N(Y)<0, \: i\tr Y <0 \}.
\end{equation*}
We are mostly interested in the bijection between functions on $\BB D^+$
and $\BB T^+$.

\begin{lem}  \label{T-invertible}
Each $Z \in \BB T^+$ is invertible.
\end{lem}

\begin{proof}
Write $Z=Y+C \in \BB T^+$ with $Y \in \BB M$ and $C \in i\operatorname{C}^+$.
Since every element of $\operatorname{C}^+$ is self-adjoint, there exists a
$u \in U(2)$ such that $uCu^*$ is diagonal:
$$
uCu^* =
\bigl(\begin{smallmatrix} \alpha & 0 \\ 0 & \beta \end{smallmatrix}\bigr),
\qquad \alpha, \beta <0.
$$
Then $uYu^* \in \BB M$, and it is easy to see that
$N(Z) = N(uYu^*+uCu^*)$ cannot be zero.
%Then $uYu^* \in \BB M$, and write
%$$
%uYu^* =
%\bigl(\begin{smallmatrix} ia & z \\ -\bar z & ib \end{smallmatrix}\bigr),
%\qquad a,b \in \BB R,\: z \in \BB C.
%$$
%If $Z$ is not invertible,
%$$
%0 = N(Z) = N(uYu^*+uCu^*)
%= (\alpha+ia)(\beta+ib) + |z|^2
%= \alpha\beta - ab + |z|^2 +i(\alpha b + \beta a)
%= \alpha\beta + \tfrac{\beta}{\alpha} a^2 + |z|^2,
%$$
%since the imaginary part of $N(Z)$ is zero.
%But the last expression is strictly positive, which produces a contradiction.
\end{proof}

\begin{cor}
The map $Z \mapsto (Z-i)(Z+i)^{-1}$ has no singularities on $\BB D^+$;
its inverse map $Z \mapsto -i(Z+1)(Z-1)^{-1}$ has no singularities on $\BB T^+$.
\end{cor}

\begin{proof}
The first map has no singularities on $\BB D^+$ because each $Z \in \BB D^+$
has eigenvalues of norm less than one, hence $Z+i$ is always invertible.
The second map has no singularities on $\BB T^+$ because, for each
$Z \in \BB T^+$, $Z-1 \in \BB T^+$ too, and, by Lemma \ref{T-invertible},
$Z-1$ is invertible.
\end{proof}

Associated to the matrices $\gamma$ and $\gamma^{-1}$ are maps on
quasi left anti regular functions
$$
\pi'_l(\gamma):\: f(Z) \to i \frac{Z-1}{N(Z-1)} f\bigl( -i(Z+1)(Z-1)^{-1} \bigr)
$$
transforming functions on $\BB D^+$ to functions on $\BB T^+$ and its inverse
$$
\pi'_l(\gamma^{-1}):\: f(Z) \to
-2 \frac{Z+i}{N(Z+i)} f\bigl( (Z-i)(Z+i)^{-1} \bigr)
$$
transforming functions on $\BB T^+$ to functions on $\BB D^+$.
We also have maps on left anti regular functions
$$
\pi_{la}(\gamma):\: f(Z) \to
-i \frac{Z-1}{N(Z-1)^2} f\bigl( -i(Z+1)(Z-1)^{-1} \bigr)
$$
transforming functions on $\BB D^+$ to functions on $\BB T^+$ and its inverse
$$
\pi_{la}(\gamma^{-1}):\: f(Z) \to
-8 \frac{Z+i}{N(Z+i)^2} f\bigl( (Z-i)(Z+i)^{-1} \bigr)
$$
transforming functions on $\BB T^+$ to functions on $\BB D^+$.
Next, we have  maps on left regular functions
$$
\pi_l(\gamma):\: f(Z) \to
8 \frac{(Z-1)^{-1}}{N(Z-1)} f\bigl( -i(Z+1)(Z-1)^{-1} \bigr)
$$
transforming functions on $\BB D^+$ to functions on $\BB T^+$ and its inverse
$$
\pi_l(\gamma^{-1}):\: f(Z) \to
i \frac{(Z+i)^{-1}}{N(Z+i)} f\bigl( (Z-i)(Z+i)^{-1} \bigr)
$$
transforming functions on $\BB T^+$ to functions on $\BB D^+$.
Finally, we have  maps on harmonic functions
$$
\pi^0_l(\gamma):\: f(Z) \to
\frac{4}{N(Z-1)} f\bigl( -i(Z+1)(Z-1)^{-1} \bigr)
$$
transforming functions on $\BB D^+$ to functions on $\BB T^+$ and its inverse
$$
\pi^0_l(\gamma^{-1}):\: f(Z) \to
\frac{-1}{N(Z+i)} f\bigl( (Z-i)(Z+i)^{-1} \bigr)
$$
transforming functions on $\BB T^+$ to functions on $\BB D^+$.

We consider the following spaces of holomorphic functions on $\BB T^+$
with values in $\BB S$:
\begin{align*}
{\cal U}^+(\BB T^+) &=\{ f: \BB T^+ \to \BB S;\:
\text{$f$ is holomorphic and $\nabla\square f =0$} \},  \\
{\cal U}^+_1(\BB T^+) = {\cal V}^+_a(\BB T^+) &= \{ f: \BB T^+ \to \BB S;\:
\text{$f$ is holomorphic and $\nabla f(Z) =0$} \},  \\
{\cal U}^+_2(\BB T^+) &= \{ f: \BB T^+ \to \BB S;\:
\text{$f$ is holomorphic and $\square f(Z) =0$} \},  \\
{\cal V}^+(\BB T^+) &= \{ f: \BB T^+ \to \BB S;\:
\text{$f$ is holomorphic and $\nabla^+ f =0$} \}.
\end{align*}
Each of these spaces has the standard topology of uniform convergence on
compact subsets.

\begin{lem}  \label{D-T-bijections}
The maps
$$
\pi'_l(\gamma) : {\cal U}^+(\BB D^+) \to {\cal U}^+(\BB T^+)
\quad \text{and} \quad
\pi'_l(\gamma^{-1}) : {\cal U}^+(\BB T^+) \to {\cal U}^+(\BB D^+)
$$
produce mutually inverse bijections of spaces of holomorphic functions
$$
{\cal U}^+(\BB D^+) \longleftrightarrow {\cal U}^+(\BB T^+), \qquad
{\cal U}^+_1(\BB D^+) \longleftrightarrow {\cal U}^+_1(\BB T^+), \qquad
{\cal U}^+_2(\BB D^+) \longleftrightarrow {\cal U}^+_2(\BB T^+).
$$
\end{lem}

\begin{proof}
The fact that $\pi'_l(\gamma)$ and $\pi'_l(\gamma^{-1})$ produce mutually
inverse bijections between ${\cal U}^+(\BB D^+)$ and ${\cal U}^+(\BB T^+)$
follows immediately from Theorem \ref{qr-action-thm} and the fact that
the Cayley transform $\gamma$ produces a bijection between $\BB D^+$ and
$\BB T^+$ (Lemma 63 in \cite{FL1}).

Next, we prove that $\pi'_l(\gamma)$ and $\pi'_l(\gamma^{-1})$
produce mutually inverse bijections between ${\cal U}^+_1(\BB D^+)$
and ${\cal U}^+_1(\BB T^+)$.
Let $f \in {\cal U}^+_1(\BB T^+)$ and observe that
$$
\bigl( \pi'_l(\gamma^{-1}) f \bigr)(Z) = 
\tfrac14 N(Z+i) \cdot \bigl( \pi_{la}(\gamma^{-1}) f \bigr)(Z)
\quad \in {\cal U}^+_1(\BB D^+),
$$
since $\tfrac14 \pi_{la}(\gamma^{-1}) f$ is left anti regular on $\BB D^+$.
Conversely, since $N(Z+i)$ never vanishes on $\BB D^+$, this formula also
shows that every function in ${\cal U}^+_1(\BB D^+)$ can be realized as
$\pi'_l(\gamma^{-1}) f$ for some $f \in {\cal U}^+_1(\BB T^+)$.

Finally, we prove that $\pi'_l(\gamma)$ and $\pi'_l(\gamma^{-1})$
produce mutually inverse bijections between ${\cal U}^+_2(\BB D^+)$
and ${\cal U}^+_2(\BB T^+)$.
Let $f \in {\cal U}^+_2(\BB T^+)$ and observe that
$$
\bigl( \pi'_l(\gamma^{-1}) f \bigr)(Z) = 
-i (Z+i) \cdot \bigl( \pi^0_l(\gamma^{-1}) f \bigr)(Z)
\quad \in {\cal U}^+_2(\BB D^+),
$$
since $-i \pi^0_l(\gamma^{-1}) f$ is harmonic on $\BB D^+$.
Conversely, since $Z+i$ is invertible on $\BB D^+$, this formula also
shows that every function in ${\cal U}^+_2(\BB D^+)$ can be realized as
$\pi'_l(\gamma^{-1}) f$ for some $f \in {\cal U}^+_2(\BB T^+)$.
\end{proof}

%\begin{lem}  \label{harmonic-sum-lem}
%Let $U$ be an open star-shaped subset of $\HC$ centered at the origin.
%\begin{enumerate}
%\item
%  Every holomorphic function $f: U \to \BB S$ such that $\square f=0$
%  can be written as a sum $f(Z)=f_1(Z)+Zf_2(Z)$, where $f_1, f_2: U \to \BB S$
%  are holomorphic, $\nabla f_1=0$ and $\nabla^+ f_2=0$.
%\item
%  Every holomorphic function $f: U \to \BB S$ such that $\nabla\square f=0$
%  can be written as a sum $f(Z)=f_1(Z)+N(Z)f_2(Z)$, where
%  $f_1, f_2: U \to \BB S$ are holomorphic, $\square f_1=0$ and $\nabla f_2=0$.
%\end{enumerate}
%\end{lem}

%\begin{proof}
%If $f: U \to \BB S$ is such that $\square f=0$, let
%$$
%f_1(Z) = f(Z) - \tfrac12 Z (\deg+2)^{-1} \nabla f(Z), \qquad 
%f_2(Z)= \tfrac12 (\deg+2)^{-1} \nabla f(Z).
%$$
%If $f: U \to \BB S$ is such that $\nabla\square f=0$, let
%$$
%f_1(Z) = f(Z) - \tfrac14 N(Z) (\deg+2)^{-1} \square f(Z), \qquad 
%f_2(Z)= \tfrac14 (\deg+2)^{-1} \square f(Z).
%$$
%The operator $(\deg+2)^{-1}$ is discussed in Subsection 2.4 of \cite{ATMP},
%and it is defined on functions with domains that are open star-shaped subsets
%of $\HC$ centered at the origin.
%\end{proof}

%\begin{lem}
%\begin{enumerate}
%\item
%The set of functions in ${\cal U}^+_2(\BB T^+)$ that extend holomorphically to
%a star-shaped open subset of $\HC$ containing $\BB T^+$ is dense in
%${\cal U}^+_2(\BB T^+)$.
%\item
%The set of functions in ${\cal U}^+(\BB T^+)$ that extend holomorphically to
%a star-shaped open subset of $\HC$ containing $\BB T^+$ is dense in
%${\cal U}^+(\BB T^+)$.
%\end{enumerate}  
%\end{lem}

\begin{lem}  \label{density-lem}
\begin{enumerate}
\item
The set of functions $f_1 \in {\cal V}^+(\BB T^+)$
that can be written as $f_1 = (\deg+2)g_1$ for some
$g_1 \in {\cal V}^+(\BB T^+)$ is dense in ${\cal V}^+(\BB T^+)$.
\item
The set of functions $f_2 \in {\cal V}^+_a(\BB T^+)$
that can be written as $f_2 = (\deg+2)g_2$ for some
$g_2 \in {\cal V}^+_a(\BB T^+)$ is dense in ${\cal V}^+_a(\BB T^+)$.
\end{enumerate}  
\end{lem}

\begin{proof}
We prove the first part only, the other case is similar.
The $\pi_l$ action of a compact group $SU(2) \times SU(2)$ on
${\cal V}^+(\BB D^+)$ is admissible, hence the set of $K$-finite elements
in ${\cal V}^+(\BB D^+)$ is dense. Since the $K$-finite elements are known
to be polynomials on $\HC$,
the set of polynomial functions are dense in ${\cal V}^+(\BB D^+)$.

The Cayley transform $\pi_l(\gamma)$ maps the polynomial functions
in ${\cal V}^+(\BB D^+)$ into a dense set in ${\cal V}^+(\BB T^+)$
that is a subset of
\begin{multline*}
{\cal V}^+ \bigl( \{ Z \in \HC ;\: N(Z-1) \ne 0 \} \bigr)  \\
= \bigl\{ f: \BB \{ Z \in \HC ;\: N(Z-1) \ne 0 \} \to \BB S;\:
\text{$f$ is holomorphic and $\nabla^+ f =0$} \bigr\}.
\end{multline*}
Thus, ${\cal V}^+ \bigl( \{ Z \in \HC ;\: N(Z-1) \ne 0 \} \bigr)$ is dense
in ${\cal V}^+(\BB T^+)$ as well.
Following Subsection 2.4 of \cite{ATMP}, on $\BB T^+$, to each
$f_1 \in {\cal V}^+ \bigl( \{ Z \in \HC ;\: N(Z-1) \ne 0 \} \bigr)$
restricted to $\BB T^+$, we can apply an operator $(\deg+2)^{-1}$
and conclude that
$$
f_1 \bigr|_{\BB T^+} = (\deg+2)g_1, \quad \text{where} \quad
g_1 = (\deg+2)^{-1} f_1 \bigr|_{\BB T^+} \in {\cal V}^+(\BB T^+).
$$
\end{proof}

\begin{prop}  \label{Poincare-iso-prop}
  The spaces ${\cal U}^+_1(\BB T^+)$ and ${\cal U}^+_2(\BB T^+)$
  are closed subspaces of ${\cal U}^+(\BB T^+)$ that are invariant under the
  $\pi'_l$ action of the Poincar\'e group $P'_{\BB R}$.
  The representations of $P'_{\BB R}$
$$
\bigl(\pi'_l, {\cal U}^+_1(\BB T^+)\bigr), \qquad
\bigl(\pi'_l, {\cal U}^+(\BB T^+)/{\cal U}^+_2(\BB T^+)\bigr)
$$
are both isomorphic to $\bigl(\pi_{la}, {\cal V}^+_a(\BB T^+)\bigr)$,
while the representation 
$$
\bigl(\pi'_l, {\cal U}^+_2(\BB T^+)/{\cal U}^+_1(\BB T^+)\bigr)
$$
is isomorphic to $\bigl(\pi_l, {\cal V}^+(\BB T^+)\bigr)$. In particular,
the three representations of the Poincar\'e group $P'_{\BB R}$
$$
\bigl(\pi'_l, {\cal U}^+_1(\BB T^+)\bigr), \qquad
\bigl(\pi'_l, {\cal U}^+_2(\BB T^+)/{\cal U}^+_1(\BB T^+)\bigr), \qquad  
\bigl(\pi'_l, {\cal U}^+(\BB T^+)/{\cal U}^+_2(\BB T^+)\bigr)
$$
are irreducible.
\end{prop}

\begin{proof}
The isomorphism $\bigl(\pi'_l, {\cal U}^+_1(\BB T^+)\bigr) \simeq
\bigl(\pi_{la}, {\cal V}^+_a(\BB T^+)\bigr)$ is obvious, since
${\cal U}^+_1(\BB T^+) = {\cal V}^+_a(\BB T^+)$ and the actions
$\pi'_l$, $\pi_{la}$ coincide on $P'_{\BB R}$.

The isomorphisms
$$
{\cal U}^+_2(\BB T^+)/{\cal U}^+_1(\BB T^+) \to {\cal V}^+(\BB T^+)
\qquad \text{and} \qquad
{\cal U}^+(\BB T^+)/{\cal U}^+_2(\BB T^+) \to {\cal V}^+_a(\BB T^+)
$$
are given by $f_1 \mapsto \nabla f_1$ and $f_2 \mapsto \square f_2$
respectively.
Clearly, both maps commute with parallel translations by elements of $\BB M$.
Next, we show that the maps commute with the actions of the
diagonal elements
$\bigl(\begin{smallmatrix} a & 0 \\ 0 & d \end{smallmatrix}\bigr)
\in P'_{\BB R}$.
If $f_1 \in {\cal U}^+_2(\BB T^+)$ and $f_2 \in {\cal U}^+(\BB T^+)$,
using identities
$$
\nabla \bigl[ f_1(a^{-1}Zd) \bigr] = d \bigl[\nabla (a^{-1} f_1)\bigr] (a^{-1}Zd)
\quad \text{and} \quad
\square \bigl[ f_2(a^{-1}Zd) \bigr] = N(a^{-1}d) [\square f_2] (a^{-1}Zd),
$$
we have:
\begin{align*}
\nabla \bigl[ \pi'_l \bigl(\begin{smallmatrix} a & 0 \\
    0 & d \end{smallmatrix}\bigr) f_1 \bigr]
&= \nabla \bigl[ a f_1(a^{-1}Zd) \bigr]
= d [\nabla f_1] (a^{-1}Zd)
= \pi_l \bigl(\begin{smallmatrix} a & 0 \\
  0 & d \end{smallmatrix}\bigr) \bigl[ \nabla f_1 \bigr],  \\
\square \bigl[ \pi'_l \bigl(\begin{smallmatrix} a & 0 \\
    0 & d \end{smallmatrix}\bigr) f_2 \bigr]
&= \square \bigl[ a f_2(a^{-1}Zd) \bigr]
= a [\square f_2] (a^{-1}Zd)
= \pi_{la} \bigl(\begin{smallmatrix} a & 0 \\
  0 & d \end{smallmatrix}\bigr) \bigl[ \square f_2 \bigr].
\end{align*}
%If $f \in {\cal U}^+_1(\BB T^+)$, write $f=f_1+Zf_2$ with $\nabla f_1=0$
%and $\nabla^+f_2=0$. Then, using identity $2(\deg+2)=\nabla Z + Z^+\nabla^+$
%(Lemma 7 from \cite{ATMP}),
%\begin{multline*}
%\nabla \bigl[ \pi'_l \bigl(\begin{smallmatrix} a & 0 \\
%    0 & d \end{smallmatrix}\bigr) f \bigr]
%= \nabla \bigl[ Zd f_2(a^{-1}Zd) \bigr]
%= (\nabla Z + Z^+\nabla^+) \bigl[ d f_2(a^{-1}Zd) \bigr]  \\
%= 2(\deg+2) \bigl[ \pi_l \bigl(\begin{smallmatrix} a & 0 \\
%    0 & d \end{smallmatrix}\bigr) f_2 \bigr]
%= \pi_l \bigl(\begin{smallmatrix} a & 0 \\
%  0 & d \end{smallmatrix}\bigr) \bigl[ 2(\deg+2) f_2 \bigr]
%= \pi_l \bigl(\begin{smallmatrix} a & 0 \\
%  0 & d \end{smallmatrix}\bigr) \bigl[ \nabla f \bigr].
%\end{multline*}
%If $f \in {\cal U}^+_2(\BB T^+)$, write $f=f_1+N(Z)f_2$ with $\square f_1=0$
%and $\nabla f_2=0$. Then, using \eqref{Laplacian-Nf},
%\begin{multline*}
%\square \bigl[ \pi'_l \bigl(\begin{smallmatrix} a & 0 \\
%    0 & d \end{smallmatrix}\bigr) f \bigr]
%= \square \bigl[ a N(Z) f_2(a^{-1}Zd) \bigr]
%= 4(\deg+2) \bigl[ a f_2(a^{-1}Zd) \bigr]  \\
%= a \bigl[ 4(\deg+2) f_2 \bigr] (a^{-1}Zd)
%= a \bigl[ \square N(Z) f_2 \bigr] (a^{-1}Zd)
%= \pi_{la} \bigl(\begin{smallmatrix} a & 0 \\
%  0 & d \end{smallmatrix}\bigr) \bigl[ \square f \bigr].
%\end{multline*}

The maps $\nabla$ and $\square$ have dense images because the functions
that appear in Lemma \ref{density-lem} are in the respective ranges of these
maps. Indeed, if $f_1=(\deg+2)g_1 \in {\cal V}^+(\BB T^+)$
and $f_2=(\deg+2)g_2 \in {\cal V}^+_a(\BB T^+)$, then, by \eqref{Laplacian-Nf}
and identity $2(\deg+2)=\nabla Z + Z^+\nabla^+$ (Lemma 7 from \cite{ATMP}),
\begin{align*}
\tfrac12 \nabla (Z g_1) &=f_1, \qquad \tfrac12 Zg_1 \in {\cal U}^+_2(\BB T^+), \\
\tfrac14 \square (N(Z) g_2) &=f_2, \qquad
\tfrac14 N(Z)g_2 \in {\cal U}^+(\BB T^+).
\end{align*}

The three representations of the Poincar\'e group $P'_{\BB R}$
$$
\bigl(\pi'_l, {\cal U}^+_1(\BB T^+)\bigr), \qquad
\bigl(\pi'_l, {\cal U}^+_2(\BB T^+)/{\cal U}^+_1(\BB T^+)\bigr), \qquad  
\bigl(\pi'_l, {\cal U}^+(\BB T^+)/{\cal U}^+_2(\BB T^+)\bigr)
$$
are irreducible because they are isomorphic as representations of the
Poincar\'e group $P_{\BB R}$ to either
$\bigl(\pi_{la}, {\cal V}^+_a(\BB D^+)\bigr)$ or
$\bigl(\pi_l, {\cal V}^+(\BB D^+)\bigr)$, and these are known to be irreducible
\cite{MT} (and a more accessible argument can be found in \cite{W}).
\end{proof}

\section{Invariant Bilinear Pairing and Pseudounitary Structures}  \label{Sect6}

\subsection{Invariant Bilinear Pairing}

In this subsection we construct a $\mathfrak{gl}(2,\HC)$-invariant bilinear
pairing between left and right quasi anti regular functions.

Let $S^3_R = \{ X \in \BB H ;\: N(X)=R^2 \}$ be the sphere of radius $R>0$
centered at the origin, and let $dS$ be the usual Euclidean volume element
on $S^3_R$.
We define a bilinear pairing $\langle f, g \rangle_{QR}$ between
$f \in {\cal U}$ and $g \in {\cal U}'$ by declaring
$$
\langle f, g \rangle_{QR} = 0 \quad \text{on ${\cal U}^+ \times {\cal U}'^+$
and ${\cal U}^- \times {\cal U}'^-$}
$$
and by the formula
\begin{multline}  \label{QR-pairing}
\langle f, g \rangle_{QR} = \frac1{8\pi^2}
\int_{X \in S^3_R} \bigl( \deg g(X) \bigr) \cdot \nabla f(X) \,\frac{dS}R  \\
- \frac1{8\pi^2}
\int_{X \in S^3_R} g(X) \cdot \deg \nabla f(X) \,\frac{dS}R  \\
- \frac{R}{8\pi^2} \int_{X \in S^3_R}
\bigl( \square g(X) \bigr) \cdot X^{-1} \cdot f(X) \,dS
\end{multline}
when $f \in {\cal U}^+$ and $g \in {\cal U}'^-$ or 
$f \in {\cal U}^-$ and $g \in {\cal U}'^+$.
(While the formula makes sense if, for example, $f \in {\cal U}^+$ and
$g \in {\cal U}'^+$, it does not produce a $\mathfrak{gl}(2,\HC)$-invariant
pairing.)
Note that the expression \eqref{QR-pairing} is inspired by the reproducing
formula for QLAR functions \eqref{Cauchy-Fueter-left}, so that
$$
f(X_0) = \biggl\langle f(X), \frac{X-X_0}{N(X-X_0)} \biggr\rangle_{QR}
\qquad \text{if } X_0 \in \BB H,\: N(X_0)<R.
$$

\begin{lem}  \label{orthog-rels-lem}
We have the following orthogonality properties among the families of
QLAR and QRAR functions:
\begin{align*}
\bigl\langle f^{(1)}_{l,m,n}(Z), \tilde g^{(1)}_{l',m',n'}(Z) \bigr\rangle_{QR}
&= \bigl\langle \tilde f^{(1)}_{l,m,n}(Z), g^{(1)}_{l',m',n'}(Z) \bigr\rangle_{QR}
= (2l+1) \delta_{ll'} \delta_{mm'} \delta_{nn'}, \\
\bigl\langle f^{(2)}_{l,m,n}(Z), \tilde g^{(2)}_{l',m',n'}(Z) \bigr\rangle_{QR}
&= \bigl\langle \tilde f^{(2)}_{l,m,n}(Z), g^{(2)}_{l',m',n'}(Z) \bigr\rangle_{QR}
= -2l(2l+1) \delta_{ll'} \delta_{mm'} \delta_{nn'}, \\
\bigl\langle f^{(3)}_{l,m,n}(Z), \tilde g^{(3)}_{l',m',n'}(Z) \bigr\rangle_{QR}
&= \bigl\langle \tilde f^{(3)}_{l,m,n}(Z), g^{(3)}_{l',m',n'}(Z) \bigr\rangle_{QR}
= -2l \delta_{ll'} \delta_{mm'} \delta_{nn'}, \\
%\langle \tilde f^{(1)}_{l,m,n}(Z), g^{(1)}_{l',m',n'}(Z) \rangle_{QR}
%&= (2l+1) \delta_{ll'} \delta_{mm'} \delta_{nn'}, \\
%\langle \tilde f^{(2)}_{l,m,n}(Z), g^{(2)}_{l',m',n'}(Z) \rangle_{QR}
%&= -2l(2l+1) \delta_{ll'} \delta_{mm'} \delta_{nn'}, \\
%\langle \tilde f^{(3)}_{l,m,n}(Z), g^{(3)}_{l',m',n'}(Z) \rangle_{QR}
%&= -2l \delta_{ll'} \delta_{mm'} \delta_{nn'}, \\
\bigl\langle f^{(1)}_{l,m,n}(Z), \tilde g^{(2)}_{l',m',n'}(Z) \bigr\rangle_{QR}
&=
\bigl\langle f^{(1)}_{l,m,n}(Z), \tilde g^{(3)}_{l',m',n'}(Z) \bigr\rangle_{QR} =0, \\
\bigl\langle f^{(2)}_{l,m,n}(Z), \tilde g^{(1)}_{l',m',n'}(Z) \bigr\rangle_{QR}
&=
\bigl\langle f^{(2)}_{l,m,n}(Z), \tilde g^{(3)}_{l',m',n'}(Z) \bigr\rangle_{QR} =0, \\
\bigl\langle f^{(3)}_{l,m,n}(Z), \tilde g^{(1)}_{l',m',n'}(Z) \bigr\rangle_{QR}
&=
\bigl\langle f^{(3)}_{l,m,n}(Z), \tilde g^{(2)}_{l',m',n'}(Z) \bigr\rangle_{QR} =0, \\
\bigl\langle \tilde f^{(1)}_{l,m,n}(Z), g^{(2)}_{l',m',n'}(Z) \bigr\rangle_{QR}
&=
\bigl\langle \tilde f^{(1)}_{l,m,n}(Z), g^{(3)}_{l',m',n'}(Z) \bigr\rangle_{QR} =0, \\
\bigl\langle \tilde f^{(2)}_{l,m,n}(Z), g^{(1)}_{l',m',n'}(Z) \bigr\rangle_{QR}
&=
\bigl\langle \tilde f^{(2)}_{l,m,n}(Z), g^{(3)}_{l',m',n'}(Z) \bigr\rangle_{QR} =0, \\
\bigl\langle \tilde f^{(3)}_{l,m,n}(Z), g^{(1)}_{l',m',n'}(Z) \bigr\rangle_{QR}
&= \bigl\langle \tilde f^{(3)}_{l,m,n}(Z), g^{(2)}_{l',m',n'}(Z) \bigr\rangle_{QR} =0.
\end{align*}
In particular, the pairing \eqref{QR-pairing} is non-degenerate and
does not depend on the choice of $R>0$.
\end{lem}

\begin{proof}
The result follows from the orthogonality relations (17) in \cite{desitter}
and intermediate computations:
$$
\nabla f^{(1)}_{l,m,n}(Z) = \nabla \tilde f^{(3)}_{l,m,n}(Z) = 0,
$$
$$
\nabla f^{(2)}_{l,m,n}(Z) = 2(2l+1) \begin{pmatrix}
(l-m) t^{l-\frac12}_{n \,\underline{m+\frac12}}(Z) \\
(l+m) t^{l-\frac12}_{n \,\underline{m-\frac12}}(Z) \end{pmatrix}, \quad 
\nabla f^{(3)}_{l,m,n}(Z) = 2 \begin{pmatrix}
t^{l-\frac12}_{n \,\underline{m+\frac12}}(Z) \\
- t^{l-\frac12}_{n \,\underline{m-\frac12}}(Z) \end{pmatrix},
$$
$$
\nabla \tilde f^{(1)}_{l,m,n}(Z) =
\frac2{N(Z)} \begin{pmatrix}
t^l_{m-\frac12 \,\underline{n}}(Z^{-1}) \\
- t^l_{m+\frac12 \,\underline{n}}(Z^{-1}) \end{pmatrix}, \quad
\nabla \tilde f^{(2)}_{l,m,n}(Z) =
-\frac{4l}{N(Z)} \begin{pmatrix}
(l-m+\frac12) t^l_{m-\frac12 \,\underline{n}}(Z^{-1}) \\
(l+m+\frac12) t^l_{m+\frac12 \,\underline{n}}(Z^{-1}) \end{pmatrix},
$$
$$
\square g^{(1)}_{l,m,n}(Z) = \square g^{(2)}_{l,m,n}(Z)
= \square \tilde g^{(2)}_{l,m,n}(Z) = \square \tilde g^{(3)}_{l,m,n}(Z) =0,
$$
$$
\square g^{(3)}_{l,m,n}(Z) \cdot Z^{-1} = \frac{8l}{N(Z)} \bigl(
(l+n) t^{l-\frac12}_{n+\frac12 \,\underline{m}}(Z),
-(l-n) t^{l-\frac12}_{n-\frac12 \,\underline{m}}(Z) \bigr),
$$
$$
\square \tilde g^{(1)}_{l,m,n}(Z) \cdot Z^{-1} = -\frac{4(2l+1)}{N(Z)^2} \bigl(
(l+n+1/2) t^l_{m\,\underline{n-\frac12}}(Z^{-1}),
-(l-n+1/2) t^l_{m\,\underline{n+\frac12}}(Z^{-1}) \bigr).
$$
\end{proof}

From these orthogonality relations and Lemma \ref{inversion-action-lem} we
immediately see the effect of the inversion on the pairing \eqref{QR-pairing}:

\begin{cor}  \label{inversion-pairing-cor}
For all $f \in {\cal U}$ and $g \in {\cal U}'$,
$$  
\Bigl\langle
\pi'_l \bigl( \begin{smallmatrix} 0 & 1 \\ 1 & 0 \end{smallmatrix} \bigr) f,
\pi'_r \bigl( \begin{smallmatrix} 0 & 1 \\ 1 & 0 \end{smallmatrix} \bigr) g
\Bigr\rangle_{QR} = - \langle f, g \rangle_{QR}.
$$
\end{cor}

We are ready to prove that the bilinear pairing is
$\mathfrak{gl}(2,\HC)$-invariant.

\begin{prop}  \label{bilinear-pairing-prop}
The equation \eqref{QR-pairing} defines an $SU(2) \times SU(2)$ and
$\mathfrak{gl}(2,\HC)$-invariant bilinear pairing between
$(\pi'_l,{\cal U}^+)$ and $(\pi'_r,{\cal U}'^-)$ and also between
$(\pi'_l,{\cal U}^-)$ and $(\pi'_r,{\cal U}'^+)$.
This pairing is non-degenerate and independent of the choice of $R>0$.
\end{prop}

\begin{proof}
We already know that this pairing is independent of the choice of $R>0$.
By direct computation, for $a, d \in \HC^{\times}$ with $N(a)=N(d)=1$, we have:
$$
\deg \bigl( f(a^{-1}Zd) \bigr) = (\deg f)(a^{-1}Zd),
$$
$$
\nabla \bigl( f(a^{-1}Zd) \bigr)
= d \bigl( \nabla (a^{-1}f) \bigr)(a^{-1}Zd),
\qquad
\nabla^+ \bigl( f(a^{-1}Zd) \bigr) = a \bigl( \nabla^+ (d^{-1}f) \bigr)(a^{-1}Zd),
$$
$$
\bigl( g(a^{-1}Zd) \bigr) \overleftarrow{\nabla}
= \bigl( (gd)\overleftarrow{\nabla} (a^{-1}Zd) \bigr) a^{-1},
\qquad
\bigl( g(a^{-1}Zd) \bigr) \overleftarrow{\nabla}^+
= \bigl( (ga)\overleftarrow{\nabla}^+ (a^{-1}Zd) \bigr) d^{-1},
$$
$$
\square \bigl( af(a^{-1}Zd) \bigr) = a (\square f) (a^{-1}Zd), \qquad
\square \bigl( g(a^{-1}Zd)d^{-1} \bigr) = \bigl((\square g)(a^{-1}Zd)\bigr)d^{-1}.
$$
Then it is easy to see that the pairing \eqref{QR-pairing}
is $SU(2) \times SU(2)$ invariant.
The invariance under the scalar  matrices 
$\bigl( \begin{smallmatrix} \lambda & 0 \\
0 & \lambda \end{smallmatrix} \bigr)$, $\lambda \in \BB R^{\times}$,
is trivial.
The invariance under the dilation matrices 
$\bigl( \begin{smallmatrix} \lambda & 0 \\
0 & \lambda^{-1} \end{smallmatrix} \bigr)$, $\lambda \in \BB R^{\times}$,
follows from the independence of the pairing \eqref{QR-pairing}
of the choice of $R>0$.

Next, we verify the invariance of the pairing under
$\bigl( \begin{smallmatrix} 0 & B \\ 0 & 0 \end{smallmatrix} \bigr)
\in \mathfrak{gl}(2,\HC)$, $B \in \HC$.
Since we have already established the $SU(2) \times SU(2)$ invariance,
it is sufficient to check for
$B = \bigl( \begin{smallmatrix} 1 & 0 \\ 0 & 0 \end{smallmatrix} \bigr)$
only. We do so by verifying that
\begin{equation}  \label{B-invariance1}
\Bigl\langle
\pi'_l \bigl( \begin{smallmatrix} 0 & B \\ 0 & 0 \end{smallmatrix} \bigr)
f^{(\alpha)}_{l,m,n}, \tilde g^{(\alpha')}_{l',m',n'} \Bigr\rangle_{QR}
+ \Bigl\langle f^{(\alpha)}_{l,m,n},
\pi'_r \bigl( \begin{smallmatrix} 0 & B \\ 0 & 0 \end{smallmatrix} \bigr)
\tilde g^{(\alpha')}_{l',m',n'} \Bigr\rangle_{QR} = 0,
\end{equation}
\begin{equation}  \label{B-invariance2}
\Bigl\langle
\pi'_l \bigl( \begin{smallmatrix} 0 & B \\ 0 & 0 \end{smallmatrix} \bigr)
\tilde f^{(\alpha)}_{l,m,n}, g^{(\alpha')}_{l',m',n'} \Bigr\rangle_{QR}
+ \Bigl\langle \tilde f^{(\alpha)}_{l,m,n},
\pi'_r \bigl( \begin{smallmatrix} 0 & B \\ 0 & 0 \end{smallmatrix} \bigr)
g^{(\alpha')}_{l',m',n'} \Bigr\rangle_{QR} = 0
\end{equation}
for the basis functions, where $\alpha,\alpha'=1,2,3$.
Note that the actions of
$\pi'_r \bigl( \begin{smallmatrix} 0 & B \\ 0 & 0 \end{smallmatrix} \bigr)$ and
$\pi'_r \bigl( \begin{smallmatrix} 0 & B \\ 0 & 0 \end{smallmatrix} \bigr)$
were spelled out in Subsection \ref{K-type-action}.
We start with $\alpha=1$, then by the orthogonality relations
(Lemma \ref{orthog-rels-lem}), both summands in \eqref{B-invariance1} are zero
unless $\alpha'=1$, $l'=l-\frac12$, $m'=m+\frac12$, $n'=n+\frac12$.
In this case,
$$
\Bigl\langle
\pi'_l \bigl( \begin{smallmatrix} 0 & B \\ 0 & 0 \end{smallmatrix} \bigr)
f^{(1)}_{l,m,n}, \tilde g^{(1)}_{l',m',n'} \Bigr\rangle_{QR}
= -2l(l-m)
= - \Bigl\langle f^{(1)}_{l,m,n},
\pi'_r \bigl( \begin{smallmatrix} 0 & B \\ 0 & 0 \end{smallmatrix} \bigr)
\tilde g^{(1)}_{l',m',n'} \Bigr\rangle_{QR}.
$$
And both summands in
\eqref{B-invariance2} are zero unless $l'=l+\frac12$, $m'=m-\frac12$,
$n'=n-\frac12$. In this case,
\begin{align*}
\Bigl\langle
\pi'_l \bigl( \begin{smallmatrix} 0 & B \\ 0 & 0 \end{smallmatrix} \bigr)
\tilde f^{(1)}_{l,m,n}, g^{(1)}_{l',m',n'} \Bigr\rangle_{QR}
&= (2l+1)(l-m+3/2)
= - \Bigl\langle \tilde f^{(1)}_{l,m,n},
\pi'_r \bigl( \begin{smallmatrix} 0 & B \\ 0 & 0 \end{smallmatrix} \bigr)
g^{(1)}_{l',m',n'} \Bigr\rangle_{QR},  \\
\Bigl\langle
\pi'_l \bigl( \begin{smallmatrix} 0 & B \\ 0 & 0 \end{smallmatrix} \bigr)
\tilde f^{(1)}_{l,m,n}, g^{(2)}_{l',m',n'} \Bigr\rangle_{QR}
&= 1
= - \Bigl\langle \tilde f^{(1)}_{l,m,n},
\pi'_r \bigl( \begin{smallmatrix} 0 & B \\ 0 & 0 \end{smallmatrix} \bigr)
g^{(2)}_{l',m',n'} \Bigr\rangle_{QR},  \\
\Bigl\langle
\pi'_l \bigl( \begin{smallmatrix} 0 & B \\ 0 & 0 \end{smallmatrix} \bigr)
\tilde f^{(1)}_{l,m,n}, g^{(3)}_{l',m',n'} \Bigr\rangle_{QR}
&= l+n
= - \Bigl\langle \tilde f^{(1)}_{l,m,n},
\pi'_r \bigl( \begin{smallmatrix} 0 & B \\ 0 & 0 \end{smallmatrix} \bigr)
g^{(3)}_{l',m',n'} \Bigr\rangle_{QR}.
\end{align*}

When $\alpha=2$, both summands in \eqref{B-invariance1} are zero unless
$\alpha'=1$ or $2$, $l'=l-\frac12$, $m'=m+\frac12$, $n'=n+\frac12$.
In this case,
\begin{align*}
\Bigl\langle
\pi'_l \bigl( \begin{smallmatrix} 0 & B \\ 0 & 0 \end{smallmatrix} \bigr)
f^{(2)}_{l,m,n}, \tilde g^{(1)}_{l',m',n'} \Bigr\rangle_{QR}
&= -2l(l-m)(l+n+1/2)
= - \Bigl\langle f^{(2)}_{l,m,n},
\pi'_r \bigl( \begin{smallmatrix} 0 & B \\ 0 & 0 \end{smallmatrix} \bigr)
\tilde g^{(1)}_{l',m',n'} \Bigr\rangle_{QR},  \\
\Bigl\langle
\pi'_l \bigl( \begin{smallmatrix} 0 & B \\ 0 & 0 \end{smallmatrix} \bigr)
f^{(2)}_{l,m,n}, \tilde g^{(2)}_{l',m',n'} \Bigr\rangle_{QR}
&= (2l-1)(2l+1)(l-m)
= - \Bigl\langle f^{(2)}_{l,m,n},
\pi'_r \bigl( \begin{smallmatrix} 0 & B \\ 0 & 0 \end{smallmatrix} \bigr)
\tilde g^{(2)}_{l',m',n'} \Bigr\rangle_{QR}.
\end{align*}
And both summands in
\eqref{B-invariance2} are zero unless $\alpha'=2$ or $3$, $l'=l+\frac12$,
$m'=m-\frac12$, $n'=n-\frac12$. In this case,
\begin{align*}
\Bigl\langle
\pi'_l \bigl( \begin{smallmatrix} 0 & B \\ 0 & 0 \end{smallmatrix} \bigr)
\tilde f^{(2)}_{l,m,n}, g^{(2)}_{l',m',n'} \Bigr\rangle_{QR}
&= -2l(2l+2)(l-m+1/2)
= - \Bigl\langle \tilde f^{(2)}_{l,m,n},
\pi'_r \bigl( \begin{smallmatrix} 0 & B \\ 0 & 0 \end{smallmatrix} \bigr)
g^{(2)}_{l',m',n'} \Bigr\rangle_{QR},  \\
\Bigl\langle
\pi'_l \bigl( \begin{smallmatrix} 0 & B \\ 0 & 0 \end{smallmatrix} \bigr)
\tilde f^{(2)}_{l,m,n}, g^{(3)}_{l',m',n'} \Bigr\rangle_{QR}
&= (l-m+1/2)(l+n)
= - \Bigl\langle \tilde f^{(2)}_{l,m,n},
\pi'_r \bigl( \begin{smallmatrix} 0 & B \\ 0 & 0 \end{smallmatrix} \bigr)
g^{(3)}_{l',m',n'} \Bigr\rangle_{QR}.
\end{align*}

When $\alpha=3$, both summands in \eqref{B-invariance1} are zero unless
$l'=l-\frac12$, $m'=m+\frac12$, $n'=n+\frac12$. In this case,
\begin{align*}
\Bigl\langle
\pi'_l \bigl( \begin{smallmatrix} 0 & B \\ 0 & 0 \end{smallmatrix} \bigr)
f^{(3)}_{l,m,n}, \tilde g^{(1)}_{l',m',n'} \Bigr\rangle_{QR}
&= -(l+n+1/2)
= - \Bigl\langle f^{(3)}_{l,m,n},
\pi'_r \bigl( \begin{smallmatrix} 0 & B \\ 0 & 0 \end{smallmatrix} \bigr)
\tilde g^{(1)}_{l',m',n'} \Bigr\rangle_{QR},  \\
\Bigl\langle
\pi'_l \bigl( \begin{smallmatrix} 0 & B \\ 0 & 0 \end{smallmatrix} \bigr)
f^{(3)}_{l,m,n}, \tilde g^{(2)}_{l',m',n'} \Bigr\rangle_{QR}
&= -1
= - \Bigl\langle f^{(3)}_{l,m,n},
\pi'_r \bigl( \begin{smallmatrix} 0 & B \\ 0 & 0 \end{smallmatrix} \bigr)
\tilde g^{(2)}_{l',m',n'} \Bigr\rangle_{QR},  \\
\Bigl\langle
\pi'_l \bigl( \begin{smallmatrix} 0 & B \\ 0 & 0 \end{smallmatrix} \bigr)
f^{(3)}_{l,m,n}, \tilde g^{(3)}_{l',m',n'} \Bigr\rangle_{QR}
&= 2l(l-m-1)
= - \Bigl\langle f^{(3)}_{l,m,n},
\pi'_r \bigl( \begin{smallmatrix} 0 & B \\ 0 & 0 \end{smallmatrix} \bigr)
\tilde g^{(3)}_{l',m',n'} \Bigr\rangle_{QR}.
\end{align*}
And both summands in
\eqref{B-invariance2} are zero unless $\alpha'=3$, $l'=l+\frac12$,
$m'=m-\frac12$, $n'=n-\frac12$. In this case,
$$
\Bigl\langle
\pi'_l \bigl( \begin{smallmatrix} 0 & B \\ 0 & 0 \end{smallmatrix} \bigr)
\tilde f^{(3)}_{l,m,n}, g^{(3)}_{l',m',n'} \Bigr\rangle_{QR}
= -(2l+1)(l-m+1/2)
= - \Bigl\langle \tilde f^{(3)}_{l,m,n},
\pi'_r \bigl( \begin{smallmatrix} 0 & B \\ 0 & 0 \end{smallmatrix} \bigr)
g^{(3)}_{l',m',n'} \Bigr\rangle_{QR}.
$$
This proves the invariance of the pairing under
$\bigl( \begin{smallmatrix} 0 & B \\ 0 & 0 \end{smallmatrix} \bigr)
\in \mathfrak{gl}(2,\HC)$, $B \in \HC$.

Finally, it follows from Corollary \ref{inversion-pairing-cor} that the
pairing is also invariant under
$\bigl( \begin{smallmatrix} 0 & 0 \\ C & 0 \end{smallmatrix} \bigr)
\in \mathfrak{gl}(2,\HC)$, $C \in \HC$.
\end{proof}

We have an alternative description of the invariant bilinear pairing
\eqref{QR-pairing}.

\begin{lem}
The $\mathfrak{gl}(2,\HC)$-invariant bilinear pairing \eqref{QR-pairing}
on ${\cal U} \times {\cal U}'$ can be described as
\begin{align*}
\bigl\langle f^{(1)}_{l,m,n}(Z), \tilde g^{(1)}_{l',m',n'}(Z) \bigr\rangle_{QR} &=
- \frac{R}{8\pi^2} \int_{X \in S^3_R}
\bigl( \square \tilde g^{(1)}_{l',m',n'}(X) \bigr) \cdot X^{-1} \cdot
f^{(1)}_{l,m,n}(X) \,dS,  \\
\bigl\langle f^{(2)}_{l,m,n}(Z), \tilde g^{(2)}_{l',m',n'}(Z) \bigr\rangle_{QR} &=
-\frac{2l}{4\pi^2} \int_{X \in S^3_R}
\tilde g^{(2)}_{l',m',n'}(X) \cdot \nabla f^{(2)}_{l,m,n}(X) \,\frac{dS}R,  \\
\bigl\langle f^{(3)}_{l,m,n}(Z), \tilde g^{(3)}_{l',m',n'}(Z) \bigr\rangle_{QR} &=
-\frac{2l}{4\pi^2} \int_{X \in S^3_R}
\tilde g^{(3)}_{l',m',n'}(X) \cdot \nabla f^{(3)}_{l,m,n}(X) \,\frac{dS}R,  \\
\bigl\langle f^{(\alpha)}_{l,m,n}(Z), \tilde g^{(\alpha')}_{l',m',n'}(Z)
\bigr\rangle_{QR} &= 0 \qquad \text{if $\alpha \ne \alpha'$};  \\
\bigl\langle \tilde f^{(1)}_{l,m,n}(Z), g^{(1)}_{l',m',n'}(Z) \bigr\rangle_{QR} &=
\frac{2l+1}{4\pi^2} \int_{X \in S^3_R}
g^{(1)}_{l',m',n'}(X) \cdot \nabla \tilde  f^{(1)}_{l,m,n}(X) \,\frac{dS}R,  \\
\bigl\langle \tilde f^{(2)}_{l,m,n}(Z), g^{(2)}_{l',m',n'}(Z) \bigr\rangle_{QR} &=
\frac{2l+1}{4\pi^2} \int_{X \in S^3_R}
g^{(2)}_{l',m',n'}(X) \cdot \nabla \tilde  f^{(2)}_{l,m,n}(X) \,\frac{dS}R,  \\
\bigl\langle \tilde f^{(3)}_{l,m,n}(Z), g^{(3)}_{l',m',n'}(Z) \bigr\rangle_{QR} &=
- \frac{R}{8\pi^2} \int_{X \in S^3_R}
\bigl( \square g^{(3)}_{l',m',n'}(X) \bigr) \cdot X^{-1} \cdot
\tilde f^{(3)}_{l,m,n}(X) \,dS,  \\
\bigl\langle \tilde f^{(\alpha)}_{l,m,n}(Z), g^{(\alpha')}_{l',m',n'}(Z)
\bigr\rangle_{QR} &= 0 \qquad \text{if $\alpha \ne \alpha'$};  \\
\bigl\langle f^{(\alpha)}_{l,m,n}(Z), g^{(\alpha')}_{l',m',n'}(Z)
\bigr\rangle_{QR} &=
\bigl\langle \tilde f^{(\alpha)}_{l,m,n}(Z), \tilde g^{(\alpha')}_{l',m',n'}(Z)
\bigr\rangle_{QR} = 0.
\end{align*}
\end{lem}

\begin{proof}
We verify by direct computation that this pairing satisfies the same
orthogonality relations as those described in Lemma \ref{orthog-rels-lem}.
Then the result follows.
\end{proof}

\subsection{Invariant Pseudounitary Structures}

In this subsection we construct pseudounitary (non-degenerate indefinite)
$\mathfrak{u}(2,2)$-invariant structures on ${\cal U}^+$, ${\cal U}^-$,
${\cal U}'^+$ and ${\cal U}'^-$.
We realize the group $U(2,2)$ as the subgroup of $GL(2,\HC)$ as in
\eqref{U(2,2)}. Then the Lie algebra of $U(2,2)$ is given by
\eqref{u(2,2)-algebra}.

We have a complex bilinear pairing
$$
\langle \,\cdot\, ,  \,\cdot\, \rangle : \BB S' \times \BB S \to \BB C,
\qquad
\Bigl\langle (s_1', s_2'), \begin{pmatrix} s_1 \\ s_2 \end{pmatrix}
\Bigr\rangle = s_1's_1 + s_2's_2.
$$
We also have $\BB C$-antilinear maps $\BB S \to \BB S'$ and $\BB S' \to \BB S$
-- matrix transposition followed by complex conjugation --
which are similar to quaternionic conjugation and so, by abuse of notation,
we use the same symbol to denote these maps:
$$
\begin{pmatrix} s_1 \\ s_2 \end{pmatrix}^+ = (\overline{s_1}, \overline{s_2}),
\qquad
(s'_1,s'_2)^+ =\begin{pmatrix} \overline{s'_1} \\ \overline{s'_2}\end{pmatrix},
\qquad
\begin{pmatrix} s_1 \\ s_2 \end{pmatrix} \in \BB S ,\: (s'_1,s'_2) \in \BB S'.
$$
Note that $(Xs)^+ = s^+ X^+$ and $(s'X)^+ = X^+ s'^+$
for all $X \in \BB H$, $s \in \BB S$, $s' \in \BB S'$.

\begin{thm}  \label{unitary-thm}
The spaces $(\pi'_l, {\cal U}^+)$, $(\pi'_l, {\cal U}^-)$,
$(\pi'_r, {\cal U}'^+)$ and $(\pi'_r, {\cal U}'^-)$ have pseudounitary
structures that are preserved by the real form $\mathfrak{u}(2,2)$ of
$\mathfrak{gl}(2,\HC)$. The invariant pseudounitary structure on ${\cal U}^+$
can be described as
\begin{align*}
\bigl( f^{(1)}_{l,m,n}(Z), f^{(1)}_{l',m',n'}(Z) \bigr)_{{\cal U}^+} &=
\frac{2l+1}{2\pi^2} \int_{X \in S^3}
\bigl( f^{(1)}_{l',m',n'}(X) \bigr)^+ \cdot f^{(1)}_{l,m,n}(X) \,dS,  \\
\bigl( f^{(2)}_{l,m,n}(Z), f^{(2)}_{l',m',n'}(Z) \bigr)_{{\cal U}^+} &=
-\frac{2l(2l+1)}{2\pi^2} \int_{X \in S^3}
\bigl( f^{(2)}_{l',m',n'}(X) \bigr)^+ \cdot f^{(2)}_{l,m,n}(X) \,dS,  \\
\bigl( f^{(3)}_{l,m,n}(Z), f^{(3)}_{l',m',n'}(Z) \bigr)_{{\cal U}^+} &=
-\frac{2l}{2\pi^2} \int_{X \in S^3}
\bigl( f^{(3)}_{l',m',n'}(X) \bigr)^+ \cdot f^{(3)}_{l,m,n}(X) \,dS,  \\
\bigl( f^{(\alpha)}_{l,m,n}(Z), f^{(\alpha')}_{l',m',n'}(Z) \bigr)_{{\cal U}^+} &= 0
\qquad \text{if $\alpha \ne \alpha'$}.
\end{align*}
And the invariant pseudounitary structures on ${\cal U}^-$, ${\cal U}'^+$,
${\cal U}'^-$ are similar.
\end{thm}

\begin{proof}
Clearly, the product $(f_1,f_2)_{{\cal U}^+}$ is $U(2) \times U(2)$ invariant.
From \eqref{t-orthog-rels} we obtain orthogonality relations for the
$f^{(\alpha)}_{l,m,n}(Z)$'s:
\begin{align*}
\bigl( f^{(1)}_{l,m,n}(Z), f^{(1)}_{l',m',n'}(Z) \bigr)_{{\cal U}^+} &=
(2l+1) \frac{(l-m)!(l+m)!}{(l-n+\frac12)!(l+n+\frac12)!}
\delta_{ll'} \delta_{mm'} \delta_{nn'},  \\
\bigl( f^{(2)}_{l,m,n}(Z), f^{(2)}_{l',m',n'}(Z) \bigr)_{{\cal U}^+} &=
-2l(2l+1) \frac{(l-m)!(l+m)!}{(l-n-\frac12)!(l+n-\frac12)!}
\delta_{ll'} \delta_{mm'} \delta_{nn'},  \\
\bigl( f^{(3)}_{l,m,n}(Z), f^{(3)}_{l',m',n'}(Z) \bigr)_{{\cal U}^+} &=
-2l \frac{(l-m-1)!(l+m-1)!}{(l-n-\frac12)!(l+n-\frac12)!}
\delta_{ll'} \delta_{mm'} \delta_{nn'}.
\end{align*}

%$$
%X=\left(\begin{smallmatrix} 0 & 0 & 1 & 0 \\ 0 & 0 & 0 & 0 \\ 1 & 0 & 0 & 0 \\
%  0 & 0 & 0 & 0 \end{smallmatrix} \right), \quad
%Y=\left(\begin{smallmatrix} 0 & 0 & i & 0 \\ 0 & 0 & 0 & 0 \\ -i & 0 & 0 & 0 \\
%  0 & 0 & 0 & 0 \end{smallmatrix} \right) \quad
%\in \mathfrak{u}(2,2).
%$$

To prove that $(f_1,f_2)_{{\cal U}^+}$ is $\mathfrak{u}(2,2)$-invariant,
it is sufficient to show that $(f_1,f_2)_{{\cal U}^+}$ is invariant under
$\pi'_l(X)$ and $\pi'_l(Y)$, where $X, Y \in \mathfrak{u}(2,2)$ are as
in \eqref{XYinu(2,2)}.
From Subsection \ref{K-type-action} we find:
\begin{multline*}
\pi'_l(X) f^{(1)}_{l,m,n}(Z) = -(l-m)f^{(1)}_{l-\frac12,m+\frac12,n+\frac12}(Z)
+ \frac{(2l+1)(l-n+\frac32)}{2l+2} f^{(1)}_{l+\frac12,m-\frac12,n-\frac12}(Z)  \\
- \frac1{(2l+1)(2l+2)} f^{(2)}_{l+\frac12,m-\frac12,n-\frac12}(Z)
- \frac{l+m}{2l+1} f^{(3)}_{l+\frac12,m-\frac12,n-\frac12}(Z),
\end{multline*}
\begin{multline*}
\pi'_l(X) f^{(2)}_{l,m,n}(Z)
= -\frac{(l-m)(l+n+\frac12)}{2l} f^{(1)}_{l-\frac12,m+\frac12,n+\frac12}(Z)
- \frac{(2l+1)(l-m)}{2l} f^{(2)}_{l-\frac12,m+\frac12,n+\frac12}(Z)  \\
+ \frac{2l(l-n+\frac12)}{2l+1} f^{(2)}_{l+\frac12,m-\frac12,n-\frac12}(Z)
- \frac{(l+m)(l-n+\frac12)}{2l+1} f^{(3)}_{l+\frac12,m-\frac12,n-\frac12}(Z),
\end{multline*}
\begin{multline*}
\pi'_l(X) f^{(3)}_{l,m,n}(Z)
= -\frac{l+n+\frac12}{2l} f^{(1)}_{l-\frac12,m+\frac12,n+\frac12}(Z)
+ \frac1{2l(2l-1)} f^{(2)}_{l-\frac12,m+\frac12,n+\frac12}(Z)  \\
-\frac{2l(l-m-1)}{2l-1} f^{(3)}_{l-\frac12,m+\frac12,n+\frac12}(Z)
+(l-n+1/2) f^{(3)}_{l+\frac12,m-\frac12,n-\frac12}(Z).
\end{multline*}
And we can find similar expressions for
$$
\pi'_l(Y) f^{(1)}_{l,m,n}(Z), \qquad \pi'_l(Y) f^{(2)}_{l,m,n}(Z), \qquad
\pi'_l(Y) f^{(3)}_{l,m,n}(Z).
$$
Then, by direct calculation,
\begin{align*}
\bigl( f^{(1)}_{l,m,n}(Z),
\pi'_l(X) f^{(1)}_{l+\frac12,m-\frac12,n-\frac12}(Z) \bigr)_{{\cal U}^+} &=
- (2l+1) \frac{(l-m+1)!(l+m)!}{(l-n+\frac12)!(l+n+\frac12)!}  \\
&= - \bigl( \pi'_l(X) f^{(1)}_{l,m,n}(Z),
f^{(1)}_{l+\frac12,m-\frac12,n-\frac12}(Z) \bigr)_{{\cal U}^+},  \\
\bigl( f^{(1)}_{l,m,n}(Z),
\pi'_l(X) f^{(2)}_{l+\frac12,m-\frac12,n-\frac12}(Z) \bigr)_{{\cal U}^+} &=
- \frac{(l-m+1)!(l+m)!}{(l-n+\frac12)!(l+n-\frac12)!}  \\
&= - \bigl( \pi'_l(X) f^{(1)}_{l,m,n}(Z),
f^{(2)}_{l+\frac12,m-\frac12,n-\frac12}(Z) \bigr)_{{\cal U}^+},  \\
\bigl( f^{(1)}_{l,m,n}(Z),
\pi'_l(X) f^{(3)}_{l+\frac12,m-\frac12,n-\frac12}(Z) \bigr)_{{\cal U}^+} &=
- (2l+1) \frac{(l-m)!(l+m)!}{(l-n+\frac12)!(l+n-\frac12)!}  \\
&= - \bigl( \pi'_l(X) f^{(3)}_{l,m,n}(Z),
f^{(1)}_{l+\frac12,m-\frac12,n-\frac12}(Z) \bigr)_{{\cal U}^+},  \\
\bigl( f^{(2)}_{l,m,n}(Z),
\pi'_l(X) f^{(2)}_{l+\frac12,m-\frac12,n-\frac12}(Z) \bigr)_{{\cal U}^+} &=
2l(2l+2) \frac{(l-m+1)!(l+m)!}{(l-n-\frac12)!(l+n-\frac12)!}  \\
&= - \bigl( \pi'_l(X) f^{(2)}_{l,m,n}(Z),
f^{(2)}_{l+\frac12,m-\frac12,n-\frac12}(Z) \bigr)_{{\cal U}^+},  \\
\bigl( f^{(2)}_{l,m,n}(Z),
\pi'_l(X) f^{(3)}_{l+\frac12,m-\frac12,n-\frac12}(Z) \bigr)_{{\cal U}^+} &=
- \frac{(l-m)!(l+m)!}{(l-n-\frac12)!(l+n-\frac12)!}  \\
&= - \bigl( \pi'_l(X) f^{(2)}_{l,m,n}(Z),
f^{(3)}_{l+\frac12,m-\frac12,n-\frac12}(Z) \bigr)_{{\cal U}^+},  \\
\bigl( f^{(3)}_{l,m,n}(Z),
\pi'_l(X) f^{(3)}_{l+\frac12,m-\frac12,n-\frac12}(Z) \bigr)_{{\cal U}^+} &=
(2l+1) \frac{(l-m)!(l+m-1)!}{(l-n-\frac12)!(l+n-\frac12)!}  \\
&= - \bigl( \pi'_l(X) f^{(3)}_{l,m,n}(Z),
f^{(3)}_{l+\frac12,m-\frac12,n-\frac12}(Z) \bigr)_{{\cal U}^+}.
\end{align*}
This proves the $\pi'_l(X)$-invariance of the pseudounitary structure.

The calculations for $\pi'_l(Y)$ are nearly identical.
\end{proof}

Recall that, by Proposition \ref{irred-prop}, $(\pi'_l, {\cal U}^+)$,
$(\pi'_l, {\cal U}^-)$, $(\pi'_r, {\cal U}'^+)$ and $(\pi'_r, {\cal U}'^-)$
are irreducible representations of $\mathfrak{gl}(2,\HC)$ and hence
$\mathfrak{u}(2,2)$, and the (pseudo)unitary structure on an irreducible
representation is unique up to rescaling. Thus we obtain:

\begin{cor}
The spaces $(\pi'_l, {\cal U}^+)$, $(\pi'_l, {\cal U}^-)$,
$(\pi'_r, {\cal U}'^+)$ and $(\pi'_r, {\cal U}'^-)$
do {\em not} have $\mathfrak{u}(2,2)$-invariant unitary structures.
\end{cor}

\section{Reproducing Kernel Expansion}  \label{Sect7}

In this section we provide two expansions of the reproducing kernel
$\frac{Z-W}{N(Z-W)}$ in terms of the basis functions given in
Propositions \ref{K-typebasis+_prop}, \ref{K-typebasis-_prop}.
These expansions should be considered as quasi regular analogues
of the matrix coefficient expansions of the Cauchy-Fueter kernel
$\frac{(Z-W)^{-1}}{N(Z-W)}$ from Proposition 26 in \cite{FL1}
(see also Proposition 113 in \cite{ATMP}).

\begin{thm}  \label{kernel-exp-thm}
We have the reproducing kernel expansions:
\begin{multline}  \label{1st-kernel-expansion}
\frac{Z-W}{N(Z-W)}
= - \sum_{\genfrac{}{}{0pt}{}{l \ge 0}{m, n}}
\frac1{2l+1} f^{(1)}_{l,m,n}(Z) \cdot \tilde g^{(1)}_{l,m,n}(W)  \\
+ \sum_{\genfrac{}{}{0pt}{}{l \ge 1/2}{m, n}} \frac1{2l(2l+1)}
f^{(2)}_{l,m,n}(Z) \cdot \tilde g^{(2)}_{l,m,n}(W)
+ \sum_{\genfrac{}{}{0pt}{}{l \ge 1}{m, n}}
\frac1{2l} f^{(3)}_{l,m,n}(Z) \cdot \tilde g^{(3)}_{l,m,n}(W),
\end{multline}
which converges uniformly on compact subsets in the region
$\{ (Z,W) \in \HC \times \HC^{\times} ; \: ZW^{-1} \in \BB D^+ \}$.
The sum is taken first over all applicable $m$ and $n$,
then over $l=0,\frac 12, 1, \frac 32,\dots$.

Similarly,
\begin{multline}  \label{2nd-kernel-expansion}
\frac{Z-W}{N(Z-W)}
= \sum_{\genfrac{}{}{0pt}{}{l \ge 0}{m, n}}
\frac1{2l+1} \tilde f^{(1)}_{l,m,n}(Z) \cdot g^{(1)}_{l,m,n}(W)  \\
- \sum_{\genfrac{}{}{0pt}{}{l \ge 1/2}{m, n}} \frac1{2l(2l+1)}
\tilde f^{(2)}_{l,m,n}(Z) \cdot g^{(2)}_{l,m,n}(W)
- \sum_{\genfrac{}{}{0pt}{}{l \ge 1}{m, n}}
\frac1{2l} \tilde f^{(3)}_{l,m,n}(Z) \cdot g^{(3)}_{l,m,n}(W),
\end{multline}
which converges uniformly on compact subsets in the region
$\{ (Z,W) \in \HC^{\times} \times \HC; \: Z^{-1}W \in \BB D^+ \}$.
The sum is taken first over all applicable $m$ and $n$,
then over $l=0,\frac 12, 1, \frac 32,\dots$.
\end{thm}

\begin{proof}
We have:
\begin{multline*}
\sum_{\genfrac{}{}{0pt}{}{l \ge 0}{m, n}}
\frac1{2l+1} f^{(1)}_{l,m,n}(Z) \cdot \tilde g^{(1)}_{l,m,n}(W)
- \sum_{\genfrac{}{}{0pt}{}{l \ge 1/2}{m, n}} \frac1{2l(2l+1)}
f^{(2)}_{l,m,n}(Z) \cdot \tilde g^{(2)}_{l,m,n}(W)
- \sum_{\genfrac{}{}{0pt}{}{l \ge 1}{m, n}}
\frac1{2l} f^{(3)}_{l,m,n}(Z) \cdot \tilde g^{(3)}_{l,m,n}(W)  \\
= \sum_{\genfrac{}{}{0pt}{}{l \ge 0}{m, n}}
\frac1{2l+1} \begin{pmatrix} t^l_{n-\frac12 \,\underline{m}}(Z) \\
- t^l_{n+\frac12 \,\underline{m}}(Z) \end{pmatrix}
\bigl( (l+m+1) t^{l+\frac12}_{m+\frac12\,\underline{n}}(W^{-1}),
-(l-m+1) t^{l+\frac12}_{m-\frac12\,\underline{n}}(W^{-1}) \bigr)  \\
- \sum_{\genfrac{}{}{0pt}{}{l \ge 1/2}{m, n}} \frac{N(W)^{-1}}{2l(2l+1)}
\begin{pmatrix} (l-n+\frac12) t^l_{n-\frac12 \,\underline{m}}(Z) \\
(l+n+\frac12) t^l_{n+\frac12 \,\underline{m}}(Z) \end{pmatrix}
\bigl( t^{l-\frac12}_{m+\frac12\,\underline{n}}(W^{-1}),
t^{l-\frac12}_{m-\frac12\,\underline{n}}(W^{-1}) \bigr)  \\
- \sum_{\genfrac{}{}{0pt}{}{l \ge 1}{m, n}} \frac{N(Z)}{2l N(W)}
\begin{pmatrix} t^{l-1}_{n-\frac12 \,\underline{m}}(Z) \\
-t^{l-1}_{n+\frac12 \,\underline{m}}(Z) \end{pmatrix}
\bigl( (l+m) t^{l-\frac12}_{m+\frac12\,\underline{n}}(W^{-1}),
-(l-m) t^{l-\frac12}_{m-\frac12\,\underline{n}}(W^{-1}) \bigr)  \\
= \sum_{\genfrac{}{}{0pt}{}{l \ge 0}{m, n}} \frac1{2l+1} \begin{pmatrix}
(l+m+1)t^l_{n-\frac12 \,\underline{m}}(Z) t^{l+\frac12}_{m+\frac12\,\underline{n}}(W^{-1}) &
-(l-m+1)t^l_{n-\frac12 \,\underline{m}}(Z) t^{l+\frac12}_{m-\frac12\,\underline{n}}(W^{-1})\\
-(l+m+1)t^l_{n+\frac12 \,\underline{m}}(Z) t^{l+\frac12}_{m+\frac12\,\underline{n}}(W^{-1}) &
(l-m+1)t^l_{n+\frac12 \,\underline{m}}(Z) t^{l+\frac12}_{m-\frac12\,\underline{n}}(W^{-1})
\end{pmatrix}  \\
- \sum_{\genfrac{}{}{0pt}{}{l \ge 1/2}{m, n}} \frac{N(W)^{-1}}{2l(2l+1)}
\begin{pmatrix} (l-n+\frac12) t^l_{n-\frac12 \,\underline{m}}(Z)
t^{l-\frac12}_{m+\frac12\,\underline{n}}(W^{-1}) &
(l-n+\frac12) t^l_{n-\frac12 \,\underline{m}}(Z)
t^{l-\frac12}_{m-\frac12\,\underline{n}}(W^{-1}) \\
(l+n+\frac12) t^l_{n+\frac12 \,\underline{m}}(Z)
t^{l-\frac12}_{m+\frac12\,\underline{n}}(W^{-1}) &
(l+n+\frac12) t^l_{n+\frac12 \,\underline{m}}(Z)
t^{l-\frac12}_{m-\frac12\,\underline{n}}(W^{-1}) \end{pmatrix}  \\
- \sum_{\genfrac{}{}{0pt}{}{l \ge 1}{m, n}} \frac{N(Z)}{2l N(W)} \begin{pmatrix}
(l+m)t^{l-1}_{n-\frac12 \,\underline{m}}(Z) t^{l-\frac12}_{m+\frac12\,\underline{n}}(W^{-1}) &
-(l-m)t^{l-1}_{n-\frac12 \,\underline{m}}(Z)t^{l-\frac12}_{m-\frac12\,\underline{n}}(W^{-1})\\
-(l+m)t^{l-1}_{n+\frac12 \,\underline{m}}(Z) t^{l-\frac12}_{m+\frac12\,\underline{n}}(W^{-1})&
(l-m)t^{l-1}_{n+\frac12 \,\underline{m}}(Z)t^{l-\frac12}_{m-\frac12\,\underline{n}}(W^{-1})
\end{pmatrix}.
\end{multline*}
We deal with the three sums separately. The first sum is
\begin{multline*}
\sum_{\genfrac{}{}{0pt}{}{l \ge 0}{m, n}}
\frac1{2l+1} f^{(1)}_{l,m,n}(Z) \cdot \tilde g^{(1)}_{l,m,n}(W)  \\ 
= \sum_{\genfrac{}{}{0pt}{}{l \ge 1/2}{m, n}} \frac1{2l} \begin{pmatrix}
(l+m)t^{l-\frac12}_{n-\frac12 \,\underline{m-\frac12}}(Z) &
-(l-m)t^{l-\frac12}_{n-\frac12 \,\underline{m+\frac12}}(Z) \\
-(l+m)t^{l-\frac12}_{n+\frac12 \,\underline{m-\frac12}}(Z) &
(l-m)t^{l-\frac12}_{n+\frac12 \,\underline{m+\frac12}}(Z) \end{pmatrix}
t^l_{m\,\underline{n}}(W^{-1})  \\
= \sum_{\genfrac{}{}{0pt}{}{l \ge 1/2}{m, n}} \frac1{2l} \partial^+_Z
t^l_{n\,\underline{m}}(Z) t^l_{m\,\underline{n}}(W^{-1})
= \partial^+_Z \sum_{l \ge \frac12,\: n} \frac1{2l} t^l_{n\,\underline{n}}(Z W^{-1}).
\end{multline*}
Assume further that $ZW^{-1}$ can be diagonalized as
$\bigl(\begin{smallmatrix} \lambda_1 & 0 \\
0 & \lambda_2 \end{smallmatrix}\bigr)$ with
$\lambda_1 \ne \lambda_2$.
This is allowed since the set of matrices with distinct eigenvalues is
dense in $\HC$.
Then the sum $\sum_{n} t^l_{n \, \underline{n}}(ZW^{-1})$
is just the character $\chi_l(ZW^{-1})$ of the irreducible representation of
$GL(2,\BB C)$ of dimension $2l+1$ and equals
$\frac{\lambda_1^{2l+1}-\lambda_2^{2l+1}}{\lambda_1-\lambda_2}$.
We drop the factors $(2l)^{-1}$ and consider
\begin{multline*}
\partial^+_Z \sum_{l \ge \frac12,\: n} t^l_{n\,\underline{n}}(Z W^{-1})
= \partial^+_Z \sum_{l \ge 0}
\frac{\lambda_1^{2l+1}-\lambda_2^{2l+1}}{\lambda_1-\lambda_2}
= \partial^+_Z
\frac{\lambda_1(1-\lambda_1)^{-1} - \lambda_2(1-\lambda_2)^{-1}}
{\lambda_1-\lambda_2}  \\
= \partial^+_Z \frac1{(1-\lambda_1)(1-\lambda_2)}
= \partial^+_Z \frac{N(W)}{N(Z-W)} = F_1(Z,W),
\end{multline*}
then
$$
\partial^+_Z \sum_{l \ge \frac12,\: n} \frac1{2l} t^l_{n\,\underline{n}}(Z W^{-1})
= \int_{s=0}^{s=1} F_1(sZ,W) \,ds.
$$

The second sum is
\begin{multline*}
- \sum_{\genfrac{}{}{0pt}{}{l \ge 1/2}{m, n}} \frac1{2l(2l+1)}
f^{(2)}_{l,m,n}(Z) \cdot \tilde g^{(2)}_{l,m,n}(W)  \\
= \sum_{\genfrac{}{}{0pt}{}{l \ge 0}{m, n}} \frac{-N(W)^{-1}}{(2l+1)(2l+2)}
\begin{pmatrix} (l-n+1) t^{l+\frac12}_{n-\frac12 \,\underline{m-\frac12}}(Z) &
(l-n+1) t^{l+\frac12}_{n-\frac12 \,\underline{m+\frac12}}(Z) \\
(l+n+1) t^{l+\frac12}_{n+\frac12 \,\underline{m-\frac12}}(Z) &
(l+n+1) t^{l+\frac12}_{n+\frac12 \,\underline{m+\frac12}}(Z) \end{pmatrix}
t^l_{m\,\underline{n}}(W^{-1}) \\
= -\frac{Z}{N(W)} \sum_{\genfrac{}{}{0pt}{}{l \ge 0}{m, n}} \frac1{2l+2}
t^l_{n\,\underline{m}}(Z) t^l_{m\,\underline{n}}(W^{-1})  \\
+ \sum_{\genfrac{}{}{0pt}{}{l \ge 1/2}{m, n}} \frac{N(Z) N(W)^{-1}}{(2l+1)(2l+2)} \begin{pmatrix}
(l+m)t^{l-\frac12}_{n-\frac12 \,\underline{m-\frac12}}(Z) &
-(l-m)t^{l-\frac12}_{n-\frac12 \,\underline{m+\frac12}}(Z) \\
-(l+m)t^{l-\frac12}_{n+\frac12 \,\underline{m-\frac12}}(Z) &
(l-m)t^{l-\frac12}_{n+\frac12 \,\underline{m+\frac12}}(Z)
\end{pmatrix} t^l_{m\,\underline{n}}(W^{-1})
\end{multline*}
The second term will be combined with the third sum.
We drop the factors $(2l+2)^{-1}$ from the first term:
\begin{multline*}
-\frac{Z}{N(W)} \sum_{\genfrac{}{}{0pt}{}{l \ge 0}{m, n}}
t^l_{n\,\underline{m}}(Z) t^l_{m\,\underline{n}}(W^{-1})
= -\frac{Z}{N(W)} \sum_{l \ge 0,\: n} t^l_{n\,\underline{n}}(ZW^{-1})  \\
= - \frac{Z}{N(W)} \sum_{l \ge 0}
\frac{\lambda_1^{2l+1}-\lambda_2^{2l+1}}{\lambda_1-\lambda_2}
= - \frac{Z}{N(W)}
\frac{\lambda_1 (1-\lambda_1)^{-1} - \lambda_2 (1-\lambda_2)^{-1}}
{\lambda_1-\lambda_2}  \\
= - \frac{Z N(W)^{-1}}{(1-\lambda_1)(1-\lambda_2)}
= - \frac{Z}{N(Z-W)} = F_2(Z,W),
\end{multline*}
then
$$
-\frac{Z}{N(W)} \sum_{\genfrac{}{}{0pt}{}{l \ge 0}{m, n}} \frac1{2l+2}
t^l_{n\,\underline{m}}(Z) t^l_{m\,\underline{n}}(W^{-1})
= \int_{s=0}^{s=1} F_2(sZ,W) \,ds.
$$

The third sum (combined with a similar term from the second sum) is
\begin{multline*}
- \frac{N(Z)}{N(W)} \sum_{\genfrac{}{}{0pt}{}{l \ge 1/2}{m, n}}
\frac1{2l+2} \begin{pmatrix}
(l+m)t^{l-\frac12}_{n-\frac12 \,\underline{m-\frac12}}(Z) &
-(l-m)t^{l-\frac12}_{n-\frac12 \,\underline{m+\frac12}}(Z) \\
-(l+m)t^{l-\frac12}_{n+\frac12 \,\underline{m-\frac12}}(Z) &
(l-m)t^{l-\frac12}_{n+\frac12 \,\underline{m+\frac12}}(Z)
\end{pmatrix} t^l_{m\,\underline{n}}(W^{-1})  \\
= - \frac{N(Z)}{N(W)}
\sum_{\genfrac{}{}{0pt}{}{l \ge 1/2}{m, n}} \frac1{2l+2} \partial^+_Z
t^l_{n\,\underline{m}}(Z) t^l_{m\,\underline{n}}(W^{-1})
= - \frac{N(Z)}{N(W)} \partial^+_Z
\sum_{l \ge \frac12,\: n} \frac1{2l+2} t^l_{n\,\underline{n}}(Z W^{-1}).
\end{multline*}
We drop the factors $(2l+2)^{-1}$:
\begin{multline*}
- \frac{N(Z)}{N(W)}
\sum_{l \ge \frac12,\: n} \partial^+_Z t^l_{n\,\underline{n}}(Z W^{-1})
= - \frac{N(Z)}{N(W)} \partial^+_Z \sum_{l \ge \frac12}
\frac{\lambda_1^{2l+1}-\lambda_2^{2l+1}}{\lambda_1-\lambda_2}  \\
= - \frac{N(Z)}{N(W)} \partial^+_Z \frac1{(1-\lambda_1)(1-\lambda_2)}
= - N(Z) \partial^+_Z \frac1{N(Z-W)} = F_3(Z,W),
\end{multline*}
then
$$
- \frac{N(Z)}{N(W)} \partial^+_Z
\sum_{l \ge \frac12,\: n} \frac1{2l+2} t^l_{n\,\underline{n}}(Z W^{-1})
= \int_{s=0}^{s=1} F_3(sZ,W) \,ds.
$$

We combine the three sums:
\begin{multline*}
F_1(Z,W) + F_2(Z,W) + F_3(Z,W)
= \bigl( N(W)-N(Z) \bigr) \partial^+_Z \frac1{N(Z-W)} - \frac{Z}{N(Z-W)}  \\
= \bigl( N(Z)-N(W) \bigr) \frac{Z-W}{N(Z-W)^2} - \frac{Z}{N(Z-W)}  \\
= \biggl( \Bigl( \frac{N(Z)}{N(W)} -1 \Bigr) \frac{ZW^{-1}-1}{N(ZW^{-1}-1)^2}
- \frac{ZW^{-1}}{N(ZW^{-1}-1)} \biggr) \frac{W}{N(W)}.
\end{multline*}
Finally,
\begin{multline*}
\int_{s=0}^{s=1} \bigl( F_1(sZ,W) + F_2(sZ,W) + F_3(sZ,W) \bigr)\,ds  \\
= \int_{s=0}^{s=1} \biggl( \Bigl( \frac{N(sZ)}{N(W)} -1 \Bigr)
\frac{sZW^{-1}-1}{N(sZW^{-1}-1)^2}
- \frac{sZW^{-1}}{N(sZW^{-1}-1)} \biggr) \frac{W}{N(W)}\,ds  \\
= \int_{s=0}^{s=1} \biggl(
\frac{\lambda_1\lambda_2 s^2 -1}{(s\lambda_1-1)^2(s\lambda_2-1)^2}
\left(\begin{smallmatrix} s\lambda_1 -1 & 0 \\
0 & s\lambda_2 -1 \end{smallmatrix}\right)
- \frac{s}{(s\lambda_1-1)(s\lambda_2-1)}
\left(\begin{smallmatrix} \lambda_1 & 0 \\
0 & \lambda_2 \end{smallmatrix}\right) \biggr) \frac{W}{N(W)}\,ds  \\
= \int_{s=0}^{s=1} \left(\begin{smallmatrix} (s\lambda_2 -1)^{-2} & 0 \\
0 & (s\lambda_1 -1)^{-2} \end{smallmatrix}\right)  \frac{W}{N(W)}\,ds
= \left(\begin{smallmatrix} (1-\lambda_2)^{-1} & 0 \\ 0 & (1-\lambda_1)^{-1}
\end{smallmatrix}\right) \frac{W}{N(W)}  \\
= \left(\begin{smallmatrix} 1-\lambda_1 & 0 \\ 0 & 1-\lambda_2
\end{smallmatrix}\right)
\frac{W}{N(W)(1-\lambda_1)(1-\lambda_2)}
= - \frac{Z-W}{N(Z-W)}.
\end{multline*}
This proves \eqref{1st-kernel-expansion}.

The proof of \eqref{2nd-kernel-expansion} is similar.
\end{proof}

\begin{rem}
Observe that the reproducing kernels for biharmonic functions
and quasi anti regular functions are related as
\begin{equation*}
\begin{split}
\nabla^+_Z \log\bigl[ N(1-ZW^{-1}) \bigr]
%= 2 \partial^+_Z \log\bigl[ N(1-ZW^{-1}) \bigr]
&= 2 \frac{Z-W}{N(Z-W)},  \\
\nabla^+_W \log\bigl[ N(1-Z^{-1}W) \bigr]
%= 2 \partial^+_W \log\bigl[ N(1-Z^{-1}W) \bigr]
&= - 2 \frac{Z-W}{N(Z-W)}.
\end{split}
\end{equation*}
\end{rem}

\section{Equivariant Maps Involving the Tensor Product of Quasi Anti Regular
  Functions ${\cal U} \otimes {\cal U}'$}  \label{Sect8}

The quasi regular functions can be regarded as a square root of the space of
quaternionic holomorphic functions ${\cal W}'$ in the same sense as regular
functions provide a square root of the dual space ${\cal W}$.
In particular, the pointwise multiplication between functions from
${\cal U}$ and ${\cal U}'$ is an equivariant map to ${\cal W}'$ in the same
way as the pointwise multiplication between functions from ${\cal V}$ and
${\cal V}'$ is an equivariant map to ${\cal W}$.
Conversely, the dual space ${\cal W}$ can be embedded into a certain
holomorphic completion of the tensor product ${\cal U} \otimes {\cal U}'$
(the same is true for embedding of ${\cal W}'$ into a completion of
${\cal V} \otimes {\cal V}'$).
In this section we study the constituents of these equivariant maps in detail.

\subsection{Equivariant Maps ${\cal W} \to {\cal U} \otimes {\cal U}'$
and ${\cal U} \otimes {\cal U}' \to {\cal W}'$}  \label{Subsect8.1}

Recall that in \cite{FL1} we introduced a space of
$$
\bigl\{\text{$\HC$-valued polynomial functions on $\HC^{\times}$}\bigr\}
= \HC \otimes \BB C[z_{11},z_{12},z_{21},z_{22}, N(Z)^{-1}]
$$
%\begin{align*}
%{\cal W} = {\cal W}' &= \bigl\{\text{$\HC$-valued polynomial functions on
%$\HC^{\times}$}\bigr\}  \\
%&= \HC \otimes \BB C[z_{11},z_{12},z_{21},z_{22}, N(Z)^{-1}]
%\end{align*}
equipped with two actions of the Lie algebra $\g{gl}(2,\HC)$ obtained by
differentiating the following group actions:
\begin{align*}
\rho_2(h): \: F(Z) \quad &\mapsto \quad \bigl( \rho_2(h)F \bigr)(Z) =
\frac {(cZ+d)^{-1}}{N(cZ+d)} \cdot F \bigl( (aZ+b)(cZ+d)^{-1} \bigr) \cdot
\frac {(a'-Zc')^{-1}}{N(a'-Zc')},  \\
\rho'_2(h): \: G(Z) \quad &\mapsto \quad \bigl( \rho'_2(h)G \bigr)(Z) =
\frac {a'-Zc'}{N(a'-Zc')} \cdot G \bigl( (aZ+b)(cZ+d)^{-1} \bigr)
\cdot \frac {cZ+d}{N(cZ+d)},
\end{align*}
%where $F \in {\cal W}$, $G \in {\cal W}'$,
where $F, G \in  \HC \otimes \BB C[z_{11},z_{12},z_{21},z_{22}, N(Z)^{-1}]$,
$h = \bigl(\begin{smallmatrix} a' & b' \\ c' & d' \end{smallmatrix}\bigr)
\in GL(2,\HC)$ and 
$h^{-1} = \bigl(\begin{smallmatrix} a & b \\ c & d \end{smallmatrix}\bigr)$.
We denote by ${\cal W}$ and ${\cal W}'$ the corresponding modules over
$\g{gl}(2,\HC)$.

As a representation of $\mathfrak{gl}(2,\HC)$,
$(\rho_2,{\cal W})$ was decomposed into thirteen irreducible components
in Subsection 5.1 in \cite{ATMP}.
According to Corollary 56 in \cite{ATMP}, the quotient
$\bigl( \rho_2, {\cal W}/\ker (\tr \circ \partial^+) \bigr)$
has five irreducible components that are isomorphic to
\begin{equation}  \label{5-irred}
(\rho, \Sh^+), \quad (\rho, \Sh^-), \quad (\rho, {\cal J}), \quad
\bigl( \rho, {\cal I}^+/(\Sh^+ \oplus {\cal J}) \bigr), \quad
\bigl( \rho, {\cal I}^-/(\Sh^- \oplus {\cal J}) \bigr).
\end{equation}
And $\ker (\tr \circ \partial^+) \subset {\cal W}$ has eight more:
three ``large'' components
\begin{align*}
{\cal Q}^+ &= \BB C\text{-span of }
\bigl\{ N(Z)^k \cdot {\bf F_{l,m,n}}(Z),\: N(Z)^k \cdot {\bf G_{l,m,n}}(Z),\:
{\bf H_{k,l,m,n}}(Z);\: k \ge 0 \bigr\}, \\
{\cal Q}^- &= \BB C\text{-span of }
\bigl\{ N(Z)^k \cdot {\bf F_{l,m,n}}(Z),\: N(Z)^k \cdot {\bf G_{l,m,n}}(Z),\:
{\bf H_{k,l,m,n}}(Z);\: k \le -(2l+3) \bigr\}, \\
{\cal Q}^0 &= \BB C\text{-span of }
\biggl\{ \begin{matrix}
N(Z)^k \cdot {\bf F_{l,m,n}}(Z),\: N(Z)^k \cdot {\bf G_{l,m,n}}(Z)
\text{ with } -(2l+1) \le k \le -2, \\
{\bf H_{k,l,m,n}}(Z) \text{ with } -(2l+2) \le k \le -1 \end{matrix} \biggr\};
\end{align*}
four components isomorphic to doubly regular functions
\begin{align*}
{\cal F}^+ &= \BB C\text{-span of }
\bigl\{ N(Z)^{-1} \cdot {\bf F_{l,m,n}}(Z);\: l \ge 1/2 \bigr\}, \\
{\cal F}^- &= \BB C\text{-span of }
\bigl\{ N(Z)^{-(2l+2)} \cdot {\bf G_{l,m,n}}(Z) ;\: l \ge 1/2 \bigr\}  \\
&\qquad \qquad = \BB C\text{-span of }
\bigl\{ N(Z)^{-1} \cdot {\bf F'_{l,m,n}}(Z);\: l \ge 1/2 \bigr\}, \\
{\cal G}^+ &=  \BB C\text{-span of }
\bigl\{ N(Z)^{-1} \cdot {\bf G_{l,m,n}}(Z);\: l \ge 1/2 \bigr\}, \\
{\cal G}^- &= \BB C\text{-span of }
\bigl\{ N(Z)^{-(2l+2)} \cdot {\bf F_{l,m,n}}(Z) ;\: l \ge 1/2 \bigr\}  \\
&\qquad \qquad = \BB C\text{-span of }
\bigl\{ N(Z)^{-1} \cdot {\bf G'_{l,m,n}}(Z);\: l \ge 1/2 \bigr\}
\end{align*}
and the trivial one-dimensional representation
$$
\BB C\text{-span of } \bigl\{ N(Z)^{-2} \cdot Z^+ \bigr\}.
$$
The way these irreducible components fit together can be described as follows.
Theorem 31 and Lemma 55 in \cite{ATMP} illustrate how the five irreducible
components \eqref{5-irred} ``nest'' inside
${\cal W}/\ker (\tr \circ \partial^+)$.

\begin{prop}[Section 5.1 in \cite{ATMP}]
The $\rho_2$-invariant subspace
$\ker (\tr \circ \partial^+) \subset {\cal W}$ has three irreducible
subrepresentations: ${\cal Q}^+$, ${\cal Q}^-$, ${\cal Q}^0$.
And the quotient
$$
\frac{\ker (\tr \circ \partial^+ : {\cal W} \to \Sh)}
{{\cal Q}^- \oplus {\cal Q}^0 \oplus {\cal Q}^+}
\simeq {\cal F}^+ \oplus {\cal F}^- \oplus {\cal G}^+ \oplus {\cal G}^-
\oplus \BB C\text{-span of } \bigl\{ N(Z)^{-2} \cdot Z^+ \bigr\}.
$$
\end{prop}

Since $(\rho'_2,{\cal W}')$ is the dual representation of $(\rho_2,{\cal W})$,
it also has thirteen irreducible components; they are described in
Subsection 5.2 in \cite{ATMP}.
The composition structure of $(\rho'_2,{\cal W}')$ ``mirrors'' that of
$(\rho_2,{\cal W})$ and can be described as follows.

\begin{rem}
The representation $(\rho'_2,{\cal W}')$ contains a $\rho'_2$-invariant
subspace $\partial^+(\Sh')$ which is isomorphic to $(\rho',\Sh'/{\cal I}'_0)$
and decomposes further into five irreducible components that are
dual to \eqref{5-irred}:
$$
(\rho', {\cal BH}^+/{\cal I}'_0), \quad (\rho', {\cal BH}^-/{\cal I}'_0), \quad
\bigl( \rho', {\cal J}'/({\cal BH}^++{\cal BH}^-) \bigr),
$$
$$
(\rho', \Sh^+/{\cal BH}^+), \quad (\rho', \Sh'^-/{\cal BH}^-);
$$
Theorem 32 in \cite{ATMP} illustrates how these components fit together. 
The subspace $\partial^+(\Sh')$ sits inside another $\rho'_2$-invariant
subspace $\ker (\M : {\cal W}' \to {\cal W})$, and the quotient
\begin{multline*}
\frac{\ker (\M : {\cal W}' \to {\cal W})}{\partial^+(\Sh')} \simeq  \\
({\cal W}'/ \ker \tau_a^+) \oplus ({\cal W}'/ \ker \tau_a^-) \oplus
({\cal W}'/ \ker \tau_s^+) \oplus ({\cal W}'/ \ker \tau_s^-) \oplus
\BB C\text{-span of } \bigl\{ N(Z)^{-1} \cdot Z \bigr\}
\end{multline*}
-- a direct sum of irreducible components isomorphic to
${\cal F}^+$, ${\cal F}^-$, ${\cal G}^+$, ${\cal G}^-$ and
the trivial one-dimensional representation respectively
(Propositions 37, 47 and Theorem 49 in \cite{ATMP}).

Finally, the quotient $(\rho'_2,{\cal W}'/\ker \M)$ is the direct sum of three
components ${\cal Q}'^- \oplus {\cal Q}'^0 \oplus {\cal Q}'^+$
that are dual to ${\cal Q}^+$, ${\cal Q}^0$, ${\cal Q}^-$ respectively
(Theorem 62 in \cite{ATMP}).
\end{rem}

\begin{rem}  \label{Q-bar-remark}
Using Lemma 59 in \cite{ATMP}, it is easy to show that the minimal
$\rho'_2$-invariant subspace $\operatorname{Span-}{\cal Q}'^0$
of ${\cal W}'$ that contains ${\cal Q}'^0$ (the dual of ${\cal Q}^0$)
has four irreducible components:
$$
{\cal Q}'^0, \quad \BB C\text{-span of } \bigl\{ N(Z)^{-1} \cdot Z \bigr\},
\quad (\rho', {\cal BH}^+/{\cal I}'_0), \quad (\rho', {\cal BH}^-/{\cal I}'_0).
$$
It has the following proper invariant subspaces:
$$
\partial^+({\cal BH}^+/{\cal I}'_0), \quad \partial^+({\cal BH}^-/{\cal I}'_0),
\quad
\partial^+({\cal BH}^+/{\cal I}'_0) \oplus \partial^+({\cal BH}^-/{\cal I}'_0),
$$
$$
\partial^+({\cal BH}^+/{\cal I}'_0) \oplus \partial^+({\cal BH}^-/{\cal I}'_0)
\oplus \BB C\text{-span of } \bigl\{ N(Z)^{-1} \cdot Z \bigr\}.
$$
We denote by $\bar{\cal Q}'^0$ the quotient
$$
\bar{\cal Q}'^0 = \frac{\operatorname{Span-}{\cal Q}'^0}
{\partial^+({\cal BH}^+/{\cal I}'_0) \oplus \partial^+({\cal BH}^-/{\cal I}'_0)}.
$$
Then $\bar{\cal Q}'^0$ is an indecomposable subquotient of ${\cal W}'$ with
exactly two irreducible components: the $1$-dimensional subrepresentation
spanned by $N(Z)^{-1} \cdot Z$ and ${\cal Q}'^0$.
\end{rem}

We can form a tensor product representation
$(\pi'_l \otimes \pi'_r, \tilde{\cal U} \otimes \tilde{\cal U}')$
and consider a larger space
$$
\widetilde{{\cal U} \otimes {\cal U}'} = \left\{ \begin{matrix}
\text{holomorphic $\HC$-valued functions in two variables} \\
\text{$Z_1,Z_2 \in \HC$ (possibly with singularities) that are} \\
\text{QLAR with respect to $Z_1$ and QRAR with respect to $Z_2$}
\end{matrix} \right\}.
$$
The action of $GL(2,\HC)$ on these functions is given by
\begin{multline*}
(\pi'_l \otimes \pi'_r)(h): \: F(Z_1,Z_2) \quad \mapsto \quad
\bigl( (\pi'_l \otimes \pi'_r)(h)F \bigr)(Z_1,Z_2) \\
= \frac{a'-Z_1c'}{N(a'-Z_1c')}
F \bigl( (a'-Z_1c')(-b'+Z_1d')^{-1}, (aZ_2+b)(cZ_2+d)^{-1} \bigr) \cdot
\frac{cZ_2+d}{N(cZ_2+d)},
\end{multline*}
where $F \in \widetilde{{\cal U} \otimes {\cal U}'}$,
$h = \bigl(\begin{smallmatrix} a' & b' \\ c' & d' \end{smallmatrix}\bigr)
\in GL(2,\HC)$ and 
$h^{-1} = \bigl(\begin{smallmatrix} a & b \\ c & d \end{smallmatrix}\bigr)$.
Differentiating, we obtain the corresponding action of the Lie algebra
$\mathfrak{gl}(2,\HC)$ also denoted by $\pi'_l \otimes \pi'_r$.

We denote by $\DR$ the restriction to the diagonal map on
$\widetilde{{\cal U} \otimes {\cal U}'}$:
\begin{equation}  \label{DiagRes}
\DR: F(Z_1,Z_2) \mapsto F(Z,Z).
\end{equation}
Clearly, $\DR$ intertwines the actions of $\pi'_l \otimes \pi'_r$ and $\rho'_2$.
When restricted to ${\cal U} \otimes {\cal U}'$, the map $\DR$ is just the
multiplication map
$$
\mult: {\cal U} \otimes {\cal U}' \to {\cal W}', \quad
f(Z_1) \otimes g(Z_2) \mapsto f(Z) \cdot g(Z).
$$
We have the following analogue of Lemma 8 in \cite{desitter}
and Proposition 89 in \cite{ATMP}:

\begin{prop}  \label{image-prop}
Under the multiplication maps
$\mult: (\pi'_l, {\cal U}^\pm) \otimes (\pi'_r, {\cal U}'^\pm) \to
(\rho'_2, {\cal W}')$,
\begin{enumerate}
\item
The image of ${\cal U}^+ \otimes {\cal U}'^+$ in ${\cal W}'$ is
a vector space $\HC \otimes \Zh^+$;
its irreducible components as a subrepresentation of ${\cal W}'$ are
$$
(\rho', {\cal BH}^+/{\cal I}_0), \quad
(\rho', \Sh^+/{\cal BH}^+), \quad
(\rho'_2, {\cal W}'/ \ker \tau_a^+), \quad
(\rho'_2, {\cal W}'/ \ker \tau_s^+), \quad (\rho'_2, {\cal Q}'^+).
$$
\item
The image of ${\cal U}^- \otimes {\cal U}'^-$ in ${\cal W}'$ is
a vector space $\HC \otimes (N(Z) \cdot \Zh^-)$;
its irreducible components as a subrepresentation of ${\cal W}'$ are
$$
(\rho', {\cal BH}^-/{\cal I}_0), \quad
(\rho',\Sh'^-/{\cal BH}^-\bigr), \quad
(\rho'_2, {\cal W}'/ \ker \tau_a^-\bigr), \quad
(\rho'_2, {\cal W}'/ \ker \tau_s^-), \quad (\rho'_2, {\cal Q}'^-).
$$
\item
The image of ${\cal U}^+ \otimes {\cal U}'^-$ and
${\cal U}^- \otimes {\cal U}'^+$ in ${\cal W}'$ is
a vector space $\HC \otimes (\Zh^0 \oplus {\cal H}^+)$;
its irreducible components as a subrepresentation of ${\cal W}'$ are
\begin{multline*}
(\rho', {\cal BH}^+/{\cal I}_0\bigr), \quad
(\rho', {\cal BH}^-/{\cal I}_0), \quad
\bigl(\rho', {\cal J}'/({\cal BH}^++{\cal BH}^-)\bigr), \quad
\BB C\text{-span of } \bigl\{ N(Z)^{-1} \cdot Z \bigr\}, \\
(\rho'_2, {\cal W}'/ \ker \tau_a^+\bigr), \quad
(\rho'_2, {\cal W}'/ \ker \tau_a^-\bigr), \quad
(\rho'_2, {\cal W}'/ \ker \tau_s^+), \quad
(\rho'_2, {\cal W}'/ \ker \tau_s^-), \quad (\rho'_2, {\cal Q}'^0).
\end{multline*}
\end{enumerate}
\end{prop}

\begin{proof}
Since the harmonic functions are contained in the spaces of quasi regular
functions:
$$
\BB S \otimes {\cal H}^{\pm} \subset {\cal U}^{\pm}, \qquad
\BB S' \otimes {\cal H}^{\pm} \subset {\cal U}'^{\pm},
$$
by Lemma 8 in \cite{desitter}, we have containments of vector spaces:
$$
\HC \otimes \Zh^+ \subset \mult ({\cal U}^+ \otimes {\cal U}'^+), \qquad
\HC \otimes \Zh^- \subset \mult ({\cal U}^- \otimes {\cal U}'^-),
$$
\begin{equation}  \label{obvious-containment}
\HC \otimes \Zh^0 \subset \mult ({\cal U}^+ \otimes {\cal U}'^-), \qquad
\HC \otimes \Zh^0 \subset \mult ({\cal U}^- \otimes {\cal U}'^+).
\end{equation}
The fact that
$$
\mult ({\cal U}^+ \otimes {\cal U}'^+) \subset \HC \otimes \Zh^+
$$
is clear.
The decomposition of $\HC \otimes \Zh^+$ into irreducible components
follows from Lemma 63 in \cite{ATMP}. This proves the first part.

Applying the inversion to both sides of
$\mult ({\cal U}^+ \otimes {\cal U}'^+) = \HC \otimes \Zh^+$
and using Lemma 63 in \cite{ATMP}, we see that
\begin{multline*}
\mult ({\cal U}^- \otimes {\cal U}'^-) =
\partial^+\bigl({\cal BH}^-/{\cal I}_0\bigr) \oplus
\partial^+\bigl(\Sh^-/{\cal BH}^-\bigr) \oplus
\bigl({\cal W}'/ \ker \tau_a^-\bigr) \oplus
\bigl({\cal W}'/ \ker \tau_s^-\bigr) \oplus {\cal Q}'^-  \\
= \HC \otimes (N(Z) \cdot \Zh^-).
\end{multline*}
This proves the second part.

It is easy to see that $\mult ({\cal U}^+ \otimes {\cal U}'^-)$ contains the
generators of the irreducible components
$$
\partial^+\bigl({\cal BH}^+/{\cal I}_0\bigr), \qquad
{\cal W}'/ \ker \tau_a^+, \qquad {\cal W}'/ \ker \tau_s^+
$$
and hence the whole components. Therefore, by \eqref{obvious-containment},
\begin{multline}  \label{mult-eqn}
\mult ({\cal U}^+ \otimes {\cal U}'^-) \supset
(\HC \otimes \Zh^0) \oplus \partial^+\bigl({\cal BH}^+/{\cal I}_0\bigr) \oplus
\bigl({\cal W}'/ \ker \tau_a^+\bigr) \oplus
\bigl({\cal W}'/ \ker \tau_s^+\bigr)  \\
= \HC \otimes (\Zh^0 \oplus {\cal H}^+).
\end{multline}
On the other hand, the quasi regular functions are biharmonic, so,
by Proposition \ref{biharmonic-prop}, we have a containment of vector spaces
$$
{\cal U}^{\pm} \subset
\BB S \otimes \bigl( {\cal H}^{\pm} \oplus N(Z) \cdot {\cal H}^{\pm} \bigr),
\qquad {\cal U}'^{\pm} \subset
\BB S' \otimes \bigl( {\cal H}^{\pm} \oplus N(Z) \cdot {\cal H}^{\pm} \bigr)
$$
and
$$
\mult ({\cal U}^+ \otimes {\cal U}'^-) \subset \HC \otimes
\bigl( \Zh^0 \oplus N(Z) \cdot \Zh^0 \oplus N(Z)^2 \cdot \Zh^0 \bigr).
$$
It follows that $\mult ({\cal U}^+ \otimes {\cal U}'^-)$ cannot contain
these irreducible irreducible components of ${\cal W}'$:
$$
\partial^+\bigl(\Sh^+/{\cal BH}^+\bigr), \qquad
\partial^+\bigl(\Sh'^-/{\cal BH}^-\bigr), \qquad
{\cal Q}'^+, \qquad {\cal Q}'^-.
$$
Therefore,
$$
\mult ({\cal U}^+ \otimes {\cal U}'^-)
= \HC \otimes (\Zh^0 \oplus {\cal H}^+).
$$
The decomposition of $\HC \otimes (\Zh^0 \oplus {\cal H}^+)$
into irreducible components follows from \eqref{mult-eqn} and
Lemma 63 in \cite{ATMP}.

The remaining case of $\mult ({\cal U}^- \otimes {\cal U}'^+)$ is similar.
\end{proof}

Let $U(2)_R = \{ RZ ;\: Z \in U(2) \}$, and
consider the following analogue of equation (57) in \cite{ATMP}:
\begin{equation}  \label{fork'}
(J'_R F)(Z_1,Z_2) = \frac{i}{2\pi^3} \int_{W \in U(2)_R}
\frac{W-Z_1}{N(W-Z_1)} \cdot F(W) \cdot \frac{W-Z_2}{N(W-Z_2)} \,dV,
\qquad F \in {\cal W}.
\end{equation}
Recall open domains $\BB D_R^{\pm}$ defined by equation (22) in \cite{desitter}):
\begin{align*}  %\label{D_R}
\BB D^+_R &= \{ Z \in \HC ;\: ZZ^*<R^2 \} = \{ RZ ;\: Z \in \BB D^+ \}, \\
\BB D^-_R &= \{ Z \in \HC ;\: ZZ^*>R^2 \} = \{ RZ ;\: Z \in \BB D^- \}.
\end{align*}
If $Z_1, Z_2 \in \BB D^-_R \sqcup \BB D^+_R$, the integrand has no singularities,
and the result is a holomorphic function in two variables $Z_1, Z_2$ that are
QLAR with respect to $Z_1$ and QRAR with respect to $Z_2$,
hence an element of $\widetilde{{\cal U} \otimes {\cal U}'}$.
Similarly to the case of $J_R$ considered in \cite{ATMP},
the integral \eqref{fork'} depends on whether $Z_1$ and $Z_2$
are both in $\BB D^+_R$, both in $\BB D^-_R$ or one is in $\BB D^+_R$ and
the other is in $\BB D^-_R$.
Thus the expression \eqref{fork'} determines four different maps.

Recall from Section 2 of \cite{desitter} that the group $U(2,2)_R$ is a
conjugate of $U(2,2)$, which is a real form of $GL(2,\HC)$ preserving
$U(2)_R$, $\BB D_R^+$ and $\BB D_R^-$.
We have the following analogue of Proposition 64 in \cite{ATMP}:

\begin{prop}  \label{fork-prop}
The maps $F \mapsto (J'_R F)(Z_1,Z_2)$ are $U(2,2)_R$ and
$\mathfrak{gl}(2,\HC)$-equivariant maps from $(\rho_2,{\cal W})$
to $(\pi'_l \otimes \pi'_r, \widetilde{{\cal U} \otimes {\cal U}'})$.
\end{prop}

\begin{proof}
We need to show that, for all $h \in U(2,2)_R$, the maps \eqref{fork'}
commute with the action of $h$. Writing
$h= \bigl(\begin{smallmatrix} a' & b' \\ c' & d' \end{smallmatrix}\bigr)$,
$h^{-1}= \bigl(\begin{smallmatrix} a & b \\ c & d \end{smallmatrix}\bigr)$,
$$
\tilde Z_1 = (aZ_1+b)(cZ_1+d)^{-1}, \qquad
\tilde Z_2 = (aZ_2+b)(cZ_2+d)^{-1}, \qquad
\tilde W = (aW+b)(cW+d)^{-1}
$$
and using Lemmas 10 and 61 from \cite{FL1} we obtain:
\begin{multline*}
\int_{W \in U(2)_R} \frac{W-Z_1}{N(W-Z_1)} \cdot (\rho_2(h)F)(W) \cdot
\frac{W-Z_2}{N(W-Z_2)} \,dV  \\
= \int_{W \in U(2)_R} \frac{W-Z_1}{N(W-Z_1)} \cdot
\frac{(cW+d)^{-1}}{N(cW+d)} \cdot F(\tilde W) \cdot
\frac{(a'-Wc')^{-1}}{N(a'-Wc')} \cdot \frac{W-Z_2}{N(W-Z_2)} \,dV  \\
= \int_{W \in U(2)_R} \frac{(a'-Z_1c') \cdot (\tilde W - \tilde Z_1) \cdot
F(\tilde W) \cdot (\tilde W - \tilde Z_2) \cdot (cZ_2+d) \,dV}
{N(a'-Z_1c') \cdot N(cW+d)^2 \cdot N(\tilde W - \tilde Z_1) \cdot
N(\tilde W - \tilde Z_2) \cdot N(a'-Wc')^2 \cdot N(cZ_2+d)}  \\
= \frac{a'-Z_1c'}{N(a'-Z_1c')} \cdot \int_{\tilde W \in U(2)_R}
\frac{\tilde W - \tilde Z_1}{N(\tilde W - \tilde Z_1)} \cdot F(\tilde W)
\cdot \frac{\tilde W - \tilde Z_2}{N(\tilde W - \tilde Z_2)} \,dV
\cdot \frac{cZ_2+d}{N(cZ_2+d)}.
\end{multline*}
This proves the $U(2,2)_R$-equivariance.
The $\mathfrak{gl}(2,\HC)$-equivariance then follows since
$\mathfrak{gl}(2,\HC) \simeq \BB C \otimes \mathfrak{u}(2,2)_R$.
\end{proof}

These maps $J'_R$ are closely related to the maps $I_R$ given by
equation (34) in Chapter 6 of \cite{desitter}
\begin{equation*}%  \label{I_R}
\Zh \ni f \quad \mapsto \quad (I_R f)(Z_1,Z_2) =
\frac i{2\pi^3} \int_{W \in U(2)_R} \frac{f(W) \,dV}{N(W-Z_1) \cdot N(W-Z_2)}
\quad \in \widetilde{{\cal H} \otimes {\cal H}},
\end{equation*}
where $\widetilde{{\cal H} \otimes {\cal H}}$ denotes the space of
holomorphic $\BB C$-valued functions in two variables $Z_1,Z_2 \in \HC$
(possibly with singularities) that are harmonic in each variable separately.
Indeed, $I_R$ extends to a map on $\HC \otimes \Zh = {\cal W}$, and
\begin{equation}  \label{I-J'-relation}
J'_R F(Z_1,Z_2) =
I_R \bigl( WF(W)W - Z_1F(W)W - WF(W)Z_2 + Z_1F(W)Z_2 \bigr)(Z_1,Z_2).
\end{equation}

%Of particular importance are the two irreducible components of ${\cal W}'$
%isomorphic to
%$$
%{\cal BH}^+/{\cal I}'_0 \qquad \text{and} \qquad {\cal BH}^-/{\cal I}'_0
%$$
%via the $\mathfrak{gl}(2,\HC)$-equivariant map
%$\partial^+: (\rho', \Sh') \to (\rho'_2, {\cal W}')$.

We use notations $J_R^{'++}$ and $J_R^{'--}$ to signify
$Z_1, Z_2 \in \BB D^+_R$ and $Z_1, Z_2 \in \BB D^-_R$ respectively.
We have an analogue of Theorem 65 in \cite{ATMP} for these maps:

\begin{thm}  \label{J'-thm}
The $\mathfrak{gl}(2,\HC)$-equivariant maps $F \mapsto (J'_R F)(Z_1,Z_2)$
have the following properties:
\begin{enumerate}
\item  \label{one}
If $Z_1, Z_2 \in \BB D^+_R$, then $J_R^{'++}$ maps ${\cal W}$ into
${\cal U}^+ \otimes {\cal U}'^+ \subset \widetilde{{\cal U} \otimes {\cal U}'}$,
annihilates all irreducible components of $(\rho_2,{\cal W})$, except for
\begin{align*}
{\cal Q}^+ &= \BB C\text{-span of }
\bigl\{ N(Z)^k \cdot {\bf F_{l,m,n}}(Z),\: N(Z)^k \cdot {\bf G_{l,m,n}}(Z),\:
{\bf H_{k,l,m,n}}(Z);\: k \ge 0 \bigr\}, \\
{\cal F}^+ &= \BB C\text{-span of }
\bigl\{ N(Z)^{-1} \cdot {\bf F_{l,m,n}}(Z);\: l \ge 1/2 \bigr\},  \\
{\cal G}^+ &= \BB C\text{-span of }
\bigl\{ N(Z)^{-1} \cdot {\bf G_{l,m,n}}(Z);\: l \ge 1/2 \bigr\}
\end{align*}
and the components isomorphic to $\Sh^+$ and
${\cal I}^+/(\Sh^+ \oplus {\cal J})$.

The composition $\mult \circ J_R^{'++}$ maps ${\cal W}$ into ${\cal W}'$,
annihilates all irreducible components of $(\rho_2,{\cal W})$, except for
the component isomorphic to ${\cal I}^+/(\Sh^+ \oplus {\cal J})$.
The image of $\mult \circ J_R^{'++}$ is the irreducible component of ${\cal W}'$
isomorphic to ${\cal BH}^+/{\cal I}'_0$.

\item
If $Z_1, Z_2 \in \BB D^-_R$, then $J_R^{'--}$ maps ${\cal W}$ into
${\cal U}^- \otimes {\cal U}'^- \subset
\widetilde{{\cal U} \otimes {\cal U}'}$, annihilates all irreducible
components of $(\rho_2,{\cal W})$, except for
\begin{align*}
{\cal Q}^- &= \BB C\text{-span of }
\bigl\{ N(Z)^k \cdot {\bf F_{l,m,n}}(Z),\: N(Z)^k \cdot {\bf G_{l,m,n}}(Z),\:
{\bf H_{k,l,m,n}}(Z);\: k \le -(2l+3) \bigr\},  \\
{\cal F}^- &= \BB C\text{-span of }
\bigl\{ N(Z)^{-1} \cdot {\bf F'_{l,m,n}}(Z);\: l \ge 1/2 \bigr\},  \\
{\cal G}^- &= \BB C\text{-span of }
\bigl\{ N(Z)^{-1} \cdot {\bf G'_{l,m,n}}(Z);\: l \ge 1/2 \bigr\}
\end{align*}
and the components isomorphic to $\Sh^-$ and
${\cal I}^-/(\Sh^- \oplus {\cal J})$.

The composition $\mult \circ J_R^{'--}$ maps ${\cal W}$ into ${\cal W}'$,
annihilates all irreducible components of $(\rho_2,{\cal W})$, except for
the component isomorphic to ${\cal I}^-/(\Sh^- \oplus {\cal J})$.
The image of $\mult \circ J_R^{'--}$ is the irreducible component of ${\cal W}'$
isomorphic to ${\cal BH}^-/{\cal I}'_0$.
\end{enumerate}
\end{thm}

\begin{proof}
We prove part \ref{one} only, the proof of the other part is similar.
So, suppose that $Z_1, Z_2 \in \BB D^+_R$.
By Proposition \ref{fork-prop}, the map $J_R^{'++}$ is
$\mathfrak{gl}(2,\HC)$-equivariant.
It follows immediately from Theorem 12 in \cite{desitter} and equations
\eqref{Zt-identity}, \eqref{I-J'-relation}
that the image of $J_R^{'++}$ lies in ${\cal U}^+ \otimes {\cal U}'^+$ and that
$J_R^{'++}$ annihilates the irreducible components of $(\rho_2,{\cal W})$
that overlap with
$$
\HC \otimes \BB C \text{-span of }
\bigl\{ t^l_{n\,\underline{m}}(Z) \cdot N(Z)^k;\: k \le -3 \bigr\}.
$$
This proves that $J_R^{'++}$ annihilates all irreducible components of
$(\rho_2,{\cal W})$, except possibly
$$
{\cal Q}^+, \quad {\cal F}^+, \quad {\cal G}^+, \quad
\BB C\text{-span of } \bigl\{ N(Z)^{-2} \cdot Z^+ \bigr\},
$$
$$
\text{comp. isomorphic to } \Sh^+, \quad
\text{comp. isomorphic to } {\cal I}^+/(\Sh^+ \oplus {\cal J}).
$$
%\begin{align*}
%&\BB C\text{-span of }
%\bigl\{ N(Z)^{-1} \cdot {\bf F_{l,m,n}}(Z);\: l \ge 1/2 \bigr\}, \\
%&\BB C\text{-span of }
%\bigl\{ N(Z)^{-1} \cdot {\bf G_{l,m,n}}(Z);\: l \ge 1/2 \bigr\}, \\
%&\BB C\text{-span of } \bigl\{ N(Z)^{-2} \cdot Z^+ \bigr\},  \\
%&\text{component isomorphic to }
%{\cal I}^+/(\Sh^+ \oplus {\cal J}).
%\end{align*}
We find the effect of $J_R^{'++}$ on these components by evaluating
$J_R^{'++}$ on the following generators
$$
{\bf H_{0,0,-1,-1}}(Z) = 2 \left(\begin{smallmatrix}
1 & 0 \\ 0 & 0 \end{smallmatrix}\right),
$$
$$
N(Z)^{-1} \cdot {\bf F_{\frac12,\frac12,-\frac12}}(Z)
= \tfrac2{N(Z)} \cdot \left(\begin{smallmatrix}
0 & 0 \\ -z_{22} & z_{12} \end{smallmatrix}\right), \qquad
N(Z)^{-1} \cdot {\bf G_{\frac12,-\frac12,\frac12}}(Z)
= \tfrac1{N(Z)} \cdot \left(\begin{smallmatrix}
0 & -z_{22} \\ 0 & z_{21} \end{smallmatrix}\right),
$$
$$
N(Z)^{-2} \cdot Z^+, \qquad N(Z)^{-1} \cdot Z^+
$$
of ${\cal Q}^+$, ${\cal F}^+$, ${\cal G}^+$,
$\BB C\text{-span of } \bigl\{ N(Z)^{-2} \cdot Z^+ \bigr\}$,
and the component isomorphic to ${\cal I}^+/(\Sh^+ \oplus {\cal J})$
respectively. We have:
\begin{multline*}
J_R^{'++} \bigl( {\bf H_{0,0,-1,-1}}(W) \bigr)(Z_1,Z_2)  \\
= -\tfrac23 (Z_1-Z_2)
\left(\begin{smallmatrix} 1 & 0 \\ 0 & 0 \end{smallmatrix}\right) (Z_1-Z_2)
+ \tfrac13 \left(\begin{smallmatrix} ((z_{11})_1-(z_{11})_2)^2 & 0 \\
  0 & ((z_{11})_1-(z_{11})_2)((z_{22})_1-(z_{22})_2) \end{smallmatrix}\right),
\end{multline*}
$$
\mult \circ J_R^{'++} \bigl( {\bf H_{0,0,-1,-1}}(W) \bigr) = 0;
$$
$$
J_R^{'++} \bigl( N(W)^{-1} \cdot {\bf F_{\frac12,\frac12,-\frac12}}(W) \bigr)(Z_1,Z_2)
= (Z_1-Z_2) \left(\begin{smallmatrix} 0 & 0 \\ 1 & 0 \end{smallmatrix}\right),
$$
$$
\mult \circ J_R^{'++}
\bigl( N(W)^{-1} \cdot {\bf F_{\frac12,\frac12,-\frac12}}(W) \bigr) = 0;
$$
$$
J_R^{'++} \bigl( N(W)^{-1} \cdot {\bf G_{\frac12,-\frac12,\frac12}}(W) \bigr)(Z_1,Z_2)
= \tfrac12  \left(\begin{smallmatrix} 0 & 1 \\ 0 & 0 \end{smallmatrix}\right)
(Z_1-Z_2),
$$
$$
\mult \circ J_R^{'++}
\bigl( N(W)^{-1} \cdot {\bf G_{\frac12,-\frac12,\frac12}}(W) \bigr) = 0;
$$
$$
J_R^{'++} \bigl( N(W)^{-2} \cdot W^+ \bigr) = 0;
$$
$$
J_R^{'++} \bigl( N(W)^{-1} \cdot W^+ \bigr)(Z_1,Z_2) = -\tfrac12 (Z_1+Z_2),
$$
$$
\mult \circ J_R^{'++} \bigl( N(W)^{-1} \cdot W^+ \bigr)(Z) = -Z.
$$
Since $-Z = - \partial^+ N(Z)$ and $N(Z)$ generates ${\cal BH}^+ \subset \Sh'$,
it follows that the image of $\mult \circ J_R^{'++}$ is the irreducible component
of ${\cal W}'$ isomorphic to ${\cal BH}^+/{\cal I}'_0$.

The component isomorphic to $\Sh^+$ does not appear in ${\cal W}$ as a
subrepresentation. Any subrepresentation $(\rho_2, {\cal W}_{sub})$ of
$(\rho_2,{\cal W})$ -- such as $\ker J_R^{'++}$ -- containing $(\rho, \Sh^+)$
must also contain $(\rho_2,{\cal Q}^+)$,
and ${\cal Q}^+$ is not in the kernel of $J_R^{'++}$.
This proves that $J_R^{'++}$ is not zero on the component isomorphic to $\Sh^+$.
Finally, it was observed in Subsection 4.4 in \cite{FL1} that by
Theorem 70 in \cite{FL1} and equation \eqref{I-J'-relation},
$\mult \circ J_R^{'++}$ annihilates the irreducible components of
$(\rho_2,{\cal W})$ that overlap with vector space
${\cal W}^+ = \HC \otimes \Zh^+$. This vector space ${\cal W}^+$ contains
${\cal Q}^+$ and the component isomorphic to $\Sh^+$.
\end{proof}

Recall that in Section 4 of \cite{desitter} we decomposed $(\rho_1, \Zh)$
into three irreducible components:
\begin{equation}  \label{Zh-decomp}
(\rho_1,\Zh) \simeq (\rho_1,\Zh^-) \oplus (\rho_1,\Zh^0) \oplus (\rho_1,\Zh^+).
\end{equation}
Corresponding to this decomposition, we have equivariant projections
\begin{equation}  \label{Zh-proj}
\proj^- : \Zh \twoheadrightarrow \Zh^-, \qquad
\proj^0 : \Zh \twoheadrightarrow \Zh^0 \qquad \text{and} \qquad
\proj^+ : \Zh \twoheadrightarrow \Zh^+.
\end{equation}
Tensoring with $\HC$, we obtain a vector space decomposition
$$
{\cal W} = \HC \otimes \Zh
= (\HC \otimes \Zh^+) \oplus  (\HC \otimes \Zh^0) \oplus (\HC \otimes \Zh^-)
$$
and projections
\begin{equation}  \label{W-proj}
\proj^- : {\cal W} \twoheadrightarrow \HC \otimes \Zh^-, \qquad
\proj^0 : {\cal W} \twoheadrightarrow \HC \otimes \Zh^0
\qquad \text{and} \qquad
\proj^+ : {\cal W} \twoheadrightarrow \HC \otimes \Zh^+.
\end{equation}
Note that these decomposition and projections $\proj^-$, $\proj^0$ are
{\em not} $\mathfrak{gl}(2,\BB C)$-invariant.
We introduce operators $\Xm^+$ and $\Xm^-: {\cal W} \to {\cal W}'$:
\begin{align}
(\Xm^+ F)(Z) &= \proj^+(ZFZ) - Z\proj^+(FZ) - \proj^+(ZF)Z + Z\proj^+(F)Z,
\label{Xm^+}  \\
(\Xm^- F)(Z) &= \proj^-(ZFZ) - Z\proj^-(FZ) - \proj^-(ZF)Z + Z\proj^-(F)Z, 
\qquad F \in {\cal W}.
\label{Xm^-}
\end{align}
From Proposition \ref{fork-prop}, Theorem \ref{J'-thm} and equation
\eqref{I-J'-relation} we obtain an analogue of Theorem 77 in \cite{FL1}:

\begin{cor}  \label{Xm-cor}
\begin{enumerate}
\item
The operator $\Xm^+: (\rho_2, {\cal W}) \to (\rho'_2, {\cal W}')$ is
$\mathfrak{gl}(2,\BB C)$-equivariant and annihilates all irreducible
components of $(\rho_2,{\cal W})$, except for the component isomorphic
to ${\cal I}^+/(\Sh^+ \oplus {\cal J})$.
The image of $\Xm^+$ is the irreducible component of ${\cal W}'$
isomorphic to ${\cal BH}^+/{\cal I}'_0$.
For $Z \in \BB D^+_R$, we have an integral presentation
$$
(\Xm^+ F)(Z) = \frac{i}{2\pi^3} \int_{W \in U(2)_R}
\frac{(W-Z) \cdot F(W) \cdot (W-Z)}{N(W-Z)^2} \,dV,
\qquad F \in {\cal W}.
$$

\item
The operator $\Xm^-: (\rho_2, {\cal W}) \to (\rho'_2, {\cal W}')$ is
$\mathfrak{gl}(2,\BB C)$-equivariant and annihilates all irreducible
components of $(\rho_2,{\cal W})$, except for the component isomorphic
to ${\cal I}^-/(\Sh^- \oplus {\cal J})$.
The image of $\Xm^-$ is the irreducible component of ${\cal W}'$
isomorphic to ${\cal BH}^-/{\cal I}'_0$.
For $Z \in \BB D^-_R$, we have an integral presentation
$$
(\Xm^- F)(Z) = \frac{i}{2\pi^3} \int_{W \in U(2)_R}
\frac{(W-Z) \cdot F(W) \cdot (W-Z)}{N(W-Z)^2} \,dV,
\qquad F \in {\cal W}.
$$
\end{enumerate}
\end{cor}

\subsection{Equivariant Maps
  ${\cal W} \to \widetilde{{\cal U} \otimes {\cal U}'}$}

In this subsection we continue to study equivariant maps from ${\cal W}$
to the tensor product ${\cal U} \otimes {\cal U}'$, and now we need to
consider the completion $\widetilde{{\cal U} \otimes {\cal U}'}$.
We use notations $J_R^{'+-}$ and $J_R^{'-+}$ to
signify $Z_1 \in \BB D^+_R$, $Z_2 \in \BB D^-_R$ and $Z_1 \in \BB D^-_R$,
$Z_2 \in \BB D^+_R$ respectively.
Our first result is about these maps $J_R^{'+-}$ and $J_R^{'-+}$ applied
to the generator of the trivial (one-dimensional) component.

\begin{prop}  \label{J-gen-prop}
We have:
\begin{multline}  \label{J-formula1}
J_R^{'+-} (N(W)^{-2} \cdot W^+) =
- \sum_{l \ge 0,m,n} \frac1{(2l+1)^2}
f^{(1)}_{l,m,n}(Z_1) \cdot \tilde g^{(1)}_{l,m,n}(Z_2)  \\
- \sum_{l \ge \frac12,m,n} \frac1{(2l)^2(2l+1)^2}
f^{(2)}_{l,m,n}(Z_1) \cdot \tilde g^{(2)}_{l,m,n}(Z_2)
- \sum_{l \ge 1,m,n} \frac1{(2l)^2}
f^{(3)}_{l,m,n}(Z_1) \cdot \tilde g^{(3)}_{l,m,n}(Z_2);
\end{multline}
\begin{multline}  \label{J-formula2}
J_R^{'-+} (N(W)^{-2} \cdot W^+) =
- \sum_{l \ge 0,m,n} \frac1{(2l+1)^2}
\tilde f^{(1)}_{l,m,n}(Z_1) \cdot g^{(1)}_{l,m,n}(Z_2)  \\
- \sum_{l \ge \frac12,m,n} \frac1{(2l)^2(2l+1)^2}
\tilde f^{(2)}_{l,m,n}(Z_1) \cdot g^{(2)}_{l,m,n}(Z_2)
- \sum_{l \ge 1,m,n} \frac1{(2l)^2}
\tilde f^{(3)}_{l,m,n}(Z_1) \cdot g^{(3)}_{l,m,n}(Z_2).
\end{multline}
\end{prop}

\begin{proof}
By \eqref{I-J'-relation},
\begin{multline*}
J_R^{'+-} (N(W)^{-2} \cdot W^+)
= I_R^{+-} (N(W)^{-1} \cdot W) - Z_1 \cdot I_R^{+-} (N(W)^{-1})  \\
-I_R^{+-}(N(W)^{-1}) \cdot Z_2 + Z_1 \cdot I_R^{+-}(N(W)^{-2} \cdot W^+) \cdot Z_2.
\end{multline*}
We compute each of the four terms separately starting from
$$
I_R^{+-}(N(W)^{-1}) (Z_1,Z_2)
= \sum_{l \ge 0,m,n} \frac1{2l+1} t^l_{n \, \underline{m}}(Z_1) \cdot
N(Z_2)^{-1} \cdot t^l_{m \, \underline{n}}(Z_2^{-1})
$$
(Section 6 in \cite{desitter}).

By Lemma 3 in \cite{desitter}, for $C \in \HC$,
$$
\rho_1 \bigl(\begin{smallmatrix} 0 & 0 \\ C & 0 \end{smallmatrix} \bigr)
N(W)^{-1} = \tr\bigl( C \cdot N(W)^{-1} \cdot W \bigr).
$$
Hence, by Lemma 11 in \cite{desitter}, Lemma 17 in \cite{FL1}
and identities \eqref{Ct}-\eqref{Ct-inverse}, we have:
\begin{multline}  \label{J1}
I_R^{+-} (N(W)^{-1} \cdot W)  \\
= \sum_{l \ge 0,m,n} \frac1{2l+1}
\begin{pmatrix} (l-n+1) t^{l+\frac12}_{n-\frac12\,\underline{m-\frac12}}(Z_1) &
(l-n+1) t^{l+\frac12}_{n-\frac12\,\underline{m+\frac12}}(Z_1) \\
(l+n+1) t^{l+\frac12}_{n+\frac12\,\underline{m-\frac12}}(Z_1) &
(l+n+1) t^{l+\frac12}_{n+\frac12\,\underline{m+\frac12}}(Z_1) \end{pmatrix}
\cdot \frac1{N(Z_2)} \cdot t^l_{m \, \underline{n}}(Z_2^{-1})  \\
- \sum_{l \ge \frac12,m,n} \frac1{2l+1} t^l_{n \, \underline{m}}(Z_1) \cdot
\frac1{N(Z_2)} \begin{pmatrix}
(l-n) t^{l-\frac12}_{m+\frac12\,\underline{n+\frac12}}(Z_2) &
(l-n) t^{l-\frac12}_{m-\frac12\,\underline{n+\frac12}}(Z_2) \\
(l+n) t^{l-\frac12}_{m+\frac12\,\underline{n-\frac12}}(Z_2) &
(l+n) t^{l-\frac12}_{m-\frac12\,\underline{n-\frac12}}(Z_2) \end{pmatrix}  \\
%= \sum_{l \ge \frac12,m,n} \frac1{2l}
%t^l_{n \, \underline{m}}(Z_1) \cdot
%\frac1{N(Z_2)} \begin{pmatrix}
%(l-n) t^{l-\frac12}_{m+\frac12\,\underline{n+\frac12}}(Z_2) &
%(l-n) t^{l-\frac12}_{m-\frac12\,\underline{n+\frac12}}(Z_2) \\
%(l+n) t^{l-\frac12}_{m+\frac12\,\underline{n-\frac12}}(Z_2) &
%(l+n) t^{l-\frac12}_{m-\frac12\,\underline{n-\frac12}}(Z_2) \end{pmatrix}  \\
%- \sum_{l \ge \frac12,m,n} \frac1{2l+1} t^l_{n \, \underline{m}}(Z_1) \cdot
%\frac1{N(Z_2)} \begin{pmatrix}
%(l-n) t^{l-\frac12}_{m+\frac12\,\underline{n+\frac12}}(Z_2) &
%(l-n) t^{l-\frac12}_{m-\frac12\,\underline{n+\frac12}}(Z_2) \\
%(l+n) t^{l-\frac12}_{m+\frac12\,\underline{n-\frac12}}(Z_2) &
%(l+n) t^{l-\frac12}_{m-\frac12\,\underline{n-\frac12}}(Z_2) \end{pmatrix}  \\
= \sum_{l \ge \frac12,m,n} \frac1{2l(2l+1)}
t^l_{n \, \underline{m}}(Z_1) \cdot
\frac1{N(Z_2)} \begin{pmatrix}
(l-n) t^{l-\frac12}_{m+\frac12\,\underline{n+\frac12}}(Z_2) &
(l-n) t^{l-\frac12}_{m-\frac12\,\underline{n+\frac12}}(Z_2) \\
(l+n) t^{l-\frac12}_{m+\frac12\,\underline{n-\frac12}}(Z_2) &
(l+n) t^{l-\frac12}_{m-\frac12\,\underline{n-\frac12}}(Z_2) \end{pmatrix}  \\
%= \sum_{l \ge \frac12,m,n} \frac1{2l(2l+1)}
%\begin{pmatrix} (l-n+\frac12) t^l_{n-\frac12\,\underline{m}}(Z_1) \\
%(l+n+\frac12) t^l_{n+\frac12\,\underline{m}}(Z_1) \end{pmatrix}
%\cdot \frac1{N(Z_2)} \bigl( t^{l-\frac12}_{m+\frac12\,\underline{n}}(Z_2^{-1}), 
%t^{l-\frac12}_{m-\frac12\,\underline{n}}(Z_2^{-1}) \bigr)  \\
= \sum_{l \ge \frac12,m,n} \frac1{2l(2l+1)} f^{(2)}_{l,m,n}(Z_1) \cdot
\tilde g^{(2)}_{l,m,n}(Z_2).
\end{multline}

By Lemma 11 in \cite{desitter}, Lemma 17 in \cite{FL1}
and \eqref{Zt-identity}, we have:
\begin{multline}  \label{J2}
Z_1 \cdot I_R^{+-} (N(W)^{-1})
= Z_1 \cdot \sum_{l \ge 0,m,n} \frac1{2l+1} t^l_{n \, \underline{m}}(Z_1) \cdot
N(Z_2)^{-1} \cdot t^l_{m \, \underline{n}}(Z_2^{-1})  \\
= \sum_{l \ge 0,m,n} \frac1{(2l+1)^2} \begin{pmatrix}
(l-n+1) t^{l+\frac12}_{n-\frac12\,\underline{m-\frac12}}(Z_1) &
(l-n+1) t^{l+\frac12}_{n-\frac12\,\underline{m+\frac12}}(Z_1) \\
(l+n+1) t^{l+\frac12}_{n+\frac12\,\underline{m-\frac12}}(Z_1) &
(l+n+1) t^{l+\frac12}_{n+\frac12\,\underline{m+\frac12}}(Z_1) \end{pmatrix} \cdot
\frac1{N(Z_2)} \cdot t^l_{m \, \underline{n}}(Z_2^{-1})  \\
+ \sum_{l \ge \frac12,m,n} \frac{N(Z_1)}{(2l+1)^2} \cdot \begin{pmatrix}
(l+m) t^{l-\frac12}_{n-\frac12\,\underline{m-\frac12}}(Z_1) &
-(l-m) t^{l-\frac12}_{n-\frac12\,\underline{m+\frac12}}(Z_1) \\
-(l+m) t^{l-\frac12}_{n+\frac12\,\underline{m-\frac12}}(Z_1) &
(l-m) t^{l-\frac12}_{n+\frac12\,\underline{m+\frac12}}(Z_1) \end{pmatrix} \cdot
\frac1{N(Z_2)} \cdot t^l_{m \, \underline{n}}(Z_2^{-1})  \\
%= \sum_{l\ge \frac12,m,n} \frac1{(2l)^2} \begin{pmatrix}
%(l-n+\frac12) t^l_{n-\frac12\,\underline{m}}(Z_1) \\
%(l+n+\frac12) t^l_{n+\frac12\,\underline{m}}(Z_1) \end{pmatrix} \cdot
%\frac1{N(Z_2)} \cdot \bigl( t^{l-\frac12}_{m+\frac12 \, \underline{n}}(Z_2^{-1}),
%t^{l-\frac12}_{m-\frac12 \, \underline{n}}(Z_2^{-1}) \bigr)  \\
%+ \sum_{l \ge 1,m,n} \frac{N(Z_1)}{(2l)^2} \cdot \begin{pmatrix}
%t^{l-1}_{n-\frac12\,\underline{m}}(Z_1) \\
%-t^{l-1}_{n+\frac12\,\underline{m}}(Z_1) \end{pmatrix} \cdot
%\frac1{N(Z_2)} \bigl( (l+m)t^{l-\frac12}_{m+\frac12 \, \underline{n}}(Z_2^{-1}),
%-(l-m)t^{l-\frac12}_{m-\frac12 \, \underline{n}}(Z_2^{-1}) \bigr)  \\
= \sum_{l \ge \frac12,m,n} \frac1{(2l)^2} f^{(2)}_{l,m,n}(Z_1) \cdot
\tilde g^{(2)}_{l,m,n}(Z_2)
+ \sum_{l \ge 1,m,n} \frac1{(2l)^2} f^{(3)}_{l,m,n}(Z_1) \cdot
\tilde g^{(3)}_{l,m,n}(Z_2);
\end{multline}
\begin{multline}  \label{J3}
I_R^{+-} (N(W)^{-1}) \cdot Z_2
= \sum_{l \ge 0,m,n} \frac1{2l+1} t^l_{n \, \underline{m}}(Z_1) \cdot
N(Z_2)^{-1} \cdot t^l_{m \, \underline{n}}(Z_2^{-1}) \cdot Z_2  \\
= \sum_{l \ge 0,m,n} \frac1{(2l+1)^2} t^l_{n \, \underline{m}}(Z_1) \cdot
\begin{pmatrix}
(l+m+1) t^{l+\frac12}_{m+\frac12\,\underline{n+\frac12}}(Z_2^{-1}) &
-(l-m+1) t^{l+\frac12}_{m-\frac12\,\underline{n+\frac12}}(Z_2^{-1}) \\
-(l+m+1) t^{l+\frac12}_{m+\frac12\,\underline{n-\frac12}}(Z_2^{-1}) &
(l-m+1) t^{l+\frac12}_{m-\frac12\,\underline{n-\frac12}}(Z_2^{-1}) \end{pmatrix}  \\
+ \sum_{l \ge \frac12,m,n} \frac1{(2l+1)^2} t^l_{n \, \underline{m}}(Z_1) \cdot
\frac1{N(Z_2)} \begin{pmatrix}
(l-n) t^{l-\frac12}_{m+\frac12\,\underline{n+\frac12}}(Z_2^{-1}) &
(l-n) t^{l-\frac12}_{m-\frac12\,\underline{n+\frac12}}(Z_2^{-1}) \\
(l+n) t^{l-\frac12}_{m+\frac12\,\underline{n-\frac12}}(Z_2^{-1}) &
(l+n) t^{l-\frac12}_{m-\frac12\,\underline{n-\frac12}}(Z_2^{-1}) \end{pmatrix}  \\
%= \sum_{l\ge 0,m,n} \frac1{(2l+1)^2}
%\begin{pmatrix} t^l_{n-\frac12\,\underline{m}}(Z_1) \\
%-t^l_{n+\frac12\,\underline{m}}(Z_1) \end{pmatrix} \cdot
%\bigl( (l+m+1) t^{l+\frac12}_{m+\frac12\,\underline{n}}(Z_2^{-1}),
%-(l-m+1) t^{l+\frac12}_{m-\frac12\,\underline{n}}(Z_2^{-1}) \bigr)  \\
%+ \sum_{l \ge \frac12,m,n} \frac1{(2l+1)^2}
%\begin{pmatrix} (l-n+\frac12) t^l_{n-\frac12\,\underline{m}}(Z_1) \\
%(l+n+\frac12) t^l_{n+\frac12\,\underline{m}}(Z_1) \end{pmatrix} \cdot
%\frac1{N(Z_2)} \bigl( t^{l-\frac12}_{m+\frac12\,\underline{n}}(Z_2^{-1}),
%t^{l-\frac12}_{m-\frac12\,\underline{n}}(Z_2^{-1}) \bigr)  \\
= \sum_{l \ge 0,m,n} \frac1{(2l+1)^2} f^{(1)}_{l,m,n}(Z_1) \cdot
\tilde g^{(1)}_{l,m,n}(Z_2)
+ \sum_{l \ge \frac12,m,n} \frac1{(2l+1)^2} f^{(2)}_{l,m,n}(Z_1) \cdot
\tilde g^{(2)}_{l,m,n}(Z_2).
\end{multline}

Since $\partial N(W)^{-1} = -N(W)^{-2} \cdot W^+$, using
\eqref{dt}-\eqref{dt-inverse}, we obtain:
\begin{multline*}
I_R^{+-}(N(W)^{-2} \cdot W^+) (Z_1,Z_2)
= - \partial_{Z_1} \bigl[ I_R^{+-}(N(W)^{-1}) (Z_1,Z_2) \bigr]
- \bigl[ I_R^{+-}(N(W)^{-1}) (Z_1,Z_2) \bigr]
\overleftarrow{\partial}_{Z_2}  \\
%= - \sum_{l \ge 0,m,n} \frac1{2l+1} \partial_{Z_1} t^l_{n \, \underline{m}}(Z_1) \cdot
%N(Z_2)^{-1} \cdot t^l_{m \, \underline{n}}(Z_2^{-1})  \\
%- \sum_{l \ge 0,m,n} \frac1{2l+1} t^l_{n \, \underline{m}}(Z_1) \cdot
%\bigl( N(Z_2)^{-1} \cdot t^l_{m \, \underline{n}}(Z_2^{-1}) \bigr)
%\overleftarrow{\partial}_{Z_2} \\
= - \sum_{l \ge \frac12,m,n} \frac1{2l+1}
\begin{pmatrix} (l-m) t^{l-\frac12}_{n+\frac12\,\underline{m+\frac12}}(Z_1) &
(l-m) t^{l-\frac12}_{n-\frac12\,\underline{m+\frac12}}(Z_1) \\
(l+m) t^{l-\frac12}_{n+\frac12\,\underline{m-\frac12}}(Z_1) &
(l+m) t^{l-\frac12}_{n-\frac12\,\underline{m-\frac12}}(Z_1) \end{pmatrix}
\cdot \frac1{N(Z_2)} \cdot t^l_{m \, \underline{n}}(Z_2^{-1})  \\
+ \sum_{l \ge 0,m,n} \frac1{2l+1} t^l_{n \, \underline{m}}(Z_1) \cdot
\frac1{N(Z_2)} \begin{pmatrix}
(l-m+1) t^{l+\frac12}_{m-\frac12\,\underline{n-\frac12}}(Z_2^{-1}) &
(l-m+1) t^{l+\frac12}_{m-\frac12\,\underline{n+\frac12}}(Z_2^{-1})  \\
(l+m+1) t^{l+\frac12}_{m+\frac12\,\underline{n-\frac12}}(Z_2^{-1}) &
(l+m+1) t^{l+\frac12}_{m+\frac12\,\underline{n+\frac12}}(Z_2^{-1}) \end{pmatrix}  \\
%= - \sum_{l \ge \frac12,m,n} \frac1{2l+1}
%\begin{pmatrix} (l-m) t^{l-\frac12}_{n+\frac12\,\underline{m+\frac12}}(Z_1) &
%(l-m) t^{l-\frac12}_{n-\frac12\,\underline{m+\frac12}}(Z_1) \\
%(l+m) t^{l-\frac12}_{n+\frac12\,\underline{m-\frac12}}(Z_1) &
%(l+m) t^{l-\frac12}_{n-\frac12\,\underline{m-\frac12}}(Z_1) \end{pmatrix}
%\cdot \frac1{N(Z_2)} \cdot t^l_{m \, \underline{n}}(Z_2^{-1})  \\
%+ \sum_{l \ge \frac12,m,n} \frac1{2l}
%\begin{pmatrix} (l-m) t^{l-\frac12}_{n+\frac12\,\underline{m+\frac12}}(Z_1) &
%(l-m) t^{l-\frac12}_{n-\frac12\,\underline{m+\frac12}}(Z_1) \\
%(l+m) t^{l-\frac12}_{n+\frac12\,\underline{m-\frac12}}(Z_1) &
%(l+m) t^{l-\frac12}_{n-\frac12\,\underline{m-\frac12}}(Z_1) \end{pmatrix}
%\cdot \frac1{N(Z_2)} \cdot t^l_{m \, \underline{n}}(Z_2^{-1})  \\
= \sum_{l \ge \frac12,m,n} \frac1{2l(2l+1)}
\begin{pmatrix} (l-m) t^{l-\frac12}_{n+\frac12\,\underline{m+\frac12}}(Z_1) &
(l-m) t^{l-\frac12}_{n-\frac12\,\underline{m+\frac12}}(Z_1) \\
(l+m) t^{l-\frac12}_{n+\frac12\,\underline{m-\frac12}}(Z_1) &
(l+m) t^{l-\frac12}_{n-\frac12\,\underline{m-\frac12}}(Z_1) \end{pmatrix}
\cdot \frac1{N(Z_2)} \cdot t^l_{m \, \underline{n}}(Z_2^{-1}).
\end{multline*}
Then, applying Lemma 23 from \cite{FL1} and \eqref{Zt-identity}, we obtain:
\begin{multline}  \label{J4}
Z_1 \cdot I_R^{+-}(N(W)^{-2} \cdot W^+) \cdot Z_2  \\
%= \sum_{l \ge \frac12,m,n} \frac{Z_1}{2l(2l+1)}
%\begin{pmatrix} (l-m) t^{l-\frac12}_{n+\frac12\,\underline{m+\frac12}}(Z_1) &
%(l-m) t^{l-\frac12}_{n-\frac12\,\underline{m+\frac12}}(Z_1) \\
%(l+m) t^{l-\frac12}_{n+\frac12\,\underline{m-\frac12}}(Z_1) &
%(l+m) t^{l-\frac12}_{n-\frac12\,\underline{m-\frac12}}(Z_1) \end{pmatrix}
%\cdot \frac{Z_2}{N(Z_2)} \cdot t^l_{m \, \underline{n}}(Z_2^{-1})  \\
= \sum_{l \ge \frac12,m,n} \frac1{2l(2l+1)}
\begin{pmatrix} (l-n) t^l_{n\,\underline{m}}(Z_1) &
(l-n+1) t^l_{n-1\,\underline{m}}(Z_1) \\
(l+n+1) t^l_{n+1\,\underline{m}}(Z_1) &
(l+n) t^l_{n\,\underline{m}}(Z_1) \end{pmatrix}
\cdot \frac{Z_2}{N(Z_2)} \cdot t^l_{m \, \underline{n}}(Z_2^{-1})  \\
= \sum_{l \ge \frac12,m,n} \frac1{2l(2l+1)^2}
\begin{pmatrix} (l-n) t^l_{n\,\underline{m}}(Z_1) &
(l-n+1) t^l_{n-1\,\underline{m}}(Z_1) \\
(l+n+1) t^l_{n+1\,\underline{m}}(Z_1) &
(l+n) t^l_{n\,\underline{m}}(Z_1) \end{pmatrix}  \\
\hskip1in \times \begin{pmatrix}
(l+m+1) t^{l+\frac12}_{m+\frac12\,\underline{n+\frac12}}(Z_2^{-1}) &
-(l-m+1) t^{l+\frac12}_{m-\frac12\,\underline{n+\frac12}}(Z_2^{-1}) \\
-(l+m+1) t^{l+\frac12}_{m+\frac12\,\underline{n-\frac12}}(Z_2^{-1}) &
(l-m+1) t^{l+\frac12}_{m-\frac12\,\underline{n-\frac12}}(Z_2^{-1}) \end{pmatrix}  \\
+ \sum_{l \ge \frac12,m,n} \frac1{2l(2l+1)^2}
\begin{pmatrix} (l-n) t^l_{n\,\underline{m}}(Z_1) &
(l-n+1) t^l_{n-1\,\underline{m}}(Z_1) \\
(l+n+1) t^l_{n+1\,\underline{m}}(Z_1) &
(l+n) t^l_{n\,\underline{m}}(Z_1) \end{pmatrix}   \\
\hskip1in \times \frac1{N(Z_2)} \begin{pmatrix}
(l-n) t^{l-\frac12}_{m+\frac12\,\underline{n+\frac12}}(Z_2^{-1}) &
(l-n) t^{l-\frac12}_{m-\frac12\,\underline{n+\frac12}}(Z_2^{-1}) \\
(l+n) t^{l-\frac12}_{m+\frac12\,\underline{n-\frac12}}(Z_2^{-1}) &
(l+n) t^{l-\frac12}_{m-\frac12\,\underline{n-\frac12}}(Z_2^{-1}) \end{pmatrix}  \\
%= \sum_{l \ge \frac12,m,n} \frac1{2l(2l+1)^2}
%\begin{pmatrix} (l-n+\frac12) t^l_{n-\frac12\,\underline{m}}(Z_1) &
%(l-n+\frac12) t^l_{n-\frac12\,\underline{m}}(Z_1) \\
%(l+n+\frac12) t^l_{n+\frac12\,\underline{m}}(Z_1) &
%(l+n+\frac12) t^l_{n+\frac12\,\underline{m}}(Z_1) \end{pmatrix}  \\
%\hskip1in \times \begin{pmatrix}
%(l+m+1) t^{l+\frac12}_{m+\frac12\,\underline{n}}(Z_2^{-1}) &
%-(l-m+1) t^{l+\frac12}_{m-\frac12\,\underline{n}}(Z_2^{-1}) \\
%-(l+m+1) t^{l+\frac12}_{m+\frac12\,\underline{n}}(Z_2^{-1}) &
%(l-m+1) t^{l+\frac12}_{m-\frac12\,\underline{n}}(Z_2^{-1}) \end{pmatrix}  \\
%+ \sum_{l \ge \frac12,m,n} \frac1{2l(2l+1)^2}
%\begin{pmatrix} (l-n+\frac12) t^l_{n-\frac12\,\underline{m}}(Z_1) &
%(l-n+\frac12) t^l_{n-\frac12\,\underline{m}}(Z_1) \\
%(l+n+\frac12) t^l_{n+\frac12\,\underline{m}}(Z_1) &
%(l+n+\frac12) t^l_{n+\frac12\,\underline{m}}(Z_1) \end{pmatrix}   \\
%\hskip1in \times \frac1{N(Z_2)} \begin{pmatrix}
%(l-n+\frac12) t^{l-\frac12}_{m+\frac12\,\underline{n}}(Z_2^{-1}) &
%(l-n+\frac12) t^{l-\frac12}_{m-\frac12\,\underline{n}}(Z_2^{-1}) \\
%(l+n+\frac12) t^{l-\frac12}_{m+\frac12\,\underline{n}}(Z_2^{-1}) &
%(l+n+\frac12) t^{l-\frac12}_{m-\frac12\,\underline{n}}(Z_2^{-1}) \end{pmatrix}  \\
%= \sum_{l \ge \frac12,m,n} \frac1{2l(2l+1)}
%\begin{pmatrix} (l-n+\frac12) t^l_{n-\frac12\,\underline{m}}(Z_1) \\
%(l+n+\frac12) t^l_{n+\frac12\,\underline{m}}(Z_1) \end{pmatrix}
%\cdot \frac1{N(Z_2)}
%\bigl( t^{l-\frac12}_{m+\frac12\,\underline{n}}(Z_2^{-1}), 
%t^{l-\frac12}_{m-\frac12\,\underline{n}}(Z_2^{-1}) \bigr)  \\
= \sum_{l \ge \frac12,m,n} \frac1{2l(2l+1)} f^{(2)}_{l,m,n}(Z_1) \cdot
\tilde g^{(2)}_{l,m,n}(Z_2).
\end{multline}

Combining \eqref{J1}-\eqref{J4}, we obtain \eqref{J-formula1}.
Applying the inversion
$\bigl( \begin{smallmatrix} 0 & 1 \\ 1 & 0 \end{smallmatrix} \bigr)
\in GL(2,\HC)$
and Lemma \ref{inversion-action-lem} to \eqref{J-formula1},
we obtain \eqref{J-formula2}.
\end{proof}

Note that the expressions \eqref{J-formula1}-\eqref{J-formula2} are
similar to the reproducing kernel expansions
\eqref{1st-kernel-expansion}-\eqref{2nd-kernel-expansion}.

Recall that, by Lemma 59 in \cite{ATMP}, $N(Z)^{-2} \cdot Z^+$ generates an
indecomposable subrepresentation of ${\cal W}$ that has exactly two irreducible
components: ${\cal Q}^0$ and the trivial one-dimensional representation.
Hence, by equivariance (Proposition \ref{fork-prop}), applying operators
of the form $(\pi'_l \otimes \pi'_r)(X)$, $X \in \mathfrak{gl}(2,\HC)$,
to the expansions \eqref{J-formula1}-\eqref{J-formula2}
one can obtain -- at least theoretically -- $J_R^{'+-} F$ and $J_R^{'-+} F$,
for all $F \in {\cal Q}^0$. In particular, we can conclude:

\begin{cor}  \label{Q^0-not-in-ker-cor}
  The irreducible component ${\cal Q}^0$ of $(\rho_2,{\cal W})$ is {\em not}
  in the kernel of $J_R^{'+-}$ and {\em not} in the kernel of $J_R^{'-+}$.
\end{cor}

We have an analogue of Proposition 66 in \cite{ATMP}.

\begin{thm}  \label{J+-_thm}
Let $Z_1 \in \BB D^+_R$ and $Z_2 \in \BB D^-_R$, then the kernel of $J_R^{'+-}$
has exactly four irreducible components of $(\rho_2,{\cal W})$:
\begin{equation}  \label{J+-_kernel}
(\rho, \Sh^+), \qquad (\rho, \Sh^-), \qquad {\cal Q}^+, \qquad {\cal Q}^-.
\end{equation}

Similarly, if $Z_1 \in \BB D^-_R$ and $Z_2 \in \BB D^+_R$, the kernel of
$J_R^{'-+}$ has the same four irreducible components \eqref{J+-_kernel}
of $(\rho_2,{\cal W})$.
In other words, $\ker J_R^{'+-} = \ker J_R^{'-+}$.
\end{thm}

\begin{proof}
We prove the first part only, the proof of the other part is similar.
So, suppose that $Z_1 \in \BB D^+_R$, $Z_2 \in \BB D^-_R$.
By Proposition \ref{fork-prop}, the map $J_R^{'+-}$ is
$\mathfrak{gl}(2,\HC)$-equivariant.
By the results of Section 6 in \cite{desitter} and equations
\eqref{Zt-identity}, \eqref{I-J'-relation},
$J_R^{'+-}$ annihilates the irreducible components of $(\rho_2,{\cal W})$
that overlap with vector spaces
$$
\HC \otimes \Zh^+ \qquad \text{and} \qquad
\HC \otimes \BB C \text{-span of }
\bigl\{ t^l_{n\,\underline{m}}(Z) \cdot N(Z)^k;\: k \le -(2l+4) \bigr\}.
$$
Comparing this with the list of irreducible components of $(\rho_2, {\cal W})$
given at the beginning of Subsection \ref{Subsect8.1}, we see that $J_R^{'+-}$
annihilates the components \eqref{J+-_kernel} of $(\rho_2,{\cal W})$.

By Proposition \ref{J-gen-prop} and Corollary \ref{Q^0-not-in-ker-cor},
the trivial $1$-dimensional component and ${\cal Q}^0$ are not in the kernel
of $J_R^{'+-}$.
We need to check that the remaining seven irreducible components
\begin{equation}  \label{7-list}
{\cal F}^+, \quad {\cal F}^-, \quad {\cal G}^+, \quad {\cal G}^-, \quad
(\rho, {\cal J}), \quad
\bigl( \rho, {\cal I}^+/(\Sh^+ \oplus {\cal J}) \bigr), \quad
\bigl( \rho, {\cal I}^-/(\Sh^- \oplus {\cal J}) \bigr)
\end{equation}
of $(\rho_2,{\cal W})$ are not annihilated by $J_R^{'+-}$.
Theorem 31 in \cite{ATMP} implies that if a subrepresentation
$(\rho_2, {\cal W}_{sub})$ of $(\rho_2,{\cal W})$
-- such as $\ker J_R^{'+-}$ -- contains either of the last
two irreducible components on the list, then it contains $(\rho, {\cal J})$.
Thus, it is sufficient to show that $J_R^{'+-}$ does not annihilate the first
five irreducible components.
In all five cases we find a $K$-type generator of the respective component
that also generates ${\cal Q}^0$.
By Proposition 58 and the conclusion of Subsection 5.1 in \cite{ATMP},
any subrepresentation ${\cal W}_{sub}$ of ${\cal W}$ that contains
an irreducible component isomorphic to ${\cal F}^+$, ${\cal F}^-$,
${\cal G}^+$ or ${\cal G}^-$, must contain a respective generator
\begin{align*}
& N(Z)^{-1} \cdot {\bf F_{\frac12,-\frac12,-\frac12}}(Z), \qquad 
N(Z)^{-3} \cdot {\bf G_{\frac12,-\frac12,-\frac12}}(Z),  \\
& N(Z)^{-1} \cdot {\bf G_{\frac12,-\frac12,-\frac12}}(Z), \qquad 
N(Z)^{-3} \cdot {\bf F_{\frac12,-\frac12,-\frac12}}(Z).
\end{align*}
By equations (52)-(53) in \cite{ATMP},
\begin{align*}
\partial_{11} \bigl( N(Z)^{-1} \cdot {\bf F_{\frac12,-\frac12,-\frac12}}(Z) \bigr)
&= - \tfrac13 N(Z)^{-2} \cdot {\bf F_{1,\frac12,\frac12}}(Z)
+ \tfrac12 {\bf H_{-1,0,0,0}}(Z) \quad \in {\cal Q}^0,  \\
\partial_{11} \bigl( N(Z)^{-3} \cdot {\bf G_{\frac12,-\frac12,-\frac12}}(Z) \bigr)
&= -2  N(Z)^{-4} \cdot {\bf G_{1,\frac12,\frac12}}(Z)
+ \tfrac12 {\bf H_{-1,0,0,0}}(Z) \quad \in {\cal Q}^- \oplus {\cal Q}^0,  \\
\partial_{11} \bigl( N(Z)^{-1} \cdot {\bf G_{\frac12,-\frac12,-\frac12}}(Z) \bigr)
&= - \tfrac23 N(Z)^{-2} \cdot {\bf G_{1,\frac12,\frac12}}(Z)
+ \tfrac12 {\bf H_{-1,0,0,0}}(Z)  \quad \in {\cal Q}^0,  \\
\partial_{11} \bigl( N(Z)^{-3} \cdot {\bf F_{\frac12,-\frac12,-\frac12}}(Z) \bigr)
&= - N(Z)^{-4} \cdot {\bf F_{1,\frac12,\frac12}}(Z)
+ \tfrac12 {\bf H_{-1,0,0,0}}(Z) \quad \in {\cal Q}^- \oplus {\cal Q}^0.
\end{align*}
In either of the four cases, it follows that ${\cal W}_{sub}$ contains
${\bf H_{-1,0,0,0}}(Z) \in {\cal Q}^0$ and ${\cal Q}^0 \subset {\cal W}_{sub}$.

We turn our attention to the component $(\rho, {\cal J})$.
Introduce elements of ${\cal W}$:
$$
{\bf I_{l,m,n}}(Z) = \begin{pmatrix}
(l+n+1) t^l_{n+1\,\underline{m+1}}(Z) & -(l-n) t^l_{n\,\underline{m+1}}(Z) \\
-(l+n+1) t^l_{n+1\,\underline{m}}(Z) & (l-n) t^l_{n\,\underline{m}}(Z)
\end{pmatrix}, \qquad
\begin{smallmatrix} %k=0, \pm1, \pm2, \dots, \\
l=\frac12,1,\frac32,\dots, \\
-l \le m,n \le l-1. \end{smallmatrix}
$$
The functions
$$
N(Z)^k \cdot {\bf I_{l,m,n}}(Z), \qquad
k \in \BB Z,\quad l=\tfrac12,1,\tfrac32,\dots, \quad -l \le m,n \le l-1,
$$
form a family of $K$-types of $(\rho_2,{\cal W})$.
More precisely, as representations of $SU(2) \times SU(2)$,
for $k$ and $l$ fixed,
$$
V_{l-\frac12} \boxtimes V_{l-\frac12} = \BB C\text{-span of }
\bigl\{ N(Z)^k \cdot {\bf I_{l,m,n}}(Z) ;\: -l \le m,n \le l-1 \bigr\}
$$
(cf. Proposition 58 in \cite{ATMP} describing the $K$-types of the kernel of
$\tr \circ \partial^+ : {\cal W} \to \Sh$).
By direct calculation, using equation (27) in \cite{desitter}, we have:
\begin{equation}  \label{Tr-d-NI}
\tr \circ \partial^+ \bigl( N(Z)^k \cdot {\bf I_{l,m,n}}(Z) \bigr)
= (2l+1)(2l+k+1) N(Z)^k \cdot t^{l-\frac12}_{n+\frac12\,\underline{m+\frac12}}(Z).
\end{equation}
The $K$-types of $(\rho, {\cal J})$ can be read from Theorem 31 in \cite{ATMP}
-- they are the functions \eqref{Tr-d-NI}
with $l=\frac32,2,\frac52,\dots$ and $-2l \le k \le -3$.
Note that, unlike the $K$-types of ${\cal F}^+$, ${\cal F}^-$, ${\cal G}^+$
and ${\cal G}^-$, these $K$-types appear in $(\rho_2,{\cal W})$ with
multiplicity two.
This means that any subrepresentation ${\cal W}_{sub}$ of ${\cal W}$
that contains an irreducible component isomorphic to $(\rho, {\cal J})$
must contain some linear combinations of
$$
N(Z)^k \cdot {\bf I_{l,m,n}}(Z) \quad \text{and} \quad {\bf H_{k+1,l-1,m,n}}(Z),
\qquad \begin{smallmatrix} l=\frac32,2,\frac52,\dots, \\
-2l \le k \le -3, \\ -l \le m,n \le l-1. \end{smallmatrix}
$$
By direct calculation, using equation (27) in \cite{desitter}, we have:
\begin{multline*}
\partial_{11} \bigl( N(Z)^k \cdot {\bf I_{l,m,n}}(Z) \bigr)
%= \frac{2l+k+1}{2l+1} N(Z)^k \begin{pmatrix}
%  (l-m-1)(l+n+1) t^{l-\frac12}_{n+\frac32\,\underline{m+\frac32}}(Z) &
%  -(l-m-1)(l-n) t^{l-\frac12}_{n+\frac12\,\underline{m+\frac32}}(Z) \\
%  -(l-m)(l+n+1) t^{l-\frac12}_{n+\frac32\,\underline{m+\frac12}}(Z) &
%  (l-m)(l-n) t^{l-\frac12}_{n+\frac12\,\underline{m+\frac12}}(Z)
%\end{pmatrix}  \\
%+ \frac{k(l+n+1)}{2l+1} N(Z)^{k-1} \begin{pmatrix}
%  (l+n+2) t^{l+\frac12}_{n+\frac32\,\underline{m+\frac32}}(Z) &
%  -(l-n) t^{l+\frac12}_{n+\frac12\,\underline{m+\frac32}}(Z) \\
%  -(l+n+2) t^{l+\frac12}_{n+\frac32\,\underline{m+\frac12}}(Z) &
%  (l-n) t^{l+\frac12}_{n+\frac12\,\underline{m+\frac12}}(Z) \end{pmatrix}  \\
= \tfrac{2l+k+1}{(2l)^2} N(Z)^k \cdot {\bf F_{l-\frac12,m+\frac12,n+\frac12}}(Z) \\
+ \tfrac{(2l+k+1)(l-m-1)(l+n+1)}{(2l)^2} N(Z)^k
\cdot {\bf G_{l-\frac12,m+\frac12,n+\frac12}}(Z)  \\
+ \tfrac{l+n+1}{2l(2l+1)} {\bf H_{k,l-\frac12,m+\frac12,n+\frac12}}(Z)
+ \tfrac{k(2l+1)(l+n+1)}{2l(2l+2)}
N(Z)^{k-1} \cdot {\bf I_{l+\frac12,m+\frac12,n+\frac12}}(Z).
\end{multline*}
On the other hand, by the calculation given between equations (53) and (54)
in \cite{ATMP},
\begin{multline*}
\partial_{11} \bigl( {\bf H_{k+1,l-1,m,n}}(Z) \bigr)
= - \tfrac{(k+1)(2l+k+1)}{l(2l-1)(2l+1)} N(Z)^k
\cdot {\bf F_{l-\frac12,m+\frac12,n+\frac12}}(Z) \\
- \tfrac{(k+1)(2l+k+1)(l-m-1)(l+n+1)}{l(2l-1)(2l+1)} N(Z)^k
\cdot {\bf G_{l-\frac12,m+\frac12,n+\frac12}}(Z)  \\
+ \tfrac{(2l-2)(2l+k+1)(l-m-1)}{(2l-1)^2}
{\bf H_{k+1,l-\frac32,m+\frac12,n+\frac12}}(Z)
+ \tfrac{(k+1)(2l+2)(l+n+1)}{(2l+1)^2} {\bf H_{k,l-\frac12,m+\frac12,n+\frac12}}(Z).
\end{multline*}
In particular,
\begin{multline*}
\partial_{11} \bigl( N(Z)^{-3} \cdot {\bf I_{\frac32,-\frac32,-\frac32}}(Z) \bigr)
= \tfrac19 N(Z)^{-3} \cdot {\bf F_{1,-1,-1}}(Z) \\
+ \tfrac29 N(Z)^{-3} \cdot {\bf G_{1,-1,-1}}(Z)
+ \tfrac{1}{12} {\bf H_{-3,1,-1,-1}}(Z)
- \tfrac45 N(Z)^{-4} \cdot {\bf I_{2,-1,-1}}(Z),
\end{multline*}
\begin{multline*}
\partial_{11} \bigl( {\bf H_{-2,\frac12,-\frac32,-\frac32}}(Z) \bigr)
= \tfrac16 N(Z)^{-3} \cdot {\bf F_{1,-1,-1}}(Z) \\
+ \tfrac13 N(Z)^{-3} \cdot {\bf G_{1,-1,-1}}(Z)
+ \tfrac12 {\bf H_{-2,0,-1,-1}}(Z)
- \tfrac58 {\bf H_{-3,1,-1,-1}}(Z).
\end{multline*}
%The components $N(Z)^{-3} \cdot {\bf F_{1,-1,-1}}(Z)$,
%$N(Z)^{-3} \cdot {\bf G_{1,-1,-1}}(Z)$, ${\bf H_{-2,0,-1,-1}}(Z)$
%and ${\bf H_{-3,1,-1,-1}}(Z)$ belong to ${\cal Q}^0$.
We conclude that any subrepresentation ${\cal W}_{sub}$ of ${\cal W}$
that contains an irreducible component isomorphic to $(\rho, {\cal J})$
must contain a non-trivial linear combination of
$$
N(Z)^{-3} \cdot {\bf I_{\frac32,-\frac32,-\frac32}}(Z) \quad \text{and} \quad
{\bf H_{-2,\frac12,-\frac32,-\frac32}}(Z),
$$
hence also a non-trivial linear combination of
$$
{\bf H_{-2,0,-1,-1}}(Z) \quad \text{and} \quad
{\bf H_{-3,1,-1,-1}}(Z) \qquad \in {\cal Q}^0,
$$
and ${\cal Q}^0 \subset {\cal W}_{sub}$.
\end{proof}

%This case when $Z_1 \in \BB D^+_R$ and $Z_2 \in \BB D^-_R$ (or the other way
%around) is very subtle and may be analyzed further in a future paper.

\subsection{Operator $\Xm^0: {\cal W} \to {\cal W}'$ Dual to the Maxwell
  Operator}

In this subsection we consider an analogue of Theorems 15 from \cite{desitter}
and 67 from \cite{ATMP}.
As was stated at the end of Subsection 6.2 in \cite{ATMP}, these results
can be interpreted as mathematical versions of the regularization of vacuum
polarization in QED and scalar QED.

Recall an open subset $\Omega_0 \subset \HC^{\times} \times \HC^{\times}$
introduced in Subsection 6.2 in \cite{ATMP}:
$$
\Omega_0 = \left\{ (Z_1,Z_2) \in \HC^{\times} \times \HC^{\times} ;\:
\begin{smallmatrix} \lambda_1 \ne 1,\: \lambda_2 \ne 1,\:
\text{neither $\frac{1-\lambda_1}{1-\lambda_2}$ nor} \\
\text{$\frac{1-\lambda_1^{-1}}{1-\lambda_2^{-1}}$ is a negative real number}
\end{smallmatrix}\right\},
$$
where $\lambda_1$ and $\lambda_2$ denote the eigenvalues of $Z_1Z_2^{-1}$.
From Section 6 of \cite{desitter} and the relation \eqref{I-J'-relation}
between $J'_R$ and $I_R$, we see that, for any $F \in {\cal W}$,
$(J_R^{'+-}F)(Z_1,Z_2)$ and $(J_R^{'-+}F)(Z_1,Z_2)$ extend analytically across
$\Omega_0$, and we have well defined operators $J_R^{'+-}$ and $J_R^{'-+}$
on ${\cal W}$:
\begin{align*}
(J_R^{'+-} F)(Z_1,Z_2) &= \text{analytic extension of $(J'_RF)(Z_1,Z_2)$
from $\BB D^+_R \times \BB D^-_R$ to $\Omega_0$},  \\
(J_R^{'-+} F)(Z_1,Z_2) &= \text{analytic extension of $(J'_RF)(Z_1,Z_2)$
from $\BB D^-_R \times \BB D^+_R$ to $\Omega_0$}.
\end{align*}
The operators $J_R^{'+-}$, $J_R^{'-+}$ is $U(2) \times U(2)$ and
$\mathfrak{gl}(2,\HC)$-equivariant (which follows from
Proposition \ref{fork-prop}) and both have kernel with four irreducible
components \eqref{J+-_kernel}.

We continue to use notation introduced in \cite{ATMP}:
if $\lambda \in \BB C$, let
$$
\sgn (\im \lambda) = \begin{cases}
1 & \text{ if $\lambda$ lies in the upper half plane of $\BB C$}, \\
-1 & \text{ if $\lambda$ lies in the lower half plane of $\BB C$}, \\
\text{undefined} & \text{ if $\lambda \in \BB R$}.
\end{cases}
$$

\begin{thm}  \label{Xm^0-operator-thm}
We have a well defined operator on ${\cal W}$
\begin{equation}  \label{Xm^0-operator}
(\Xm^0 F)(Z) = \lim_{\genfrac{}{}{0pt}{}{Z_1, Z_2 \to Z, \: N(Z_1-Z_2) \ne 0}
{\sgn (\im \lambda_1) = \sgn (\im \lambda_2)}}
- \bigl( (J_R^{'+-} + J_R^{'-+}) F \bigr)(Z_1,Z_2),
\qquad Z \in U(2)_R,
\end{equation}
where $\lambda_1$ and $\lambda_2$ are the eigenvalues of $Z_1Z_2^{-1}$.
The operator $\Xm^0$ can be expressed as
\begin{equation}  \label{Xm^0-P-rel}
(\Xm^0 F)(Z) = \proj^0(ZFZ) - Z\proj^0(FZ) - \proj^0(ZF)Z + Z\proj^0(F)Z,
\qquad F \in {\cal W},
\end{equation}
where $\proj^0 : {\cal W} \twoheadrightarrow \HC \otimes \Zh^0$ was introduced
in \eqref{W-proj}, and we have a relation between operators:
$$
\Xm^+ + \Xm^0 + \Xm^- =0.
$$
In particular, the operator $\Xm^0$ has values in ${\cal W}'$,
is $\mathfrak{gl}(2,\HC)$-equivariant, annihilates all irreducible components
of $(\rho_2,{\cal W})$, except for the components isomorphic
to ${\cal I}^+/(\Sh^+ \oplus {\cal J})$ and
${\cal I}^-/(\Sh^- \oplus {\cal J})$, and maps these components into
the irreducible components of ${\cal W}'$ isomorphic to
${\cal BH}^+/{\cal I}'_0$ and ${\cal BH}^-/{\cal I}'_0$ respectively.
\end{thm}

Note that the space ${\cal W}'$ consists of rational functions, and rational
functions on $\HC$ as well as analytic ones are completely determined by
their values on $U(2)_R$.

\begin{proof}
First, we show that the limit \eqref{Xm^0-operator} exists.
As was observed in the proof of Theorem 67 in \cite{ATMP}, the operator
$(I_R^{+-} + I_R^{-+})$ on $\Zh$ annihilates $\Zh^+$, $\Zh^-$, and
the image of $(I_R^{+-} + I_R^{-+})$ is generated by
\begin{equation*}  %\label{I(N(W)^{-1})}
\bigl( (I_R^{+-}+ I_R^{-+}) N(W)^{-1} \bigr)(Z_1,Z_2)
= -\frac1{N(Z_2)} \cdot
\begin{cases} \frac{\log\lambda_2-\log\lambda_1}{\lambda_2-\lambda_1} &
\text{if $\lambda_1 \ne \lambda_2$;} \\
\lambda^{-1} & \text{if $\lambda_1 = \lambda_2 = \lambda$, $\lambda \ne 1$,}
\end{cases}
\end{equation*}
where $\log$ denotes the branch of logarithm with a cut along the positive
real axis. If we restrict $\bigl( (I_R^{+-}+ I_R^{-+}) N(W)^{-1} \bigr)(Z_1,Z_2)$
to the open set where $\sgn (\im \lambda_1) = \sgn (\im \lambda_2)$,
we see that this restriction is 
$$
\frac1{N(Z_2)} \cdot \bigl( \text{function of
  $\lambda_1, \lambda_2 \in \BB C$ that is holomorphic near
  $(\lambda_1,\lambda_2)=(1,1)$} \bigr).
$$
Since the map $(I_R^{+-}+ I_R^{-+})$ is $\mathfrak{gl}(2,\HC)$-equivariant,
it follows that the same is true for
$\bigl( (I_R^{+-}+ I_R^{-+}) f \bigr)(Z_1,Z_2)$ with any $f \in \Zh$.
Hence the limit
$$
\lim_{\genfrac{}{}{0pt}{}{Z_1, Z_2 \to Z, \: N(Z_1-Z_2) \ne 0}
{\sgn (\im \lambda_1) = \sgn (\im \lambda_2)}}
- \bigl( (I_R^{+-} + I_R^{-+}) f \bigr)(Z_1,Z_2),
\qquad f \in \Zh,\: Z \in U(2)_R,
$$
exists and equals $\proj^0 f$, where $\proj^0 : \Zh \twoheadrightarrow \Zh^0$
is the projection introduced
in \eqref{Zh-proj}.
Then \eqref{I-J'-relation} implies that the limit \eqref{Xm^0-operator}
also exists, and we have an expression \eqref{Xm^0-P-rel}.

We express the projection
$\proj^0 : {\cal W} \twoheadrightarrow \HC \otimes \Zh^0$ as
$$
\proj^0 F = F - \proj^+ F - \proj^- F, \qquad F \in {\cal W}.
$$
Substituting this expression into \eqref{Xm^0-P-rel} and using
\eqref{Xm^+}-\eqref{Xm^-} yields $\Xm^0 = -(\Xm^+ + \Xm^-)$.
Then the result follows from Corollary \ref{Xm-cor}.
  
Alternatively, one can prove that $\Xm^0 = -(\Xm^+ + \Xm^-)$
by checking that both sides coincide on suitable generators of each
irreducible component of ${\cal W}$.
\end{proof}

\section{Invariant Algebra Structures on $\Zh$, ${\cal W}$ and ${\cal W}'$}  \label{IntertwiningOper-section}

\subsection{Invariant Algebra Structures on $\Zh$, ${\cal W}$ and ${\cal W}'$}  \label{IntertwiningOper-subsection}

In this subsection we study $\mathfrak{gl}(2,\HC)$-invariant maps
$$
(\rho_1 \otimes \rho_1, \Zh \otimes \Zh) \to (\rho_1, \Zh), \quad
(\rho_2 \otimes \rho_2, {\cal W} \otimes {\cal W}) \to (\rho_2, {\cal W})
\quad \text{and} \quad
(\rho'_2 \otimes \rho'_2, {\cal W}' \otimes {\cal W}') \to (\rho'_2, {\cal W}')
$$
and related $\mathfrak{gl}(2,\HC)$-invariant $3$-linear forms
$$
\Zh \times \Zh \times \Zh \to \BB C, \qquad
{\cal W} \times {\cal W} \times {\cal W} \to \BB C
\quad \text{and} \quad   
{\cal W}' \times {\cal W}' \times {\cal W}' \to \BB C.
$$

\subsection{The Scalar Case $\Zh$}

We start with the scalar case:
$$
(\rho_1 \otimes \rho_1, \Zh \otimes \Zh) \to (\rho_1, \Zh).
$$
Recall that $(\rho_1,\Zh)$ decomposes into a direct sum \eqref{Zh-decomp}.
Thus, an equivariant map
$\mu: (\rho_1 \otimes \rho_1, \Zh \otimes \Zh) \to (\rho_1, \Zh)$
can be decomposed into components
$$
\mu: (\rho_1 \otimes \rho_1, \Zh^{\alpha} \otimes \Zh^{\beta}) \to
(\rho_1, \Zh^{\gamma}), \qquad \text{where }
\alpha, \beta, \gamma \in \{ -,0,+ \}.
$$

For $f, g \in \Zh$, we consider maps given by the formula
\begin{equation}  \label{6easy-scalar}
f \otimes g \mapsto \Bigl( \frac{i}{2\pi^3} \Bigr)^2
\iint_{\genfrac{}{}{0pt}{}{W_1 \in U(2)_{R_1}}{W_2 \in U(2)_{R_2}}}
\frac{f(W_1) \cdot g(W_2)\,dV_{W_1}dV_{W_2}}{N(Z-W_1) N(W_1-W_2) N(W_2-Z)}
\end{equation}
with different choices of radii $R_1$, $R_2$ and regions containing $Z$.
By Lemmas 8, 11 and Theorems 12, 15 in \cite{desitter}, we have the following
$\mathfrak{gl}(2,\HC)$-equivariant maps:
\begin{itemize}
\item
If $0<R_1<R_2$, $Z \in \BB D^+_{R_1}$, the integral \eqref{6easy-scalar}
produces a multiplication $\Zh \otimes \Zh \to \Zh^+$ which descends to
$\Zh^0 \otimes \Zh^+ \to \Zh^+$;
\item
If $0<R_1<R_2$, $Z \in \BB D^-_{R_1} \cap \BB D^+_{R_2}$, the integral
\eqref{6easy-scalar} produces a multiplication $\Zh \otimes \Zh \to \Zh^0$
which descends to $\Zh^- \otimes \Zh^+ \to \Zh^0$;
\item
If $0<R_1<R_2$, $Z \in \BB D^-_{R_2}$, the integral
\eqref{6easy-scalar} produces a multiplication $\Zh \otimes \Zh \to \Zh^-$
which descends to $\Zh^- \otimes \Zh^0 \to \Zh^-$;
\item
If $0<R_2<R_1$, $Z \in \BB D^+_{R_2}$, the integral \eqref{6easy-scalar}
produces a multiplication $\Zh \otimes \Zh \to \Zh^+$ which descends to
$\Zh^+ \otimes \Zh^0 \to \Zh^+$;
\item
If $0<R_2<R_1$, $Z \in \BB D^+_{R_1} \cap \BB D^-_{R_2}$, the integral
\eqref{6easy-scalar} produces a multiplication $\Zh \otimes \Zh \to \Zh^0$
which descends to $\Zh^+ \otimes \Zh^- \to \Zh^0$;
\item
If $0<R_2<R_1$, $Z \in \BB D^-_{R_1}$, the integral
\eqref{6easy-scalar} produces a multiplication $\Zh \otimes \Zh \to \Zh^-$
which descends to $\Zh^0 \otimes \Zh^- \to \Zh^-$.
\end{itemize}
It is clear that these maps are non-trivial.

\begin{thm}  \label{Zh-mult-thm}
An $\mathfrak{sl}(2,\HC)$-equivariant map
$\mu: (\rho_1 \otimes \rho_1, \Zh \otimes \Zh) \to (\rho_1, \Zh)$ maps
\begin{center}
\begin{tabular}{lll}
$\Zh^- \otimes \Zh^+ \to \Zh^0$, & \qquad & $\Zh^+ \otimes \Zh^- \to \Zh^0$, \\
$\Zh^- \otimes \Zh^0 \to \Zh^0$, & \qquad & $\Zh^0 \otimes \Zh^- \to \Zh^-$, \\
$\Zh^0 \otimes \Zh^+ \to \Zh^0$, & \qquad & $\Zh^+ \otimes \Zh^0 \to \Zh^+$, \\
$\Zh^+ \otimes \Zh^+ \to \{0\}$, & \qquad & $\Zh^- \otimes \Zh^- \to \{0\}$.
\end{tabular}
\end{center}
Equivariant maps
\begin{center}
\begin{tabular}{lll}
  $(\rho_1 \otimes \rho_1, \Zh^- \otimes \Zh^+) \to (\rho_1, \Zh^0)$, & \qquad &
  $(\rho_1 \otimes \rho_1, \Zh^+ \otimes \Zh^-) \to (\rho_1, \Zh^0)$, \\
  $(\rho_1 \otimes \rho_1, \Zh^- \otimes \Zh^0) \to (\rho_1, \Zh^0)$, & \qquad &
  $(\rho_1 \otimes \rho_1, \Zh^0 \otimes \Zh^-) \to (\rho_1, \Zh^-)$, \\
  $(\rho_1 \otimes \rho_1, \Zh^0 \otimes \Zh^+) \to (\rho_1, \Zh^0)$, & \qquad &
  $(\rho_1 \otimes \rho_1, \Zh^+ \otimes \Zh^0) \to (\rho_1, \Zh^+)$
\end{tabular}
\end{center} 
are unique (up to scaling).
\end{thm}

%The proof of this theorem will be given elsewhere.
Missing from this theorem is the more difficult case of $\Zh^0 \otimes \Zh^0$.
We will provide the proof of this theorem as well as study the case of
$\Zh^0 \otimes \Zh^0$ in a future work.

It is often useful to consider $\mathfrak{gl}(2,\HC)$-invariant $3$-forms
$\Zh \times \Zh \times \Zh \to \BB C$.
They are related to the invariant multiplications $\Zh \otimes \Zh \to \Zh$
as follows.
Recall that there is a $\mathfrak{gl}(2,\HC)$-invariant non-degenerate
symmetric bilinear pairing on $(\rho_1,\Zh)$ given by
Proposition 69 in \cite{FL1}; let us denote this pairing by
$\langle \:,\: \rangle_{\Zh}$.
Given a multiplication $\mu: \Zh \otimes \Zh \to \Zh$, define a $3$-form
$c: \Zh \times \Zh \times \Zh \to \BB C$
by the rule
$$
c(f_1,f_2,f_3) = \bigl\langle \mu(f_1 \otimes f_2), f_3 \bigr\rangle_{\Zh}.
$$
If $\mu$ is $\mathfrak{gl}(2,\HC)$-invariant, so is $c$.
Conversely, a $3$-form
$$
c: \Zh \times \Zh \times \Zh \to \BB C
$$
defines a map
$$
c^*: \Zh \otimes \Zh \to \Zh^*,
$$
where $\Zh^*$ is the linear dual of $\Zh$.
If $c$ is $\mathfrak{gl}(2,\HC)$-invariant, so is $c^*$, and
$c^*(f_1 \otimes f_2)$ is $K$-finite for all $f_1, f_2 \in \Zh$:
$$
c^*(f_1 \otimes f_2) \bigl( N(Z)^k \cdot t^l_{n\,\underline{m}}(Z) \bigr)
= c \bigl( f_1, f_2, N(Z)^k \cdot t^l_{n\,\underline{m}}(Z) \bigr) =0
$$
for all but finitely many basis elements
$N(Z)^k \cdot t^l_{n\,\underline{m}}(Z) \in \Zh$.
Therefore, there exists a unique element $\mu(f_1 \otimes f_2) \in \Zh$
such that
$$
c^*(f_1 \otimes f_2)(g) = c(f_1,f_2,g)
= \bigl\langle \mu(f_1 \otimes f_2), g \bigr\rangle_{\Zh}
\qquad \text{for all } g \in \Zh.
$$
Thus, with the help of the bilinear pairing $\langle \:,\: \rangle_{\Zh}$
on $(\rho_1,\Zh)$, we have a bijection between $\mathfrak{gl}(2,\HC)$-invariant
$3$-linear forms $c: \Zh \times \Zh \times \Zh \to \BB C$
and $\mathfrak{gl}(2,\HC)$-equivariant multiplications
$\mu: \Zh \otimes \Zh \to \Zh$.

The $\mathfrak{gl}(2,\HC)$-invariant $3$-linear form corresponding to
\eqref{6easy-scalar} is
\begin{equation}  \label{6easy-scalar-3form}
\langle f_1,f_2,f_3 \rangle \mapsto \Bigl( \frac{i}{2\pi^3} \Bigr)^3
\iiint_{\begin{smallmatrix} W_1 \in U(2)_{R_1} \\ W_2 \in U(2)_{R_2} \\
W_3 \in U(2)_{R_3} \end{smallmatrix}}
\frac{f_1(W_1) \cdot f_2(W_2) \cdot f_3(W_3)\,dV_{W_1}dV_{W_2}dV_{W_3}}
{N(W_1-W_2) N(W_2-W_3) N(W_3-W_1)},
\end{equation}
$f_1, f_2, f_3 \in \Zh$, with different choices of radii $R_1$, $R_2$ and $R_3$.
Then we have six combinations of the radii $R_1$, $R_2$ and $R_3$ and,
by Theorems 12, 15 in \cite{desitter},
\begin{itemize}
\item
If $0<R_3<R_1<R_2$, the integral \eqref{6easy-scalar-3form} produces an
invariant $3$-form on $\Zh$ which descends to $\Zh^0 \times \Zh^+ \times \Zh^-$;
\item
If $0<R_1<R_3<R_2$, the integral\eqref{6easy-scalar-3form} produces an
invariant $3$-form on $\Zh$ which descends to $\Zh^- \times \Zh^+ \times \Zh^0$;
\item
If $0<R_1<R_2<R_3$, the integral \eqref{6easy-scalar-3form} produces an
invariant $3$-form on $\Zh$ which descends to $\Zh^- \times \Zh^0 \times \Zh^+$;
\item
If $0<R_3<R_2<R_1$, the integral \eqref{6easy-scalar-3form} produces an
invariant $3$-form on $\Zh$ which descends to $\Zh^+ \times \Zh^0 \times \Zh^-$;
\item
If $0<R_2<R_3<R_1$, the integral \eqref{6easy-scalar-3form} produces an
invariant $3$-form on $\Zh$ which descends to $\Zh^+ \times \Zh^- \times \Zh^0$;
\item
If $0<R_2<R_1<R_3$, the integral \eqref{6easy-scalar-3form} produces an
invariant $3$-form on $\Zh$ which descends to $\Zh^0 \times \Zh^- \times \Zh^+$.
\end{itemize}

Missing from this list are the two impossible combinations of the radii
$$
0<R_1<R_2<R_3<R_1 \qquad \text{and} \qquad R_1>R_2>R_3>R_1>0.
$$
In a future work we will construct two additional invariant $3$-forms
on $\Zh$ that descend to $\Zh^0 \times \Zh^0 \times \Zh^0$ and can be
regarded as counterparts of these impossible combinations of the radii.

\subsection{The Spinor Case ${\cal W}'$}

By analogy with \eqref{6easy-scalar-3form}, for $F_1, F_2, F_3 \in {\cal W}'$,
we can consider
\begin{multline}  \label{6easy-W'-3form}
\langle F_1, F_2, F_3 \rangle_{{\cal W}'} = \Bigl( \frac{i}{2\pi^3} \Bigr)^3
\iiint_{\begin{smallmatrix} W_1 \in U(2)_{R_1} \\ W_2 \in U(2)_{R_2} \\
W_3 \in U(2)_{R_3} \end{smallmatrix}}
\tr \biggl[ F_1(W_1) \cdot \frac{(W_1-W_2)^{-1}}{N(W_1-W_2)} \cdot F_2(W_2) \\
\times \frac{(W_2-W_3)^{-1}}{N(W_2-W_3)} \cdot F_3(W_3) \cdot
\frac{(W_3-W_1)^{-1}}{N(W_3-W_1)} \biggr] \,dV_{W_1}dV_{W_2}dV_{W_3}
\end{multline}
with different choices of radii $R_1$, $R_2$ and $R_3$.
By Proposition 64 in \cite{ATMP}, this is a
$\mathfrak{gl}(2,\HC)$-invariant $3$-linear form
$$
{\cal W}' \times {\cal W}' \times {\cal W}' \to \BB C.
$$

Recall that $\bar{\cal Q}'^0$ denotes the indecomposable subquotient of
${\cal W}'$ introduced in Remark \ref{Q-bar-remark}.
By Theorem 65 and Proposition 66 in \cite{ATMP},
\begin{itemize}
\item
If $0<R_3<R_1<R_2$, the integral \eqref{6easy-W'-3form} produces an
invariant $3$-form on ${\cal W}'$ which descends to
$\bar{\cal Q}'^0 \times {\cal Q}'^+ \times {\cal Q}'^-$;
\item
If $0<R_1<R_3<R_2$, the integral \eqref{6easy-W'-3form} produces an
invariant $3$-form on ${\cal W}'$ which descends to
${\cal Q}'^- \times {\cal Q}'^+ \times \bar{\cal Q}'^0$;
\item
If $0<R_1<R_2<R_3$, the integral \eqref{6easy-W'-3form} produces an
invariant $3$-form on ${\cal W}'$ which descends to
${\cal Q}'^- \times \bar{\cal Q}'^0 \times {\cal Q}'^+$;
\item
If $0<R_3<R_2<R_1$, the integral \eqref{6easy-W'-3form} produces an
invariant $3$-form on ${\cal W}'$ which descends to
${\cal Q}'^+ \times \bar{\cal Q}'^0 \times {\cal Q}'^-$;
\item
If $0<R_2<R_3<R_1$, the integral \eqref{6easy-W'-3form} produces an
invariant $3$-form on ${\cal W}'$ which descends to
${\cal Q}'^+ \times {\cal Q}'^- \times \bar{\cal Q}'^0$;
\item
If $0<R_2<R_1<R_3$, the integral \eqref{6easy-W'-3form} produces an
invariant $3$-form on ${\cal W}'$ which descends to
$\bar{\cal Q}'^0 \times {\cal Q}'^- \times {\cal Q}'^+$.
\end{itemize}
It is clear that these maps are non-trivial.

By Proposition 80 in \cite{FL1}, we have a $\mathfrak{gl}(2,\HC)$-invariant
non-degenerate bilinear pairing between $(\rho_2, {\cal W})$ and
$(\rho'_2, {\cal W}')$. Thus, the ``multiplication'' corresponding to
\eqref{6easy-W'-3form} and an analogue of \eqref{6easy-scalar} is a map
\begin{equation}  \label{6easy-Wprime}
F \otimes G \mapsto \Bigl( \frac{i}{2\pi^3} \Bigr)^2
\iint_{\genfrac{}{}{0pt}{}{W_1 \in U(2)_{R_1}}{W_2 \in U(2)_{R_2}}}
\frac{(Z-W_1)^{-1}}{N(Z-W_1)} \cdot F(W_1) \cdot
\frac{(W_1-W_2)^{-1}}{N(W_1-W_2)} \cdot G(W_2) \cdot
\frac{(W_2-Z)^{-1}}{N(W_2-Z)} \,dV_{W_1}dV_{W_2},
\end{equation}
$F, G \in {\cal W}'$, with different choices of radii $R_1$, $R_2$ and
regions containing $Z$.
By Proposition 64 in \cite{ATMP}, this is an equivariant map
$$
(\rho'_2 \otimes \rho'_2, {\cal W}' \otimes {\cal W}') \to (\rho_2, {\cal W}).
$$
If we want this map to have values in $(\rho'_2, {\cal W}')$,
we can compose it with an equivariant map
$(\rho_2, {\cal W}) \to (\rho'_2, {\cal W}')$, such as $\Xm^+$ or $\Xm^-$.

%By Lemma 8 in \cite{desitter} and Lemma 63, Theorem 65,
%Propositions 66, 89 in \cite{ATMP}, we have the following
%$\mathfrak{gl}(2,\HC)$-equivariant maps:
%\begin{itemize}
%\item
%If $0<R_1<R_2$, $Z \in \BB D^+_{R_1}$, the integral \eqref{6easy-W'}
%produces a map ${\cal W}' \otimes {\cal W}' \to {\cal W}$
%which descends to $\bar{\cal Q}'^0 \otimes {\cal Q}'^+ \to {\cal Q}^+$;
%\item
%If $0<R_1<R_2$, $Z \in \BB D^-_{R_1} \cap \BB D^+_{R_2}$, the integral
%\eqref{6easy-W'} produces a map
%${\cal W}' \otimes {\cal W}' \to {\cal W}$
%which descends to ${\cal Q}'^- \otimes {\cal Q}'^+ \to {\cal Q}^0$;
%\item
%If $0<R_1<R_2$, $Z \in \BB D^-_{R_2}$, the integral
%\eqref{6easy-W'} produces a map
%${\cal W}' \otimes {\cal W}' \to {\cal W}$
%which descends to ${\cal Q}'^- \otimes \bar{\cal Q}'^0 \to {\cal Q}^-$;
%\item
%If $0<R_2<R_1$, $Z \in \BB D^+_{R_2}$, the integral \eqref{6easy-W'}
%produces a map ${\cal W}' \otimes {\cal W}' \to {\cal W}$
%which descends to ${\cal Q}'^+ \otimes \bar{\cal Q}'^0 \to {\cal Q}^+$;
%\item
%If $0<R_2<R_1$, $Z \in \BB D^+_{R_1} \cap \BB D^-_{R_2}$, the integral
%\eqref{6easy-W'} produces a map
%${\cal W}' \otimes {\cal W}' \to {\cal W}$
%which descends to ${\cal Q}'^+ \otimes {\cal Q}'^- \to {\cal Q}^0$;
%\item
%If $0<R_2<R_1$, $Z \in \BB D^-_{R_1}$, the integral
%\eqref{6easy-W'} produces a map
%${\cal W}' \otimes {\cal W}' \to {\cal W}$
%which descends to $\bar{\cal Q}'^0 \otimes {\cal Q}'^- \to {\cal Q}^-$.
%\end{itemize}
%It is clear that these maps are non-trivial.

\subsection{The Spinor Case ${\cal W}$}

By analogy with \eqref{6easy-scalar-3form} and \eqref{6easy-W'-3form},
for $F_1, F_2, F_3 \in {\cal W}$,
we can consider
\begin{multline}  \label{6easy-W-3form}
\langle F_1, F_2, F_3 \rangle_{\cal W} = \Bigl( \frac{i}{2\pi^3} \Bigr)^3
\iiint_{\begin{smallmatrix} W_1 \in U(2)_{R_1} \\ W_2 \in U(2)_{R_2} \\
W_3 \in U(2)_{R_3} \end{smallmatrix}}
\tr \biggl[ F_1(W_1) \cdot \frac{W_1-W_2}{N(W_1-W_2)} \cdot F_2(W_2) \\
\times \frac{W_2-W_3}{N(W_2-W_3)} \cdot F_3(W_3) \cdot
\frac{W_3-W_1}{N(W_3-W_1)} \biggr] \,dV_{W_1}dV_{W_2}dV_{W_3}
\end{multline}
with different choices of radii $R_1$, $R_2$ and $R_3$.
By Proposition \ref{fork-prop}, this is a
$\mathfrak{gl}(2,\HC)$-invariant $3$-linear form
$$
{\cal W} \times {\cal W} \times {\cal W} \to \BB C.
$$

Recall that the kernels of $J_R^{'\pm\pm}$ are described in
Theorems \ref{J'-thm} and \ref{J+-_thm}. We have:
\begin{itemize}
\item
If $0<R_3<R_1<R_2$, the integral \eqref{6easy-W-3form} produces an
invariant $3$-form on ${\cal W}$ which descends to
$$
{\cal W}/\ker J_{R_1}^{'+-} \times {\cal W}/\ker J_{R_2}^{'++}
\times {\cal W}/\ker J_{R_3}^{'--};
$$
\item
If $0<R_1<R_3<R_2$, the integral \eqref{6easy-W-3form} produces an
invariant $3$-form on ${\cal W}$ which descends to
$$
{\cal W}/\ker J_{R_1}^{'--} \times {\cal W}/\ker J_{R_2}^{'++}
\times {\cal W}/\ker J_{R_3}^{'-+};
$$
\item
If $0<R_1<R_2<R_3$, the integral \eqref{6easy-W-3form} produces an
invariant $3$-form on ${\cal W}$ which descends to
$$
{\cal W}/\ker J_{R_1}^{'--} \times {\cal W}/\ker J_{R_2}^{'+-}
\times {\cal W}/\ker J_{R_3}^{'++};
$$
\item
If $0<R_3<R_2<R_1$, the integral \eqref{6easy-W-3form} produces an
invariant $3$-form on ${\cal W}$ which descends to
$$
{\cal W}/\ker J_{R_1}^{'++} \times {\cal W}/\ker J_{R_2}^{'-+}
\times {\cal W}/\ker J_{R_3}^{'--};
$$
\item
If $0<R_2<R_3<R_1$, the integral \eqref{6easy-W-3form} produces an
invariant $3$-form on ${\cal W}$ which descends to
$$
{\cal W}/\ker J_{R_1}^{'++} \times {\cal W}/\ker J_{R_2}^{'--}
\times {\cal W}/\ker J_{R_3}^{'+-};
$$
\item
If $0<R_2<R_1<R_3$, the integral \eqref{6easy-W-3form} produces an
invariant $3$-form on ${\cal W}$ which descends to
$$
{\cal W}/\ker J_{R_1}^{'-+} \times {\cal W}/\ker J_{R_2}^{'--}
\times {\cal W}/\ker J_{R_3}^{'++}.
$$
\end{itemize}
It is clear that these maps are non-trivial.

Since we have a $\mathfrak{gl}(2,\HC)$-invariant
non-degenerate bilinear pairing between $(\rho_2, {\cal W})$ and
$(\rho'_2, {\cal W}')$, the ``multiplication'' corresponding to
\eqref{6easy-W-3form} and an analogue of \eqref{6easy-scalar},
\eqref{6easy-Wprime} is a map
\begin{equation*}
F \otimes G \mapsto \Bigl( \frac{i}{2\pi^3} \Bigr)^2
\iint_{\genfrac{}{}{0pt}{}{W_1 \in U(2)_{R_1}}{W_2 \in U(2)_{R_2}}}
\frac{Z-W_1}{N(Z-W_1)} \cdot F(W_1) \cdot \frac{W_1-W_2}{N(W_1-W_2)} \cdot
G(W_2) \cdot \frac{W_2-Z}{N(W_2-Z)} \,dV_{W_1}dV_{W_2}
\end{equation*}
with different choices of radii $R_1$, $R_2$ and regions containing $Z$.
By Proposition \ref{fork-prop}, this is an equivariant map
$$
(\rho_2 \otimes \rho_2, {\cal W} \otimes {\cal W}) \to (\rho'_2, {\cal W}').
$$
If we want this map to have values in $(\rho_2, {\cal W})$,
we can compose it with an equivariant map
$(\rho'_2, {\cal W}') \to (\rho_2, {\cal W})$, such as $\M$.

\section{Factorization of Intertwining Operators via Clifford Algebras}  \label{Sect10}

The integral formulas \eqref{6easy-W'-3form}, \eqref{6easy-W-3form}
for the invariant $3$-linear forms
$$
{\cal W} \times {\cal W} \times {\cal W} \to \BB C \qquad \text{and} \qquad
{\cal W}' \times {\cal W}' \times {\cal W}' \to \BB C
$$
suggest their factorizations via Clifford algebras and their modules.
Such factorizations are well known in the setting of spinor representations
of loop algebras and their central extensions.
In this section we try to express an analogy between our new constructions
in four dimensions and the classical facts of the two-dimensional theory.

Let $V$ be a complex vector space (possibly of infinite dimension) with basis
$\{\beta_i\}_{i \in I}$ and let $V'$ be another copy of $V$ with basis
$\{\gamma_i\}_{i \in I}$. Define a symmetric bilinear form on $V \oplus V'$ by
$$
\langle\beta_i,\gamma_j\rangle = \langle\gamma_j,\beta_i\rangle
= -\tfrac12 \delta_{ij},
\qquad \langle\beta_i,\beta_j\rangle = \langle\gamma_i,\gamma_j\rangle = 0.
$$
Associated to this bilinear form is the (universal)  Clifford algebra
$Cl(V \oplus V')$. It is generated by $1$, $\{\beta_i\}_{i \in I}$ and
$\{\gamma_i\}_{i \in I}$ subject to the relations
$$
\{ \beta_i, \gamma_j \} = \delta_{ij}, \qquad
\{ \beta_i, \beta_j \} = \{ \gamma_i, \gamma_j \} = 0, \qquad
\text{for all $i, j \in I$},
$$
where $\{x,y\}=xy+yx$ is the anticommutator.

The exterior algebras $\Lambda V$ and $\Lambda V'$ provide examples of
modules for the Clifford algebra $Cl(V \oplus V')$.
However, in the theory of representations of loop algebras we also need to
consider more general Clifford modules obtained from "intermediate"
polarizations of $V \oplus V'$ as follows.
Consider a partition of $I$ -- the set indexing the bases of $V$ and $V'$ --
into two disjoint subsets $I_+$ and $I_-$, and let $V_+$ and $V_-$ be the spans
of the basis vectors $\{\beta_i\}_{i \in I_+}$ and $\{\beta_i\}_{i \in I_-}$
respectively. Similarly, let $V'_{\pm}$ be the spans of the basis vectors
$\{\gamma_i\}_{i \in I_{\pm}}$. We set
\begin{equation}  \label{H-def}
H = \Lambda(V_- \oplus V'_+).
\end{equation}
Then $H$ has a natural Clifford module structure.
Using the physics terminology, we call
\begin{align*}
&\{\beta_i\}_{i \in I_-} \quad \text{and} \quad \{\gamma_i\}_{i \in I_+} \quad
\text{creation operators},  \\
&\{\beta_i\}_{i \in I_+} \quad \text{and} \quad \{\gamma_i\}_{i \in I_-} \quad
\text{annihilation operators}.
\end{align*}
The creation operators act on $H$ by left multiplication,
and the annihilation operators act by contraction (or interior product)
and thus annul the vacuum vector $1 \in H$.
This defines an action of $V \oplus V'$ on $H$ that extends uniquely to
$Cl(V \oplus V')$ via the universal property.

Let $V$ be a certain space of functions.
In our first example we set
$$
V = V' = \BB C [z,z^{-1}]
$$
with a bilinear pairing between $V$ and $V'$
$$
\langle g, f \rangle = \frac1{2\pi i} \oint g(z) f(z) \,dz,
\qquad f \in V,\: g \in V'.
$$
Let $\{f_i(z)\}_{i \in I}$ be a basis of $V$ and $\{g_i(z)\}_{i \in I}$
be a dual basis of $V'$, i.e.
$$
\langle g_i, f_j \rangle = \delta_{ij}.
$$
In our first example, we can choose
$$
I = \BB Z, \qquad f_n(z) = z^n, \qquad g_n(z) = z^{-n-1}, \quad n \in I.
$$
Next, we consider a Clifford algebra $Cl(V \oplus V')$ with the ``abstract''
basis $\beta_i$ corresponding to $f_i(z)$ and $\gamma_i$ corresponding to
$g_i(Z)$, $i \in I$.
Let us fix a partition of $I$ into $I_+$ and $I_-$, so we have a decomposition
$$
V = V_+ \oplus V_-, \qquad V' = V'_+ \oplus V'_-.
$$
In our first example, we choose $I_+ = \BB Z_{\ge 0}$ and $I_-=\BB Z_{<0}$. Then
$$
V_+ = \BB C[z], \qquad V_- = z^{-1} \BB C[z^{-1}], \qquad
V'_+ = z^{-1} \BB C[z^{-1}], \qquad V'_- = \BB C[z],
$$
$$
H = \Lambda \bigl( z^{-1} \BB C[z^{-1}] \oplus z^{-1} \BB C[z^{-1}] \bigr).
$$

A fundamental concept in the quantum field theory is the notion of
generating functions, called quantum fields:
\begin{equation}  \label{quantum_fields-def}
\beta(z) = \sum_{i \in I} \beta_i g_i(z) \qquad \text{and} \qquad
\gamma(z) = \sum_{i \in I} \gamma_i f_i(z).
\end{equation}
In our first example, these are just
\begin{equation} \label{quantum_fields}
\beta(z) = \sum_{i \in \BB Z} \beta_i z^{-i-1} \qquad \text{and} \qquad
\gamma(z) = \sum_{i \in \BB Z} \gamma_i z^i.
\end{equation}
In the 2-dimensional quantum field theory, these fields form the so-called
$\beta\gamma$-system.
One can construct more complicated algebras by introducing the normal
ordering of polynomials of the quantum fields \eqref{quantum_fields}
and their derivatives.
The normal ordering $:\::$ of two elements is defined by
$$
:\beta_i\gamma_i: = \begin{cases} \beta_i\gamma_i, & i \in I_-, \\
  -\gamma_i\beta_i, & i \in I_+, \end{cases} \qquad
:\beta_i\gamma_j: = \beta_i\gamma_j = -\gamma_j\beta_i
\quad \text{if } i \ne j.
$$
In general, the normal ordering of a monomial amounts to moving the creation
operators to the left side of the product and the annihilation operators to
the right and changing the sign of the product with each transposition.
Here is a key calculation of the algebra of $\beta\gamma$-system involving
the normal ordering of quantum fields.

\begin{prop}
  Let $|z|>|w|$, then
  $$
  \beta(z)\gamma(w) = :\beta(z)\gamma(w): + \frac1{z-w} Id, \qquad
  \gamma(z)\beta(w) = :\gamma(z)\beta(w): + \frac1{z-w} Id.
  $$
\end{prop}

Now we consider an example of $V$ and $V'$ arising from the theory of
quaternionic regular functions, namely $V$ is the space of left
regular functions with basis $\{f_i(Z)\}_{i \in I}$ and $V'$ is the dual
space of right regular functions with basis $\{g_i(Z)\}_{i \in I}$.
Then we define the quantum fields $\beta(Z)$ and $\gamma(Z)$ as in
\eqref{quantum_fields-def}, and the normal ordering in this case yields

\begin{prop}  \label{Cauchy-Fueter-normal-ordering-prop}
  If $Z^{-1}W \in \BB D^+$, then
  $$
  \beta(Z)\gamma(W) = :\beta(Z)\gamma(W): + \frac{(Z-W)^{-1}}{N(Z-W)} Id, \qquad
  \gamma(Z)\beta(W) = :\gamma(Z)\beta(W): + \frac{(Z-W)^{-1}}{N(Z-W)} Id.
  $$
\end{prop}

This fact is well-known in the quantization theory of massless Dirac fields
in the context of the 4-dimensional conformal field theory.
From the point of view of quaternionic analysis, it follows from the matrix
coefficient expansion of the Cauchy-Fueter formula.

Our final example of quantum fields comes from replacing the spaces
of regular functions ${\cal V}$ and ${\cal V}'$ with similar spaces of
quasi anti regular functions ${\mathcal U}$ and ${\mathcal U}'$ that we
introduced in Subsection \ref{polynomial-subsection}.
Then the quantum fields $\beta(Z)$ and $\gamma(Z)$ can be constructed using
the basis functions from Propositions \ref{K-typebasis+_prop},
\ref{K-typebasis-_prop}.
And the reproducing kernel expansions
\eqref{1st-kernel-expansion}-\eqref{2nd-kernel-expansion} yield a
counterpart of Proposition \ref{Cauchy-Fueter-normal-ordering-prop}:

\begin{prop}
  If $Z^{-1}W \in \BB D^+$, then
  $$
  \beta(Z)\gamma(W) = :\beta(Z)\gamma(W): + \frac{Z-W}{N(Z-W)} Id, \qquad
  \gamma(Z)\beta(W) = :\gamma(Z)\beta(W): + \frac{Z-W}{N(Z-W)} Id.
  $$
\end{prop}

Note however an important difference between the regular and
quasi anti regular cases. In the case of regular functions, the spaces
${\cal V}$ and ${\cal V}'$ have positive definite unitary structures that are
invariant under the action of the conformal group $SU(2,2)$, which in turn
yield a positive definite unitary structure on the space $H$ defined by
equation \eqref{H-def}.
On the other hand, in the quasi anti regular case, the spaces
${\mathcal U}$ and ${\mathcal U}'$ have only pseudounitary structures
that are $SU(2,2)$-invariant, and, therefore, $H$ is no longer a
(pre-)Hilbert space. Such spaces, however, do appear in quantum field theory
and are expected to have important applications to physics.
%In particular, the restriction of ${\mathcal U}$, ${\mathcal U}'$ and $H$
%to the Poincar\'e subgroup $SO(3,1) \ltimes \BB R^{3,1}$ does have an
%invariant positive definite unitary structure.

As we mentioned at the beginning of this section, the quantum fields
$\beta(z)$ and $\gamma(z)$ can be used to construct representations of
loop algebras. For example, for
$F(z) \in \operatorname{Mat}_{n \times n} (\BB C[z,z^{-1}])$,
one can define \cite{Fr}
\begin{equation}  \label{pi(F)}
\pi(F) = \oint :\beta(z) F(z) \gamma(z):\, dz,
\end{equation}
where we take $\BB C^n$-valued fields $\beta$ and $\gamma$
($\beta$ being a row vector and $\gamma$ being a column vector).
The commutation relations can be easily found; they also follow from
the correlation functions, such as
\begin{equation}  \label{correlation-eqn}
  \bigl( 1, \pi(F)\pi(G)\pi(H) 1\bigr)
  = \iiint_{\begin{smallmatrix} |z_1|=R_1 \\ |z_2| = R_2 \\
    |z_3| =R_3 \end{smallmatrix}}
\tr \left[ F(z_1) \frac1{z_1-z_2} G(z_2) \frac1{z_2-z_3} H(z_3) \frac1{z_3-z_1}
  \right]\,dz_1dz_2dz_3,
\end{equation}
where $R_1>R_2>R_3$.
This integral can be easily computed and yields the commutation relations
of the loop algebra.
Another way to treat \eqref{correlation-eqn} is as an intertwining operator
in the triple product of the loop algebra
$L\mathfrak{g} = \operatorname{Mat}_{n \times n} (\BB C[z,z^{-1}])$
$$
L\mathfrak{g} \otimes L\mathfrak{g} \otimes L\mathfrak{g} \to \BB C.
$$
Clearly, one can generalize the correlation functions to
correlation functions of any number of copies of the loop algebra and obtain
complete information about its representation theory \cite{Fr}.

If we replace $F$ in equation \eqref{pi(F)} with an element of ${\cal W}'$
or ${\cal W}$ and let $\beta(Z)$, $\gamma(Z)$ be the quantum fields for
the spaces regular or quasi anti regular functions, we can factorize our
intertwining operators from Section \ref{IntertwiningOper-section} using
the currents, such as $\pi(F)$.
In fact, the right hand side of equation \eqref{correlation-eqn} yields
precisely the intertwining operators\footnote{Instead of ${\cal W}'$ and
  ${\cal W}$ we can also consider $\operatorname{Mat}_{n \times n} ({\cal W}')$
  and $\operatorname{Mat}_{n \times n} ({\cal W})$, as in the loop algebra case.}
\eqref{6easy-W'-3form} and \eqref{6easy-W-3form}.
%formulas for the intertwining operators with the kernels
%$\frac{(Z-W)^{-1}}{N(Z-W)}$ and $\frac{Z-W}{N(Z-W)}$!
Unfortunately, unlike the case of $L\mathfrak{g}$, neither ${\cal W}'$
nor ${\cal W}$ inherits an interesting structure of a Lie algebra.
Recall, however, that we have constructed only six elementary intertwining
operators, which we were able to reproduce as correlation functions of
currents. There are two more intertwining operators on the triple products
of ${\cal W}'$ and ${\cal W}$ that we did not study in this paper.
Thus, we are led to a natural question of whether these missing intertwining
operators on ${\cal W}'$ or ${\cal W}$ yield an interesting Lie algebra
structure with a profound representation theory.
This question will be subject of our future research.

\noindent
{\em Department of Mathematics, Yale University,
P.O. Box 208283, New Haven, CT 06520-8283}\\
{\em Department of Mathematics, Indiana University,
Rawles Hall, 831 East 3rd St, Bloomington, IN 47405}   

\end{document}